\documentclass[a4paper,11pt,openany]{amsbook}
\usepackage[T1]{fontenc}
\usepackage{newtxtext,newtxmath}
\usepackage{tikz}

\usepackage{hyperref}

\usepackage{imakeidx}
\makeindex
\makeindex[name=notation,title=Index of notation,columns=1] 

\renewcommand{\epsilon}{\varepsilon}
\renewcommand{\phi}{\varphi}

\newtheoremstyle{mytheorem}
{\baselineskip}{\baselineskip}{\itshape}{}{\scshape}{.}{.5em}{}

\newtheoremstyle{mydefinition}
{\baselineskip}{\baselineskip}{}{}{\scshape}{.}{.5em}{}        

\theoremstyle{mytheorem}
\newtheorem{theorem}{Theorem}
\numberwithin{theorem}{chapter}
\newtheorem{lemma}[theorem]{Lemma}
\newtheorem{corollary}[theorem]{Corollary}

\theoremstyle{mydefinition}
\newtheorem{definition}[theorem]{Definition}
\newtheorem{remark}[theorem]{Remark}

\numberwithin{equation}{chapter}
\numberwithin{section}{chapter}

\usepackage{etoolbox}
\makeatletter
\patchcmd{\@makechapterhead}{\bfseries}{\scshape}{}{}  
\patchcmd{\@makeschapterhead}{\bfseries}{\scshape}{}{}
\def\section{\@startsection{section}{1}\z@{.7\linespacing\@plus\linespacing}{.5\linespacing}{\normalfont\scshape\centering}}
 \def\subsection{\@startsection{subsection}{2}\z@{.5\linespacing\@plus.7\linespacing}{1ex}{\normalfont\itshape}}

\usepackage{xstring}
\renewcommand{\tocappendix}[3]{\indentlabel{\IfStrEq{#3}{Bibliography}{}{\IfStrEq{#3}{Index}{}{\IfStrEq{#3}{Index of notation}{}{#1}}}\@ifnotempty{#2}{ #2.\quad}}#3}
\makeatother

\makeatletter
\newcommand{\setref}[2]{\phantomsection
  #1\def\@currentlabel{\unexpanded{#1}}\label{#2}}
\makeatother

\title{Regularity and uniqueness result for generated Jacobian equation}
\author{Cale Rankin}

\begin{document}

\thispagestyle{empty}\clearpage{}{\begingroup \scshape \begin{center}

\begin{minipage}{35em}\begin{center}
            \vspace{60pt}
{\LARGE Regularity and uniqueness results for  \\ generated Jacobian equations \\}
\vspace{60pt}
{\Large Cale Rankin \\}
{\large December 2021}
\vspace{30pt}
\end{center}\end{minipage}

\vspace{350pt}
{\large    A thesis submitted for the degree of \\
    Doctor of Philosophy of \\
     The Australian National University  \\}
\end{center}
\endgroup} \pagebreak

\clearpage{}
\thispagestyle{empty}\clearpage{}\begin{minipage}{35em}\end{minipage}
\vspace{-500pt}

{\begingroup \begin{center}
 \vspace{2cm}
    
    The work in this thesis is my own except where otherwise indicated. \\
 \hspace{-84pt}  This thesis contains approximately 25,000 words.\\
    \vspace{1cm}
    \hspace{234pt}  Cale Rankin\\ 
  \end{center}
      \vspace{2cm} \textsc{Version Info} \\
      November 28, 2021: Corrections to errors/typos that were present when submitted for marking.\\
      December 14, 2021: Corrections based on examiner feedback.\\
      January 07, 2022: Expository comment \S 6.1, removed blank pages, arXiv v1
\endgroup \pagebreak}

\clearpage{}
\thispagestyle{empty}\clearpage{}\chapter*{Acknowledgements}

Thank you to all my friends. Thank you Anthony and Jaklyn for many enjoyable Friday nights. Chris and Feng it's been a joy completing our degrees together. There's no one I'd have rather shared an office with. 

I owe thanks to a number of academic staff. Professor Seick Kim helped organise my participation in The 9th Korea PDE School and Professor Jiakun Liu invited me to speak at the UOW PDE Seminar Series. Both experiences left me feeling mathematically-renewed. Professor Ben Andrews, Professor Shibing Chen, and Professor Xu-Jia Wang helpfully answered mathematical questions. I thank them for this.   John Urbas was free whenever I needed him for mathematical discussions. John I enjoyed and benefited from our meetings, thank you. Neil, working with you has been, and will remain, a high point of my academic life. Thank you for all you have taught me, both about mathematics and what it means to be a mathematician. 

Finally to my family, with all my love, thank you for your believing in me.

\clearpage{}

\pagestyle{plain}
\clearpage{}\chapter*{Abstract}

This is a thesis about generated Jacobian equations; our purpose is twofold. First, we provide an introduction to these equations, whilst, at the same time, collating some results scattered throughout the literature. The other goal is to present the author's own results on these equations. These results all concern solutions of the PDE
\begin{align}
  \label{eq:a:gje}\tag{1} \det DY(\cdot,u,Du) = \frac{f(\cdot)}{f^*(Y(\cdot,u,Du))} \text{ in } \Omega,
\end{align}
for a particular family of vector fields $Y:\mathbf{R}^n\times \mathbf{R}\times \mathbf{R}^n\rightarrow\mathbf{R}^n$ and $\Omega$ a domain in $\mathbf{R}^n$. Usually this PDE is paired with the second boundary value problem, which requires the image of the mapping $Y$ be prescribed. That is,
\begin{align}
  \label{eq:a:2bvp}\tag{2} Y(\cdot,u,Du)(\Omega) = \Omega^*,
\end{align}
for prescribed domains $\Omega,\Omega^* \subset \mathbf{R}^n$. We summarise our results as follows, though warn the word convexity must be understood in the generalised $g$-convexity sense. In all but the third point we assume the A3w condition. 
\begin{itemize}
\item In two dimensions when $f/f^* \geq \lambda > 0$ and $\Omega$ is convex, solutions are strictly convex. Also in two dimensions, when $f/f^* \leq \Lambda < \infty$ and $\Omega^*$ is convex, solutions are $C^{1}$.
\item When $0 < \lambda \leq f/f^* \leq \Lambda < \infty$, $\Omega^*$ is convex, and the generating function is defined on a convex domain containing $\overline{\Omega}$, solutions are strictly convex and in $C^{1}(\Omega)$. The same result holds provided $\Omega$ is uniformly convex and the  generating function is  defined on a domain containing $\overline{\Omega}$. 
\item If two $C^{1,1}(\Omega)$ solutions of \eqref{eq:a:gje} and \eqref{eq:a:2bvp} intersect, then they are the same solution. In addition, $C^{1,1}(\overline{\Omega})$ solutions of the Dirichlet problem associated with \eqref{eq:a:gje} are unique. 
\item Under appropriate uniform convexity hypothesis on $\Omega,\Omega^*$ and smoothness conditions on $f,f^*$, \textit{all} Aleksandrov solutions are globally smooth.
\item Solutions of the parabolic version of the generated Jacobian equation remain uniformly bounded, independent of time. 
\end{itemize}

We emphasise that the results of the second point have appeared in the works of Guillen and Kitagawa under stronger domain hypothesis. Our contribution is the weaker domain hypothesis which are essential for our applications to the global regularity. With regards to which, conditions for the \textit{existence} of globally smooth solutions are due to Jiang and Trudinger. Our contribution is showing that all weak (i.e.\ Aleksandrov) solutions are globally smooth under the same conditions.

\clearpage{}
\tableofcontents
\clearpage{}\chapter{Introduction}

Generated Jacobian equations are a relatively new family of partial differential equations. They model a variety of situations in which the underlying problem is how to move between two prescribed densities. As examples we list reflection and refraction problems in geometric optics, the optimal transport problem, and the Minkowski problem from geometry. All of these problems can be expressed in the language of generated Jacobian equations (GJEs), and their solutions written explicitly in terms of the solution to the corresponding GJE.

More precisely, GJEs are a generalisation of the Monge--Amp\`ere equation and the Monge--Amp\`ere type equations from optimal transport. The framework of generated Jacobian equations was introduced with the goal of extending, and generalising, the well developed framework and techniques from optimal transport to problems in geometric optics. In this regard much has been done --- yet much remains. The goal of this thesis is to continue these developments. We prove new regularity results and structure results akin to uniqueness. 

The equations we study are of the form
\begin{equation}
  \label{eq:i:gje}
   \det DY(\cdot,u,Du) = \psi(\cdot,u,Du) \text{ in }\Omega,
\end{equation}
where $Y:\mathbf{R}^n\times\mathbf{R}\times\mathbf{R}^n\rightarrow \mathbf{R}^n$, $\psi:\mathbf{R}^n\times\mathbf{R}\times\mathbf{R}^n \rightarrow \mathbf{R} $ and $\Omega$ is a domain in $\mathbf{R}^n$. Precise conditions on $Y$ and $\psi $are given in Chapter \ref{chap:g}. However, the prototypical example is $Y(\cdot,u,Du) = Du$  in which case \eqref{eq:i:gje} is the Monge--Amp\`ere equation. The other well known example is from optimal transport. In this setting $Y$ and $\psi$ are independent of $u$, so \eqref{eq:i:gje} takes the form
\[ \det DY(\cdot,Du) = \psi(\cdot,Du).\]
The boundary condition that arises in applications is the second boundary value problem. For this we prescribe the image of $\Omega$ under the mapping $x \mapsto Y(x,u(x),Du(x))$. That is, we require
\[ Y(\cdot,u,Du)(\Omega) = \Omega^*,\]
for $\Omega^*$ a prescribed domain.

The results in this thesis mirror those in the Monge--Amp\`ere and optimal transport setting. So we'll begin with an outline of the regularity theory in those cases. First, the Monge--Amp\`ere equation. This equation has a long history. It dates back (unsurprisingly) to the works of Monge \cite{Monge1784} and Amp\`ere \cite{Ampere1819}.  The regularity theory for Monge--Amp\`ere equations of the form
\begin{align*}
  \det D^2u &= \psi(\cdot,u,Du) \text{ in }\Omega, 
\end{align*}
is well developed in the elliptic setting, that is, when the solution $u:\Omega\rightarrow \mathbf{R}$ is a convex function. Convexity plays other roles in the regularity theory: to prove solutions are smooth convexity conditions on the domains are essential. 

In the elliptic setting the modern theory began with the work of Minkowski \cite{Minkowski03} and Aleksandrov \cite{Aleksandrov42,Aleksandrov42a}. They studied the Monge--Amp\`ere equation in conjunction with problems in geometry. In his early work Aleksandrov introduced (in two dimensions) a notion of weak solution and techniques for studying the uniqueness and $C^1$ regularity of such weak solutions. The analogues of all this play an important role in our study of GJEs. Aleksandrov's definition of weak solution was generalised to higher dimensions by both himself \cite{Aleksandrov58} and Bakelman \cite{Bakelman57}. However the interior regularity of these solutions is a difficult problem. It was not solved until Pogorelov \cite{Pogorelov71,Pogorelov71a,Pogorelov78}, and also Cheng and Yau \cite{ChengYau76,ChengYau77} deduced interior estimates that showed strictly convex solutions are smooth in the interior. The idea here, that strict convexity paves the way to higher regularity, plays an important role in this thesis. Indeed, strict convexity is tantamount to uniform ellipticity, and lets us treat the equation locally as a Dirichlet problem with affine boundary values. It was later that Caffarelli gave (very weak) conditions for solutions to be strictly convex \cite{Caffarelli}.

We are also interested in global regularity for the second boundary value problem. For the Monge--Amp\`ere equation this can be attacked in two completely different ways. One is the geometric theory of Caffarelli \cite{Caffarelli96}, the other, which we pursue here, is the method of continuity as used by Urbas \cite{Urbas1997}. Key to this is the derivation of a number of apriori estimates. Throughout all this the regularity theory is aided by uniqueness results: with appropriate conditions on $\psi$ (for example, no $u$ dependence), weak solutions of the Dirichlet problem are unique, and weak solutions of the second boundary value problem all differ by a constant. An important consequence is that the existence of a solution with certain regularity properties implies all solutions are equally regular. Heuristically, regularity for one implies regularity for all. 

Next we discuss the optimal transport case. The optimal transport problem, which is well known, originated with Monge in 1781 \cite{Monge1781}. Though at that point the connection to the Monge--Amp\`ere equation was unknown. Indeed, it was Kantorovich \cite{Kantorovitch1942,Kantorovich2004} who showed the optimal transport problem can be solved in terms of certain functions known as potentials. Later, Brenier \cite{Brenier1991} showed these are a weak solution of the Monge--Amp\`ere equation. Thus regularity of the optimal transport map could be studied via the associated Monge--Amp\`ere equation. However, at that time the regularity theory was only known in the special case where the cost of transport was the distance squared. It was known that for more general costs the optimal transport map could be written in terms of the solution to a more general Monge--Amp\`ere type equation. Yet the regularity theory for these equations was entirely unknown. Nevertheless, the early work of R\"{u}schendorf \cite{Ruschendorf91} and later work by Gangbo and McCann \cite{GangboMcCann96} showed there was an underlying convexity theory.  It was Ma, Trudinger, and Wang \cite{MTW05} who (based upon earlier works of Wang \cite{Wang96}) stated a condition on the cost of transport which allowed them to prove the interior regularity of optimal transport maps.

Immediately following this, a number of authors proved results that all contributed to the basic fact that Ma, Trudinger, and Wang's condition, known as A3 or A3w, was essential for the underlying convexity theory to inherit desirable, familiar, properties from standard convexity theory \cite{Loeper09,TrudingerWang09a,Villani09,KimMcCann10}. These include local conditions for testing the convexity of functions and domains. This  marked a relative explosion of regularity results for optimal transport maps. A number of results, all analogues of the corresponding results for Monge--Amp\`ere equations, appeared shortly thereafter: Trudinger and Wang proved global regularity for the second boundary value problem (following the basic approach of Urbas) \cite{TrudingerWang09}. Liu and Trudinger showed that, as in the Monge--Amp\`ere case, strict convexity implies interior regularity \cite{LiuTrudinger14}. A number of authors proved the Caffarelli style strict convexity under various hypothesis \cite{FKM13,GuillenKitagawa15,Vetois15,ChenWang16}. In all these results the convexity theory induced by the cost and the A3w condition play an essential role.

So techniques from the Monge--Amp\`ere equation had been used with extreme success to study optimal transport. It was natural to ask whether this same success could be recreated with applications to geometric optics --- here the equations have a similar, but more complicated, form. In his early work \cite{Trudinger14} Trudinger laid the framework for this extension. He introduced the convexity theory and proved regularity results in the same vein as his work with Ma and Wang. After which, he and Feida Jiang proved the existence of globally smooth solutions \cite{JiangTrudinger14,JiangTrudinger18}. However, unlike the optimal transport case this does not suffice to prove all weak solutions are globally smooth. Further work on the regularity was done by Guillen and Kitagawa \cite{GuillenKitagawa17} who proved the strict $g$-convexity and $C^{1,\alpha}$ regularity under the same hypothesis as their earlier work on optimal transport. The regularity theory has been further developed by Jhaveri \cite{Jhaveri17} and  Jeong \cite{Jeong21}. There is also work on numerical methods for approximating the solutions of GJEs \cite{AbedinGutierrez17,GMT21}. 

Our goal in this thesis is to continue the development of the regularity theory for generated Jacobian equations.  The author's results begin with strict convexity and $C^{1}$ regularity under weaker hypothesis on the domains than in the literature. We also prove particularly strong strict convexity and $C^1$ differentiability results in two dimensions (some well known counterexamples show these results hold only in two dimensions). Next, we consider adaptations of the uniqueness results for Monge--Amp\`ere equations, proving that if two smooth solutions of the second boundary value problem intersect, then they are the same solution. We also prove uniqueness for the Dirichlet problem. Finally, by a modification of Jiang and Trudinger's proof of global regularity in conjunction with our earlier results, we show all Aleksandrov solutions are globally smooth. To summarise, we develop the global regularity theory for Aleksandrov solutions of the second boundary value problem, building up along the way a number of other, necessary, results. Here is a more detailed outline of the thesis.

Chapters \ref{chap:g} and \ref{chap:w} are introductory and largely an exposition of the works of Trudinger \cite{Trudinger14,Trudinger20}. In Chapter \ref{chap:g} we introduce the underlying theory of $g$-convexity. We give details on the structure of the equations we consider and define generalisations of convex functions and convex sets. Moreover we emphasise that the well known A3w condition lies at the heart of this convexity theory: it ensures a number of essential properties. Then, in Chapter \ref{chap:w}, we give the definition of Aleksandrov solutions to generated Jacobian equations. We prove the existence of Aleksandrov solutions to the second boundary value problem that take a prescribed value at a given point in the domain. That is, solutions  satisfying $u(x_0) = u_0$ for an $x_0,u_0$ of our choosing. This is similar to being able to add a constant to any solution of the second boundary value problem for Monge--Amp\`ere equations. We prove that when these solutions are strictly $g$-convex they are smooth provided the data is smooth. 

The author's results begin in Chapter \ref{chap:sc}. Here we show  Aleksandrov solutions are strictly convex. This was achieved in the earlier work of Guillen and Kitagawa \cite{GuillenKitagawa17}. Our contributions are weaker domain hypothesis: we show strict convexity under the analogues of Chen and Wang's \cite{ChenWang16} hypothesis from optimal transport. Furthermore, we show that in two dimensions the situation is comparable to the Monge--Amp\`ere and optimal transport case --- much weaker hypothesis guarantee the strict convexity. The work in this chapter appears in \cite{Rankin2020b} and \cite{Rankin2021}.  We conclude the chapter with $C^1$ differentiability, treated as a consequence of the strict convexity. 

In Chapter \ref{cha:u} we consider the structure of the solution set, that is, results which are substitutes for uniqueness. The key tool is an Aleksandrov style lemma. This lemma states that on the boundary of the set where one solution lies above another, the gradients of the two solutions agree. In particular this implies the maximum of two Aleksandrov solutions is still an Aleksandrov solution. We combine this with classical tools from elliptic PDE (the Harnack inequality and Hopf's boundary point lemma) to obtain a full characterisation of the solution set in the case of smooth solutions: If two $C^{1,1}(\Omega)$ solutions of the second boundary value problem intersect then they are the same solution. Moreover solutions of the Dirichlet problem are unique provided they satisfy a particular extension property. The results for the second boundary value problem appeared in the author's work \cite{Rankin2020}. 

Our final result extends the global regularity result of Jiang and Trudinger. They proved the existence of globally smooth solutions. We modify their construction so as to prove the existence of a globally smooth solution which takes a given value at a point. Then, under the hypothesis for the existence of a globally smooth solution we have the following procedure to show Aleksandrov solutions are globally smooth. We begin by taking an arbitrary Aleksandrov solution. Under the given hypothesis our results in Chapter \ref{chap:sc}  imply this solution is strictly convex and, subsequently, smooth on the interior.  Now our construction implies the existence of a globally smooth solution intersecting the Aleksandrov solution. Since both solutions are smooth on the interior our uniqueness result for intersecting solutions imply they are the same solution. Thus the Aleksandrov solution is globally smooth. The construction of a globally smooth solution through a given point, which is via degree theory, uses a great deal of background material. Thus Chapters \ref{chap:r1} and \ref{chap:r2} are largely assembling results from the literature. Our contributions are in Section \ref{sec:glob-regul-via}. 

We conclude the thesis with some appendices. Along with the usual index of notation we include some proofs that by virtue of being generally well known, or similar to other parts of the thesis, were left out of the main text. We also, briefly, outline some work on parabolic generated Jacobian equations. Here we are able to show the natural parabolic analog of GJEs has a solution which is bounded independently of time. We then outline how this implies existence of $C^2$ solutions to the parabolic problem for all finite time intervals. 

\section{Conventions and notes to the reader}
\label{sec:conv-notes-read}

It will be helpful (if not essential) to know the basics of the Monge--Amp\`ere equation. Standard references are the books by Guti\'errez \cite{Gutierrez16} and Figalli \cite{Figalli17}. The author is also fond of the surveys by Trudinger and Wang \cite{TrudingerWang08} and Liu and Wang \cite{LiuWang15}. The material we need from elliptic PDE, along with most of our notation, is from \cite{GilbargTrudinger01}. Further introductory material for generated Jacobian equations can be found in the work of Guillen and Kitagawa \cite{GuillenKitagawa17} or Guillen's survey \cite{Guillen19}. 

\subsection*{Conventions}
\begin{enumerate}
  \item A \textit{domain} is an open connected subset of $\mathbf{R}^n$.
  \item When $E$ is a matrix $E_{ij}$ denotes the $ij^{th}$ element and $E^{ij}$ denotes the $ij^{th}$ element of $E^{-1}$.
\item For a symmetric matrix $A$, $A \geq 0$ means the eigenvalues of $A$ are nonnegative. For symmetric matrices $A,B$ we write $A \geq B$ provided $A-B \geq 0$. 
    
  \item We frequently deal with a function $g(x,y,z)$ where $(x,y,z) \in \mathbf{R}^n\times \mathbf{R}^n\times \mathbf{R}$. Derivatives with respect to $x$ are indicated as subscripts before a comma (or without a comma), derivatives with respect to $y$ are subscripts after a comma, and derivatives with respect to $z$ are indicated by a subscript $z$. As examples $D_{x_ix_jy_kz}g = g_{ij,k,z}$ and $D_{y_ky_lz}g = g_{,kl,z}$. On occasion to be explicit we subscript as in $D_{x_ix_jy_k}g = g_{x_ix_jy_k}$.
    \item The summation convention, that repeated indices are summed over, is employed throughout. 
\item Contrary to the literature when we use use the phrase generating function we are referring to a function that satisfies the conditions  A0, A1, A1$^*$, and A2 (see Chapter \ref{chap:g}). This differs from most of the literature where these conditions are explicitly stated.
\item Our definition of $g$-convex is stronger than standard. For those familiar with the literature what we call a $g$-convex function is Guillen and Kitagawa's \cite{GuillenKitagawa17} definition of ``very nice $g$-convex function''.
\end{enumerate}

\clearpage{}
\clearpage{}\chapter{Generating functions and $g$-convexity}
\label{chap:g}

Generated Jacobian equations are a generalisation of Monge--Amp\`ere equations, the study of which relies heavily on convexity theory. Thus, it's no surprise that the study of GJE depends on the introduction of a generalisation of convexity that is well adapted to these new equations. That's what we do in this chapter: We present the framework of $g$-convexity largely following  \cite{Trudinger14,Trudinger20}. Whilst this chapter is definition heavy, the payoff is powerful tools for generalising results from the study of the Monge--Amp\`ere equation. We'll see this in action in the coming chapters. 

\section{Generating functions}
\label{sec:generating-functions}
 We begin by introducing generating functions which we regard  as nonlinear extensions of affine planes.

   \begin{definition} \index{generating function}
  A \textit{generating function} is a function $g$ satisfying the conditions A0, A1, A1$^*$, and A2. 
\end{definition}

  \textbf{A0.} \index{A0}\index[notation]{$g$, $g(x,y,z)$, \ \ generating function} $g \in C^4(\overline{\Gamma})$ where $\Gamma \subset \mathbf{R}^n\times \mathbf{R}^n\times \mathbf{R}$\index[notation]{$\Gamma$, \ \ domain of definition for $g$} is a bounded open set whose points are denoted $(x,y,z)$ and which satisfies that for all $x,y\in \mathbf{R}^n$ the set
  \[ I_{x,y} := \{z ; (x,y,z) \in \Gamma\}, \]
  is a (possibly empty) open interval. Moreover we assume there are domains $U,V \subset \mathbf{R}^n$ and a nonempty open interval $J$\index[notation]{$J$, \ \  ``allowable'' $u$ values}\index[notation]{$U$, \ \ ``allowable'' $x$ values}\index[notation]{$V$, \ \ ``allowable'' $y$ values} such that when $x \in \overline{U}, y \in \overline{V}$ $J\subset g(x,y,I_{x,y})$.   \\
  \textbf{A1.}\index{A1} For all $(x,u,p)\in \mathcal{U}$,\index[notation]{$\mathcal{U}$, \ \  one-jet (with respect to $x$ derivative) of the generating function} which we define as
  \[\mathcal{U}:= \{(x,g(x,y,z),g_x(x,y,z)); \ (x,y,z) \in \Gamma\}, \]
  there is a unique $(x,y,z) \in \Gamma$ such that
  \[ g(x,y,z) = u \quad\text{ and }\quad g_x(x,y,z) = p.\]
  That is, for each fixed $x$ the mapping $(y,z) \mapsto (g(x,y,z),g_x(x,y,z))$ is injective.\\
  \textbf{A1$^*$.}\index{A1$^*$} For each $y,z$ the mapping
  \[ x \mapsto \frac{g_y}{g_z}(x,y,z), \]
  is injective on its domain of definition. Our reason for the notation A1$^*$ is explained in Section \ref{sec:g:convexity}.  \\
  \textbf{A2.}\index{A2} On $\overline{\Gamma}$ there holds $g_z < 0$ and the matrix $E$ with entries\footnote{Our subscript convention is explained in Section \ref{sec:conv-notes-read}.}\index[notation]{$E$,$E_{ij}$, \ \ nonsingular matrix defined by $E_{ij}:=g_{i,j}-g_z^{-1}g_{i,z}g_{,j}$}
  \[ E_{ij}:=g_{i,j}-g_z^{-1}g_{i,z}g_{,j}\]
  satisfies $\det E \neq 0$.

From this point on $g$ denotes a generating function. The A0 condition requires explanation. The condition $g \in C^4(\overline{\Gamma})$ ensures rudimentary estimates, for example $\sup|g_{xx}| \leq C$. This idea is from \cite{FKM13}.  Heuristically $U,V$ is the ``universe'' of allowable $x,y$ values and $J$ the set of allowable heights. The set $J$ is adapted from Guillen and Kitagawa's definition of uniform admissibility \cite[Definition 4.1]{GuillenKitagawa17}. Indeed, $J$ represents the set of allowable heights, which we may hit by an appropriate choice of the arguments of $g$. We will introduce further conditions on $g$ later, these are A3w (page \pageref{ref:a3w}), A4w (page \pageref{ref:a4w}), and A5 (page \pageref{ref:a5}). 

We use condition A1 to define mappings $Y:\mathcal{U} \rightarrow \mathbf{R}^n$ and $Z:\mathcal{U} \rightarrow \mathbf{R}$\index[notation]{$Y$, $Y(x,u,p)$, \ \ $Y$-mapping}\index[notation]{$Z$, $Z(x,u,p)$, \ \ $Z$-mapping}  by requiring they solve
\begin{align}
  \label{eq:g:yzdef1}  g(x,Y(x,u,p),Z(x,u,p)) &= u,\\
  \label{eq:g:yzdef2} g_x(x,Y(x,u,p),Z(x,u,p)) &= p.
\end{align}
We assume if $(x,u,p_0),(x,u,p_1) \in \mathcal{U}$ then  $(x,u,tp_1+(1-t)p_0)) \in \mathcal{U}$ for $t \in [0,1]$. This condition serves the same purpose as $g^*$-subconvexity \cite{Trudinger20}. 

\begin{definition}\index{generated Jacobian equation}\index{GJE}
  The partial differential equation
  \begin{align}
         \det DY(\cdot,u,Du) = \psi(\cdot,u,Du), 
  \end{align}
  is called a \textit{generated Jacobian equation} (GJE) provided the mapping $Y$ derives from solving \eqref{eq:g:yzdef1} and \eqref{eq:g:yzdef2} for some generating function. Here $\psi$ is a real valued function defined on a subset of $\mathbf{R}^n\times\mathbf{R}\times\mathbf{R}^n$.
\end{definition}

For applications GJEs are paired with the second boundary value problem\index{second boundary value problem}. That is, the condition
\begin{equation}
  \label{eq:g:2bvp}
  Y(\cdot,u,Du)(\Omega) = \Omega^*. \tag{2BVP}
\end{equation}
We let $\Omega,\Omega^*$ denote domains in $\mathbf{R}^n$ satisfying $\Omega \subset U$ and $\Omega^* \subset V$. When we study the second boundary value problem $\psi$ has the form
\begin{equation}
  \label{eq:g:separable}
  \psi(\cdot,u,Du) = \frac{f(\cdot)}{f^*(Y(\cdot,u,Du))}
\end{equation}
where $f,f^*$ are nonnegative functions on $\Omega,\Omega^*$. The resulting equation is 
  \begin{equation}
    \label{eq:g:gje}
       \det DY(\cdot,u,Du) = \frac{f(\cdot)}{f^*(Y(\cdot,u,Du))}. \tag{GJE}
  \end{equation}
If $\psi$ has the form \eqref{eq:g:separable} and $x \mapsto Y(x,u,Du)$ is a diffeomorphism, the change of variables formula implies a necessary condition for solvability of \eqref{eq:g:gje} subject to \eqref{eq:g:2bvp} is the mass balance\index{mass balance} condition
\begin{equation}
  \label{eq:g:mass_balance}
   \int_\Omega f = \int_{\Omega^*}f^*. 
\end{equation}
We always assume \eqref{eq:g:mass_balance} holds.  We occasionally study the Dirichlet problem. However this is primarily as a means of locally approximating solutions of \eqref{eq:g:gje} subject to \eqref{eq:g:2bvp}.

For a function $u:\Omega \rightarrow \mathbf{R}$ to solve a generated Jacobian equation in $\Omega$ it is necessary that $(x,u(x),Du(x)) \in \mathcal{U}$ for each $x \in \Omega$. We say more about the class of solutions to GJEs in Section \ref{sec:g:convexity}. For now we show condition A2 allows us to rewrite  \eqref{eq:g:gje} as a Monge--Amp\`ere type equation. 

To begin we assume $u \in C^2(\Omega)$ solves \eqref{eq:g:gje} on $\Omega$. Evaluating \eqref{eq:g:yzdef1} and \eqref{eq:g:yzdef2} at $u=u(x),p=Du(x)$ we have
\begin{align}
  \label{eq:g:matederive1}  g(x,Y(x,u,Du),Z(x,u,Du)) &= u,\\
  \label{eq:g:matederive2} g_{i}(x,Y(x,u,Du),Z(x,u,Du)) &= D_iu.
\end{align}
Differentiating with respect to $x_j$ we obtain
\begin{align}
  \label{eq:g:matederive3} g_{j}+g_{,k}D_jY^k+g_zD_jZ &= D_ju\\
  \label{eq:g:matederive4} g_{ij}+g_{i,k}D_jY^k+g_{i,z}D_jZ &= D_{ij}u.
\end{align}
Where the arguments of $g,Y,Z$ are unchanged from \eqref{eq:g:matederive1} and \eqref{eq:g:matederive2}. 
Comparing \eqref{eq:g:matederive2} and \eqref{eq:g:matederive3} we see
\begin{equation}
  \label{eq:g:dz-sub}
  g_{,k}D_jY^k+g_zD_jZ = 0. 
\end{equation}
Using this to eliminate $D_jZ$ from \eqref{eq:g:matederive4} yields
\[ g_{ij}+g_{i,k}D_jY^k-\frac{g_{i,z}g_{,k}}{g_z}D_jY^k = D_{ij}u.\]
Recalling the definition of $E$ from condition A2 we obtain
\begin{equation}
DY(x,u,Du) = E^{-1}[D^2u-g_{xx}(x,Y(x,u,Du),Z(x,u,Du))],\label{eq:g:premate}
\end{equation}
and thus $u$ satisfies the Monge--Amp\`ere type equation \index{Monge--Amp\`ere type equation}
\begin{equation}
  \label{eq:g:mate}
  \tag{MATE}
  \det[D^2u-A(\cdot,u,Du)] = B(\cdot,u,Du),
\end{equation}
for \index[notation]{$A$, $A(x,u,p)$, $A_{ij}(x,u,p)$, \ \ augmenting matrix for MATE} \index[notation]{$B$, $B(x,u,p)$, \ \ right-hand side for MATE}
\begin{align}
  \label{eq:g:Adef} A(x,u,Du) &= g_{xx}(x,Y(x,u,Du),Z(x,u,Du)),\\
\label{eq:g:Bdef}  B(x,u,Du) &= \det E \ \psi(x,u,Du).
\end{align}

The two distinct ways of writing the equation, \eqref{eq:g:gje} and \eqref{eq:g:mate}, correspond to two distinct families of techniques for studying the equation. Indeed we use \eqref{eq:g:gje} mainly with techniques from convexity theory. On the other hand \eqref{eq:g:mate} is the correct form to apply tools from elliptic PDE, though of course this requires the equation be elliptic. So we study \eqref{eq:g:mate} only when it is elliptic, that is,  when the matrix
\begin{equation}
  \label{eq:g:elliptic}
   D^2u - g_{xx}(x,Y(x,u,Du),Z(x,u,Du))
\end{equation}
is positive definite (see \cite[Ch. 17]{GilbargTrudinger01} for the definition of ellipticity for nonlinear PDE). This condition is a generalised notion of convexity in the sense we now make clear. 

\section{Convexity theory}
\label{sec:g:convexity}

\subsection*{$g$-convex functions}\index{$g$-convex function}

The prototypical example of a generating function is $g(x,y,z) = x\cdot y - z$. In this case $Y(\cdot,u,Du) = Du$. The resulting GJE is the Monge--Amp\`ere equation
\[\det D^2u = \psi(\cdot,u,Du). \]
The ellipticity condition \eqref{eq:g:elliptic} reduces to the positivity of $D^2u,$ which implies $u$ is locally convex. It is well known that finite convex functions have a supporting hyperplane at each point on the interior of their domain. That is, if $u:\Omega \rightarrow \mathbf{R}$ is convex then for each $x_0 \in \Omega$ there is $y_0\in \mathbf{R}^n$ and $z_0 \in \mathbf{R}$ such that
\begin{align}
  \label{eq:g:convdef1} u(x_0) &= x_0\cdot y_0 - z_0,\\
  \label{eq:g:convdef2} u(x) &\geq x\cdot y_0 - z_0 \text{ for any }x\in\Omega.
\end{align}
Recognising that, in this case, the right hand side is the generating function we generalise as follows.
\begin{definition}
  A function $u:\Omega \rightarrow \mathbf{R}$ is called $g$\textit{-convex} provided for all $x_0 \in \Omega$ there exists $y_0 \in V$ and $z_0 \in \cap_{x \in \Omega}I_{x,y_0}$ such that
\begin{align}
  \label{eq:g:gconvdef1} u(x_0) &= g(x_0,y_0,z_0),\\
  \label{eq:g:gconvdef2} u(x) &\geq g(x,y_0,z_0) \text{ for any }x\in \Omega,
\end{align}
and  whenever \eqref{eq:g:gconvdef1} and \eqref{eq:g:gconvdef2} are satisfied $g(\overline{\Omega},y_0,z_0)\subset J$.

Furthermore if the inequality in \eqref{eq:g:gconvdef2} is strict for $x \neq x_0$, then $u$ is called strictly $g$-convex. \index{strict $g$-convexity}
\end{definition}
\noindent We call $g(\cdot,y_0,z_0)$ a $g$-support at $x_0$ provided \eqref{eq:g:gconvdef1}, \eqref{eq:g:gconvdef2} hold and $g(\overline{\Omega},y_0,z_0)\subset J$. More generally we call functions of the form $x \mapsto g(x,y,z)$ $g$-affine.  The containment condition $g(\overline{\Omega},y_0,z_0)\subset J$ is due to Guillen and Kitagawa \cite{GuillenKitagawa17}, and functions satisfying it are referred to as ``very nice''.  It ensures any two supports of $u$ can be vertically shifted to pass through the same point; in practice this is a useful tool. 

If $u \in C^1(\Omega)$ is $g$-convex and $g(\cdot,y_0,z_0)$ is a $g$-support at $x_0$ then  $x \mapsto u(x)-g(x,y_0,z_0)$ has a minimum at $x_0$.  Subsequently,
\[ Du(x_0) = g_x(x_0,y_0,z_0).\]
Since this equation along with \eqref{eq:g:gconvdef1} is equivalent to \eqref{eq:g:yzdef1} and \eqref{eq:g:yzdef2} we must have $y_0 = Y(x_0,u(x_0),Du(x_0))$ and $z_0 = Z(x_0,u(x_0),Du(x_0))$. If, furthermore, $u \in C^2(\Omega)$ then, again because there is a minimum at $x_0$, the matrix
\[ D^2u(x_0) - g_{xx}[x_0,Y(x_0,u(x_0),Du(x_0)),Z(x,u(x_0),Du(x_0))]  \]
is nonnegative definite. Thus \eqref{eq:g:mate} is degenerate elliptic whenever $u$ is a $C^2$ $g$-convex function.

When $u$ is $g$-convex but not differentiable at $x_0$ then \eqref{eq:g:gconvdef1} and \eqref{eq:g:gconvdef2} hold for more than one $y_0$. We use the notation \index[notation]{$Yu$, \ \ $Y$-mapping for $g$-convex function $u$}\index[notation]{$Zu$, \ \ $Z$-mapping for $g$-convex function $u$}
\begin{align}
  \label{eq:g:yudef}
  Yu(x_0) &:= \{y_0 \in V; \text{ there exists } z_0 \in \mathbf{ R}\text{ such that \eqref{eq:g:gconvdef1} and \eqref{eq:g:gconvdef2} hold}\},\\
  \label{eq:g:yudef} Zu(x_0) &:= \{z_0 \in \mathbf{R}; \text{ there exists } y_0 \in V\text{ such that \eqref{eq:g:gconvdef1} and \eqref{eq:g:gconvdef2} hold}\}.
\end{align}
When these sets are singletons we identify them with their single element. The mapping $Yu$ is referred to as the $g$-subgradient, $g$-normal mapping or $g$-exponential mapping (or just the $Y$ mapping). Finally, when $x \in \partial \Omega$,  $Yu(x)$ is defined as the set of all $y = \lim y_k$ where $y_k \in Yu(x_k)$ for $\{x_k\}_{k=1}^\infty$ a sequence in $\Omega$ with limit $x$. 
\begin{remark} \label{rem:g:semiconvex}\index{semiconvexity}
We note any $g$-convex function, $u$, is semiconvex. That is, $u+C|x|^2$ is convex for some $C$ and, in fact, $C$ may be chosen independently of $u$. To see this note
\[g_C(x,y,z):=g(x,y,z)+C|x|^2\]
is convex for $C = \sup_{\overline{\Gamma}}|g_{xx}|$. Thus if $u$ is $g$-convex, $x_0$ is given, and $g(\cdot,y_0,z_0)$ is a support at $x_0$, then by the convexity of $g_C$
\begin{align*}
  u(x)+C|x|^2 &\geq g(x,y_0,z_0)+C|x|^2\\
  &\geq Dg_C(x_0,y_0,z_0)\cdot(x-x_0)+g_C(x_0,y_0,z_0).
\end{align*}
Thus $u+C|x|^2$ has a supporting plane at every $x_0 \in \Omega$ and so is convex. Semiconvexity implies local Lipschitz continuity, differentiability almost everywhere, and the existence of second derivatives almost everywhere \cite[\S 6.3,6.4]{EvansGariepy15}.
\end{remark}

\subsection*{The dual generating function}\index{dual generating function}\index[notation]{$g^*$, $g^*(x,y,u)$, \ \ dual generating function}

If $g(x,y,z)=x\cdot y -z$, $u$ is a convex function, $x_0$ is given and $y_0 \in Yu(x_0)$ then the corresponding support is
\[ x \mapsto x\cdot y_0 - (x_0 \cdot y_0 - u(x_0)). \]
That is, the $z_0$ for which $g(\cdot,y_0,z_0)$ is a support is found by solving $u(x_0) = g(x_0,y_0,z_0)$. Now we introduce a function $g^*$ that gives the $z_0$ solving $g(x_0,y_0,z_0) = u_0$. 

\begin{definition}
  Given a generating function $g$, the \textit{dual generating function}, denoted $g^*$, is the unique function defined on
  \[ \Gamma^*:= \{ (x,y,g(x,y,z)); (x,y,z) \in \Gamma\},\]
  by the requirement it satisfy
  \begin{equation}
    \label{eq:g:gstardef1}
       g^*(x,y,g(x,y,z)) = z.
  \end{equation}
\end{definition}

\begin{remark}\label{rem:g:support}
When $u$ is $g$-convex and $y_0 \in Yu(x_0)$ there is a support of the form $g(\cdot,y_0,z_0)$. Since the support satisfies $g(x_0,y_0,z_0) = u(x_0)$ we must have $z_0 = g^*(x_0,y_0,u(x_0))$. Thus the support at $x_0$ with $g$-subgradient $y_0$ is
\[ g(\cdot,y_0,g^*(x_0,y_0,u(x_0))).\]
\end{remark}

Note $g^*$ is well defined by A2 and more explicitly $g^*(x,y,\cdot) = g(x,y,\cdot)^{-1}$. Denoting points in $\Gamma^*$ as $(x,y,u)$, we see by applying $g(x,y,\cdot)$ to \eqref{eq:g:gstardef1} that 
\begin{equation}
  \label{eq:g:gstardef2}
g(x,y,g^*(x,y,u)) = u,
\end{equation}
whenever $u=g(x,y,z)$ for $z \in I_{x,y}$.
We now explain why the A1$^*$ condition was called such. We claim it is equivalent to the following condition on $g^*$. \index{A1$^*$}\index[notation]{$\mathcal{U}^*$, \ \ one-jet (with respect to $y$ derivative) of the dual generating function}\\
\textbf{A1$^*$.}\textit{ (equivalent form)} For each $(y,z,q)$ in the set
\[ \mathcal{U}^* = \{(y,g^*(x,y,u),g_{y}^*(x,y,u)) ; (x,y,u) \in \Gamma^*\},\]
there exists a unique $(x,y,u)$ such that
\[  g^*(x,y,u) = z \quad\text{ and }\quad g_y^*(x,y,u) = q.\]

As before, we assume whenever $(y,z,q_0),(y,z,q_1) \in \mathcal{U}^*$ then for $t \in [0,1]$ we have $(y,z,tq_1+(1-t)q_0)) \in \mathcal{U}^*$.

That the above is equivalent to our earlier statement of A1$^*$ can be seen by differentiating \eqref{eq:g:gstardef2} and obtaining the identities
\begin{align}
  &g_y^* = \frac{-g_y}{g_z},  &&g^*_x =\frac{-g_x}{g_z}, &g^*_u = \frac{1}{g_z}, \label{eq:g:deriv-equiv}
\end{align}
where $g/g^*$ terms are evaluated at  $(x,y,g^*(x,y,u))/(x,y,u)$ respectively.

Analogously to the mapping $Y(x,u,p)$ we define the mappings $X(y,z,q)$. At this point we hope the following seems natural. \index[notation]{$X$, $X(y,z,q)$, \ \ $X$-mapping}

\begin{definition}\index{$g^*$-convex}
  A function $v:\Omega^*\rightarrow \mathbf{R}$ is called $g^*$\textit{-convex} provided for each $y_0 \in \Omega^*$ there exists $x_0 \in U$ and $u_0 \in J$ such that
  \begin{align}
  \label{eq:g:gstarconvdef1} v(y_0) &= g^*(x_0,y_0,u_0),\\
  \label{eq:g:gstarconvdef2} v(y) &\geq g^*(x_0,y,u_0) \text{ for any }y \in \Omega^*.
\end{align}
\end{definition}
As for $g$-convex functions, when $v$ is differentiable and $g^*(x_0,\cdot,u_0)$ is a support at $y_0$ then $x_0 = X(y_0,v(y_0),Dv(y_0))$. However when $v$ is not differentiable we denote the set of all such $x_0$ by $Xv(y_0)$. \index[notation]{$Xv$, \ \ $X$-mapping for $g^*$-convex function $v$}

\subsection*{The $g$ and $g^*$ transforms}
The functions $g^*$ and $g$ induce mappings between functions in the $x$ and $y$ variables that generalise the Legendre--Fenchel transformation. 

\begin{definition}\index{$g^*$-transform}\index{$g$-transform}
  Let $u:\Omega \rightarrow \mathbf{R}$ be $g$-convex. The $g^*$-transform of $u$ is the function $v:V\rightarrow\mathbf{R}$ defined by
  \begin{equation}
    \label{eq:g:g-trans-def}
       v(y) = \sup_{x \in \Omega}g^*(x,y,u(x)).
  \end{equation}
  Similarly if $v:\Omega^* \rightarrow \mathbf{R}$ is $g^*$-convex its $g$-transform is the function $u$ defined on $\Omega$ by
  \[ u(x) = \sup_{y \in \Omega^*}g(x,y,v(y)).\]
\end{definition}

The $g$ and $g^*$ transform serve a number of roles in the study of GJEs. One example is that they map between our original problem, and the dual problem. This is encapsulated in the following lemma which we use in Appendix \ref{cha:bound-parab-gener} where we give its proof. 

\begin{lemma}\label{lem:g:smooth_duality}
  Assume $u \in C^2(\Omega)$ is a $g$-convex solution of
  \begin{align*}
    \det DYu = \frac{f(\cdot)}{f^*(Yu(\cdot))},\\
    Yu(\Omega) = \Omega^*,
  \end{align*}
  for positive $f,f^*$. Let $v$ denote the $g^*$-transform of $u$. Then $v|_{\Omega^*}$  is a $C^2(\Omega^*)$ solution of
  \begin{align*}
    \det DXv = \frac{f^*(\cdot)}{f(Xv(\cdot))},\\
    Xv(\Omega^*) = \Omega.
  \end{align*}
\end{lemma}

When $u$ is merely $g$-convex we instead obtain the following.
\begin{lemma}\label{lem:g:gstarsubdiff}
  Let $u:\Omega\rightarrow \mathbf{R}$ be $g$-convex and $v$ its $g^*$-transform. If $y_0 \in Yu(x_0) $ then  $x_0 \in Xv(y_0)$ and
  \begin{equation}
    \label{eq:g:point-def}
       v(y_0) = g^*(x_0,y_0,u(x_0)).
  \end{equation}
\end{lemma}
\begin{proof}
  Since $y_0 \in Yu(x_0)$ we have for every $x \in \Omega$
  \[ u(x) \geq g(x,y_0,g^*(x_0,y_0,u(x_0))).\]
Applying  $g(x,y_0,\cdot)$ to both sides (recall $g_z < 0$), we obtain
\[ g(x,y_0,u(x)) \leq g^*(x_0,y_0,u(x_0)).\]
Thus the supremum defining $v(y_0)$ is obtained at $x=x_0$ and \eqref{eq:g:point-def} holds.  On the other for any $y \in V$ the  supremum defining $v(y)$ is taken over a set containing $g^*(x_0,\cdot,u(x_0))$ so  $v(y) \geq g^*(x_0,y,u(x_0))$ and $x_0 \in Xv(y_0)$. 
\end{proof}

A consequence, which we'll frequently employ, is the following.

\begin{lemma}\label{lem:g:sharedy}\index[notation]{$\mathcal{Z}$, \ \ set of points in $Yu(x)$ for more than one $x$}
  Let $u:\Omega\rightarrow\mathbf{R}$ be $g$-convex. The set
  \[ \mathcal{Z}:= \{ y \in V; y \in Yu(x_0) \cap Yu(x_1) \text{ for distinct }x_0,x_1 \in \Omega\}\]
  has Lebesgue measure 0. 
\end{lemma}
\begin{proof}
  Take $y \in \mathcal{Z}$ and let $v$ denote the $g^*$-transform of $u$. Necessarily there is distinct $x_0,x_1$ with $y \in Yu(x_0) \cap Yu(x_1)$. Our previous lemma implies $x_0,x_1 \in Xv(y)$. This means $v$ is not differentiable at $y$. Thus $\mathcal{Z}$ is a subset of the nondifferentiable points for $v$. By semiconvexity this latter set has measure 0.
\end{proof}

Further duality properties are considered in Section \ref{sec:w:weak_exist}.

\subsection*{Domain convexity}
We conclude our presentation of $g$-convexity by generalising line segments and subsequently obtaining a generalisation of convex sets.

\begin{definition}\index{$g$-segment}\index[notation]{$x_\theta$, $\{x_\theta\}_{\theta \in [0,1]}$ \ \ a $g$-segment}
  Let $[a,b]$ be a nonempty interval. A set of points $\{x_\theta\}_{\theta \in [a,b]} \subset U$ is called a $g$-segment with respect to $y_0,z_0$ provided that either/both of the equivalent expressions
  \begin{align*}
    g_y^*(x_\theta,y_0,g(x_\theta,y_0,z_0))\\
    \text{ and } -\frac{g_y}{g_z}(x_\theta,y_0,z_0)
  \end{align*}
  are line segments in $\theta$.
\end{definition}
If we are given points $x_a,x_b$ then the $g$-segment joining $x_a$ to $x_b$ parametrised by $[a,b]$ is the set  $\{x_\theta\}_{\theta \in [a,b]}$ obtained by solving
\[\frac{g_y}{g_z}(x_\theta,y_0,z_0) = \frac{\theta-b}{a-b}\frac{g_y}{g_z}(x_a,y_0,z_0)+\frac{a-\theta}{a-b}\frac{g_y}{g_z}(x_b,y_0,z_0). \]
Let $y_0,z_0$ and two points $x_0,x_1$ be given. There is, up to reparametrisation, at most one $g$-segment with respect to $y_0,z_0$ joining $x_0$ to $x_1$. This follows from A1$^*$ and the corresponding uniqueness of the line segment joining $x_0$ to $x_1$. 

\begin{definition}\index{$g$-convex}
  A set $A \subset U$ is called $g$-convex with respect to $y_0,z_0$ provided for every $x_0,x_1 \in A$ the $g$-segment joining $x_0$ to $x_1$ with respect to $y_0,z_0$ lies in $A$. 
\end{definition}

An equivalent definition is that $g_y^*(\cdot,y_0,g(\cdot,y_0,z_0))(A)$ is convex. We have the following dual notions of $g$-segments and $g$-convexity.

\begin{definition}\index{$g^*$-segment}\index[notation]{$y_\theta$, $\{y_\theta\}_{\theta \in [0,1]}$, \ \ a $g^*$-segment}
  A set of points $\{y_\theta\}_{\theta\in[a,b]} \subset V$ is called a $g^*$-segment with respect to $x_0,u_0$ provided that the either/both of the equivalent expressions
    \begin{align*}
    g_x(x_0,y_\theta,g^*(x_0,y_\theta,u_0))\\
    \text{ and } -\frac{g^*_x}{g^*_u}(x_0,y_\theta,u_0)
  \end{align*}
  are line segments in $\theta$. 
\end{definition}
\begin{definition}\index{$g^*$-convex}
  A set $B\subset V$ is called $g^*$-convex with respect to $x_0,u_0$ provided that for every $y_0,y_1 \in B$, the $g^*$-segment joining $y_0$ to $y_1$ with respect to $x_0,u_0$ lies in $B$.
  
\end{definition}

The statement of a number of theorems are streamlined by the following definitions.

\begin{definition} \label{def:g:g-conv-wrt-fung}
  A domain $B \subset V$ is called $g^*$-convex with respect to a $g$-convex function $u:\Omega \rightarrow \mathbf{R}$ provided $B$ is $g$-convex with respect to each $x,u(x)$ for each $x \in \Omega$.

  A domain $A \subset U$ is called $g$-convex with respect to $u:\Omega \rightarrow \mathbf{R}$ provided $A$ is $g$-convex with respect to each $y,z$ for $y \in Yu(x)$, $z = g^*(x,y,u(x))$ and $x \in \Omega$.
\end{definition}

Finally we define uniform $g$ and $g^*$-convexity.  Recall a $C^2$ domain is uniformly convex if it is convex and its boundary curvatures are bounded below by a positive constant.
\begin{definition}\label{def:g:g-conv-unif}\index{uniform $g$-convexity of domains} \index{uniform $g^*$-convexity of domains}
  A domain $A \subset U$ is called uniformly $g$-convex with respect to $y,z$ provided $\frac{g_y}{g_z}(A,y,z)$ is uniformly $g$-convex. Similarly $A$ is uniformly $g$-convex with respect to a $g$-convex function $u:\Omega \rightarrow \mathbf{R}$ provided $A$ is uniformly $g$-convex with respect to each $y,z$ for $y \in Yu(x)$, $z = g^*(x,y,u(x))$ and $x \in \overline{\Omega}$. 
\end{definition}
Uniform $g^*$-convexity and uniform $g^*$-convexity with respect to $u$ are defined in the expected way based on Definitions \ref{def:g:g-conv-wrt-fung} and \ref{def:g:g-conv-unif}.

\section{The A3w condition}
\label{sec:a3w-convexity}

There is an interplay between convex functions and convex sets. For example the set where a convex function lies below a plane is convex, the subdifferential of a convex function a convex set, and, finally, a $C^2$ locally convex function ($D^2u\geq0$) on a convex domain is convex. These innocuous seeming facts have important implications for the regularity of Monge--Amp\`ere equations. In this section we introduce an assumption on $g$, called A3w, that ensures  the natural extension of these results holds in the $g$-convex setting. The condition is the following.

\textbf{A3w.}\index{A3w}\setref{}{ref:a3w} Let the matrix $A$ be as given in \eqref{eq:g:Adef}. Then for all $\xi,\eta \in \mathbf{R}^n$ satisfying $\xi\cdot\eta=0$ and $(x,u,p) \in \mathcal{U}$ there holds
\begin{align}
  \label{eq:g:a3w}
  D_{p_kp_l}A_{ij}(x,u,p)\xi_i\xi_j\eta_k\eta_l \geq 0. \tag{A3w}
\end{align}

This condition, though with a strict inequality, first appeared as the third assumption in \cite{MTW05}.  There it was used to obtain interior second derivative estimates. Its implications for the underlying convexity theory were studied later \cite{FKM13,KimMcCann10,TrudingerWang09a,Liu09,Loeper09}. We will have occasion to use A3w for vectors $\xi,\eta$ which don't necessarily satisfy $\xi\cdot\eta = 0$. In this case A3w implies
\begin{equation}
  \label{eq:g:a3w-equiv}
  D_{p_kp_l}A_{ij}(x,u,p)\xi_i\xi_j\eta_k\eta_l \geq -C |\xi||\eta|(\xi\cdot\eta),
\end{equation}
for $C = -K|D_{p_kp_l}A_{ij}(x,u,p)|$. 
To prove \eqref{eq:g:a3w-equiv} from A3w let $\xi,\eta$ be arbitrary unit vectors and apply A3w to the vectors $\xi$ and $\eta-(\xi\cdot\eta)\xi$.

We begin with two theorems which illustrate the interplay between $g$-convex functions and $g$-convex sets under A3w.

\begin{theorem}\label{thm:g:gconvsection}
  Suppose $g$ is a generating function satisfying A3w and $u$ is a $g$-convex function. Suppose $y_0 \in Yu(x_0)$ and $z_0=g^*(x_0,y_0,u(x_0))$. For $h>0$ define
  \[ S^{x_0,y_0}_h:= \{x ; u(x) < g(x,y_0,z_0-h)\},\]
  and for $h=0$ define
  \[ S^{x_0,y_0}_0:= \{x; u(x) = g(x,y_0,z_0)\}.\]
(This notation is not perfectly consistent.) For all $h\geq0$ sufficiently small the set $S^{x_0,y_0}_h$ is $g$-convex with respect to $y_0,z_0-h$.  
\end{theorem}
Most of the proof is accomplished via the following Lemma.
\begin{lemma}\label{lem:g:maindiffineq}
  Assume $x_0,x_1 \in U,$ $y_0 \in V$ and $z_0$ are given and $g$ satisfies A3w. Let $\{x_\theta\}_{\theta \in [0,1]}$ be the $g$-segment with respect to $y_0,z_0$ that joins $x_0$ to $x_1$. Let $u$ be a $C^2$ $g$-convex function and define
  \[ h(\theta) = u(x_\theta)-g(x_\theta,y_0,z_0).\]
Then
  \begin{align}
  \label{eq:g:maindiffineq} \frac{d^2}{d \theta^2}h(\theta) &\geq [D_{ij}u(x_\theta) - g_{ij}(x_\theta,Yu(x_\theta),Zu(x_\theta))]\dot{(x_\theta)}_i\dot{(x_\theta)}_j\\
          \nonumber              &\quad+D_{p_kp_l}A_{ij}\dot{(x_\theta)}_i\dot{(x_\theta)}_jD_kh(\theta)D_lh(\theta)\\
\nonumber  &\quad\quad +A_{ij,u}\dot{(x_\theta)}_i\dot{(x_\theta)}_j h(\theta)-  C|h'(\theta)|.
\end{align}
where we've used the shorthand $D_kh(\theta) = u_k(x_\theta)-g_k(x_\theta,y_0,z_0)$. Here $K$ depends on $g$ and its derivatives on $(x_\theta,y_0,z_0)$ and $\dot{x_\theta} = \frac{d}{d\theta}x_\theta$.   In particular there holds
\begin{equation}
  \label{eq:g:main-diff-useful}
\frac{d^2}{d \theta^2}h(\theta) \geq - K|h'| - K_0|h|.
\end{equation}

\end{lemma}
\begin{proof}
   We first compute a differentiation formula for second derivatives along $g$-segments. We suppose
  \begin{equation}
    \label{eq:g:gseglemdef} \frac{g_y}{g_z}(x_\theta,y_0,z_0) = \theta q_1+(1-\theta)q_0,
  \end{equation}
and set $q = q_1-q_0$. We begin with a formula for first derivatives. Since
  \begin{equation}
    \label{eq:g:3chain}
     \frac{d}{d\theta} = (\dot{x_\theta})_iD_{x_i}
  \end{equation}
  we need to compute $(\dot{x_\theta})_i$. Differentiate \eqref{eq:g:gseglemdef} with respect to $\theta$ and obtain  \[\left[\frac{g_{i,m}}{g_z}-\frac{g_{i,z}g_{,m}}{g_z^2}\right](\dot{x_\theta})_i = q_m,\]
  from which it follows that
  \begin{equation}
    \label{eq:g:xdot}
       (\dot{x_\theta})_i = g_zE^{m,i}q_m.
  \end{equation}
Thus \eqref{eq:g:3chain} becomes
  \begin{equation}
    \label{eq:g:3qderiv}
     \frac{d}{d\theta} = g_zE^{m,i}q_mD_{x_i}.
  \end{equation}
  Using this expression to compute second derivatives we have
  \begin{align*}
  \frac{d^2}{d\theta^2}&= g_zE^{n,j}D_{x_j}(g_zE^{m,i}D_{x_i})q_mq_n\\
                  &= g_{z}^2E^{n,j}E^{m,i}q_mq_nD_{x_ix_j} +g_z^2q_mq_nE^{n,j}D_{x_j}(E^{m,i})D_{x_i}\\
    &\quad\quad+ g_zg_{j,z}E^{n,j}E^{m,i}q_mq_nD_{x_i}.
  \end{align*}
  Equation \eqref{eq:g:xdot} and the formula for differentiating an inverse yield
  \begin{align}
      \label{eq:g:3tosub}   \frac{d^2}{d\theta^2} = (\dot{x_\theta})_i(\dot{x_\theta})_j&D_{x_ix_j} -g_z^2q_mq_nE^{n,j}E^{m,a}D_{x_j}(E_{ab})E^{b,i}D_{x_i}\\&+ g_zg_{j,z}E^{n,j}E^{m,i}q_mq_nD_{x_i}.\nonumber
  \end{align}
  Now compute
  \begin{align}
    D_{x_j}(E_{ab}) &= D_{x_j}\left[g_{a,b}-\frac{g_{a,z}g_{,b}}{g_z}\right] \nonumber\\
    &= g_{aj,b}-\frac{g_{aj,z}g_{,b}}{g_z}-\frac{g_{a,z}g_{j,b}}{g_z}+\frac{g_{j,z}g_{a,z}g_{,b}}{g_z^2}\nonumber\\
    &= -\frac{g_{a,z}}{g_z}E_{jb} +E_{l,b}D_{p_l}g_{aj} \label{eq:g:3invderiv}.
  \end{align}
  Here we've used 
  \[ E_{l,b}D_{p_l}g_{aj}(\cdot,Y(\cdot,u,p),Z(\cdot,u,p)) = g_{aj,b}-\frac{g_{aj,z}g_{,b}}{g_z}, \]
  which follows by computing $D_{p_l}g_{aj}$, differentiating \eqref{eq:g:yzdef1} with respect to $p$ to express  $Z_p$ in terms of $Y_p$, and employing
  \begin{equation}
    \label{eq:g:e-yp}
    E^{i,j} = D_{p_j}Y^i
  \end{equation}
  (which is obtained via calculations similar to those for \eqref{eq:g:premate}).

   Substitute \eqref{eq:g:3invderiv} into \eqref{eq:g:3tosub} to obtain
  \begin{align*}
   \frac{d^2}{d\theta^2} &= (\dot{x_\theta})_i(\dot{x_\theta})_j D_{x_ix_j} -g_z^2q_mq_nE^{n,j}E^{m,a}E_{l,b}D_{p_l}g_{aj}E^{b,i}D_{x_i}\\
               &\quad\quad+[g_zg_{a,z}E^{n,j}E^{m,a}E_{j,b}E^{b,i}D_{x_i+}g_zg_{j,z}E^{n,j}E^{m,i}D_{x_i}]q_mq_n\\
    &= (\dot{x_\theta})_i(\dot{x_\theta})_j D_{x_ix_j} -g_z^2q_mq_nE^{n,j}E^{m,a}D_{p_i}g_{aj}D_{x_i}\\
               &\quad\quad+[g_zg_{a,z}E^{n,i}E^{m,a}D_{x_i+}g_zg_{j,z}E^{n,j}E^{m,i}D_{x_i}]q_mq_n\\
    &= (\dot{x_\theta})_i(\dot{x_\theta})_j(D_{x_i,x_j}-D_{p_k}g_{ij}D_{x_k})\\
               &\quad\quad+g_{j,z}\Big(E^{m,j}q_m\frac{d}{d\theta}+E^{n,j}q_n\frac{d}{d\theta}\Big),
  \end{align*}
  where in the last equality we swapped  the dummy indices $i$ and $a$ on the second term to allow us to collect like terms and also used \eqref{eq:g:3qderiv}.

  Now let's use this identity to compute $h''(\theta)$. We have 
  \begin{align*}
    h''(\theta) &= \big[D_{ij}u(x_{\theta})-g_{ij}(x_{\theta},y_0,z_0)\\
           &\quad-D_{p_k}g_{ij}(x_{\theta},y_0,z_0)(D_ku(x_\theta)-D_kg(x_{\theta},y_0,z_0))\big](\dot{x_\theta})_i(\dot{x_\theta})_j \\
           &\quad\quad+g_{j,z}(E^{m,j}q_mh'+E^{n,j}q_nh'). 
  \end{align*}
  Terms on the final line are bounded below by $-K|h'(\theta)|$. Adding and subtracting $g_{ij}(x_{\theta},y,z)$ for $y=Yu(x_\theta),z=Zu(x_\theta)$ yields
        \begin{align}
         \label{eq:g:need-app} h''(\theta) &\geq  \big[D_{ij}u(x_{\theta}) -g_{ij}(x_{\theta},y,z)\big](\dot{x_\theta})_i(\dot{x_\theta})_j +  [g_{ij}(x_{\theta},y,z)-g_{ij}(x_{\theta},y_0,z_0)\\
      \nonumber           &\quad-D_{p_k}g_{ij}(x_{\theta},y_0,z_0)(D_ku(x_\theta)-D_kg(x_{\theta},y_0,z_0))\big](\dot{x_\theta})_i(\dot{x_\theta})_j \\
      \nonumber      &\quad\quad-K|h'(\theta)|.
        \end{align}
        Set $u_0 = g(x_{\theta},y_0,z_0), \ u_1=u(x_\theta)$, $p_0 = g_x(x_{\theta},y_0,z_0),$ and $ p_1 = Du(x_\theta)$. Then rewriting in terms of the matrix $A$ we have
\begin{align*}
  h''(\theta) &\geq  \big[D_{ij}u(x_{\theta}) -g_{ij}(x_{\theta},y,z)\big](\dot{x_\theta})_i(\dot{x_\theta})_j  +\big[A_{ij}(x_{\theta},u_1,p_1)\\
         &\quad-A_{ij}(x_{\theta},u_0,p_0)-D_{p_k}A_{ij}(x_{\theta},u_0,p_0)(p_1-p_0)\big](\dot{x_\theta})_i(\dot{x_\theta})_j -K|h'(\theta)|\\
         &= \big[D_{ij}u(x_{\theta}) -g_{ij}(x_{\theta},y,z)\big](\dot{x_\theta})_i(\dot{x_\theta})_j +\big[A_{ij}(x_{\theta},u_0,p_1) \\
         &\quad-A_{ij}(x_{\theta},u_0,p_0)-D_{p_k}A_{ij}(x_{\theta},u_0,p_0)(p_1-p_0)\big](\dot{x_\theta})_i(\dot{x_\theta})_j\\
         &\quad\quad+ A_{ij,u}(x_\theta,u_\tau,p)(u_1-u_0)(\dot{x_\theta})_i(\dot{x_\theta})_j-K|h'(\theta)|
        \end{align*} 
        Here $u_\tau = \tau u +(1-\tau)u_0$ for some $\tau \in [0,1]$ results from a Taylor series. Applying another Taylor series for $f(t):= A_{ij}(x_{\theta},u_0,t p_1+(1-t)p_0)$, we obtain
        \begin{align*}
           h''(\theta) &\geq  \big[D_{ij}u(x_{\theta}) -g_{ij}(x_{\theta},y,z)\big](\dot{x_\theta})_i(\dot{x_\theta})_j  + A_{ij,u}(u_1-u_0)(\dot{x_\theta})_i(\dot{x_\theta})_j\\ &-K|h'(\theta)| + D_{p_kp_l}A_{ij}(x_{\theta},u_0,p_t)(\dot{x_\theta})_i(\dot{x_\theta})_j(p_1-p_0)_k(p_1-p_0)_l. 
        \end{align*}
        This is \eqref{eq:g:maindiffineq}. Inequality \eqref{eq:g:main-diff-useful} follows from non-negativity of $D_{ij}u(x_{\theta}) -g_{ij}(x_{\theta},y,z)$ and \eqref{eq:g:a3w-equiv}.
      \end{proof}

\begin{proof}[Proof: (Theorem \ref{thm:g:gconvsection})]
  The proof of $g$-convexity of sections now follows.
  We fix our $g$-convex function $u$ and any two $x_0,x_1$ in $S^{x_0,y_0}_h$, where, for now, we assume $h>0$. Let $\{x_\theta\}_{\theta \in [0,1]}$ denote the $g$-segment joining $x_0$ to $x_1$ with respect to $y_0,z_0-h$. Suppose, for a contradiction, that at some  $x_T$ on the $g$-segment 
  \begin{equation}
  u(x_T) \geq g(x_T,y_0,z_0-h)\label{eq:g:tocont}.
\end{equation}
There exists a support $g(\cdot,y_T,z_T)$ for $u$ at $x_T$ which, by virtue of being less than $u$, satisfies
\[ g(x_0,y_T,z_T) \leq u(x_0) < g(x_0,y_0,z_0-h),\]
and similarly at $x_1.$ Choose $\delta$ so large that $\sup_{\theta}g(x_\theta,y_T,z_T+\delta) - g(x_0,y_0,z_0-h) = 0$ (that is shift the $g$-affine function down until it just contacts $g(x_0,y_0,z_0-h)$). This step relies on A1 and that $g(\Omega,y_0,z_0-h) \subset J$ for $h$ sufficiently small.  Applying Lemma \ref{lem:g:maindiffineq} with $u = g(\cdot,y_T,z_T+\delta)$ and $z_0$ replaced by $z_0+h$ we obtain that
\[ h(\theta): =g(x_\theta,y_T,z_T+\delta) - g(x_0,y_0,z_0-h) \]
satisfies $h(\theta)\leq0$ for $\theta \in [0,1]$ and $h(\theta_0) = 0$ at some $\theta_0 \in (0,1)$. Since Lemma \ref{lem:g:maindiffineq} implies this function satisfies
\[ h'' \geq -K|h'| -K_0|h|,\]
we contradict the maximum principle (in say the form \cite[Theorem 2.10]{HanLin97}). When $h=0$ for a contradiction we instead assume \eqref{eq:g:tocont} is a strict inequality. This ensures when we shift down it is by positive $\delta$, giving again strict inequality at the end points and proving the result by the same contradiction. 
\end{proof}

\begin{theorem}\label{thm:g:gconvymapping}
  Let $g$ be a generating function satisfying A3w. Let $u \in C^0(\Omega)$ be a $g$-convex function. For each $x \in \Omega$ the set $Yu(x)$ is $g^*$-convex with respect to $x,u(x)$. 
\end{theorem}
\begin{proof}
  We use that Theorem \ref{thm:g:gconvsection} holds in the dual form. This follows because  A3w implies A3w$^*$ \index{A3w$^*$}, that is, the A3w condition for $g^*$ (see Lemma \ref{lem:a:a3w-star} for a proof). Let $v$ be the $g^*$ transform of $u$. We claim
  \begin{equation}
    \label{eq:g:yu-claim}
       Yu(x_0) := \{y ; v(y) = g^*(x_0,y,u(x_0))\},
  \end{equation}
  the latter set being $g^*$-convex with respect to $x_0,u(x_0)$ yields the result.

  To show \eqref{eq:g:yu-claim} take $y_0 \in Yu(x_0)$ and note Lemma \ref{lem:g:gstarsubdiff} implies $v(y_0) = g^*(x_0,y_0,u(x_0))$. If instead we take $y_0$ satisfying $v(y_0) = g^*(x_0,y_0,u(x_0))$ then the definition of $v$ implies for any $x \in \Omega$
  \[ g^*(x,y_0,u(x)) \leq g^*(x_0,y_0,u(x_0)).\]
  Applying $g(x,y_0,\cdot)$ to both sides we have
  \[ u(x) \geq g(x,y_0,g^*(x_0,y_0,u(x_0)),\]
  so $y_0 \in Yu(x_0)$. 
\end{proof}

The following result, due to Loeper \cite{Loeper09}, played a foundational role in understanding the A3w condition. We derive it from Theorem \ref{thm:g:gconvymapping}, however we could have just as easily proved Corollary \ref{cor:g:loeper} first and then derived Theorem \ref{thm:g:gconvymapping}. 
\begin{corollary}[Loeper's Maximum Principle]\index{Loeper's maximum principle}\label{cor:g:loeper}
  Suppose $x_0 \in \Omega$, $y_0,y_1 \in \Omega^*$ and $u_0 \in J$ are given. Let $\{y_\theta\}_{\theta \in [0,1]}$ denote the $g^*$-segment joining $y_0$ to $y_1$ with respect to $x_0,u_0$ and set $z_\theta = g^*(x_0,y_\theta,u(x_0))$. Then for any $x \in \Omega$
\[  g(x,y_\theta,z_\theta) \leq \max\{g(x,y_0,z_0),g(x,y_1,z_1)\} .\]
\end{corollary}

\begin{proof}
  Apply Theorem \ref{thm:g:gconvymapping} to the $g$-convex function
  \[ u(x):= \max\{g(x,y_0,z_0),g(x,y_1,z_1)\}. \]
  That is, we note $y_0,y_1 \in Yu(x_0)$ and by $g^*$-convexity of this set so too is $y_\theta$. 
\end{proof}

Our next theorems yield local characterisations of $g$-convexity under A3w. 

  \begin{theorem}\label{thm:g:locglob}
    Suppose $u \in C^2(\Omega)$ is locally $g$-convex on $\Omega$, that is, for each $x \in \Omega$
    \[D^2u(x) \geq g_{xx}(x,Yu(x),Zu(x)).\]
    If $\Omega$ is $g$-convex with respect to $(y,z)$ for all $y \in Yu(\Omega)$ and $z \in g^*(\cdot,y,u)(\Omega)$ then $u$ is $g$-convex.
\end{theorem}
\begin{proof}
 Fix $x_0 \in \Omega$. We show $g(\cdot,Yu(x_0),Zu(x_0))$ is a $g$-support at $x_0$. Let $x_1 \in \Omega$ be arbitrary and $\{x_\theta\}_{\theta \in [0,1]}$ be the $g$-segment with respect to $y_0 = Yu(x_0)$ and $z_0 = Zu(x_0)$ joining $x_0$ to $x_1$. Again, the function $h(\theta):=u(x_\theta)-g(x_\theta,y_0,z_0)$ satisfies
  \begin{equation}
    \label{eq:g:quick_diff_ineq}
       h''(\theta) \geq -K|h'(\theta)|-K_0|h(\theta)|.
  \end{equation}
  We assume initially  for $\theta$ in some small neighbourhood $(0,\kappa)$ that $h(\theta) \geq 0$. If $h(\theta) \leq 0$ in $[0,1]$ then by the maximum principle \cite[Theorem 2.10]{HanLin97} $h(\theta) \equiv 0$ and we are done. Otherwise we suppose $h(\theta) > 0$ for some $\theta \in (0,1)$ yet $h(1) < 0$. Set
  \[ h_\delta(\theta) := u(x_\theta^\delta)-g(x_\theta^\delta,y_0,z_0+\delta),\]
  where $x_\theta^\delta$ is the $g$-segment with respect to $y_0,z_0+\delta$. Because, $h_0(0) = 0, h_0(1) < 0$ for $\delta \geq 0$ large enough $h_\delta$ attains a zero interior maximum, again contradicting the maximum principle.
  
  Now we show $h(\theta) \geq 0$ on a small neighbourhood $(0,\kappa)$. To this end consider
  \[ \tilde{u}(x):= u(x)+\epsilon|x-x_0|^2/2.\]
  Now $\tilde{u} \geq g(\cdot,y_0,z_0)$ on a neighbourhood of $x_0$ (this neighbourhood depends on $\epsilon$). However, using a Taylor Series
  \begin{align*}
    D^2\tilde{u}(x) - A_{ij}(x,\tilde{u},D\tilde{u}) &= D^2u - A_{ij}(\cdot,u,Du) \\&\quad\quad+ \epsilon[I - A_{ij,u}|x-x_0|^2+A_{ij,p_k}\cdot (x-x_0)].
  \end{align*}
  The final matrix is positive definite on a sufficiently small neighbourhood which depends on $|A_{ij,u}|,|A_{ij,p_k}|$ but is independent of $\epsilon$. Thus $D^2\tilde{u}(x) - A_{ij}(x,\tilde{u},D\tilde{u}) \geq 0$ on a neighbourhood independent of $\epsilon$. Subsequently $\tilde{h}(\theta) := \tilde{u}(x_\theta)- g(x_\theta,y_0,z_0)$ is initially greater than 0 on a small neighbourhood $(0,\tau)$ for $\tau$ depending on $\epsilon$, and satisfies \ref{eq:g:quick_diff_ineq} on $[0,\kappa]$ for $\kappa$ independent of $\epsilon$. The maximum principle, used as above, implies $\tilde{h} \geq 0$ on $[0,\kappa]$ and sending $\epsilon\rightarrow0$ we have the same for $h$, as was required.    
\end{proof}

The previous theorem says that $g$-convexity can be checked locally. Thus the following theorem, which relates $Yu(x)$ to a local quantity --- the subdifferential --- should come as no surprise. We recall the subdifferential is defined for any semi convex function $u$ by
\[ \partial u(x_0) = \{ p \in \mathbf{R}^n; u(x) \geq u(x_0)+p\cdot(x-x_0)+o(|x-x_0|)\}.\]
We  use that if $p$ is an extreme point\footnote{An extreme point of a convex set $A \subset \mathbf{R}^n$ is a point $x \in \partial A$ such that there is a plane $P$ with $P \cap \partial A = \{x_0\}$.} of the (classically) convex set $\partial u(x_0)$ there is a sequence $x_k \rightarrow x_0$ with $u$  differentiable at $x_k $ and $Du(x_k) \rightarrow p$. \index{subdifferential}\index{extreme point} 
  \begin{theorem}\label{thm:g:y-sub}
    Suppose $g$ is a generating function satisfying A3w and $u:\Omega \rightarrow \mathbf{R}$ is a  $g$-convex function. Then for every $x \in \Omega$
    \[ Yu(x) = Y(x,u(x),\partial u(x)).\]
  \end{theorem}
  \begin{proof}
    Fix $x \in \Omega$ and set $u_0 = u(x)$. Since $y \mapsto g_x(x,y,g^*(x,y,u_0))$ is injective, it suffices to prove
    \begin{align*}
      g_x(x,\cdot,g^*(x,\cdot,u_0))(Yu(x)) &= g_x(x,\cdot,g^*(x,\cdot,u_0))(Y(x,u_0,\partial u(x))) \\
      &= \partial u(x).
    \end{align*}

    Take $y \in Yu(x)$ and set $z = g^*(x,y,u(x))$. Because $g(\cdot,y,z)$ is a support, a Taylor series yields for any $x' \in \Omega$
    \[ 0 \leq u(x') -g(x',y,z) = u(x')-g(x,y,z) - g_x(x,y,z)\cdot(x'-x)+o(|x'-x|). \]
    So as required $g_x(x,y,z) \in \partial u(x)$, that is $ g_x(x,\cdot,g^*(x,\cdot,u_0))(Yu(x)) \subset \partial u(x)$. 

   For the other subset relation, let $p$ be an extreme point of $\partial u(x)$ and take a sequence $x_k \rightarrow x$  with $Du(x_k) \rightarrow p$. Set $y_k = Y(x_k,u(x_k),Du(x_k))$ and note $y_k \rightarrow Y(x_0,u(x_0),p)$. On the other hand taking $k \rightarrow \infty$ in 
    \[ u(x) \geq g(x,y_k,g^*(x_0,y_k,u(x_k)))\]
    implies $Y(x_0,u(x_0),p)\in Yu(x_0)$, that is $p \in g_x(x,\cdot,g^*(x,\cdot,u_0))(Yu(x))$. By A3w $Yu(x)$ is $g^*$-convex with respect to $x,u_0$. Thus  $g_x(x,\cdot,g^*(x,\cdot,u_0))(Yu(x))$ is a convex set containing all extreme points of the convex set $\partial u(x)$. So $\partial u(x) \subset g_x(x,\cdot,g^*(x,\cdot,u_0))(Yu(x))$ and this completes the proof. 
  \end{proof}
  
\clearpage{}
\clearpage{}\chapter{Aleksandrov solutions}\label{chap:w}
With the convexity theory of Chapter \ref{chap:g} in hand we start our study of solutions to GJE. This chapter sees us introduce a notion of weak solution that is an extension of Aleksandrov solutions to Monge--Amp\`ere equations.  We prove such solutions exist, satisfy a comparison principle, and, when strictly $g$-convex,  are smooth when the data is smooth. The results in this chapter are from \cite{Trudinger14,Trudinger20}, though we provide some additional details.

First some motivation. Suppose $u$ is a $C^2$ $g$-convex solution of
\begin{equation}
  \label{eq:w:gje}
  \det DYu = \frac{f(\cdot)}{f^*(Yu)} \text{ in }\Omega,
\end{equation}
with $f,f^*>0$ and $x \mapsto Yu(x)$ a diffeomorphism onto its image. Then for all Borel $E\subset \Omega$ 
\[ \int_{E}f^*(Yu) \det DYu  = \int_{E}f .\]
The change of variables formula implies 
\begin{equation}
\int_{Yu(E)}f^* = \int_Ef.\label{eq:w:weak}
\end{equation}
 Moreover, under the assumption that $u$ is $C^2$, the converse holds: If equation \eqref{eq:w:weak} holds for all Borel $E \subset \Omega$ then $u$ satisfies \eqref{eq:w:gje}. This suggests requiring \eqref{eq:w:weak} hold for all Borel $E\subset\Omega$ provides a good notion of weak solution. We begin by considering the properties of the function $E \mapsto\int_{Yu(E)}f^*$. 
\section{The $g$-Monge--Amp\`ere measure and Aleksandrov solutions}
\label{sec:w:gmongeampere}
\begin{definition}\index{$g$-Monge--Amp\`ere measure}
  Let $u \in C^0(\Omega)$ be a $g$-convex function and $f^* \in L^1_{\text{loc}}(\mathbf{R}^n)$ be a nonnegative function. Let $\mu_{u,f^*}$ be the function defined on Borel $E\subset \Omega$ by 
  \begin{equation}
    \label{eq:w:ma-meas}
      \mu_{u,f^*}(E) := \int_{Yu(E)}f^*(y) \ dy.
  \end{equation}
  Then $\mu_{u,f^*}$ is called the $g$\textit{-Monge--Amp\`ere measure of $u$ with respect to $f^*$.}
\end{definition}
We write $\mu$ or $\mu_u$ for $\mu_{u,f^*}$ depending on what is clear from context. In this section we prove that this function is a Radon measure and behaves well with respect to convergence. The proofs we present follow those from the Monge--Amp\`ere case as in, say, Figalli's book \cite[Theorem 2.3, Proposition 2.6]{Figalli17}.

\begin{lemma}
  The function  $\mu_{u,f^*}$ is a Borel measure on $\Omega$.  Moreover because $f^* \in L^1_{\text{loc}}(\mathbf{R}^n)$ the function $\mu_{u,f^*}$ is a Radon measure.
\end{lemma}
\begin{proof}
  It is convenient, though not strictly necessary, to extend $\mu$ as an outer measure on $\mathbf{R}^n$. Define $\mu$ on any $E\subset \mathbf{R}^n$ by
  \begin{equation}
    \label{eq:w:outer}
     \mu(E) = \inf_{\substack{B\supset E\\ B \text{ Borel }}}\int_{Yu(B\cap \Omega)}f^*(y) \ dy. 
  \end{equation}
  This measure clearly agrees with our original on Borel subsets of $\Omega$. We show it satisfies Caratheodory's criterion, is Borel regular, and is finite on compact sets. 

First though, we show $\int_{Yu(B \cap \Omega)}f^*$ is well defined for all Borel $B$, specifically that $Yu(B \cap \Omega)$ is Lebesgue measurable.   This is the case when $B \subset\Omega$ is compact, since in this case $Yu(B)$ is compact. In addition if, for a sequence of sets $B_k,$ $Yu(B_k)$ is Lebesgue measurable, then  so is $Yu(\cup_k B_k) = \cup_kYu(B_k)$. By writing $\Omega$ as a countable union of compact subsets we see $Yu(\Omega)$ is Borel measurable and so is $Yu(B\cap \Omega)$ for any compact $B$. Next for any set $B$ with $Yu(B)$ Lebesgue measurable write
  \[ Yu(\Omega\setminus B) = [Yu(\Omega) \setminus Yu(B)] \cup [Yu(\Omega\setminus B) \cap Yu(B)].\]
  Using Lemma \ref{lem:g:sharedy} we see $Yu(\Omega\setminus B) \cap Yu(B)$ has measure 0. Thus $Yu(\Omega \cap B^c) = Yu(\Omega\setminus B)$ is Lebesgue measurable because it differs from the Lebesgue measurable set $Yu(\Omega) \setminus Yu(B)$ by a set of measure 0. We've shown the family of sets $B$ for which $Yu(B)$ is Lebesgue measurable is a $\sigma$-algebra containing all compact sets, and subsequently contains the Borel sets. Thus the  outer measure defined by \eqref{eq:w:outer} is well defined. 
  
  As defined $\mu$ is a subadditive function with $\mu(\emptyset) = 0$. To see subadditivity fix $\epsilon > 0$ and assume $E = \cup E_i$. Take Borel $B_i$ containing $E_i$ and satisfying $\int_{Yu(B_i \cap \Omega)}f^* \leq \mu(E_i)+\epsilon2^{-i}$. Put
  \[ \mu(E) \leq \int_{\cup Yu(B_i \cap \Omega)} f^* \leq \sum\mu(E_i)+\epsilon,\]
and send $\epsilon \rightarrow 0$. 
  
Next, $\mu$ is a Borel measure by Caratheodory's criterion and Lemma \ref{lem:g:sharedy}. Indeed, let $A,B$ be subsets of $\mathbf{R}^n$ with $\text{dist}(A,B) =\delta > 0$ and fix $\epsilon > 0$. Choose Borel measurable $C$ containing $A \cup B$ with
\[ \int_{Yu(C \cap \Omega)}f^* \leq \mu(A \cup B)+\epsilon.\]
The sets $A_\delta = C \cap \{x ; \text{dist}(x,A) < \delta/2\}$ and the corresponding $B_\delta$ are disjoint Borel sets containing, respectively, $A$ and $B$. Taking  $\mathcal{Z}$  as the measure 0 set from Lemma \ref{lem:g:sharedy}, we see $Yu(A_\delta)\setminus\mathcal{Z}$ and $Yu(B_\delta)\setminus\mathcal{Z}$ are disjoint. Thus
\begin{align*}
  \mu(A)+\mu(B) &\leq \int_{Yu(A_\delta)\setminus\mathcal{Z}}f^*+\int_{Yu(B_\delta)\setminus\mathcal{Z}}f^*\\
            &= \int_{Yu(A_\delta) \cup Yu(B_\delta)} f^*\leq \int_{Yu(C)}f^* \leq \mu(A \cup B)+\epsilon.
\end{align*}
Caratheodory's criterion follows by sending $\epsilon \rightarrow 0$.

To see $\mu$ is Borel regular take any $A \subset \mathbf{R}^n$ and a sequence of bounded Borel sets with $B_i \supset A_i$ and $\mu(A) = \lim_{i \rightarrow \infty} \mu(B_i)$. Continuity from above implies $\mu(A) = \mu(\cap B_i)$. We obtain Borel regularity since $\cap B_i$ is Borel. Finally, $\mu$ is trivially a Radon measure because, according to our definitions, $Yu$ takes values in the bounded set $V$.
\end{proof}

\begin{lemma}\label{lem:w:weak_conv}\index{weak convergence}
  Suppose $\{u_k\}_{k=1}^\infty$ is a sequence of $g$-convex functions on $\Omega$ converging pointwise  to a $g$-convex function $u$. Then the corresponding Monge--Amp\`ere measures of $u_k$ converge weakly to that of $u$. That is
  \[ \mu_{u_k,f^*} \rightharpoonup \mu_{u,f^*}\]
  as $k \rightarrow \infty$.
\end{lemma}
\begin{proof}
   The functions $u_k$ and $u$ are locally semiconvex with a constant independent of $n$ (Remark \ref{rem:g:semiconvex}). Thus the convergence is locally uniform \cite[\S 3.3]{Bakelman94}. To show weak convergence it suffices to show
  \begin{align}
    \label{eq:w:main1} \text{for each compact }K \subset \Omega \text{ there holds } &\limsup_{k \rightarrow \infty}\mu_{u_k}(K) \leq \mu_u(K)
  \end{align}
  and
  \begin{align}
  \label{eq:w:main2}\text{for each open }U\subset\subset \Omega \text{ there holds } &\liminf_{k \rightarrow \infty}\mu_{u_k}(U) \geq \mu_u(U).
  \end{align}

  Fix a compact $K\subset\Omega$. To begin we show
  \begin{align}
    \label{eq:w:conv1}
    Yu(K) \supset \bigcap_{i=1}^\infty\bigcup_{k=i}^\infty Yu_k(K).
  \end{align}
  Indeed, if we take $y$ in the right hand side then there is a sequence of $x_k \in K$ with $y \in Yu_k(x_k)$. Necessarily (Remark \ref{rem:g:support}) for all $x \in \Omega$
  \begin{equation}
    \label{eq:w:tolim}
    u_k(x) \geq g(x,y,g^*(x_k,y,u_k(x_k))).
  \end{equation}
  Up to a subsequence $x_k$ converges to some $ \overline{x} \in K$. Using locally uniform convergence to take $k \rightarrow \infty$ in \eqref{eq:w:tolim} yields
  \[ u(x) \geq g(x,y,g^*(\overline{x},y,u(\overline{x}))).\]
  This implies $y \in Yu(\overline{x})$. So \eqref{eq:w:conv1} holds and \eqref{eq:w:main1} follows by the continuity properties of the measure $E \mapsto \int_{E}f^*$. That is,
  \begin{align}
  \label{eq:w:reasoning}  \int_{Yu(K)}f^* &\geq \int_{\cap_{i=1}^\infty\cup_{k=i}^\infty Yu_k(K)} f^*\\
   \nonumber               &= \lim_{i \rightarrow \infty} \int_{\cup_{k=i}^\infty Yu_k(K)}f^*\\
   \nonumber               &\geq\lim_{i \rightarrow \infty} \limsup_{k \rightarrow \infty}\int_{Yu_k(K)}f^* = \limsup_{k \rightarrow \infty}\int_{Yu_k(K)}f^*.
  \end{align}

  For \eqref{eq:w:main2} we fix open $U\subset\subset\Omega$. As $\mu_u$ is a Radon measure it satisfies \cite[Theorem 1.8]{EvansGariepy15}
  \[ \mu_u(U) = \sup\{\mu(K); K \subset U, K \text{ compact}\}.\] Thus to show \eqref{eq:w:main2} we'll show for every compact $K\subset U$ there holds \[\liminf_{k \rightarrow \infty}\mu_{u_k}(U) \geq \mu_u(K) .\] Recalling the set $\mathcal{Z}$ from \ref{lem:g:sharedy}, it suffices to show
  \begin{equation}
    \label{eq:w:conv2}
    Yu(K) \setminus \mathcal{Z} \subset \bigcup_{i=1}^\infty \bigcap_{k=i}^\infty Yu_k(U),  
  \end{equation}
  for then we can conclude as in \eqref{eq:w:reasoning}.

  We take $y \in Yu(K) \setminus \mathcal{Z}$. There is $\tilde{x} \in K$ such that $g(\cdot,y,g^*(\tilde{x},y,u(\tilde{x})))$ is a $g$-support at $\tilde{x}$ and for all other $x$, $u(x) > g(x,y,g^*(\tilde{x},y,u(\tilde{x})))$. Hence for sufficiently small $\delta > 0$ 
  \[ \phi_\delta(\cdot) := g(\cdot,y,g^*(\tilde{x},y,u(\tilde{x}))-\delta),\]
  satisfies
  \[ \phi_\delta(\tilde{x}) > u(\tilde{x}) \text{ and }\phi_\delta(x) < u(x) \text{ for }x \in \partial U.\]
  Locally uniform convergence implies the same holds for all  $u_k$ provided $k$ is sufficiently large. For each such $u_k$ we  decrease $\delta$ to some $\delta_k$ for which $\phi_{\delta_k}$ touches $u_k$ from below at some $x' \in U$. Such a choice is possible because $u_k(\Omega) \subset J$ for $k$ sufficiently large.  Thus $y \in Yu_k(U)$ for all $k$ sufficiently large and \eqref{eq:w:conv2} follows.  
\end{proof}

Both previous results are crucial for our study of Aleksandrov solutions which we now define.

\begin{definition}\label{def:w:aleksandrov}
  Let $f \in L^1(\Omega)$, $f^* \in L^1(\Omega^*)$ and, after extending as 0, assume both are defined on $\mathbf{R}^n$. A $g$-convex function $u:\Omega \rightarrow \mathbf{R}$ is called an Aleksandrov (or generalised) solution of
  \begin{equation}
\tag{GJE}
       \det DYu(\cdot) = \frac{f(\cdot)}{f^*(Yu(\cdot))} \text{ in }\Omega,
  \end{equation}
  provided for every Borel $E\subset\Omega$ there holds
  \[ \int_{Yu(E)}f^*(y) \ dy = \int_E f(x) \ dx.\]\index{Aleksandrov solution}\index{generalised solution}

  Moreover $u$ is called a generalised solution of the second boundary value problem 
  \begin{equation}
\tag{2BVP}
    Yu(\Omega) = \Omega^*
  \end{equation}
  provided $\Omega^* \subset Yu(\overline{\Omega})$ and
  \begin{equation}
    \label{eq:w:g2bvp}
       |\{x ; f(x) > 0 \text{ and } Yu(x) \setminus\overline{\Omega^*} \text{ is nonempty}\}|  = 0. 
  \end{equation}
\end{definition}

This definition first appeared for the Monge--Amp\`ere equation in two dimensions in the works of Aleksandrov \cite{Aleksandrov42,Aleksandrov42a}. It was generalized to all dimensions in the works of Aleksandrov \cite{Aleksandrov58} and Bakelman \cite{Bakelman57,Bakelman58}. Extensions to optimal transport were made in the work of Ma, Trudinger, and Wang \cite{MTW05} based on earlier work of Wang \cite{Wang96} where the above definition of the generalized second boundary value problem first appeared. The extension to generated Jacobian equations, that is the definition above, is due to Trudinger \cite{Trudinger14}.

\section{Existence of Aleksandrov solutions}\label{sec:w:weak_exist}

Let's start our study of Aleksandrov solutions with the fact that, under very general hypothesis, they exist. We make use of an additional condition on the generating function.\\
\index{A5}
\textbf{A5. }\setref{}{ref:a5} There exist $K_0$ depending on $\Omega,\Omega^*, J$ such that whenever $x\in\Omega , y \in \Omega^*$ and $u:=g(x,y,z) \in J$ then $|g_x(x,y,z) |,|g^*_y(x,y,u)| \leq K_0$. Here $J$ is as given in A0.\\

This condition ensures the second boundary value problem implies a gradient bound. This is trivial in the Monge--Amp\`ere case where the second boundary value problem is $Du(\Omega) = \Omega^*.$

The goal of this section is to prove the following pair of theorems.\index{existence of weak solutions} 

\begin{theorem}\label{thm:w:weakexist}
  Let $g$ be a generating function satisfying $A5$. Let $f \in L^1(\Omega)$ and $f^*\in L^1(\Omega^*)$ be positive functions satisfying the mass balance condition. Then for any $x_0\in \Omega$ and $u_0 \in J$ satisfying $u_0+K_0 \text{diam}(\Omega) \in J$ there exists an Aleksandrov solution of \eqref{eq:g:gje} subject to \eqref{eq:g:2bvp} satisfying $u(x_0) = u_0.$
 \end{theorem}

 \begin{theorem}\label{thm:w:g2bvp}
   Suppose $u$ is an Aleksandrov solution of \eqref{eq:g:gje}. Suppose $f>0$ and $\Omega^*$ is $g^*$-convex with respect to $u$. Then $Yu(\Omega) \subset \overline{\Omega^*}$. 
 \end{theorem}
 
There are a number of steps to the proof of these theorems. Here is an outline:
The condition $u(x_0) = u_0$ is easier to enforce in its dual form --- that $g^*(x_0,\cdot,u_0)$ is a support of $v$ (the $g^*$transform of $u$). Thus our approach is to solve the dual problem, that is the corresponding equation for $v,$ in its generalized form. This, in turn, is achieved by solving a sequence of finite approximations. This approach is due to Caffarelli and Oliker \cite{CaffarelliOliker08}, and the extension to generated Jacobian equations due to Trudinger \cite{Trudinger14}. 

To begin we consider the duality structure of solutions. For GJE this result is from \cite{Trudinger14}, though we use some details from the optimal transport case \cite{MTW05}. \index{dual problem}

\begin{lemma}\label{lem:w:dualsolutions}
  Suppose $u:\Omega \rightarrow \mathbf{R}$ is a generalized solution of \eqref{eq:g:gje} subject to \eqref{eq:g:2bvp} with $f \in L^1(\Omega)$, $f^*\in L^1(\Omega^*)$ nonnegative functions satisfying the mass balance condition. Let $v$ denote the $g^*$-transform of $u$. Then $v|_{\Omega^*}$ is an Aleksandrov solution of
  \begin{align}
    \tag{GJE*} \det DXv = \frac{f^*(\cdot)}{f(Xv(\cdot))} \text{ on }\Omega^*. \label{eq:w:gjestar}
  \end{align}
  If, in addition, $f > 0$ a.e on $\Omega$ then $v$ is a generalized solution of the second boundary value problem
  \begin{equation}
     \tag{2BVP*} Xv(\Omega^*) = \Omega. \label{eq:w:2bvpstar}
   \end{equation}
\end{lemma}

\begin{proof} The idea is that heuristically $Xv = Yu^{-1}$. The proof is all about dealing with the measure 0 sets which prevent this from being more than heuristic. \\
  \textit{Step 1. The $g$-transform is an involution (on the domain of definition of $u$)}\\
  We have
  \begin{align*}
    v(y):= \sup_{x \in \Omega}g^*(x,y,u(x)),\\
    v^*(x) := \sup_{y \in V}g(x,y,v(y)).
  \end{align*}
  Take $x \in \Omega$ and $y \in Yu(x)$. From Lemma \ref{lem:g:gstarsubdiff} we have $x \in Xv(y)$ and $v(y) = g^*(x,y,u(x))$. The same argument, this time with $v$ and $v^*$, implies
  \[v^*(x) = g(x,y,v(y)) = g(x,y,g^*(x,y,u(x))) = u(x).\]

  \textit{Step 2. $Yu$ is measure preserving}\\
  We fix an $E^* \subset \Omega^*$ and aim to show
  \begin{equation}
    \label{eq:w:meas_pres}
    \int_{E^*}f^* =  \int_{Yu^{-1}(E^*)}f,
  \end{equation}
  where $Yu^{-1}(E^*) = \{x \in \Omega ; Yu(x) \cap E^* \neq \emptyset\}.$ Let $E_u$ denote the measure 0 set of points at which $u$ is not differentiable. Then because $u$ is an Aleksandrov solution
  \[ \int_{Yu(Yu^{-1}(E^*)\setminus E_u)}f^* = \int_{Yu^{-1}(E^*)\setminus E_u} f = \int_{Yu^{-1}(E^*)}f,\]
  and, provided the left most integral equals $\int_{E^*}f^*$, \eqref{eq:w:meas_pres} follows. For this we note that, with $\mathcal{Z}$ as defined in Lemma \ref{lem:g:sharedy}, there holds
  \begin{equation}
    \label{eq:w:set_rel}
       E^*\setminus[\mathcal{Z}\cup Yu(E_u)] \subset Yu(Yu^{-1}(E^*)\setminus E_u) \subset E^*.  
  \end{equation}
  Indeed if $y$ is in the left most set then $y \in E^* \subset \Omega^*$ so by the generalized second boundary value problem $y \in Yu(x)$ for some $x \in \overline{\Omega}$. However based on the sets we have excluded this is the only $x$ for which $y \in Yu(x)$ and, furthermore, $Yu(x)$ is a singleton. That is, $x \in Yu^{-1}(E^*)\setminus E_u$ and so $y \in Yu(Yu^{-1}(E^*)\setminus E_u)$. If instead $y \in Yu(Yu^{-1}(E^*)\setminus E_u)$ then $y = Yu(x)$ where $x$ satisfies that $Yu(x)$ is a singleton in $E^*$. So \eqref{eq:w:set_rel} holds and \eqref{eq:w:meas_pres} follows by integrating $f^*$ over the sets in \eqref{eq:w:set_rel}.

  \textit{Step 3. Integrals over $Xv(E^*)$ agree with those over $Yu^{-1}(E^*)$}\\
  Now we show for each $E^* \subset \Omega^*$ that
  \begin{equation}
    \label{eq:w:int-rel}
     \int_{Xv(E^*)} f = \int_{Yu^{-1}(E^*)}f.
  \end{equation}
  We recall $f$ is extended as 0 outside $\Omega$. Thus it suffices to show $Xv(E^*)\cap\Omega = Yu^{-1}(E^*)$. Indeed if $x \in Xv(E^*)\cap \Omega$ then $x \in Xv(y)$ for $y \in E^*$. So $y \in Yv^*(x)$ and because $v^* \equiv u$ on $\Omega$ we have $x \in Yu^{-1}(E^*)$. Alternatively, if $x \in Yu^{-1}(E^*)$ then $x \in \Omega$ and $Yu(x)$ contains $y \in E^*$. Thus $x \in Xv(y) \subset Xv(E^*)$. This shows \eqref{eq:w:int-rel} and completes the proof that $v$ is an Aleksandrov solution of \eqref{eq:w:gjestar}. 

  \textit{Step 4. $v$ is a generalised solution of \eqref{eq:w:2bvpstar}}\\
Now suppose $f > 0$ a.e on $\Omega$ and take $x \in \Omega$. Then, for every $\epsilon > 0$, the set $\{x' \in B_\epsilon(x); f(x') > 0\}$ has positive measure. So, because $u$  is an Aleksandrov solution satisfying the generalized second boundary value problem, there is a sequence $x_k \rightarrow x$ with $Yu(x_k) =y_k$ a singleton in $\overline{\Omega^*}$. In particular $x_k \in Xv(y_k)$ so that after taking a subsequence we see $x \in Xv(\overline{y})$ for some $\overline{y}\in \overline{\Omega^*}$.  This proves the first part of the definition of generalized solution.

  Next we show 
  \begin{equation}
    \label{eq:w:meas-0-cond}
     |\{y \in \Omega^*; f^*(y) > 0 \ \text{ and } \ Xv(y) \setminus \overline{\Omega} \neq \emptyset\} | = 0.
  \end{equation}
  Take $y$ in this set. Either $v$ is not differentiable at $y$, or there is a neighbourhood, $B_\epsilon(y)$ such that for almost all  $y' \in B_\epsilon(y)$ we have $Xv(y')$ is also disjoint from\footnote{If not there is a sequence of $y_k \rightarrow y$ with $x_k \in Xv(y_k)\cap \overline{\Omega}$ so that $x_k \rightarrow x = Xv(y) \in \overline{\Omega}$, a contradiction.} $\overline{\Omega}$. In the latter case
  \[ 0 = \int_{Xv(B_\epsilon(y))}f = \int_{B_\epsilon(y)}f^*,\]
 Dividing by $|B_\epsilon|$ and sending $\epsilon \rightarrow 0$ we see  $y$ is not a Lebesgue point of $f^*$.  We've shown $y$ must lie in one of two sets of measure 0. This gives \eqref{eq:w:meas-0-cond}.
\end{proof}

\begin{remark}\label{rem:w:brenier}
  If $u$ is $g$-convex and satisfies for every $E^* \subset \Omega^*$
  \begin{equation}
    \label{eq:w:bren-def}
     \int_{Yu^{-1}(E^*)}f = \int_{E^*}f^*,
  \end{equation}
  then $u$ is called a Brenier solution of \eqref{eq:g:gje}\index{Brenier solution}. The above proof shows, under our hypothesis, Aleksandrov solutions are Brenier solutions. Specifically we are using that $f,f^*$ are extended outside $\Omega,\Omega^*$ as 0. These assumptions, along with $f,f^*>0$ on $\Omega,\Omega^*$ respectively, imply the converse: Brenier solutions are Aleksandrov solutions. Indeed if $u$ is a Brenier solution then by \eqref{eq:w:int-rel} and \eqref{eq:w:bren-def} $v$ is an Aleksandrov solution of \eqref{eq:w:gjestar}. Then a straight forward modification of step 4 of the previous proof, using now that $v$ is an Aleksandrov solution, implies $v$ is a generalised solution of \eqref{eq:w:2bvpstar}. Then using the dual of Lemma \ref{lem:w:dualsolutions} $u$ is an Aleksandrov solution of \eqref{eq:g:gje} subject to the generalised second boundary value problem. 

  Explicitly, by extending $f^*$ as 0 outside $\Omega^*$ any $g$-Monge--Amp\`ere mass outside $\Omega^*$ is ignored. When $f^*$ is nonzero outside $\Omega^*$ the convexity conditions of Theorem \ref{thm:w:g2bvp} are needed to ensure Brenier solutions are Aleksandrov solutions. 
\end{remark}

We now prove an existence result for finite approximations of the dual problem. 
\begin{lemma} \label{lem:w:finiteexist}
  Suppose $u_0 \in J$ is such that $u_0+K_0 \text{diam}(\Omega) \in J$. Let $\mu_0 = \sum_{i=0}^Nf_i\delta_{x_i}$ for $x_0,\dots,x_N \in \Omega$ with $f_0 > 0$. Assume $f^* \in L^1(\Omega^*)$ is a positive function satisfying 
  \[ \int_{\Omega^*}f^* = \sum_{i=0}^Nf_i = \mu_0(\Omega).\]
  There exists a $g^*$-convex function $v:\overline{\Omega^*}\rightarrow\mathbf{R}$ such that for each $i=1,\dots,N$
  \[ \int_{Xv^{-1}(x_i)}f^*(y) \ dy = f_i. \]
  Moreover the function $g^*(x_0,\cdot,u_0)$ is a support of $v$ at some $y \in \Omega^*$.
\end{lemma}
\begin{proof}
  It is convenient to work with functions indexed by\footnote{In this section, and this section only, we use bold face for elements of $\mathbf{R}^{N+1}$ so as to avoid confusion with functions.} a vector $\mathbf{u} \in \mathbf{R}^{N+1}$. Indeed, let $\mathbf{u} = (u_0,u_1,\dots,u_N)$ where $u_0$ is fixed, and $u_1,\dots,u_N \in J$. Define 
\begin{align}
 \nonumber  v_{\mathbf{u}}&:\overline{\Omega^*}\rightarrow\mathbf{R}\\
    v_\mathbf{u}(y) &:= \sup\{g^*(x_i,y,u_i); i = 0,\dots,N\}.
    \label{eq:w:vdef}
  \end{align}
  The function $v_{\mathbf{u}}$ is $g^*$-convex since it is a supremum of $g^*$-affine functions with $u_i \in J$. Thus the (multi-valued) mapping $Xv$ is well defined. We consider the set $\mathcal{V}$ consisting of functions $v=v_\mathbf{u}$ for some $\mathbf{u}$ which satisfy:
  \begin{enumerate}
  \item The function $g^*(x_0,\cdot,u_0)$ is a support at some $y \in \Omega^*$.\\
  \item For $i=1,\dots,N$ there holds $ \int_{Xv^{-1}(x_i)}f^*(y) \ dy \leq f_i$.\\
    \item $\int_{Xv^{-1}(x_0)}f^*(y)  = \mu_0(\Omega) - \sum_{i=1}^n\int_{Xv^{-1}(x_i)}f^*(y) \ dy $.
    \end{enumerate}

Condition 3 is superfluous for functions of the form \eqref{eq:w:vdef}, but it will be helpful to have taken special note of it.

    We claim $\mathcal{V}$ is nonempty. Indeed for a particular choice of $u_1,\dots,u_N$ sufficiently large we obtain  $v_{\mathbf{u}}(\cdot) = g^*(x_0,\cdot,u_0)$ which is clearly in $\mathcal{V}$. To check such a choice of $u_1,\dots,u_N$ is possible note we must have, for each $y \in \Omega^*$
    \[ g^*(x_i,y,u_i) \leq g^*(x_0,y,u_0).\]
    Applying $g(x_i,y,\cdot)$ to both sides it suffices to ensure
    \[ u_i \geq g(x_i,y,g^*(x_0,y,u_0)).\]
    However because, by condition A5, $g(x_i,y,g^*(x_0,y,u_0)) \leq K_0\text{diam}(\Omega)+u_0 \in J$ (via a Taylor series), such a choice is always possible. 

    Since we now know $\mathcal{V}$ is nonempty we may pick any element $v_{\tilde{\mathbf{u}}}$ and further restrict our attention to the set $\tilde{\mathcal{V}}$ defined as the set of $v_{\mathbf{u}}$ lying above $v_{\tilde{\mathbf{u}}}$, that is those satisfying $u_i \leq \tilde{u_i}.$
    In addition,  the condition that $g^*(x_0,\cdot,u_0)$ is a support implies a lower bound on $u_1,\dots,u_n$.  That is, there is some $y \in \Omega^*$ with
    \[ g^*(x_i,y,u_i) \leq g^*(x_0,y,u_0) \quad \text{ for }i=1,2,\dots,N. \] By applying $g(x_i,y,\cdot)$ to both sides we see
    \[ u_i = g(x_i,y,g^*(x_i,y,u_i)) \geq g(x_i,y,g^*(x_0,y,u_0)).\]
    A lower bound for the $u_i$ follows by taking the infimum over $x_i$ and $y\in\Omega^*$.

    Thus the set of $\mathbf{u}$ with $v_{\mathbf{u}} \in \tilde{\mathcal{V}}$ is bounded. Define $H:\tilde{\mathcal{V}}\rightarrow\mathbf{R}$ by
    \[ H(v_{\mathbf{u}}) = \sum u_i.\]
    Then there exists a sequence $\{\mathbf{u}^k\}_{k=1}^\infty\subset \mathbf{R}^{N+1}$ such that $v_{\mathbf{u}^k} \in \tilde{\mathcal{V}}$ and
    \[ H(v_{\mathbf{u}^k}) \rightarrow \inf_{v_\mathbf{u} \in \tilde{\mathcal{V}}}H(v_\mathbf{u}).\]
    Up to a subsequence $\mathbf{u}^k $ converges to some $ \overline{\mathbf{u}}$. We show $v_{\overline{\mathbf{u}}}$ is the desired function.

    Necessarily $v_{\mathbf{u}^k}\rightarrow v_{\overline{\mathbf{u}}}$ and since these functions are semiconvex this convergence is locally uniform. Moreover the resulting function satisfies condition 1 and condition 3.

    We argue that $v_{\mathbf{\overline{u}}}$ satisfies condition 2 as follows. Let $E_v$ denote the set of points where $v_{\overline{\mathbf{u}}}$ is not differentiable. We claim
    \begin{equation}
    X^{-1}_{v_{\overline{\mathbf{u}}}}(x_i)\setminus E_v   \subset  \bigcup_{j=1}^\infty\bigcap_{k=j}^\infty X^{-1}_{v_{\mathbf{u}^k}}(x_i).\label{eq:w:Xident}
   \end{equation}
   Indeed if $y$ is in the left hand side we have $v_{\overline{\mathbf{u}}}(y) > g(x_j,y,\overline{u}_j)$ for $j \neq i$. Via locally uniform convergence $v_{\mathbf{u}^k}(y) > g(x_j,y,u^k_j)$ for all $k$ sufficiently large. Subsequently for all such $k$ the $g^*$-support of $v_{\mathbf{u}^k}$ at $y$ is $g^*(x_i,y,u^k_i)$ proving \eqref{eq:w:Xident}. It follows that $v_{\overline{\mathbf{u}}}$ satisfies condition 2 (cf. \eqref{eq:w:conv2},\eqref{eq:w:reasoning}).

Now suppose $v_{\overline{\mathbf{u}}}$ is not the desired solution. Then for some $i \neq 0$ we have
    \begin{equation}
      \label{eq:w:strictineq}
            \int_{Xv^{-1}(x_i)}f^*(y) \ dy < f_i.  
    \end{equation}

    The new function obtained by replacing $\overline{u}_i$ with ${u}_i -\epsilon$ for $\epsilon$ sufficiently small strictly decreases $H$ and, we claim, still lies in $\tilde{\mathcal{V}}$. To check this new function still lies in $\tilde{\mathcal{V}}$ we note condition 3 and \eqref{eq:w:strictineq} implies $g^*(x_0,\cdot,y_0)$ is a support of $v_{\overline{\mathbf{u}}}$ on a set of positive measure. Thus provided $\epsilon$ is taken small enough our new function still has $g^*(x_0,\cdot,y_0)$ as a support on a set of positive measure. Condition 2 is also satisfied since decreasing $u_i$ decreases $\int_{Xv^{-1}(x_k)}f^*(y) $ for $k \neq i$, and, for small enough $\epsilon$, \eqref{eq:w:strictineq} is still satisfied. 
 \end{proof}

 Theorem \ref{thm:w:weakexist} now follows from Lemmas \ref{lem:w:dualsolutions} and \ref{lem:w:finiteexist}.

 \begin{proof} [Proof (Theorem \ref{thm:w:weakexist}).]
   With all quantities as in Theorem \ref{thm:w:weakexist} we  consider a sequence of approximating problems for $K \in \mathbf{N}$. Begin by dividing the domain $\Omega$ into a finite number of Borel sets with positive measure and diameter less than $1/K$. Call these sets $\omega^K_i$ for $i=0,1,\dots,N$ for some $N \in \mathbf{N}$. Take $x_i^K \in \omega^K_i$ and relabel as necessary to ensure $x_0 = x_0^K \in \omega^K_0$. Set
   \[ f_i^K = \int_{\omega^K_i}f(x) \ dx,\]
   and use these to define $\mu_0^K = \sum_{i=0}^Nf_i^K\delta_{x^K_i}$.
   Employ Lemma \ref{lem:w:finiteexist} to solve the dual problem for some function $v_K$.

   Since $v_k$ is a sequence of bounded $g^*$-convex functions, equicontinuous by condition A5,  up to a subsequence they converge uniformly to some function $v$.
   We check that $v$ satisfies
   \begin{enumerate}
   \item There is $\overline{y} \in \overline{\Omega^*}$ such that $v(\overline{y}) = g^*(x_0,\overline{y},u_0)$ and for all other $y \in \Omega^* $there holds $v(y) \geq  g^*(x_0,y,u_0)$.
     \item For every $E \subset \Omega$ there holds $\int_{Xv^{-1}(E)}f^* = \int_E f$.  
     \item $v$ is a generalized solution of \eqref{eq:w:2bvpstar}. 
     \end{enumerate}
     For the first let $y_k$ satisfy $v_k(y_k) = g^*(x_0,y_k,u_0)$, noting elsewhere $v(y) \geq g^*(x_0,y,u_0)$, and take $k \rightarrow \infty$ using uniform convergence.

    The second point follows because $Xv^{-1}$ satisfies the same convergence properties  as $Xv$ (Lemma \ref{lem:a:1}) and $\mu^K_0 $ converges weakly to the measure $\mu$ defined by $\mu(E) = \int_E f$. As noted in Remark \ref{rem:w:brenier} this implies that infact $v$ is an Aleksandrov solution of \eqref{eq:w:gjestar}.
     
 For the third take any $x \in \Omega$. There is a sequence $x^K_i \rightarrow x$ with $x^K_i \in Xv(y_K)$ for some $y_K$. Subsequently $x \in Xv(\lim y_K)$ (again by uniform convergence). That is, $\Omega\subset Xv(\overline{\Omega^*})$. Furthermore, by construction, if $v$ is differentiable at $y$ then $Xv(y) \in \overline{\Omega}$. Thus the second requirement to be a generalized solution to \eqref{eq:w:2bvpstar} is satisfied.

 Hence using Lemma \ref{lem:w:dualsolutions} (though in the dual form) we obtain that the $g$-transform is an Aleksandrov solution of \eqref{eq:g:gje} with $u_0 = u(x_0)$. 
\end{proof}

Let's conclude this section with the proof of Theorem \ref{thm:w:g2bvp}. Our proof is taken from \cite[Lemma 5.1]{MTW05}. 

\begin{proof}[Proof (Theorem \ref{thm:w:g2bvp}).] We now suppose $u$ is a generalized solution of \eqref{eq:g:gje} with $f > 0$ on $\Omega$ and $\Omega^*$ is $g^*$-convex with respect to $u$. We show $Yu(\Omega) \subset \overline{\Omega^*}$. First we show if $Yu(x)$ is a singleton then $Yu(x) \in \overline{\Omega^*}$. Indeed suppose, for a contradiction, $Yu(x) \notin \overline{\Omega^*}$. Let $E_u$ be the points where $u$ is not differentiable. We claim there is $\epsilon > 0$ such that $Yu(B_\epsilon(x)\setminus E_u)$ is disjoint from $\overline{\Omega^*}$. If not there is a sequence $x_k\rightarrow x$ with $Yu(x_k)$ a singleton in $\overline{\Omega^*}$. Then $Yu(x)  = \lim_{k \rightarrow \infty} Yu(x_k) \in \overline{\Omega^*}$. Since $f^*$ is 0 outside $\overline{\Omega^*}$ we have
  \[ \int_{B_\epsilon(x)}f = \int_{B_\epsilon(x)\setminus E_u} f =\int_{Yu(B_\epsilon(x)\setminus E_u)}f^* = 0 .\]
  This contradicts $f>0$ on $\Omega$.

  Now if $Yu(x)$ is not a singleton, then $\partial u(x)$ is a convex set containing more than one point. Let $p$ be an extreme point of $\partial u(x)$. There is (as in Theorem \ref{thm:g:y-sub}) a sequence $x_k \rightarrow x$ with $u$ differentiable at $x_k$ and $Du(x_k)\rightarrow p$. We note $Yu(x_k) = Y(x_k,u(x_k),Du(x_k)) \in \overline{\Omega^*}$ so $Y(x,u(x),p) \in \overline{\Omega^*}$. Thus, by the $g^*$-convexity of $\Omega^*$, $\{p; Y(x,u(x),p) \in \overline{\Omega^*}\}$ is a convex set containing the extreme points of $\partial u(x)$. Subsequently it contains $\partial u(x)$ and $Yu(x) = Y(x,u(x),\partial u(x)) \subset \overline{\Omega^*}$.
  
\end{proof}

\section{Comparison principle}

Maximum and comparison principles are central to elliptic PDE. However, when our differential operator depends on the lowest order term $u$ in an unkown way, these results are lost. A cornerstone of GJE is that despite the unwieldy $u$ dependence the convexity theory yields comparison principles. \index{comparison principle}

\begin{lemma}\label{lem:w:comp}
  Let $u,v$ be $g$-convex functions with $u = v$ on $\partial \Omega$. If $u \leq v$ in $\Omega$ then $Yv(\Omega)\subset Yu(\Omega)$.
\end{lemma}
\begin{proof}
  Take $y_0 \in Yv(\Omega)$ then $g(\cdot,y_0,z_0)$ supports $v$ at some $x_0 \in \Omega$. If $g(\cdot,y_0,z_0) \leq u$ in $\Omega$ then, because $u \leq v$, we have 
\[g(x_0,y_0,z_0) \leq u(x_0) \leq v(x_0) = g(x_0,y_0,z_0),\]
  so that these are equalities and subsequently $y_0 \in Yu(x_0)$.  Otherwise shift the support until it supports $u$. Explicitly, we set $z = \sup_{x \in \Omega}g^*(x,y_0,u(x))$, which is well defined because $u(\overline{\Omega}) \subset J$. Then for each $x \in \Omega$
  \[ u(x) = g(x,y,g^*(x,y,u(x))) \geq g(x,y,z),\]
  with equality at some interior contact point. (Interior because we have shifted down by a positive amount.) Thus $y \in Yu(\Omega)$. 
\end{proof}

\begin{corollary}\label{cor:w:comp}
  Let $u,v$ be $g$-convex functions on $\Omega$ with $u\leq v$ on $\partial \Omega$. Assume for $f^*\in L^1_{\text{loc}}(V)$ that, in the Aleksandrov sense,
  \[ f^*(Yu)\det DYu > f^*(Yv)\det DYv.\]
  By this we mean for all Borel $E \subset \Omega$ of positive measure
  \begin{align}
    \label{eq:w:aleks-ineq}
    \int_{Yu(E)}f^* > \int_{Yv(E)}f^*. 
  \end{align}
  Then $u \leq v$ in $\Omega$.
\end{corollary}
\begin{proof}
  If not we may consider a maximal connected component of the set $\{ u > v\}$. Call this component $\Omega'$. Note $\Omega' \subset \Omega$ with $u = v$ on $\partial \Omega'$. Our previous lemma implies $Yv(\Omega') \supset Yu(\Omega')$. This contradicts \eqref{eq:w:aleks-ineq} with $E= \Omega'$. 
\end{proof}

We consider extensions to uniqueness results proper in Chapter \ref{cha:u}. 

\section{Interior regularity}
\label{sec:interior-regularity}

Initially Aleksandrov solutions have no regularity beyond what is guaranteed by $g$-convexity. In this section we show if $g$ satisfies A3w  the only real hindrance to regularity is a lack of strict $g$-convexity. That is, we show if $u$ is a strictly $g$-convex Aleksandrov solution and $f,f^*$ are $C^2$ then $u \in C^3(\Omega)$. Higher regularity follows from the elliptic theory provided $g,f,f^*$ have higher regularity.  

The result is proved using standard tools of elliptic PDE: the method of continuity and apriori estimates.  The estimates follow from Pogorelov type arguments. These arguments are a powerful method for obtaining $C^2$ estimates. However their proofs are tedious and we use a number of these estimates.  So we save their proofs for Chapter \ref{chap:r2}.\index{pogorelov estimates}

\begin{theorem}\label{thm:w:pogorelov_boundary}
 Let $u \in C^4(\Omega) \cap C^2(\overline{\Omega})$ be an elliptic solution of
  \begin{align}
   \label{eq:r2:mate-dirichlet} \det [D^2u -A(\cdot,u,Du)] &= B(\cdot,u,Du) \text{ in }\Omega\\
  \label{eq:r2:dirichlet}  u &= \phi \text{ on }\partial \Omega. 
  \end{align}
  where $A,B$ are $C^2$, $A$ satisfies A3w and $B>0$. Assume there exists a barrier $\underline{u} \in C^2(\overline{\Omega})$ satisfying $\underline{u} = \phi$ on $\partial \Omega$ along with
  \begin{align}
\label{eq:r2:underu2}   D^2\underline{u} - A(\cdot,u,D\underline{u}) &\geq 0,\\
  \label{eq:r2:underu1}  \text{ and }\quad\det[D^2\underline{u} - A(\cdot,u,D\underline{u})] &\geq B(\cdot,u,D\underline{u}).
  \end{align}
    Then there is $C$ depending only on $\Omega,\underline{u},A,B,\Vert u \Vert_{C^1(\Omega)}$ such that
\[   \sup_{\partial \Omega}|D^2u| \leq C. \] 
\end{theorem}

\begin{theorem}\label{thm:w:pogorelov_strictly}
    Assume that $u \in C^4(\Omega) \cap C^2(\overline{\Omega})$ is an elliptic solution of
    \begin{align*}
      \det[D^2u-A(\cdot,u,Du)] = B(\cdot,u,Du) \text{ in }\Omega, \\
      u = g(\cdot,y,z) \text{ on }\partial \Omega,
    \end{align*}
where $A,B$ are $C^2$, $A$ satisfies A3w and $B>0$. 
  Then there exists $\beta,d,C > 0$ such that provided $\text{diam}(\Omega) < d$ we have the estimate 
  \begin{align*}
       \sup_{\Omega}(g(\cdot,y,z)-u)^{\beta}|D^2u| \leq C,
  \end{align*}
  where $C$ depends on $\Omega,A,B,\Vert u \Vert_{C^1(\Omega)}$.

\end{theorem}

With these estimates in hand we prove the following. 

\begin{theorem}\label{thm:w:regularity}
  Let $u$ be a strictly $g$-convex Alexandrov solution of \eqref{eq:g:gje} subject to \eqref{eq:g:2bvp}. Assume $g$ satisfies A3w, $f,f^*$ are $C^2$ functions satisfying  $0 < \lambda < f,f^* \leq \Lambda \leq \infty$, and $\Omega^*$ is $g^*$-convex with respect to $u$. Then $u \in C^3(\Omega)$.
\end{theorem}
\begin{proof}
  We solve, via the method of continuity, a family of approximating problems. These have smooth solutions which converge to a solution of \eqref{eq:g:gje}. Using the estimates in Theorem \ref{thm:w:pogorelov_strictly} we show this limiting function, a posteriori our original function, retains this smoothness.

  \textit{Step 1. Construction of approximating problem}\\
  Fix  $x_0 \in \Omega$, without loss of generality $x_0=0$, and put $B_r:=B_r(0)$ for $r$ to be chosen small. Via semiconvexity $u+C|x|^2/2$ is convex for $C$ chosen large, depending only on $g$.  Let $\tilde{u}_m$ denote a sequence of mollifications of $u+C|x|^2/2$ and put
  \[ u_m = \tilde{u}_m-\frac{C}{2}|x|^2.\]
  Because the mollification of a convex function is convex, $D^2u_m \geq -CI$. Moreover by choosing $C$ larger as necessary we have
  \begin{equation}
    \label{eq:w:um-choices}
    |Du_m| +|u_m|\leq C_0 \text{ on }B_r. 
  \end{equation}
  This is possible because $u_m \rightarrow u$ and pointwise $Du_m(x_0) \rightarrow p \in \partial u(x_0)$ with such $p$ bounded by the local Lipschitz property of semiconvex functions. In addition this convergence implies $(x,u_m(x),Du_m(x)) \in \mathcal{U}$ for $m$ sufficiently large. 

  We consider for $\epsilon>0$ small $w_m \in C^{4,\alpha}(\overline{B_r})$ solving
  \begin{align}
    \label{eq:w:approx-eqn} \det DYw_m &= \frac{(1+\epsilon)f(\cdot)}{f^*(Yw_m(\cdot))}  \quad \text{ in }B_r\\
    \label{eq:w:approx-bc} w_m &= u_m \quad\quad\quad\quad\quad\text{ on }\partial B_r.
  \end{align}

  \textit{Step 2. Solvability of \eqref{eq:w:approx-eqn} subject to \eqref{eq:w:approx-bc}}\\
  To begin we show, using the method of continuity, that \eqref{eq:w:approx-eqn} subject to \eqref{eq:w:approx-bc} has a solution $w_m \in C^{4,\alpha}(\overline{B_r})$. As the start point of the method of continuity, set
  \begin{equation}
    \label{eq:v-def}
    v = u_m + \frac{k}{2}(|x|^2-r^2),
  \end{equation}
  for $k$ to be chosen. Using a Taylor series we compute
  \begin{align*}
    \det [D^2v - A(\cdot,v,Dv)] &= \det \big[D^2u_m - A(\cdot,u_m,Du_m)\\
    &\quad+k(I - A_{ij,u}(|x|^2-r^2)-A_{ij,p_k}\cdot x)\big],\\
    B(\cdot,u,Du) &= B(\cdot,u_m,Du_m)+\frac{k}{2}B_u(|x|^2-r^2)+kB_{p_k}\cdot x.
  \end{align*}
  We can ensure, by a choice of first $k$ large then $r$ and subsequently $|x|$ small, that $(x,v(x),Dv(x)) \in \mathcal{U}$ and, in addition,
  \begin{equation}
    \label{eq:w:strict-ineq}
       \det[D^2v-A(\cdot,v,Dv)] \geq B(\cdot,v,Dv)+c,
  \end{equation}
  for some small $c$. Moreover, again by $k$ large then $r$ small, we have $D^2v,D^2v-A(\cdot,v,Dv)\geq0$. This ensures convexity and, by Lemma \ref{thm:g:locglob}, $g$-convexity.  We consider for each $t \in [0,1]$ an elliptic solution $w \in C^{4,\alpha}(\overline{B_r})$ of 
  \begin{align}
    \label{eq:w:6}
    &\det [D^2w - A(\cdot,w,Dw)] = tB(\cdot,w,Dw)\\
 \nonumber   &\quad\quad+ (1-t)\det [D^2v - A(\cdot,v,Dv)]\text{ in }B_r,\\
    &w = u_m \text{ on }\partial B_r.
  \end{align}
  For $t=0$ the problem is solved, by construction, by $v$. We obtain a solution for $t=1$ by showing the set of $t$ for which \eqref{eq:w:6} is solvable is open and closed in $[0,1]$. Closure follows from Arzel\`a--Ascoli and uniform estimates in $C^{4,\alpha}$. These, in turn, follow from $C^2$ estimates along with the Evans-Krylov boundary estimates \cite[Theorem 17.26']{GilbargTrudinger01} and elliptic Schauder estimates (implicit in \cite[Lemma 6.19]{GilbargTrudinger01}). So we're left to obtain the $C^2$ estimates. These follow from Theorem \ref{thm:w:pogorelov_boundary} provided we have $C^1$ estimates. The function $v$ satisfies the barrier requirement for Theorem \ref{thm:w:pogorelov_boundary}. More precisely, for the barrier requirement, note $A,B$ are bounded functions but $D^2v$ is as large as desired (by a choice of $k$ large). Finally, the $C^1$ estimates are as follows. 

  First we show $w \geq v$. Combining \eqref{eq:w:strict-ineq} and we have \eqref{eq:w:6}
  \begin{align*}
    0 &\leq \log\det[D^2v-A(\cdot,v,Dv)] - \log\det[D^2w-A(\cdot,w,Dw)] \\
    &\quad\quad\tilde{B}(\cdot,v,Dv) - \tilde{B}(\cdot,w,Dw)),
  \end{align*}
  where 
  \[ \tilde{B}(x,u,p):= -\log\big[tB(x,u,p)+(1-t)\det[D^2v-A(\cdot,v,Dv)]\big]. \]
So by linearising\footnote{For details on the linearisation see Lemma \ref{lem:u:harn}.}
  \begin{align*}
    0 \leq a^{ij}D_{ij}(v-w) + b^kD_k(v-w)+c(v-w)
  \end{align*}
  where, with $w_\tau := \tau v+(1-\tau)w$
  \begin{align*}
    a^{ij} &= [D^2w-A(\cdot,w,Dw)]^{ij},\\
    b^i &= -a^{ij}A_{ij,p_k}(\cdot,w_\tau,Dw_\tau)+\tilde{B}_{p_k}(\cdot,w_\tau,Dw_\tau),\\
    c &= -a^{ij}A_{ij,u}(\cdot,w_\tau,Dw_\tau)+\tilde{B}_{u}(\cdot,w_\tau,Dw_\tau).
  \end{align*}
  Using the maximum principle on small domains  we obtain $w \geq v$. The precise form of the maximum principle is given in Lemma \ref{lem:a:mp}. It requires a further choice of $r$ small (importantly this choice is independent of $m$), and the observation $\det Dw >c$ and subsequently $\text{trace}{(w)} \geq c$ for a small constant independent of $m$. 

  Now, $w \geq v$ in $B_r$ implies
  \[ Dw(\partial B_r) \subset Dv(\overline{B_r}),\]
  using the convexity of $v$. Indeed we note any tangent plane to $w$ on the boundary is either already a tangent plane to $v$ or can be made one by shifting down and using the convexity. Now, $\det DYw \neq 0$ in $B_r$ implies $|Yw|^2$ cannot have an interior max (at an interior max $DYw = 0$). However  $Yw = Y(x,w(x),Dw(x))$ and both $w$ and $Dw$ are under control on the boundary. To obtain an upper bound for $w$ the positivity of $D^2w-A(\cdot,w,Dw)$ implies
  \[ \Delta w \geq \sum_{i}g_{ii}(\cdot,Yw,Zw).\]
  Thus for $k$ large depending only on $\Vert g \Vert_{C^2}$ we have that $\Delta(w+K|x|^2) \geq 0$ and the maximum principle for Laplace's equation implies $w \leq C$ depending only on $r,K$ and $\sup |u_m|$ which, for $m$ large, can be estimated independently of $m$. We've established estimates for $|Yw|$ and $|w|$ independent of $m$. Writing
  \[ Dw = g_{x}(x,Yw(x),g^*(x,Yw(x),w(x))),\]
  we obtain the desired $\Vert w \Vert_{C^1(\overline{B_r})}$ estimates.

  The set of $t$ for which \eqref{eq:w:6} is open follows by the implicit function theorem in Banach spaces provided we have unique solvability of the associated linearized operators \cite[Theorem 17.6]{GilbargTrudinger01}. The Schauder theory (say in the form \cite[Theorem 6.15]{GilbargTrudinger01}) reduces the solvability to uniqueness. The uniqueness, however, holds by the same maximum principle as before (Lemma \ref{lem:a:mp}). Here again we are using the smallness of the balls $B_r$, though the required choice of $r$ is independent of $m$.

  Thus we conclude the solvability of \eqref{eq:w:approx-eqn} subject to \eqref{eq:w:approx-bc}

  \textit{Step 3. Convergence of approximating solution}\\
  We take $m \rightarrow \infty$. Immediately we obtain, from the uniform $C^1$ bounds, that there is a subsequence of the $w_m$ that converges uniformly to an Aleksandrov solution $w$ of \eqref{eq:w:approx-eqn} with $w =u$ on $\partial B_r$. The comparison principle implies $w \leq u$ in $B_r$ (using crucially the factor $(1+\epsilon)f$). Now fix any $x_0 \in B_r$ and a support $g(\cdot,y_0,z_0)$ of our original solution $u$. Strict $g$-convexity implies for small $h$ the section $G^h_u := \{ u < g(\cdot,y_0,z_0-h)\}$ is strictly contained in $B_r$. Moreover since $w_{m}$ converges to $w \leq u$ the sections $G^h_{w_m},G^h_w$ are also strictly contained (for $m$ large). Theorem \ref{thm:w:pogorelov_strictly} implies local $C^2$ estimates for the $w_m$ that are stable under the convergence. 
  Then we obtain, via Theorem \ref{thm:w:pogorelov_strictly}, interior $C^2$ estimates for $w_m$ independent of $m$. Subsequently via the elliptic regularity theory and Arzela--Ascoli $w \in C^3(B_r)$.
  At this point we send $\epsilon \rightarrow 0$ (our estimates are independent of $\epsilon$) and obtain a $C^3(B_r)$ solution of the Dirichlet problem equal to $u$ on $\partial B_r$.
  
\textit{Step 4. Approximating solution is our original solution (i.e. uniqueness)}\\
We need to show $w \equiv u$. The inequality $w \leq u$ follows from our construction. We suppose for a contradiction $w < u$ somewhere in $B_r.$ Then, provided $\epsilon >0$ is sufficiently small, the same is true for
\[ u_\epsilon := u + \frac{\epsilon}{2}(|x|^2-r^2). \]
That is, $\Omega' := \{w < u_\epsilon\}$ is nonempty and a subset of $B_r$ (not excluding the possibility $\Omega' = B_r$). At any point of $C^2$ differentiability
\begin{align}
\label{eq:w:use-later}  \det DYu_\epsilon &= \det E^{-1}(\cdot,u_\epsilon,Du_\epsilon)[D^2u + \epsilon I - A(\cdot,u_\epsilon,Du_\epsilon)],
\end{align}
and $|u - u_\epsilon|,|Du- Du_\epsilon| \leq C\epsilon r$. Thus the $\epsilon I$ term means provided $r$ was initially chosen small, depending on the derivatives of $f,f^*,$ and $g$, we obtain the strict inequality
\begin{equation}
  \label{eq:w:aleksandrov-contradict}
   f^*(Yu_\epsilon)\det DYu_\epsilon > f^*(Yu) \det DYu. 
\end{equation}
Now the $g$-Monge--Amp\`ere measure is concentrated on the set of points where $u$ is twice differentiable. Here we are using the semiconvexity to convclude the Aleksandrov twice differentiability almost everywhere. Thus if $E_u$ denotes the measure 0 set of points where $u$ is not twice differentiable we have
\[ \int_{Yu(E\setminus E_u)} f^* = \int_{Yu(E)}f^*.\]
So \eqref{eq:w:aleksandrov-contradict} holds in the Aleksandrov sense. Since, in addition, $\Omega'$ is nonempty we contradict Corollary \ref{cor:w:comp} and see we must have $w=u$. This completes the proof of $C^3$ differentiability for $u$. 
\end{proof}

\clearpage{}
\clearpage{}\chapter{Strict convexity and its consequences}
\label{chap:sc}

In the close of the previous chapter we emphasized the importance of strict $g$-convexity. Here we prove the strict $g$-convexity of Aleksandrov solutions of the second boundary value problem. Apart from the two dimensional case, which we pay special attention to, the ideas in this chapter originated in the works of Caffarelli \cite{Caffarelli,Caffarelli92a}. His results were extended to the optimal transport case, by Figalli, Kim and McCann \cite{FKM13}, Guillen and Kitagawa \cite{GuillenKitagawa15}, V\'etois \cite{Vetois15}, and Chen and Wang \cite{ChenWang16}. Each proves the strict $c$-convexity under different hypothesis. Guillen and Kitagawa have extended their results to Generated Jacobian equations \cite{GuillenKitagawa17}.

In this chapter we follow the ideas of Chen and Wang, though adapted to the GJE case, and prove the strict $g$-convexity under their hypothesis. These hypothesis are natural and seem almost optimal: $g^*$-convexity of the target (which is necessary), and that the source domain $\Omega$ is compactly contained in a $g$-convex domain. To emphasize, there is a convexity condition on a domain containing $\Omega$, but no convexity condition on $\Omega$. An immediate corollary is the strict $g$-convexity when $\Omega$ is uniformly $g$-convex and the generating function is defined on a domain of $x$ values containing $\overline{\Omega}$. This is the form in which we'll use the result for applications to global regularity in Chapters \ref{chap:r1} and \ref{chap:r2}. Guillen and Kitagawa have also proved the strict convexity for GJEs. Our domain hypothesis are weaker, but an advantage of their result is it only requires a $C^2$ generating function. 

This is a long chapter. Here's an outline. In Section \ref{sec:sc:transformations-1} we introduce a  transformation  of the coordinates and generating function. These make the generating function asymptotically close to $g(x,y,z) = x\cdot y - z$. In Section \ref{sec:sc:g-cones-subd} we introduce an object called a $g$-cone and derive estimates for its $Y$ mapping. We use these in Section \ref{sec:uniform-estimates} to derive estimates for the size of sections in terms of their height. These estimates are the crucial tool used to derive strict $g$-convexity in Section \ref{sec:sc:strict-conv}. We prove some stronger two dimensional results in Section \ref{sec:sc:2d} and conclude with $C^1$ differentiability as a consequence of strict convexity in Section \ref{sec:sc:furth-cons-strict}.

\section{Transformations}
\label{sec:sc:transformations-1}
In this section we show that in appropriate coordinates the generating function is close to $x\cdot y -z $. Assume $x_0 \in U,y_0 \in V,u_0\in J$ and $h\geq0$ are given. For context, we usually have a $g$-convex function $u:\Omega \rightarrow \mathbf{R}$ and consider $x_0 \in \Omega$ with $y_0 \in Yu(x_0),$ $u_0 = u(x_0)$ and $h$ a shift of the support. Without loss of generality $x_0,y_0,u_0=0$.  Set $z_h=g^*(0,0,h)$. After replacing $g$ by the function which maps  $(x,y,z)$ to $ g(x,y,z+g^*(0,0,h))$ we assume $z_h = 0$. Furthermore by working in the coordinates $y' := E(0,0,0)y$  we have $E(0,0,0) = \text{Id}$. (We recall $E$ is the matrix from assumption A2)

\subsection*{Transformed coordinates}
\label{sec:sc:transf-coord}

Define
\begin{align}
  q(x) &:= g_z(0,0,0)\left[\frac{g_y}{g_z}(x,0,0)-\frac{g_y}{g_z}(0,0,0)\right],\label{eq:sc:xdef}\\
  p(y) &:= g_x(0,y,g^*(0,y,h))-g_x(0,0,0). \label{eq:sc:ydef}
\end{align}
Conditions A1,A1$^*$, and A2  imply $x \mapsto q(x)$ and $y\mapsto p(y)$ are diffeomorphisms, so we may write $q=q(x)$, or $x=x(q)$ as necessary, similarly for $y$ and $p$. The Jacobian of the first transform is
\[ \frac{\partial q_i}{\partial x_j} =  \frac{g_z(0,0,0)}{g_z(x,0,0)}E_{ji}(x,0,0).\]
Because $g_z \det E \neq 0$ on $\overline{\Gamma}$ there is a constant $C_{xy} > 0$ such that for any set $D$ in the $x$ coordinates and $D_q$ its image in the $q$ coordinates there holds
\begin{equation}
  \label{eq:sc:xxbarest}
  C_{xy}^{-1}|D| \leq | D_q|\leq C_{xy}|D|. 
\end{equation}
The same estimate hold for the $y$ to $p$ transformation.

\subsection*{Generating function transformation}
\label{sec:sc:gener-funct-nearly}

Set
\[ \tilde{g}(x,y,z) = \frac{g_z(0,0,0)}{g_z(x,0,0)}[g(x,y,g^*(0,y,h-z)) - g(x,0,0)] ,\]
and subsequently \index[notation]{$\tilde{g}$, \ \ transformed generating function}\index[notation]{$\overline{g}$, \ \ transformed generating function}
\[ \overline{g}(q,p,z) = \tilde{g}(x(q),y(p),z),\]
where $x,q$ and $y,p$ satisfy \eqref{eq:sc:xdef} and \eqref{eq:sc:ydef} respectively. As motivation note in the optimal transport case, where $g(x,y,z) = c(x,y)-z$ is a cost function, we have
\[ \tilde{g}(x,y,z) = [c(x,y)-c(0,y)] - [c(x,0)-c(0,0)]-z,\]
which is a frequently used transformation \cite{FKM13,LTW15,ChenWang16} and is the inspiration for $\overline{g}$. We note a different transformed generating function is used by Jhaveri \cite{Jhaveri17}. 

The key facts concerning the function $\overline{g}$ are summarized in the following lemma.
\begin{lemma}\label{lem:sc:transform_facts}
  Let $g$ be a generating function satisfying A3w. Then:\\
  (1) $\overline{g}$ is a $C^3$ generating function satisfying A3w.\\ 
  (2) A function $u$ is $g$-convex if and only if the corresponding function
  \begin{equation}
    \label{eq:sc:u_transform}
         \overline{u}(q):= \frac{g_z(0,0,0)}{g_z(x(q),0,0)}[u(x(q)) - g(x(q),0,0)], 
  \end{equation}
    is $\overline{g}$-convex. Moreover, with $\overline{Y}$ defined for $\overline{g}$ just as $Y$ was for $g$, we have  $y \in Yu(x)$ if and only if $p \in \overline{Y}\overline{u}(q)$.\\
  (3) We have the following expansion for $\overline{g}$
    \begin{align}
      \label{eq:sc:genexp}  &\overline{g}(q,p,\overline{z}) =q\cdot p-z +a_{ij,kl}(q,p)q_iq_jp_kp_l  \\
     &+z[b_{ij}(q,p)q_iq_j+c_{ij}(q,p)q_ip_j+d_{ij}(q,p)p_ip_j]+f(x,y,z)z^2. \nonumber
    \end{align}
    Here the functions $a$ through $f$ represent remainder terms of Taylor series. Using the integral expression for remainder terms these are $C^1$. 
\end{lemma}
\begin{proof}
  Point (2) follows from a direct calculation. Indeed any $g$-support, $g(\cdot,y,z)$, of $u$ gives rise to a $\overline{g}$-support of $\overline{u}$ of the form $\overline{g}(\cdot,p(y),\overline{z})$ for $\overline{z}=h-g(0,y,z)$. Similarly a $\overline{g}$ support of a function $\overline{u}$ gives rise to a $g$-support of the function $u$ defined by solving \eqref{eq:sc:u_transform}.  \\
  To show (1) we verify conditions A1,A1$^*$,A2, and A3w. Beginning with A1, we note 
  \begin{align}   
  \label{eq:sc:newgx}  \tilde{g}_x(x,y,z) &= \frac{-g_{xz}(x,0,0)}{g_z(x,0,0)}\tilde{g}(x,y,z) \\
\nonumber   &\quad\quad+ \frac{g_z(0,0,0)}{g_z(x,0,0)}[g_x(x,y,g^*(0,y,h-z))-g_x(x,0,0)],
  \end{align}
  so that \textit{for fixed} $x$ the mapping $(y,z) \mapsto(\tilde{g}(x,y,z),\tilde{g}_x(x,y,z))$ is injective by A1. Similarly for A1$^*$ by computing
  \begin{align*}
    -\frac{\tilde{g}_y}{\tilde{g}_z}(x,y,z) = \frac{1}{g^*_u(0,y,h-z)}\frac{g_y}{g_z}(x,y,g^*(0,y,h-z)) + \frac{g^*_y}{g^*_u}(0,y,h-z),
  \end{align*}
  we see \textit{for fixed} $(y,z)$ the mapping $x\mapsto \frac{\tilde{g}_y}{\tilde{g}_z}(x,y,z)$ is injective by A1$^*$. We use the notation $\overline{Y}(q,U,P),\overline{Z}(q,U,P)$ to denote $\overline{Y},\overline{Z}$ solving
  \begin{align*}
    \overline{g}(q,\overline{Y}(q,U,P),\overline{Z}(q,U,P)) = U,\\
    \overline{g}_q(q,\overline{Y}(q,U,P),\overline{Z}(q,U,P)) =P .
  \end{align*}

  The calculation $\tilde{g}_z<0$ implies $\overline{g}_z<0$. Moreover to check $\det \overline{E} \neq 0$ it suffices to check $ \det D_P\overline{Y} \neq 0$ (recall \eqref{eq:g:e-yp}). However this follows by computing $\overline{Y}$ in terms of $Y$ and subsequently $D_P\overline{Y}$ in terms of $D_pY$. Indeed by direct calculation
  \begin{align}
 \nonumber \overline{Z}(q,U,P) = h - &g\left[0,\overline{Y}(x,U,P),g^*\left(x,\overline{Y}(x,U,P),\frac{g_z(x,0,0)}{g_z(0,0,0)}U+g(x,0,0)\right)\right]\\
   \label{eq:sc:ny} \overline{Y}(q,U,P) &= p\Big[Y\Big(x,\frac{g_z(x,0,0)}{g_z(0,0,0)}U+g(x,0,0),\\
   \nonumber     &\quad\quad\frac{g_z(x,0,0)}{g_z(0,0,0)}\frac{\partial q}{\partial x}P+\frac{g_{x,z}(x,0,0)}{g_z(0,0,0)}U+g(x,0,0)\Big)\Big]
  \end{align}
  So the A2 condition follows. The key point is that, despite the unwieldy expression, $\overline{Y}(q,U,P) = p(Y(x,l_1(U),l_2(P)))$ for some function $l_2(P)$  which is linear in P. 

  We found verifying A3w by direct calculation difficult \footnote{It was pointed out by one of the examiners of this thesis that this result has already appeared in the work of Zhang\cite{Zhang18}. In fact Zhang proves the more powerful, and conceptually simpler result, that A3w is tensorial, and thus coordinate independent (proved for cost functions by Kim and McCann \cite{KimMcCann10}). }
  . However, by the work of Loeper and Trudinger \cite[Theorem 2.1]{LoeperTrudinger21} it suffices to verify that $\overline{g}$ satisfies the Loeper maximum principle. That is we need to verify for each $q,q',U$
  \begin{align*}
   \overline{g}(q',&\overline{Y}(q,U,P_\theta),\overline{Z}(q,U,P_\theta)) \\&\leq \max\{\overline{g}(q',\overline{Y}(q',U,P_0),\overline{Z}(q,U,P_0)),\overline{g}(q,\overline{Y}(q,U,P_1),\overline{Z}(q,U,P_1))\},
  \end{align*}
whenever $\{P_\theta\}_{\theta \in[0,1]} $ is a line segment for which the above quantities are well defined. This follows from a direct calculation using the definition of $\overline{g}$ and \eqref{eq:sc:ny}.
  
Finally we prove the expansion in \eqref{eq:sc:genexp}. First, a Taylor series in $z$, and then another in $(p,q)$ implies
  \begin{align}
  \label{eq:sc:1-ineq}  \overline{g}&(q,p,z) = \overline{g}(q,p,0)+\overline{g}_z(q,p,0)z+\frac{1}{2}\overline{g}_{zz}(q,p,\tau z)z^2\\
\nonumber                        &= \overline{g}(q,p,0) + \overline{g}_z(0,0,0)z+\overline{g}_{q_i,z}(0,0,0)q_iz+\overline{g}_{p_i,z}(0,0,0)p_iz\\
\nonumber    &+z[b_{ij}(q,p)q_ip_j+c_{ij}(q,p)q_jq_i+d_{ij}(q,p)p_jp_i]+f(q,p,z)z^2.
  \end{align}
  Here $b,c,d,f$ arise as Taylor series remainder terms, and are $C^1$ by the integral form of the remainder term. Using $\overline{g}_z(0,p,0) = \overline{g}_z(q,0,0) = -1,$ and subsequently  $\overline{g}_{q_i,z}(0,0,0) = \overline{g}_{p_i,z}(0,0,0) = 0$, this simplifies further to
  \begin{align}
    \nonumber    \overline{g}(q,p,z) &=  \overline{g}(q,p,0)  -z\\
    &\quad\quad+z[a_{ij}(q,p)q_ip_j+b_{ij}(q,p)q_jq_i+c_{ij}(q,p)p_jp_i]+d(q,p,z)z^2. \label{eq:sc:sub-to-fin}
  \end{align}

  We're left to deal with the term $\overline{g}(q,p,0)$. Set $\tilde{c}(x,y) := \tilde{g}(x,y,0)$ and $c(q,p) = \tilde{c}(x(q),y(p))$. The virtue of this highly suggestive notation is that $c$ is a cost function as in optimal transport. That is, by freezing a height we can regard $g$ as a cost function and reuse some calculations from the optimal transport case. Note $\tilde{c}$ satisfies $\tilde{c}_x(0,y) = p$ and $\tilde{c}_y(x,0) = q$ along with $\tilde{c}(x,0) \equiv 0, $ $\tilde{c}(0,y) \equiv 0$ and $\tilde{c}_{i,j}(0,0) = \delta_{ij}$. In particular these imply
  \begin{align}
 \label{eq:sc:r1}   &c(0,0) = 0 &&c_{q}(0,0) = 0 &&&c_{p}(0,0) = 0\\
\label{eq:sc:r2}    &c_{p_i,q_j}(0,0) = \delta_{ij} && c_{qq,p}(q,0) = 0 &&& c_{q,pp}(0,p) = 0.
  \end{align}
  Thus via a Taylor series
  \begin{align}
   \nonumber    c(q,p) &= c(0,0) + c_{q_i}(0,0)q_i + c_{p_j}(0,0)p_j\\
    \label{eq:sc:t1}    &\quad\quad +\frac{1}{2}c_{q_iq_j}(tq,tp)q_iq_j+c_{q_i,p_j}(tq,tp)q_ip_j + \frac{1}{2}c_{p_ip_j}(tq,tp)p_ip_j.
  \end{align}
  Further Taylor series yield 
  \begin{align}
    \label{eq:sc:t2}  c_{q_iq_j}(tq,tp) &=  c_{q_iq_j}(tq,0)+ c_{q_iq_j,p_k}(tq,0)p_k+a^{(1)}_{ij,kl}(q,p)p_kp_l,\\
  \label{eq:sc:t3}     c_{p_ip_j}(tq,tp) &=  c_{p_ip_j}(0,tp)+ c_{q_k,p_ip_j}(0,tp)q_k+a^{(2)}_{ij,kl}(q,p)q_kq_l
  \end{align}
  and also
  \begin{align}
 \nonumber c_{q_i,p_j}(tq,tp) &= c_{q_i,p_j}(0,0) + c_{q_iq_k,p_j}(\tau q,\tau p)q_k + c_{q_i,p_jp_k}(\tau q, \tau p)p_k\\
    \label{eq:sc:t4}    &= c_{q_i,p_j}(0,0) + c_{q_iq_k,p_j}(\tau q,0)q_k +a^{(3)}_{ij,kl}(q,p)q_kp_l \\
                              &\quad\quad+ c_{q_i,p_jp_k}(0, \tau p)p_k + a^{(4)}_{ij,kl}(q,p)q_kp_l.  \nonumber 
  \end{align}
  Using the integral form of the remainder term the $a$ terms are $C^1$. Now combining \eqref{eq:sc:t1}-\eqref{eq:sc:t4} and using the relations \eqref{eq:sc:r1} and \eqref{eq:sc:r2} we obtain
  \begin{align*}
    \overline{g}(q,p,0) = c(q,p) = q \cdot p + a_{ij,kl}q_iq_jp_kp_l.
  \end{align*}
  Substituting into \eqref{eq:sc:sub-to-fin} concludes the proof. 
\end{proof}

\section{$g$-cones}
\label{sec:sc:g-cones-subd}
Cones are a basic tool for obtaining estimates in the theory of Monge--Amp\`ere equations. A similar function was introduced in the optimal transport setting by Figalli, Kim, and McCann \cite[\S 6.2]{FKM13}. The defining feature of this so-called $c$-cone is that its $Y$ mapping is concentrated at a point, that is for some $x_0$ $Yu(x_0) = Yu(\Omega)$. The generalization to $g$-cones is due to Guillen and Kitagawa \cite{GuillenKitagawa17}. In each case we want estimates for the $Y$-mapping of the generalized cone in terms of the measure of its base and its height. These estimates were given by Guillen and Kitagawa. Here we offer an alternate derivation which follows closely that of Chen and Wang though, necessarily, uses the expansion \eqref{eq:sc:genexp}.

Let $u:\Omega \rightarrow \mathbf{R}$ be a $g$-convex function. Assume $x_0 \in \Omega, y_0 \in V$ are given and $u_0:=u(x_0)$. For $h > 0$ set $z_h = g^*(x_0,y_0,u_0+h)$ and assume
\begin{equation}
  \label{eq:sc:D-def}
  D := \{x \in \Omega;u(x) < g(x,y_0,z_h) \} \subset\subset \Omega.
\end{equation}
We define the $g$-cone with vertex $(x_0,u_0)$ and base $\{(x,g(x,y_0,z_h)); x \in \partial D\}$ by  \index{$g$-cone}\index[notation]{$\vee$, $\vee_{D,h}$, \ \ $g$-cone}
\begin{align}
  \label{eq:sc:2vdef}
  \vee(x) &= \sup\{\phi_y(x) := g(x,y,g^*(x_0,y,u_0)) ;\\
  &\quad\quad\quad  \phi_y(x) \leq g(x,y_0,z_h) \text{ on }\partial D\}. \nonumber
\end{align}
This function depends on $x_0,y_0,u_0,h$ and $D$. When we need to emphasize some of these dependencies we  include them as a subscript, e.g. $\vee_{D,h}$ if $x_0,y_0,u_0$ are clear from context. The expression \eqref{eq:sc:2vdef} does not require $D$ arise as a section, like in \eqref{eq:sc:D-def}. However because $D$ is given by \eqref{eq:sc:D-def} we have $\vee = g(\cdot,y_0,z_h)$ on $\partial D$ provided $Yu(\Omega)$ is $g^*$-convex with respect to $x_0,u_0+h$.

Our goal is to estimate $Y \vee_{D,h}(x_0)$ in terms of $D$ and $h$. As in Section \ref{sec:sc:transformations-1} we assume, without loss of generality, that $x_0,y_0,u_0,z_h = 0$. 

Using Lemma  \ref{lem:sc:transform_facts}(2) it suffices to work in the coordinates given by \eqref{eq:sc:xdef}, \eqref{eq:sc:ydef} and estimate the $\overline{Y}$ mapping of
\[ \overline{\vee}(q) := \frac{g_z(0,0,0)}{g_z(x(q),0,0)}[\vee(x(q)) - g(x(q),0,0)].\]
 By direct calculation we see $\overline{\vee}$ is the $\overline{g}$-cone with base $\partial D_q\times \{0\}$ and vertex $(0,-h)$ (recall $D_q$ is the image of $D$ in the $q$ coordinates). Thus
\begin{equation}
  \label{eq:sc:simple}
   \overline{\vee}(q) = \sup\{\phi_p(q) := \overline{g}(q,p,h); \phi_p \leq 0 \text{ on }D_q\},
\end{equation}
and $D_q$ is convex.

To simplify our notation we switch back to $x,y,g,\vee$, though now for a generating function with the expansion \eqref{eq:sc:genexp} and $\vee$ defined as in \eqref{eq:sc:simple}.  

\begin{lemma}\label{lem:sc:upper_est}
  Suppose $g$ is a generating function of the form \eqref{eq:sc:genexp} satisfying A3w and  $D$ is  a convex domain containing 0. Suppose $\vee$ is as defined in \eqref{eq:sc:simple} and $K$ is the cone with vertex $(0,-h)$ and base $\partial D\times \{0\}$. There exists $d_0,h_0,C>0$ such that if $\text{diam}(D) \leq d_0$ and $h \leq h_0$ then
  \begin{equation}
    \label{eq:sc:y-upper}
    Y \vee(0) \subset 2\partial K(0). 
  \end{equation}
Where $C,K$ depend on $\Vert g \Vert_{C^4},\text{diam}(D),\text{diam}(V)$. 
\end{lemma}

\begin{proof}
    We prove the transformed generating function satisfies
  \begin{equation}
    \label{eq:sc:almost-affine}
     g(x,y,h) \geq \frac{3}{4} x \cdot y - \frac{3h}{2},
  \end{equation}
  for $\vert x \vert,h$ sufficiently small and $x \cdot y ,h>0$,\footnote{we note if $x \cdot y$ or $h<0$ \eqref{eq:sc:almost-affine} holds with $3/4$ replaced by $5/4$ or $3/2$ replaced by $1/2$. }; \eqref{eq:sc:y-upper} is a straightforward consequence. Indeed, take $y \in Y\vee(0)$ and suppose $y \notin 2\partial K(0)$, that is $x \cdot y > 2h$ for some boundary point $x \in \partial D$. By \eqref{eq:sc:almost-affine} $g(x,y,h) > 0$ and so $g(\cdot,y,h)$ can not be a support of $\vee$.

Take $y \in V$ and rotate so $y = (0,\dots,0,y_n)$. Let $x = (x_1,\dots,x_n) \in D$ satisfy $x_ny_n >0$ and set $x'=(x_1,\dots,x_{n-1},0)$. We assume, for now, $g(x',y,0) \geq 0$ (we'll see this is a consequence of LMP). Now, \eqref{eq:sc:genexp} implies
  \begin{equation}
   g_{x_n}(x_\tau,y,0)x_n \geq x_ny_n - K\vert x \vert  x_n y_n ,\label{eq:sc:x-est-cone}
 \end{equation}
  for $x_\tau= \tau x + (1-\tau)x'$ and $\tau \in [0,1]$ where $K$ depends on $\Vert g \Vert_{C^3}$. Since the transformed generating function satisfies $g_z(0,p,0)= -1$ a choice of $\text{diam}(D),h$ sufficiently small implies
  \begin{align*}
   g(x,y,h) \geq g(x,y,0) - \frac{3h}{2}.
\end{align*}
A Taylor series for $h(t):= g(tx+(1-t)x',y,0)$, our assumption $g(x',y,0) \geq 0$, and  \eqref{eq:sc:x-est-cone} imply
\begin{align*}
g(x,y,h) &\geq g(x',y,0)+g_{x_n}(x_\tau,y,0)x_n-\frac{3}{2}h \geq x_ny_n(1- K \vert x \vert ) - \frac{3}{2}h.
  \end{align*}
  Choosing $\text{diam}(D)$ small to ensure $K \vert x \vert  \leq 1/4$ we obtain \eqref{eq:sc:almost-affine}.

  To conclude we show $g(x',y,0) \geq 0$. Since $x' \cdot y = 0$ it suffices to show whenever $x \cdot y > 0$ then $g(x,y,0) \ge 0$ and use continuity.  Note if $x \cdot y > 0$ the expression \eqref{eq:sc:genexp} implies
  \[ g(tx,y,0) > 0 \text{ and } g(-tx,y,0) < 0,\]
  for $t > 0$ sufficiently small. If $g(x,y,0) < 0$ then the $g$-convexity of the section $\{g(\cdot,y,0) < 0 = g(\cdot,0,0)\}$ with respect to $0,0$ (which is just convexity), is violated. So as required $ g(x,y,0) \ge 0$.
\end{proof}

The estimates in the other direction are formulated differently. As motivation consider the rectangle
  \begin{equation}
   R = \{ x \in \mathbf{R}^d; -b_i \leq x_i \leq a_i\},\label{eq:sc:R}
 \end{equation}
  for $a_i,b_i>0$,  and cone
 \[ K(x) = \sup\{ l(x):=p \cdot x - h; l(x) \leq 0 \text{ on }\partial R \}. \]
 Then $\partial K(0)$ contains the points $he_i/a_i,-he_i/b_i$. Thus for
   \begin{equation}
   R^* := \{x \in \mathbf{R}^d; -b_i^{-1} \leq x_i \leq a_i^{-1}\}.\label{eq:sc:Rstar}
 \end{equation}
we have
 \begin{align}
  \label{eq:sc:k-est-1} \partial K(0) &\supset C_n h R^*\\
  \label{eq:sc:k-est-2}   |\partial K(0)| &\geq C_n h^n  \prod_{i=1}^n\left(\frac{1}{b_i}+\frac{1}{a_i}\right). 
 \end{align}
 Next we decrease the base of the cone: consider a domain $D$ with $0 \in D \subset R$ and
 \[ K_D(x) :=\sup\{ l(x):=p \cdot x - h; l(x) \leq 0 \text{ on }\partial D \}. \]
 Because $\partial K(0) \subset \partial K_D(0)$, \eqref{eq:sc:k-est-1} and \eqref{eq:sc:k-est-2} hold for $K_D$. This motivates the following result. 

\begin{lemma}\label{lem:sc:lower_est}
  Suppose $g$ is a generating function of the form \eqref{eq:sc:genexp} satisfying A3w and  $D$ is a convex domain with $0 \in D \subset R$. Let $\vee$ be given by \eqref{eq:sc:simple}. There is $d_0,h_0 > 0$  such that if $\text{diam}(D) \leq d_0$ and $h \leq h_0$ then \eqref{eq:sc:k-est-1} and \eqref{eq:sc:k-est-2} hold with $\vee$ in place of $K$. The quantities $d_0,h_0$ depend on $\text{diam}(V), \Vert g \Vert_{C^4}$, 
\end{lemma}

\begin{proof}[Proof. (Lemma \ref{lem:sc:lower_est})]
  Using convexity it would suffice to show $\partial \vee(0)$ contains the points $Che_i/a_i$ and $-Che_i/b_i$ for $i=1,\dots,n$. This is beyond us, so instead we show $\partial {\vee}(0)$ contains points close to these points. That is, we show for $q := \kappa h e_n/a_n$ for some $\kappa \geq 1/4$ there is $p \in \partial \vee(0)$  satisfying
  \begin{equation}
    \label{eq:sc:pqrel}
      |p-q| \leq \frac{1}{16}|q|.
  \end{equation}
 Our proof also applies to $\kappa h e_i/a_i$ and  $-\kappa h e_i/b_i$ for $i=1,\dots,n$, so  $ChR^* \subset \partial  \vee(0)$.

  To begin, choose $\hat{x}$ realizing $\hat{x}_n = \sup\{x_n ; x = (x_1,\dots,x_n) \in D\}$. We see, by taking a limit of the $\phi_y$ used in \eqref{eq:sc:simple} to define $\vee$, that there is $\hat{y}$ for which $g(\cdot,\hat{y},h)$ supports $\vee$ at $\hat{x}$ and $0$. In particular, since $\vee = 0$ at $\hat{x}$ and is less than or equal to $0$ on $\partial D$ we have, for  $\hat{y}$ appropriately chosen and some $\beta \geq 0$,
  \begin{equation}
    \label{eq:sc:onxn}
       g_x(\hat{x},\hat{y},h) = \beta e_n. 
  \end{equation}
We'll prove that $p = g_x(0,\hat{y},h)$ and $q = (p\cdot e_n)e_n$ satisfy \eqref{eq:sc:pqrel}. 

Choose $d^*$ so that $g(d^*e_n,\hat{y},h) = 0$. We claim $d^*\leq a_n$. Indeed
\[S:= \{x; g(x,\hat{y},h) < 0 = g(x,0,0)\}  \] is convex because it is a section. Furthermore since $g_x(\hat{x},\hat{y},h) = \beta e_n$ and $g(\hat{x},\hat{y},h) = 0$ the plane  $P:=\{x; x_n = \hat{x}_n\}$ is supporting to $S$. Thus, since $S$ contains $0$ and lies on one side of $P$, $S$ is contained in $\{x;x_n \leq \hat{x}_n\}$ and $d^* \leq a_n$.

  Now \eqref{eq:sc:genexp} implies
  \begin{equation}
    \label{eq:sc:est}
        |g_x(x,\hat{y},h) - g_x(0,\hat{y},h)| \leq C(|x|+h)||\hat{y}| + Kh(h+|x|),
      \end{equation}
      where $C,K$ depend on $\Vert g \Vert_{C^4},$ $\text{diam}(V)$ and we assume $\text{diam}(D) \leq1$. Subsequently
  \begin{align}
  \nonumber  h &= g(d^*e_n,\hat{y},h) - g(0,\hat{y},h)\\
  \nonumber    &= d^*g_{x_n}(\tau d^*e_n,\hat{y},h)\\
   \label{eq:sc:d-ineq} &\leq d^* |g_x(0,\hat{y},h)|+Cd^*(d^*+h)|\hat{y}| + Kh(h+d^*).
  \end{align}
  To estimate $|\hat{y}|$ in terms of $|g_x(0,\hat{y},h)|$ write
  \begin{equation}
    \label{eq:sc:haty_est}
       |\hat{y}| = |g_x(0,\hat{y},0)| \leq |g_x(0,\hat{y},h)| + |g_{xz}(0,\hat{y},\tau h)|h. 
  \end{equation}
  Combining \eqref{eq:sc:d-ineq} and \eqref{eq:sc:haty_est} we have
  \[ h \leq d^*|g_x(0,\hat{y},h)|[1+Cd^*(d^*+h)] + Kh(h+d^*).\]
  We choose $\text{diam}(D)$ and $h$ small to ensure both $(1+Cd^*(d^*+h)) \leq 3/2$ and $K(h+d^*) \leq 1/4$. Combining with $d^* \leq a_n$ yields
  \begin{equation}
    \label{eq:sc:subdiff_est}
   \frac{h}{2a_n} \leq   |g_x(0,\hat{y},h)|. 
  \end{equation}

  Using, once again, \eqref{eq:sc:est} (with $x = \hat{x}$) and \eqref{eq:sc:haty_est} we have
  \[  |g_x(\hat{x},\hat{y},h) - g_x(0,\hat{y},h)| \leq C(|\hat{x}|+h)||g_x(0,\hat{y},h)|+Kh(|\hat{x}|+h) . \]
  Dividing through by $|g_x(0,\hat{y},h)|$, using \eqref{eq:sc:subdiff_est} and choosing $h,|\hat{x}|$ sufficiently small we can ensure 
  \[ \left|\frac{g_x(\hat{x},\hat{y},h)}{|g_x(0,\hat{y},h)|}- \frac{g_x(0,\hat{y},h)}{|g_x(0,\hat{y},h)|}\right| \leq 1/16.\]
  The first vector lies on the $e_n$ axis (recall \eqref{eq:sc:onxn}). Thus the unit vector $\frac{g_x(0,\hat{y},h)}{|g_x(0,\hat{y},h)|}$, and consequently $g_x(0,\hat{y},h)$ make angle $\theta$ with the $e_n$ axis for $\theta$ satisfying $\sin(\theta) \leq 1/16$, i.e. $\theta \leq 1/8$. This, with \eqref{eq:sc:subdiff_est} implies both
  \[ e_n \cdot g_x(0,\hat{y},h) = \cos (\theta)|g_x(0,\hat{y},h)| \geq \frac{h}{4a_n},\]
  and
  \[ |g_x(0,\hat{y},h)-(e_n \cdot g_x(0,\hat{y},h))e_n | \leq \sin (\theta) |g_x(0,\hat{y},h)| \leq \frac{1}{16}|(e_n \cdot g_x(0,\hat{y},h))e_n|,\]
  which is \eqref{eq:sc:pqrel}. \end{proof}

We also have an extension to the case when $x_0$ is close to the boundary. We make use of the minimum ellipsoid (see \cite[\S 2.1]{LiuWang15})\index{minimum ellipsoid} and the following lemma due to Figalli, Kim, and McCann. 
\begin{lemma}\cite[Lemma 6.9]{FKM13}\label{lem:sc:fkm}
  Let $D \subset \mathbf{R}^n$ be a convex domain. Assume $D$ contains a ``vertical'' line segment $\{(x',t_0+t); x' \in \mathbf{R}^{n-1} , t \in [0,d]\}$ of length $d$. Let
  \[D': = \{(x_1,\dots,x_{n-1},0); x = (x_1,\dots,x_n) \in D\}\] be the projection of $D$ onto $\mathbf{R}^{n-1}$. There is $C>0$ depending only on $n$ such that
  \[ |D|\geq  Cd \mathcal{H}^{n-1}(D'),\]
  where $\mathcal{H}^{n-1}$ is the $n-1$ dimensional Hausdorff measure. 
\end{lemma}

\begin{lemma}\label{lem:sc:lower_est_close}   Suppose $g$ is a generating function  of the form \eqref{eq:sc:genexp} satisfying A3w and $D$ is a convex domain with $0 \in D$. Let $\vee$ be given by \eqref{eq:sc:simple}.  Assume $0$ is close to the boundary, in the sense that there is a unit vector $\nu$ and positive $d$ such that
  \begin{equation}
   \sup_{x \in D}\langle x,\nu \rangle = \epsilon d, \label{eq:sc:close_cond}
 \end{equation}
 and in addition $D$ contains a line segment of length $d$ parallel to $\nu$. There is $C,d_0,h_0 > 0$ depending on $\text{diam}(V), \Vert g \Vert_{C^4}$, such that if $\text{diam}(D) \leq d_0$ and $h \leq h_0$ then
  \[ h^n \leq  C\epsilon |\partial \vee(x_0)||D|. \]
\end{lemma}
\begin{proof}
  We assume, without loss of generality that $\nu = e_n$. Let \[D' = \{x' = (x_1,\dots,x_{n-1},0); x \in D\}\] be the projection of $D$ onto $\mathbf{R}^{n-1}$. Then up to a choice of the remaining coordinates we assume the minimum ellipsoid of $D'$ (as a subset of $\mathbf{R}^{n-1}$) is
  \[ E':= \{x' = (x_1,\dots,x_{n-1},0); \sum_{i=1}^{n-1}\left(\frac{x_i-\overline{x}_i}{b_i/2}\right)^2 \leq 1\}.\] Then $\mathcal{H}^{n-1}(D') \geq C_n b_1 \dots b_{n-1}$ and
  \[ D \subset [-b_1,b_1]\times \dots \times [-b_{n-1},b_{n-1}]\times [-K,\epsilon d],\]
  for some $K>0$. Then Lemma \ref{lem:sc:lower_est} implies
  \begin{equation}
    \label{eq:sc:cb}
      |\partial \vee(0)| \geq C_n h^n \frac{1}{\epsilon d b_1 \dots b_{n-1}}. 
  \end{equation}
  On the other hand Lemma \ref{lem:sc:fkm} implies
  $|D| \geq C_n d |D'| \geq C_n d b_1 \dots b_{n-1} $, that is,
  \[ \frac{1}{ d b_1 \dots b_{n-1}} \geq \frac{1}{|D|}.\]
  Which combined with \eqref{eq:sc:cb} completes the proof. 
\end{proof}

\section{Uniform estimates}
\label{sec:uniform-estimates}

In this section we consider $g$-convex Aleksandrov solutions of
\begin{align}
\label{eq:sc:caff-style} \lambda&\leq \det DYu \leq \Lambda   \text{ in }D,\\
\label{eq:sc:dir}  u &= g(\cdot,y_0,z_0) \text{ on }\partial D.
\end{align}

Here $\lambda,\Lambda$ are positive constants and $D$ (being a section) is necessarily $g$-convex with respect to $y_0,z_0$.

Our goal is to estimate $|u-g(\cdot,y_0,z_0)|$ in D, that is, estimates on how $u$ separates from its boundary values. Thanks to the $g$-cone estimates our proofs are adaptations of the proofs in the Monge--Amp\`ere case. To use the $g$-cone estimates we assume throughout this section that $\text{diam}(D)$ and $h := \sup |u-g(\cdot,y_0,z_0)|$ are sufficiently small as required by Lemmas \ref{lem:sc:upper_est}, \ref{lem:sc:lower_est} and \ref{lem:sc:lower_est_close}. In practice we will need to check this assumption before using the theorems of this section.

\begin{theorem}\label{thm:sc:upper_uniform}
Assume $g$ is a generating function satisfying A3w.  Assume $u$ is $g$-convex and satisfies \eqref{eq:sc:caff-style}, \eqref{eq:sc:dir}. Then there is $C>0$ independent of $u$ such that
  \begin{equation}
    \label{eq:sc:upper_uniform}
       \sup_{D}|u(\cdot)-g(\cdot,y_0,z_0)|^n \leq C|D|^2.
     \end{equation}
\end{theorem}
\begin{proof} Fix any $x_0$, after translating assumed to be 0. It suffices to obtain these estimates after applying the transformations in Section \ref{sec:sc:transformations-1}. Thus we may assume $D$ is convex, $g$ is given by \eqref{eq:sc:genexp} and $u \equiv 0$ on $\partial D$.   Let the minimum ellipsoid of $D$ be given by
  \[  E := \big\{x; \sum\frac{(x_i-\overline{x}_{i})}{a_i^2} \leq 1\big\}, \]
  for some $\overline{x} \in D$. 
 Note  $D \subset R := \{-2a_i\leq x_i \leq 2a_i\}$ and $c_na_1\dots a_n \leq |D| \leq C_n a_1 \dots a_n.$

  Now consider the $g$-cone from Lemma \ref{lem:sc:lower_est} with $h = -u(0)$. That is, the $g$-cone $\vee$ with vertex $(0,u(0))$ and base $\partial D \times \{0\}$. Lemma \ref{lem:sc:lower_est} implies $|\partial \vee(0)| \geq |u(0)|^n|R^*|$. Moreover $Y\vee(0) = Y(0,\vee(0),\partial \vee(0))$ so that $|Y\vee(0)| \geq C|\partial \vee(0)|$. Subsequently
  \[ |Y\vee(0) | \geq C|u(0)|^n|R^*| .\]

  The comparison principle (Lemma \ref{lem:w:comp}) implies $Y\vee(0)\subset Yu(D)$.  Thus
  \[ |u(0)|^n|R^*| \leq C|Y\vee(0)| \leq C\Lambda|D|.\]
  Since $|R^*| = C_n |R|^{-1},$ and $R$, being defined in terms of the of the minimum ellipsoid, satisfies $C_n|R|\leq |D|$ we obtain \eqref{eq:sc:upper_uniform}. 
\end{proof}

If we instead use Lemma \ref{lem:sc:lower_est_close} we obtain the following:

\begin{theorem}\label{thm:sc:upper_uniform_close} Assume $D$ is a convex domain and $g$ has the form \eqref{eq:sc:genexp} and satisfies A3w. Let $u$ satisfy \eqref{eq:sc:caff-style} and \eqref{eq:sc:dir}. Suppose further that $x_0 \in D$ is close to the boundary in the sense that there is a unit vector $\nu$ and $\epsilon,d>0$ with
  \[ \sup_{x \in D}\langle x-x_0,\nu \rangle = \epsilon d \]
  and $D$ contains a line segment of length $d$ parallel to $\nu$. There is $C>0$ independent of $u$ such that
  
  \[ |u(x_0)-g(x_0,y_0,z_0)|^n \leq C\epsilon|D|^2.\]
\end{theorem}

\begin{remark}\label{rem:sc:ineq}
  For both Theorems \ref{thm:sc:upper_uniform} and \ref{thm:sc:upper_uniform_close} it suffices to have the inequality $u \geq g(\cdot,y_0,z_0)$ on the boundary. In this case we apply the above proofs to $D' = \{ u \leq g(\cdot,y_0,z_0)\} \subset D$. To be precise we obtain the estimate for $\text{min}\{u-g(\cdot,y_0,z_0),0\}$.  Similarly for the lower bound, Theorem \ref{thm:sc:lower_uniform}, it suffices to have $u \leq g(\cdot,y_0,z_0)$.
\end{remark}

The lower bound uses a lemma of Guillen and Kitagawa's \cite{GuillenKitagawa17}. We restate it here using our terminology, and for completeness include its proof after the proof of the lower bounds. 

\begin{lemma}\cite[Lemma 6.1]{GuillenKitagawa17}\label{lem:sc:gk}
 Let $g$ be the transformed generating function from Lemma \ref{lem:sc:transform_facts}. Suppose $u$ is $g$-convex and $D := \{u \leq 0\}$ has $0$ as the centre of its minimum ellipsoid. Set $h = \sup_{D}|u|$. There is constants $C,K > 0$ depending on $g$ such that
  \[ Yu\left(\frac{1}{K}D\right) \subset Y\vee_{\frac{1}{K}D,Ch}(0),\]
  where $\vee_{\frac{1}{K}D,Ch}$ is the $g$-cone with base $\frac{1}{K}D \times \{0\}$ and vertex $(0,-Ch)$.
\end{lemma}

\begin{theorem}\label{thm:sc:lower_uniform}
  Assume $g$ is a generating function satisfying A3w and $u$ is a $g$-convex solution of \eqref{eq:sc:caff-style} subject to \eqref{eq:sc:dir}. There is a constant $C$ independent of $u$ such that
  \[ C|D|^2 \leq \sup_{D}|u-g(\cdot,y_0,z_0)|^n.\]
\end{theorem}
\begin{proof}
  We assume $y_0,z_0 = 0$ then pick any $x_0 \in D$ and make the change
  \eqref{eq:sc:xdef} so that $D$ is convex. After further translation we assume the center of the minimum ellipsoid for $D$ is $0$. Then we make the remaining changes in Section \ref{sec:sc:transformations-1} so  $g$ has the form \eqref{eq:sc:genexp} and $u = 0$ on $\partial D$. 

Let  $\vee = \vee_{\frac{1}{K}D,Ch}$ be the $g$-cone from Lemma \ref{lem:sc:gk} and note 
  \[ \frac{\lambda}{K^n}|D| \leq \left|Yu\left(\frac{1}{K}D\right)\right| \leq |Y\vee(0)|.\]  Let $\hat{K}$ be the classical cone with the same base and vertex as $\vee$. By Lemma \ref{lem:sc:upper_est}
  \begin{equation}
    \label{eq:sc:lower-1}
     \frac{\lambda}{K^n}|D| \leq C|\partial \hat{K}(0)|. 
  \end{equation}
The estimate 
\begin{equation}
  \label{eq:sc:lower-2}
  |\partial \hat{K}(0)| \leq \frac{Ch^n}{|D|},
\end{equation}
is straight forward: Consider the rectangle $R$ whose side lengths correspond to the axis of the minimum ellipsoid. Then the cone $\tilde{K}$ with base $\frac{1}{n\sqrt{n}}R \subset D$ and vertex $(0,-Ch)$ satisfies $\partial \hat{K}(0) \subset \partial \tilde{K}(0)$ and $|\partial \tilde{K}(0)| \leq Ch^n/|D|$.

Combining \eqref{eq:sc:lower-1} and \eqref{eq:sc:lower-2} completes the proof. 
\end{proof}

The proof of Lemma \ref{lem:sc:gk} uses the A3w condition via the quantitative quasiconvexity of Guillen and Kitagawa. More precisely if A3w holds then so does the following statement: Let $x_0,x_1 \in U$ $y_0,y_1 \in V$ and $u_0 \in J$ be given. Let $\{x_\theta\}_{\theta \in [0,1]}$ denote the $g$-segment from $x_0$ to $x_1$ with respect to $y_0,z_0 = g^*(x_0,y_0,u_0)$ and set $z_1 = g^*(x_0,y_1,u_0)$. Then there is $M$ depending only on $g$ such that
  \begin{align}
    g(x_\theta,y_1,z_1) - g(x_\theta,y_0,z_0) \leq M \theta [g(x_1,y_1,z_1) - g(x_1,y_0,z_0)]_{+}.\label{eq:sc:gqq}
  \end{align}
  We prove that this is implied by A3w in Lemma \ref{lem:a:gqq}. There's more to say here: Guillen and Kitagawa proved this condition (actually, a slightly more general one) implies A3w. The interested reader should see \cite{GuillenKitagawa17}. 

\begin{proof}[Proof. (Lemma \ref{lem:sc:gk})] 
  We assume $K$ has been fixed small, to be chosen in the proof, and show there is $C$ such that
  \begin{equation}
    \label{eq:sc:gk-main}
      Yu\left(\frac{1}{K}D\right) \subset Y\vee_{\frac{1}{K}D,Ch}(0).
  \end{equation}

  To this end, fix $x \in \frac{1}{K}D$ and $y \in Yu(x)$. Let $g_0$ be the corresponding support
  \[ g_0(\cdot):= g(\cdot,y,g^*(x,y,u(x))) =g(\cdot,y,|g_0(0)|),\]
  where the second equality is because the transformed generating function satisfies $g^*(0,y,u) = -u$. 
  To prove \eqref{eq:sc:gk-main} it suffices to show there is $C$ (independent of $x,y$) such that
  \begin{equation}
    \label{eq:sc:gk-equiv}
       |g_0(0)| \leq Ch.
  \end{equation}
  For in this case the function $g(\cdot,y,Ch)$ passes through the vertex of $\vee = \vee_{\frac{1}{K}D,Ch}$ and lies below $g_0$ so is nonpositive on $\frac{1}{K}D$. Thus $y \in Y\vee(0)$.

  So let's prove \eqref{eq:sc:gk-equiv}. A Taylor series implies for any $x' \in D$
  \[ g(x',y,0) = g_0(x') - g_z(x',y,z_\tau)|g_0(0)|.\]
  Then for appropriate positive constants $C^+,C^-$ which depend only on $g_z < 0$ we have that for any $x' \in D$
  \begin{align}
     \label{eq:sc:g1-1} g(x',y,0) &\geq g_0(x') + C^-|g_0(0)|\\
   \label{eq:sc:g1-2} g(x',y,0)  &\leq g_0(x') + C^+|g_0(0)| \leq C^+|g_0(0)|.   
  \end{align}
 
  Now let $\{x_\theta\}_{\theta \in [0,1]}$ denote the $g$-segment with respect to $0,0$ that starts at $0$, passes through $x$, and hits $\partial D$ at some $x_1$. Because $x \in \frac{1}{K}D$ there is $\theta' \leq 1/K$ with $x_{\theta'} = x$. So \eqref{eq:sc:gqq} implies
  \begin{align}
       g(x,y,0) \leq M\theta' [g(x_1,y,0)]_+, \label{eq:sc:pos-part}
  \end{align}
  where we've used \eqref{eq:sc:gqq} with $y_1=y$ and $y_0,z_0,u_0 = 0$ so $g(\cdot,y_0,z_0) = 0$.
  If $g(x_1,y,0) \leq 0$, then  $g(x,y,0) \leq 0$ and we obtain \eqref{eq:sc:gk-equiv} from \eqref{eq:sc:g1-1} with $x' = x$ (because $g_0$ is a support at $x$ we have $g_0(x)=u(x) \geq -h$).

  Otherwise, combine \eqref{eq:sc:pos-part} with \eqref{eq:sc:g1-1} on the left hand side and \eqref{eq:sc:g1-2} on the right hand side to obtain
  \[ g_0(x) + C^-|g_0(0)| \leq M\theta' C^+|g_0(0)|. \]
  Recalling $g_0(x) \geq -h$ and choosing $K$ large to ensure $M\theta' C^+ \leq C^-/2$ completes the proof. \end{proof}

\section{Strict convexity assuming a $g$-convex containing domain}
\label{sec:sc:strict-conv}
In this section we prove the strict convexity under Chen and Wang's \cite{ChenWang16} hypothesis from optimal transport. The main tool in their proof is the uniform estimates. Having established these our proof follows theirs.

\begin{theorem}\label{thm:sc:contain}
  Assume $g$ is a generating function satisfying A3w and $\Omega\subset\subset U$. Assume for positive constants $\lambda,\Lambda>0$, $u:\Omega \rightarrow \mathbf{R}$ is an Aleksandrov solution of
  \[ \lambda \leq \det DYu \leq \Lambda  \text{ in }\Omega,\]
  and a generalised solution of \eqref{eq:g:2bvp}. Assume $U$ and $\Omega^*$ are, respectively, $g$ and $g^*$-convex with respect to $u$. Then $u$ is strictly $g$-convex in $\Omega$. 
\end{theorem}

\begin{proof} [Proof (Theorem \ref{thm:sc:contain}).]
  We extend $u$ to $\tilde{u}$ defined on $U$ as 
\[ \tilde{u}(x) := \sup\{g(x,y_0,z_0); g(\cdot,y_0,z_0) \text{ is a $g$-support of $u$ in }\Omega\}.\]
This extension is equal to $u$ on $\Omega$ and satisfies
\begin{align}
  \label{eq:sc:gjeext}  \lambda\chi_{\Omega} \leq \det DY\tilde{u} &\leq \Lambda\chi_{\Omega} ,\\
  \label{eq:sc:2bvpext}Y\tilde{u}(U) &= \overline{\Omega^*}.
\end{align}
We assume this is the solution we are working with, though keep the notation $u$.

  For a contradiction we suppose there is a support $g(\cdot,y_0,z_0)$ such that
  \[ G:= \{x \in  \overline{U}; u(x) = g(x,y_0,z_0)\},\]
  contains more than one point in $\Omega$. The first step of the proof is to show that, after the coordinate transform with respect to $y_0,z_0$, any extreme point of $G$ lies  on $ \partial U$. The second step is to choose a particular extreme point and obtain a contradiction from the fact that it must lie on $\partial U$. 

  \textit{Step 1. Extreme points cannot be in the interior}\\
  Without loss of generality $y_0,z_0=0$. Applying the transformation \eqref{eq:sc:xdef} of the $x$-coordinates we have that $G$ and $U$ are convex. Without loss of generality $0$ is an extreme point of $G$, where, for a contradiction, we assume $0 \in \text{interior}(U)$. After transforming the $y$ coordinates and generating function as in Section \ref{sec:sc:transformations-1} we have
   \[ G  = \{ x;u(x) = 0 = g(x,0,0) \},\]
  $u\geq0$, and $g$ has the form \eqref{eq:sc:genexp}. Let $P$ be a supporting plane of $G$ at the extreme point $0$ and choose appropriate coordinates so that $P = \{x_{1} = 0\}$ and
   \begin{align}
  G \cap P = \{0\},\quad\quad G \subset  \{x; x_1 \leq 0\} \quad \text{ and }\quad -a e_1  \in G,
   \end{align}
   for some $a >0$.
 Set
  \[ G^h := \{u < g(\cdot,0,-h)\}.\]
  Because $\Omega$ is open there is $x \in G\cap \Omega$ with $B_r(x) \subset \Omega$ for sufficiently small $r>0$. In particular $\det DYu \geq \lambda$ on $B_r(x)$. Choose $r,h$ small enough to ensure Theorem \ref{thm:sc:lower_uniform} holds on $B_r\cap G^h $ (recalling Remark \ref{rem:sc:ineq}). Then
  \begin{equation}
    \label{eq:sc:ghest1}
       C |G^h\cap B_r(x)|^{2/n} \leq h.
  \end{equation}
The $g$-convexity of $G^h$ (with respect to $0,-h$) implies
  \begin{equation}
  C |G^h| \leq  |G^h\cap B_r(x)| .\label{eq:sc:convballest}
  \end{equation}
  where $C$ depends on $r$, $g$, and some upper bound $h_0 \geq h$. The proof of \eqref{eq:sc:convballest} is given in Lemma \ref{lem:sc:convballest} at the conclusion of this proof. By \eqref{eq:sc:ghest1} and \eqref{eq:sc:convballest}
  \begin{equation}
    \label{eq:sc:to_cont}
    C |G^h|^{2/n} \leq h.
  \end{equation}

For a contradictory estimate we consider sections that behave like $\{u < t(x_1+a) \}$ in the convex case. By the derivation of \eqref{eq:sc:genexp}, taking only a first order Taylor series in \eqref{eq:sc:1-ineq}, we obtain
  \begin{equation}
    \label{eq:sc:genineq}
    x \cdot y + a_{ij,kl}x_ix_jy_ky_l - C^-z \leq g(x,y,z) \leq x \cdot y + a_{ij,kl}x_ix_jy_ky_l - C^+z,
  \end{equation}
for appropriate $C^-,C^+>0$.  Now by \eqref{eq:sc:genexp} for $t$, sufficiently small,
 \begin{equation}
   \label{eq:sc:gstar-rep}
      g(-ae_1,te_1,0) = -at + O(t^2) < 0.
 \end{equation}
  This implies $ g(\cdot,te_1,0) < g(\cdot,te_1,g^*(-ae_1,te_1,0))$ and subsequently
  \[ D_t := \{u < g(\cdot,te_1,g^*(-ae_1,te_1,0))\},\]
  satisfies $0 \in D_t$, and $-ae_1 \in \partial D_t$. On the other hand $g(\cdot,te_1,g^*(-ae_1,te_1,0)) \rightarrow 0$ as $t \rightarrow 0$ implies
  \begin{equation}
    \label{eq:sc:lim-cond}
     a^+_t := \sup\{x_1 = x\cdot e_1; x \in D_t\} \rightarrow 0.
   \end{equation}
   Finally by \eqref{eq:sc:genineq} and \eqref{eq:sc:gstar-rep} the height of $g(0,te_1,g^*(-ae_1,te_1,0))$ above $u(0)$ satisfies
   \begin{equation}
     \label{eq:sc:0est} |u(0) - g(0,te_1,g^*(-ae_1,te_1,0))| \geq Ct.
   \end{equation}

 To apply Theorem \ref{thm:sc:upper_uniform_close}, transform the $x$ coordinates according to \eqref{eq:sc:xdef} with respect to $y_t:=te_1,z_t:=g^*(-a e_1,y_t,0)$ and let $\tilde{D_t}$ denote the image of $D_t$. Take the line segment joining the images of $-ae_1$ and $0$. Its length is greater than $Ca$ for a constant depending on $g$. One of the supporting planes orthogonal to this line segment converges to the plane $P = \{x_{1} = 0\}$ as $t \rightarrow 0$. The other remains a distance of at least $Ca$ from the image of $0$.   Thus Theorem \ref{thm:sc:upper_uniform_close} implies, for some $\epsilon_t \rightarrow 0$, that
  \begin{equation}
    \label{eq:sc:aona}
     t \leq C\epsilon_t|D_t|^{2/n},
  \end{equation}
  where we've used \eqref{eq:sc:0est}.
  Finally, again by \eqref{eq:sc:genineq} and \eqref{eq:sc:gstar-rep} we obtain $D_t \subset G^{Ct}$ for some $C>0$ and $t$ sufficiently small. Thus \eqref{eq:sc:aona} with $t = h/C$ contradicts \eqref{eq:sc:to_cont}.\\ \\
  \textit{Step 2. An extreme point that cannot be on the boundary}\\
  \textit{Step 2a. Coordinate transform}
  The argument in this step requires a delicate coordinate transform, so we are explicit with the details.  We begin with the original coordinates and generating function, and the contact set $G = \{u \equiv g(\cdot,y_0,z_0)\}$ which is assumed to contain more than one point in $\Omega$. Without loss of generality $y_0,z_0= 0$. Introduce the coordinates
  \begin{equation}
    \label{eq:sc:x-exp-1}
       \tilde{x} = \frac{-g_y}{g_z}(x,0,0).
  \end{equation}
  Pick any $x_1 \in G$ and denote its image under \eqref{eq:sc:x-exp-1} by $\tilde{x_1}$. Define
  \[ \overline{y} = E^{-1}(x_1,0,0)[g_x(x_1,y,g^*(x_1,y,u(x_1))) - g_x(x_1,0,0)].\]
  The $E^{-1}$ factor implies, via a Taylor series, that
  \begin{equation}
    \label{eq:sc:y-exp}
    \overline{y} = y + O(|y|^2).
  \end{equation}
The images of $\Omega,\Omega^*$ in these coordinates are denoted by\footnote{In step 2a the overline notation is used for coordinates not closures.} $\tilde{\Omega},\overline{\Omega^*}$. Both are convex and $0 \in \overline{\Omega^*}$. By a rotation, which is applied to $\tilde{x}$ and $\overline{y}$, we assume $t_0e_1 \in \overline{\Omega^*}$ for some small $t_0$. By convexity $\overline{\Omega^*}$ contains a cone
  \begin{equation}
    \label{eq:sc:cone-arg}
       \mathcal{C} := \{t \xi; 0 < t \leq t_0, \xi \in B_r'\},
  \end{equation}
  where $B_r '$ is a geodesic ball in the hemisphere $\mathbf{S}^{n-1} \cap \{x_1 > 0\}$ centered on $(1,0,\dots,0)$.

  Now choose a specific extreme point of $\tilde{G}$ as follows: Take the paraboloid
\[ \tilde{P}_M = \{\tilde{x} = (\tilde{x}_1,\dots,\tilde{x}_n) ; \tilde{x}_1 = -\epsilon(\tilde{x}_2^2+\dots+\tilde{x}_n^2) + M\}\]
for large $M$ and small $\epsilon$. Decrease $M$ until the paraboloid first touches $\tilde{G}$, necessarily at an extreme point $\tilde{x_0}$. We take $\tilde{P}$ as the tangent plane to $\tilde{P}_M$ at $\tilde{x_0}$ and note $\tilde{P}$ supports $\tilde{G}$ at $\tilde{x_0}$. Provided $\epsilon$ is sufficiently small (depending on  $\text{diam}(G)$) the normal to $\tilde{P}$ is in $B_r'$ (see Figure \ref{fig:transform1}).

  Our final coordinate transform is
  \begin{equation}
    \label{eq:sc:x-fin}
      \overline{x} \mapsto -g_z(x_0,y_0,z_0)[\tilde{x}-\tilde{x_0}].
  \end{equation}
  which is a dilation and translation (but no rotation). In these coordinates $\overline{U}$ is convex, and provided $\epsilon$ was chosen small depending also on $|g_z|$ (but importantly independent of $x_0$) we have the normal to $\overline{P}$, the image of $\tilde{P}$ under \eqref{eq:sc:x-fin}, at $0$ is in $B_r'$. Set
  \begin{align*}
    \tilde{g}(x,y,z) &= \frac{g_z(x_0,0,0)}{g_z(x,0,0)}[g(x,y,g^*(x_0,y,u(x_0)-z)) - g(x,0,0)],
   \end{align*}
  and define $\overline{g}(\overline{x},\overline{y},z) = \tilde{g}(x,y,z)$ for $\overline{x},\overline{y}$ the image of $x,y$. A direct calculation, using a Taylor series (only in $y,z$) implies
  \begin{equation}
    \label{eq:sc:genexp-2}
       \overline{g}(\overline{x},\overline{y},z) = \overline{x} \cdot \overline{y} -z+ O(|\overline{y}|^2)+O(|\overline{y}||z|) +O(|z|^2)
  \end{equation}

  Finally we rotate the $\overline{x}$ and $\overline{y}$ coordinates so that the supporting plane $\overline{P}$ becomes
  \[\overline{P} = \{\overline{x}=(\overline{x}_1,\dots,\overline{x}_n); \overline{x}_1 = 0\},\] and $\overline{G} \subset \{ \overline{x} ; \overline{x}_1 \leq 0\}$. Because, originally, the normal to the supporting plane was in $B_r'$ after this rotation we still have $\overline{y_0}':=t_0e_1 \in \overline{\Omega^*}$. This is by inclusion of the cone \eqref{eq:sc:cone-arg} (see Figure \ref{fig:transform2}).  Finally pick any $\overline{x_0'}\in \overline{G}$ such that $\overline{x_0'} \cdot \overline{y_0 '}\neq 0$. We are in exactly the setting to use the transformation (4.7) from \cite{ChenWang16} which preserves that $\overline{g}$ has the form \eqref{eq:sc:genexp-2} and after which we have
  \begin{align*}
      &\overline{G}  \subset \{\overline{x}; \overline{x}_1 \leq 0\},  &&\overline{G} \cap \{\overline{x}_1 = 0\} = 0,\\
    &-a e_1  \in \overline{G}, && be_1 \in \overline{\Omega^*}. 
  \end{align*}
  The rest of the argument takes place in this setting. For ease of notation we drop the overline.

  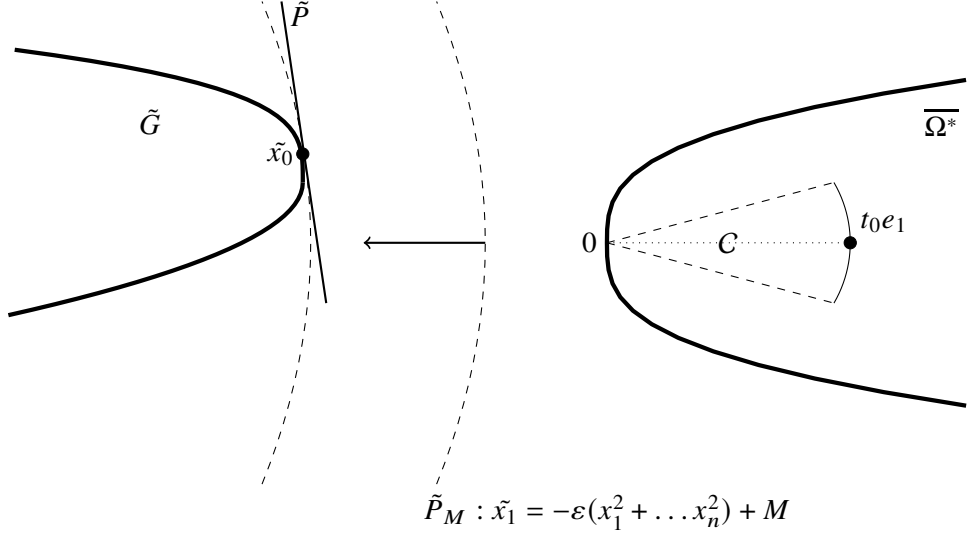
\begin{figure}
  \begin{tikzpicture}[scale=0.8]
\node (c) at (15.5,7) {$\overline{\Omega^*}$};
    \draw[ultra thick, domain=2.3:7.7, variable=\y] plot ({10+0.3*(\y-5)*(\y-5)*abs((\y-5))}, \y);
    \node [above right] at (14,5) {$t_0e_1$};
    \draw[fill] (14,5) circle [radius=0.1];
    \draw [dotted] (10,5) -- (14,5) ;
    \draw[domain=4:6, variable=\y] plot ({12+sqrt(4-pow((\y-5),2))}, \y);
    \draw [dashed] (10,5) -- (13.73205,4) ;
    \draw [dashed] (10,5) -- (13.73205,6) ;
    \node  at (12,5) {$\mathcal{C}$};
    \node  [left] at (10,5) {$0$};

\draw[ultra thick, domain=6:8.2, variable=\y] plot ({5-0.3*pow((\y-6),3.5)}, \y)  ;
    \draw[ultra thick, domain=3.8:6, variable=\y] plot ({5-pow((\y-6),2)}, \y)  ;

    \draw[dashed, domain=1:9, variable=\y] plot ({8-0.05*pow((\y-5),2)}, \y)  ;
    \draw[dashed, domain=1:9, variable=\y] plot ({5.13-0.05*pow((\y-5),2)}, \y)  ;
    \draw[fill] (5.0,6.47196) circle [radius=0.1];
    \node  at (2.5,7) {$\tilde{G}$};
    \node [below] at (10,1) {$\tilde{P}_M: \tilde{x_1} = -\epsilon(x_1^2+\dots x_n^2)+M$};
    \draw [->,thick] (8,5) -- (6,5);
    \node  [left] at (5.021666,6.47196) {$\tilde{x_0}$};
    \draw[thick,domain=4:9, variable=\y] plot ({-0.147196*\y+5.974303}, \y)  ;
     \node at (4.95,8.8) {$\tilde{P}$};
  \end{tikzpicture}
  \caption{Our choice of $\overline{y}$ coordinates implies $\overline{\Omega^*}$ contains the cone $\mathcal{C}$. By choosing $\epsilon$ sufficiently small the normal to the paraboloid at the extreme point lies in the cone. }\label{fig:transform1}
\end{figure}

\begin{figure}
 \begin{tikzpicture}[scale=0.8,rotate=-8]
\node (c) at (15,10) {$\overline{\Omega^*}$};
    \draw[ultra thick, domain=5.3:10.7, variable=\y] plot ({10+0.3*(\y-8)*(\y-8)*abs((\y-8))}, \y);
    \node [above right] at (13.92,8.57) {$t_0e_1$};
    \draw[fill] (13.92,8.57) circle [radius=0.1];
    \draw[domain=7:9, variable=\y] plot ({12+sqrt(4-pow((\y-8),2))}, \y);
    \draw [dashed] (10,8) -- (13.73205,7) ;
    \draw [dashed] (10,8) -- (13.73205,9) ;
    \draw [dotted] (10,8) -- (13.92,8.57) ;
    \node  at (12,8) {$\mathcal{C}$};
    \node  [left] at (10,8) {$0$};

\draw[ultra thick, domain=6:8.2, variable=\y] plot ({6-0.3*pow((\y-6),3.5)}, \y)  ;
    \draw[ultra thick, domain=3.8:6, variable=\y] plot ({6-pow((\y-6),2)}, \y)  ;

    \draw[fill] (6.0,6.47196) circle [radius=0.1];
    \node  at (3.5,7) {$\overline{G}$};
    \node  [left] at (6.021666,6.47196) {$0$};
    \draw[thick,domain=4:9, variable=\y] plot ({-0.147196*\y+6.974303}, \y)  ;
     \node at (5.95,8.8) {$\overline{P}$};
   \end{tikzpicture}
     \caption{We rotate so $\tilde{P}$ becomes $\overline{P} = \{\overline{x};\overline{x}_1=0\}$. After this rotation $\overline{\Omega^*}$ still contains $t_0e_1$. Moreover any $\overline{x_0'}$ in $\overline{G}$ satisfies $\overline{x_0'} \cdot (t_0e_1)\neq0$. This is exactly the situation to apply a linear transformation due to Chen and Wang \cite{ChenWang16}.}\label{fig:transform2}
\end{figure}
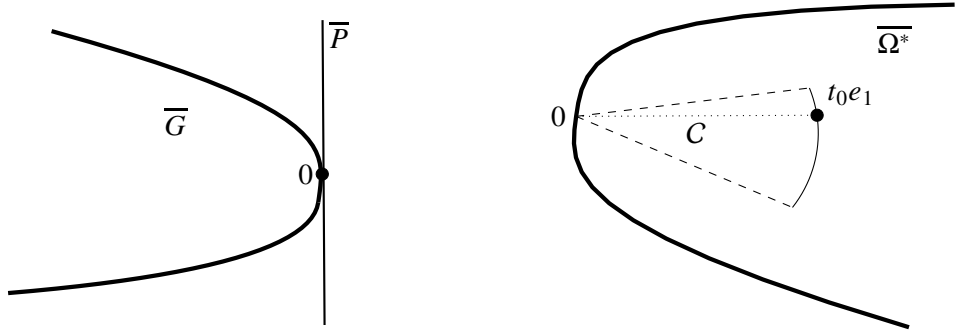

\pagebreak
\textit{Step 2b. Obtaining the contradiction in this setting}\\
  We see $F_\tau := G\cap \{x_1 \geq -\tau\}$ decreases to the point $\{0\}$ as $\tau \rightarrow 0$. Because $\Omega$ lies a positive distance from $0 \in \partial U$, for $\tau$ sufficiently small $\text{dist}(\Omega,F_\tau) >  0$. On the other hand, provided in addition, $\tau < -a/2$ the set
  \begin{equation}
    \label{eq:sc:new-set}
     \{ u < g(\cdot,te_1,g^*(-\tau e_1,te_1,u(-\tau e_1))\} ,   
  \end{equation}
 decreases to $F_\tau$ as $t\rightarrow0$. To see this note
  \[ g(x,te_1,g^*(-\tau e_1,te_1,u(-\tau e_1)) = (x_1+\tau)t+O(t^2),\]
  which follows from \eqref{eq:sc:genexp-2} and a similar expansion for $g^*$ obtained using  \eqref{eq:g:deriv-equiv}. Thus for $t$ small the set in \eqref{eq:sc:new-set} is disjoint from $\overline{\Omega}$. So for all $x \in \overline{\Omega}$,
\begin{equation}
  \label{eq:sc:final_cont}
  u(x) > g(x,te_1,g^*(-\tau e_1,te_1,u(-\tau e_1)).
\end{equation}

On the other hand $te_1 \in \Omega^*$ for $t$ small. Thus there is $x_2 \in \overline{\Omega}$ with $te_1 \in Yu(x_2)$. So
\begin{equation}
  \label{eq:sc:gives_cont}
   u(x) \geq g(x,te_1,g^*(x_2,te_1,u(x_2))),
\end{equation}
for all $x \in U$. To obtain a contradiction evaluate \eqref{eq:sc:gives_cont} at $x = -\tau e_1$, then apply
\[g(x_2,te_1,g^*(-\tau e_1,te_1,\cdot)) \]
to both sides. This contradicts \eqref{eq:sc:final_cont} with $x = x_2$.
\end{proof}

So we're left to deal with the  lemma we skipped over in the previous proof. The author relied on Figalli's book \cite[pg. 101]{Figalli17} for details in the convex case.

\begin{lemma}\label{lem:sc:convballest}
  In the context of the previous proof \eqref{eq:sc:convballest} holds. 
\end{lemma}
\begin{proof}
  Recall the setting is as follows, we have
  \begin{align*}
    G &:= \{u \equiv g(\cdot,y_0,z_0)\}\\
    G^h&:= \{u < g(\cdot,y_0,z_0-h)\},
  \end{align*}
  and $x_0 \in \Omega\cap G$ with $r$ chosen so small as to ensure $B_r(x_0)\subset\subset\Omega$.

  Switch to the coordinates centred on $x_0$ with respect to $y_0,z_0-h$ (so $G^h$ is convex). It suffices to prove the estimates in these coordinates. By \eqref{eq:sc:xxbarest} and continuity the image of $B_r$ contains a ball, which we still refer to as $B_r$.  We assume $G^h \subset B_R(0)$ for some large $R$. Now, normalize $G^h$ via John's lemma and an affine transformation $T$ so that
  \begin{equation}
    \label{eq:sc:tg-contain}
       B_1 \subset T(G^h) \subset B_n.
  \end{equation}
  Then $B_1(0) \subset T(G^h) \subset T(B_R)$, and, because $T$ is affine, this implies $B_{r/R}(Tx_0) \subset T(B_r) \subset T\Omega$.  Finally using \eqref{eq:sc:tg-contain} and $Tx_0 \in G^h$ we have, provided $r$ is sufficiently small, the estimate
  \[ |T(G^h)\cap B_{r/R}(Tx_0)|  \geq C_n \left(\frac{r}{R}\right)^n \geq C_n \left(\frac{r}{R}\right)^n |T(G^h)|. \]
The final equality uses the normalization to conclude $|T(G^h)| \leq C_n$.  
\end{proof}

\begin{remark}\label{rem:sc:unif-dom}
 A corollary of Theorem \ref{thm:sc:contain} is that when the source domain $\Omega$ is uniformly $g$-convex with respect to $u$ and $\Omega \subset\subset U$ then $u$ is strictly convex. The important part is that the uniform convexity of $\Omega$ means no convexity condition is needed on $U$. This corollary is immediate: for $\epsilon$ sufficiently small $\Omega_\epsilon := \{x ; \text{dist}(x,\Omega) < \epsilon\}$ is uniformly $g$-convex with respect to $u$ and strictly contains  $\Omega$.
\end{remark}

It is worth drawing attention to the strict convexity result of Figalli, Kim, and McCann in the optimal transport case \cite{FKM13}. Importantly, their result does not require the cost function be defined outside $\overline{\Omega} \times \overline{\Omega^*}$. Such a result is desireable for GJE. They rely crucially on a theorem that shows the $Y$ mapping maps boundaries to boundaries. I've been unable to extend their result --- it an interesting question as to whether the analog of \cite[Theorem 5.1]{FKM13} holds for generating functions.  The impediment is a troublesome Taylor series term $g_{i,j,z}x_iy_jz$ that occurs if one tries to repeat their equation (5.2). A starting point would be working out whether the result holds for the (contrived) example $g(x,y,z) = x \cdot y - z +za_{ij}x_iy_j$. 

\section{Strict convexity under only a one sided bound in 2 dimensions}
\label{sec:sc:2d}

We have stronger strict convexity results in two dimensions. Only a lower bound $\det DYu \geq \lambda$ and the $g$-convexity of $\Omega$ with respect to $u$ are required for strict convexity. Using duality  the $C^1$ differentiability holds if $\det DYu \leq \Lambda$ and the target is $g^*$-convex with respect to $u$. What's more, these results are proved by elementary means. Similar results were proved, using different techniques and stronger hypothesis on the domains, in the optimal transport case by Figalli and Loeper \cite{FigalliLoeper09}.

The results in this section require the A4w condition: \index{A4w}

\textbf{A4w.} \setref{}{ref:a4w}  For all $\xi\in \mathbf{R}^n$  and $(x,u,p) \in \mathcal{U}$ there holds
\begin{align}
   D_{u}A_{ij}(x,u,p)\xi_i\xi_j \geq 0. 
\end{align}
Using the dual matrix $A^*_{ij}$ we may define the dual condition A4w$^*$.
This condition is used to improve the differential inequality from Lemma \ref{lem:g:maindiffineq}. The condition A4w is satisfied in applications known to the author. However, it would still be worthwhile to investigate if the results in this section hold without this hypothesis. 

\begin{theorem}\label{thm:sc:main}
  Let $g$ be a generating function satisfying A3w and A4w, $\Omega \subset \mathbf{R}^2$, and $\lambda > 0$. Assume $u:\Omega \rightarrow \mathbf{R}$ be an Aleksandrov solution of
  \begin{align}
    \label{eq:sc:alekssoln}
       \det DYu \geq \lambda \text{ in }\Omega.
  \end{align}
If $\Omega$ is $g$-convex with respect to $u$ then $u$ is strictly $g$-convex. 
\end{theorem}
\begin{corollary}\label{cor:sc:c1}
  Let $g$ be a generating function satisfying A3w and A4w$^*$, $\Omega \subset \mathbf{R}^2$, and $\Lambda > 0$. Let $u:\Omega \rightarrow \mathbf{R}$ be an Aleksandrov solution
  \begin{align}
    \label{eq:alekssolnabove}
   \det DYu \leq \Lambda \text{ in }\Omega
  \end{align}
  If $\Omega^*:=Yu(\Omega)$ is $g^*$-convex with respect to $u$ then $u \in C^1(\Omega)$. 
\end{corollary}

We make use of the differential inequality from Lemma \ref{lem:g:maindiffineq}. Let $x_0,x_1$ be given and $\{x_\theta\}_{\theta \in [0,1]}$ be the $g$-segment joining $x_0$ to $x_1$ with respect to $y_0,z_0$. If A3w,A4w are satisfied this lemma implies 
\[ h(\theta) := u(x_\theta) - g(x_\theta,y_0,z_0),\]
 satisfies  \[h''(\theta) \geq -C|h'(\theta)|,\] provided $h(\theta)\geq0$. Here $C$ depends on $\sup|Du|$ and $|\dot{x_\theta}|.$  This inequality implies estimates on  $h'$ in terms of $\sup |h|$ via the following lemma. 

\begin{lemma}\label{lem:sc:diffseg}
  Let $h \in C^2([a,b])$ satisfy $h'' \geq -K|h'|$. If $t \in (a,b)$ then
  \begin{equation}
    \label{eq:sc:diffseg}
      -C_0 \sup_{[a,t]}|h|\leq h'(t) \leq C_1\sup_{[t,b]}|h|,
  \end{equation}
  where $C_0,C_1$ depend on $t-a,b-t$ respectively and $K$.
\end{lemma}
\begin{proof}
  First note if $h'(\tau) = 0$ at any $\tau \in (a,b)$ then $h'(a) \leq 0$. To see this assume $\tau$ is the infimum of points with $h'(\tau) = 0$. By continuity if $\tau=a$ we are done. Otherwise $h'$ is single signed on $(a,\tau)$ and if $h'<0$ on this interval then again by continuity we are done. Thus we assume $h'>0$ on $(a,\tau)$. The  inequality $h'' \geq -K|h'|$ implies $\frac{d}{dt}\log(h'(t)) \geq -K$ on $(a,\tau)$. Subsequently for $a<t_1<t_2<\tau$ integration gives
  \begin{equation}
    \label{eq:sc:diffint}
        h'(t_1) \leq e^{K(t_2-t_1)}h'(t_2),
  \end{equation}
  and sending $t_1 \rightarrow a, t_2 \rightarrow \tau$ gives $h'(a) \leq 0$.

  To prove \eqref{eq:sc:diffseg} we begin with the upper bound. We assume $h' >0$ on $[t,b]$ otherwise the argument just given implies $h'(t) \leq 0$. We obtain \eqref{eq:sc:diffint} for $t_1=t$ and $t_2 \in (t,b)$. Integrating with respect to $t_2$ from $t$ to $b$ establishes the result. The other inequality follows by applying the same argument to the function $k$ defined by $k(t):=h(-t)$. 
 \end{proof}

Now we present the proof of Theorem \ref{thm:sc:main}. The proof follows closely the proof of Trudinger and Wang \cite[Remark 3.2]{TrudingerWang08} who provided a proof of the result in the Monge--Amp\`ere case. The key ideas of our proof will be more transparent if the reader is familiar with their proof. There are two key steps: First we obtain a quantitative $g$-convexity estimate for $C^2$ solutions of $\det [D^2u-A(\cdot,u,Du)] \geq \lambda$ (importantly the estimate is independent of bounds on second derivatives). Then we obtain a convexity estimate for Aleksandrov solutions via a barrier argument. 
\begin{proof}[Proof. Theorem \ref{thm:sc:main}.] \textit{Step 1. Quantitative convexity for $C^2$ solutions}\\
  Initially we assume $u$ is $C^2$. Let $g(\cdot,y_0,z_0)$ be a support with $y_0 \in Yu(x),z_0 = g^*(x,y_0,u(x))$ for $x \in \Omega$. Assume for some $\sigma \geq 0$ there is distinct $x_{-1},x_{1} \in \Omega$ with
  \begin{align}
    \label{eq:sc:sigest1} u(x_{-1}) \leq g(x_{-1},y_0,z_0)+\sigma,\\
    \label{eq:sc:sigest2}    u(x_1) \leq g(x_1,y_0,z_0)+\sigma.
  \end{align}

  Let $\{x_\theta\}_{\theta \in [-1,1]}$ denote the $g$-segment with respect to $y_0,z_0$ that joins $x_{-1}$ to $x_{1}$ and set
\[ h_\sigma(x) = u(x) - g(x,y_0,z_0) - \sigma.\]
  We use the shorthand $h_\sigma(\theta) = h_\sigma(x_\theta)$. Lemma \ref{lem:g:maindiffineq} along with A3w and A4w yields $h''_0(\theta) \geq -K |h_0'(\theta)|$ with the same inequality holding for $h_\sigma$. The maximum principle implies $h_\sigma(\theta) \leq \text{max}\{h_\sigma(0),h_\sigma(1)\}$. Thus
  \begin{align}
  0 \geq h_\sigma(\theta) \geq \inf_{\theta \in [-1,1]} h_\sigma(\theta) =:-H.\label{eq:sc:heightest}
  \end{align}
  The convexity estimate we intend to derive is $H \geq C > 0$ where $C$ depends only on $|x_0-x_1|, \ \Vert u \Vert_{C^1}, \ g, \ \lambda$. We use $C$ to indicate any positive constant depending only on these quantities.  

  Because $h$ is locally Lipschitz for $\theta \in [-3/4,3/4]$ and $\xi \in \mathbf{R}^n$, \eqref{eq:sc:heightest} implies
  \begin{align}
        \label{eq:sc:lipest1}         h_\sigma(x_\theta+\xi)  &\leq h_\sigma(x_\theta) + C|\xi| \leq C|\xi|,\\
    \label{eq:sc:lipest2}         h_\sigma(x_\theta+\xi)  &\geq h_\sigma(x_\theta) - C|\xi| \geq -H - C|\xi|.
  \end{align}

  We let $\eta_\theta$ be a continuous unit normal vector field to the $g$-segment $\{x_\theta\}$. Fix $\delta > 0$ so that $x_{-1/2}+\delta \eta_{-1/2}$ and $x_{1/2}+\delta \eta_{1/2}$ lie in $\Omega$. For $\epsilon \in [0,\delta]$  let $\{x^\epsilon_\theta\}_{\theta \in [-1/2,1/2]}$ be the $g$-segment with respect to $y_0,z_0$ joining $x_{-1/2}+\epsilon \eta_{-1/2}$ to $x_{1/2}+\epsilon \eta_{1/2}$.  Using Lemma \ref{lem:sc:diffseg} for $\theta \in [-1/4,1/4]$ combined with \eqref{eq:sc:lipest1} and \eqref{eq:sc:lipest2} implies
  \begin{align}
  -C(\epsilon+H) \leq \frac{d}{d\theta}h_{\sigma}(x_\theta^\epsilon) \leq C(\epsilon+H). \label{eq:sc:diffineq2}
  \end{align}
  Here we have used that $|x_\theta - x_\theta^\epsilon| < C \epsilon$ for a Lipschitz constant independent of $\theta$.   

  This implies
  \begin{align}
\label{eq:sc:tosub1}    \int_{-1/4}^{1/4}\frac{d^2}{d\theta^2}h_{\sigma}(x_\theta^\epsilon) \ d \theta \leq C(\epsilon+H),
  \end{align}
  and we come back to this in a moment. For now note that since
  \[ \det [D^2u - A(\cdot,u,Du)]\geq  \lambda \inf\det E > 0,\] for any two orthogonal unit vectors $\xi,\eta$
 \[ [D_{\xi\xi}u-g_{\xi \xi}(x,Yu(x),Zu(x))][D_{\eta\eta}u-g_{\eta \eta}(x,Yu(x),Zu(x))] \geq C.\]
 In particular for $\eta_\theta^\epsilon$ a choice of unit normal vector field orthogonal to $\dot{x}_\theta^\epsilon$ (continuous in $\theta,\epsilon$) we have
 \begin{align*}
   C^{-1}&[D_{\dot{x}_\theta^\epsilon\dot{x}_\theta^\epsilon}u-g_{\dot{x}_\theta^\epsilon \dot{x}_\theta^\epsilon}(x,Yu(x),Zu(x))] \\
   &\geq |\dot{x}_\theta^\epsilon|^{2} [D_{\eta_\theta^\epsilon\eta_\theta^\epsilon}u-g_{\eta_\theta^\epsilon \eta_\theta^\epsilon}(x,Yu(x),Zu(x))]^{-1}.
 \end{align*}
  Employing this and \eqref{eq:sc:diffineq2} in Lemma \ref{lem:g:maindiffineq}  gives
  \begin{align}
\label{eq:sc:tosub2}   \frac{d^2}{d\theta^2}h_{\sigma}(x_\theta^\epsilon) &\geq C|\dot{x}_\theta^\epsilon|^2\big([D_{\eta_\theta^\epsilon \eta_\theta^\epsilon}u(x_\theta^\epsilon)-g_{\eta_\theta^\epsilon \eta_\theta^\epsilon}(x_\theta^\epsilon,Yu(x_\theta^\epsilon),Zu(x_\theta^\epsilon))]\big)^{-1}\\&\quad\quad - C(\epsilon+H),\nonumber
  \end{align}
  where initially this holds for $h_0$, and thus for $h_\sigma$. Note that by the equivalent form of the A3w condition (\ref{eq:g:a3w-equiv}), the $D_p^2A$ term in Lemma \ref{lem:g:maindiffineq} is bounded below by $-C|h'|$ and subsequently controlled by \eqref{eq:sc:diffineq2}.

  Substituting \eqref{eq:sc:tosub2} into \eqref{eq:sc:tosub1} we obtain
  \begin{align*}
    \int_{-1/4}^{1/4} |\dot{x}_\theta^\epsilon|^2\big([D_{ij}u(x_\theta^\epsilon)-g_{ij}(x_\theta^\epsilon)](\eta_\theta^\epsilon)_i (\eta_\theta^\epsilon)_j\big)^{-1} \ d \theta \leq C(\epsilon+H),
  \end{align*}
 where we omit that $g$ is evaluated at $(x_\theta^\epsilon,Yu(x_\theta^\epsilon),Zu(x_\theta^\epsilon))$. An application of Jensen's inequality implies
  \begin{align}
   \label{eq:sc:inttobound} &\int_0^\delta \int_{-1/4}^{1/4} |\dot{x}_\theta^\epsilon|^{-2}[D_{ij}u(x_\theta^\epsilon)-g_{ij}(x_\theta^\epsilon)](\eta_\theta^\epsilon)_i (\eta_\theta^\epsilon)_j \ d \theta \ d\epsilon \\\quad&\geq C \int_0^\delta \Big(\int_{-1/4}^{1/4}  |\dot{x}_\theta^\epsilon|^2\big([D_{ij}u(x_\theta^\epsilon)-g_{ij}(x_\theta^\epsilon)](\eta_\theta^\epsilon)_i (\eta_\theta^\epsilon)_j\big)^{-1}\ d \theta \Big)^{-1} \ d\epsilon \nonumber \\
\label{eq:sc:invint}    &\quad >  \int_0^\delta \frac{C}{\epsilon+H} \ d \epsilon. 
  \end{align}
This is the crux of the proof complete: the only way for the final integral to be bounded is if $H$ is bounded away from 0. We're left to show the integral \eqref{eq:sc:inttobound} is bounded in terms of the allowed quantities, and approximate when $u$ is not $C^2$.

   To bound \eqref{eq:sc:inttobound} use that $\det E \neq 0$ implies $|\dot{x}_\theta^\epsilon|$ is bounded below by a positive constant depending on $|x_1-x_0|$ and $g$. This gives the estimate
  \begin{align*}
    &\int_0^\delta \int_{-1/4}^{1/4} |\dot{x}_\theta^\epsilon|^{-2}[D_{ij}u(x_\theta^\epsilon)-g_{ij}(x_\theta^\epsilon)](\eta_\theta^\epsilon)_i (\eta_\theta^\epsilon)_j \ d \theta \ d\epsilon\\
   & \leq C  \int_0^\delta \int_{-1/4}^{1/4}  [D_{ij}u(x_\theta^\epsilon)-g_{ij}(x_\theta^\epsilon)](\eta_\theta^\epsilon)_i (\eta_\theta^\epsilon)_j \ d \theta \ d\epsilon\\
   & \leq  C  \int_0^\delta \int_{-1/4}^{1/4} \sum_i D_{ii}u(x_\theta^\epsilon)-g_{ii}(x_\theta^\epsilon) \ d \theta \ d \epsilon.
  \end{align*}
  The final line is obtained using positivity of $D_{ij}u-g_{ij}$ and is bounded in terms of $\Vert g \Vert_{C^2}$ and $\sup |Du|$ (compute the integral in the Cartesian coordinates and note the Jacobian for this transformation is bounded). Thus returning to \eqref{eq:sc:inttobound} and \eqref{eq:sc:invint} we obtain $H > C$ where $C$ depends on the stated quantities. 

  \textit{Step 2: Convexity estimates for Aleksandrov solutions via a barrier argument}\\
  We extend to Aleksandrov solutions via a barrier argument. Suppose $u$ is an Aleksandrov solution of \eqref{eq:sc:alekssoln} that is not strictly convex. There is a support $g(\cdot,y_0,z_0)$ touching at two points $x_1,x_{-1}$. Using Theorem \ref{thm:g:gconvsection} we  have $u \equiv g(\cdot,y_0,z_0)$ along the $g$-segment joining these points (with respect to $y_0,z_0$). Balls with sufficiently small radius are $g$-convex. This follows because, as noted in \cite[\S 2.2]{LiuTrudinger16}, $g$-convexity requires the boundary curvatures minus a function depending only on $\Vert g \Vert_{C^3}$ are positive. Thus we  assume $x_1,x_{-1}$ are sufficiently close to ensure that $B$, the ball with radius $|x_{1}-x_{-1}|/2$ and centre $(x_1+x_{-1})/2$ is $g$-convex with respect to $u$. Let $\epsilon > 0$ be given, and let $u_h$ be the mollification of $u$ with $h$ taken small enough to ensure $|u-u_h| < \epsilon/2$ on $\partial B.$ The Dirichlet theory for GJE, either Step 2. of the proof of Theorem \ref{thm:w:regularity} or, in a more explicit form, \cite[Lemma 4.6]{Trudinger14}, yields a $C^3$ solution of
  \begin{align*}
    \det DY(\cdot,v_h,Dv_h) = c/2 \text{ in }B,\\
    v_h = u_h+\epsilon \text{ on } \partial B,
  \end{align*}
  satisfying an estimate $|Dv_h| \leq K$ where $K$ depends on the local Lipschitz constant of $u$. Since $v_h \geq u$ on $\partial B$ the comparison principle, Corollary \ref{cor:w:comp} implies $v_h \geq u$ in $B$. Thus  our previous argument implies strict $g$-convexity of $v_h$ provided we note \eqref{eq:sc:sigest1} and \eqref{eq:sc:sigest2} are satisfied for $\sigma = 2\epsilon$. Hence at $x_{\theta_\epsilon}$ a point on the $g$-segment where the infimum defining $H$ is obtained we have 
  \[ u(x_{\theta_\epsilon}) - g(x_{\theta_\epsilon},y_0,z_0) - 2\epsilon \leq v_h(x_{\theta_\epsilon}) - g(x_{\theta_\epsilon},y_0,z_0) - 2\epsilon \leq -H < 0.\]
  As $\epsilon \rightarrow 0$ we contradict that $g(\cdot,y_0,z_0)$ supports $u$.  
 \end{proof}

 This result along with duality yields a quick proof of Corollary \ref{cor:sc:c1}.
 Let $u:\Omega \rightarrow \mathbf{R}$ be a $g$-convex function, put $\Omega^*:=Yu(\Omega)$, and let $v:V \rightarrow\mathbf{R}$ be its $g^*$-transform. Recall  if $y \in Yu(x)$ then $x \in Xv(y)$. We use this as follows. Suppose in addition $u$ satisfies that for all $E \subset \Omega$
\begin{align*}
  |Yu(E)| \leq c^{-1}|E|.
\end{align*}
Take $A \subset \Omega^*$ and let $E_u$ denote the measure $0$ set of points where $u$ is not differentiable. Necessarily $A\setminus Yu(E_u) = Yu(E)$ for  some $E\subset \Omega$. Our above reasoning implies $E \subset Xv(Yu(E))$. Hence 
\begin{align}
 \label{eq:sc:dualineq} |Xv(A)|\geq |Xv(A\setminus Yu(E_u))| \geq |E| \geq c|A|.
\end{align}

Corollary \ref{cor:sc:c1} follows: Let $u$ be the function given in Corollary \ref{cor:sc:c1} and $v$ its $g^*$ transform restricted to $\Omega^*$. Theorem \ref{thm:sc:main} holds in the dual form, that is, provided the relevant hypothesis are changed to their starred equivalents Theorem \ref{thm:sc:main} implies strict $g^*$-convexity. Thus the hypothesis of Corollary \ref{cor:sc:c1} along with \eqref{eq:sc:dualineq} allow us to conclude $v$ is strictly $g^*$-convex. 

  Suppose for a contradiction $u$ is not $C^1$. Then for some $x_0$ the set $Yu(x_0)$ contains two distinct points, say $y_0,y_1$. Our above working implies $g^*(x_0,\cdot,u(x_0))$ is a support touching at $y_0,y_1$. This contradicts strict $g^*$-convexity and proves the corollary.

\section{$C^1$ differentiability as a consequence of strict convexity}
\label{sec:sc:furth-cons-strict}

Once the strict convexity has been proved we obtain a short proof of the $C^1$ differentiability of strictly convex solutions using similar techniques to Theorem \ref{thm:sc:contain}.

\begin{theorem}
  Assume $g$ satisfies A3w and $\lambda,\Lambda > 0$. Suppose $u:\Omega \rightarrow \mathbf{R}$ is a strictly $g$-convex Aleksandrov solution of $ \lambda \leq \det DYu \leq \Lambda.$  Then $u \in C^1(\Omega)$.
\end{theorem}
\begin{proof}
  Suppose for a contradiction at some $x_0$, assumed to be $0$, $\partial u(0)$ contains more than one point. Let $p_0$ be an extreme point of $\partial u(0)$ with $y_0 := Y(0,u(0),p_0)$ and $g(\cdot,y_0,z_0)$ the corresponding support. Without loss of generality $u(0),y_0 = 0$. Fix $h>0$ small, and apply the transformations in Section \ref{sec:sc:transformations-1} so $g$ satisfies \eqref{eq:sc:genexp}. We consider
  \[  G^h = \{ u < 0\}, \]
  which is a section of height $h$ and thus convex.
  
  After these transformations $p=g_x(0,0,h)$ is an extreme point of $\partial u(0)$ satisfying $|p| = O(h^2)$. Thus after subtracting $-h+p\cdot x$ from both $u$ and the generating function we have
  \begin{equation}
    \label{eq:sc:c1-subset}
       \{u < h/2\} \subset  G^h = \{u < h - p \cdot x \} \subset \{u < 3h/2\}.
  \end{equation}
  Because by an initial choice of the section small  $|x|$ is as small as desired. Moreover, $0$ is an extreme point of $\partial u(0)$ and, after a rotation, for some small $a>0$
  \begin{align}
  \nonumber  \partial u(0) &\subset \{ p; p_1 \geq 0\}\\
 \label{eq:sc:subdiff2}   \partial u(0) &\supset \{te_1; 0 \leq t \leq a\}.
  \end{align}

This implies  $u(-te_1) = o(t)$. Thus $ \{ u < h/2\},$
  contains $-R(h)he_1$ for some positive function $R$ satisfying $R(h) \rightarrow \infty$ as $h \rightarrow 0$ (see Lemma \ref{lem:a:big-vec}).

  Using the inequality \eqref{eq:sc:almost-affine}, with $y_a = Y(0,ae_1,u(0))$
  \[ u > g(x,y_a,h) \geq x_1a -Ch, \]
  where we've used $Y(0,ae_1,u(0)) = ae_1+O(h)$. Thus, the second subset relation in \eqref{eq:sc:c1-subset} implies $\sup_{G^h}x_1 \leq Ch/a$.

  Thus by Theorem \ref{thm:sc:upper_uniform_close}
  \[ h \leq \frac{C}{R(h)}|G^h|^{2/n}.\]
  On the other hand Theorem \ref{thm:sc:lower_uniform} yields $|G^h|^{2/n} \leq Ch$, which is a contradiction as $h \rightarrow 0$.
 \end{proof}

Guillen and Kitagawa have proved the stronger result that strictly $g$-convex solutions are $C^{1,\alpha}$. Thus we also obtain that result under the new domain hypothesis in Theorem \ref{thm:sc:contain}.

\clearpage{}
\clearpage{}\chapter{Uniqueness results}
\label{cha:u}

In this chapter we consider uniqueness results, and variants there of, for generated Jacobian equations. Our main result is for the second boundary value problem. Back in Theorem \ref{thm:w:weakexist} we proved  there exists an Aleksandrov solution of the second boundary value problem taking a given value at a given point, that is satisfying $u(x_0) = u_0$ for a given $x_0,u_0$. Here we show, if this solution is sufficiently regular, it is unique. Phrased differently: if two solutions of the second boundary value problem intersect, then they are the same solution. Our methods yield a uniqueness result for sufficiently regular solutions of the Dirichlet problem satisfying an extension property. This chapter follows the author's paper \cite{Rankin2020}. 

 The results in this chapter are based on an old lemma of Aleksandrov's and some ideas from elliptic PDE. We can only use Aleksandrov's lemma at points of differentiability. The elliptic theory requires the associated linearisations are uniformly elliptic, and this requires our solutions are $C^{1,1}$. An interesting question is whether the results of this chapter hold for $g$-convex solutions with less regularity.

Here are the theorems we prove in this chapter. 

\begin{theorem}\label{thm:u:main}
  Let $f\in C^1(\Omega),\ f^*\in C^1(\Omega^*) $ be positive densities satisfying the mass balance condition. Let $u,v \in C^{1,1}(\Omega)$ be $g$-convex Aleksandrov solutions of \eqref{eq:g:gje} subject to \eqref{eq:g:2bvp}. If there is $x_0 \in \Omega$ such that $u(x_0) =v(x_0)$ then $u \equiv v$. 
\end{theorem}

We have a similar result if the contact point is on the boundary, provided we assume a convexity condition on the target.

\begin{theorem}\label{thm:u:boundary}
    Let  $f\in C^1(\Omega),\ f^*\in C^1(\Omega^*)$  be positive densities satisfying the mass balance condition. Let $u,v \in C^{1,1}(\overline{\Omega})$ be $g$-convex Aleksandrov solutions of \eqref{eq:g:gje} subject to \eqref{eq:g:2bvp}. Assume $\Omega^*$ is a $C^2$ domain and is $g^*$-convex with respect to $u$ and $v$. If there is $x_0 \in \overline{\Omega}$ with $u(x_0) = v(x_0)$ then $u \equiv v$.
  \end{theorem}

We include the following result, which is a corollary of the techniques used for the second boundary value problem. 
  \begin{theorem}\label{thm:u:dirichlet}
   Suppose  $f\in C^1(\Omega),\ f^*\in C^1(\Omega^*)$ are $C^1$ positive functions and $u,v \in C^{1,1}(\overline{\Omega})$ solve \eqref{eq:g:gje} subject to $u = v$ on $\partial \Omega$. Suppose further $u,v$ have $C^1$ $g$-convex extensions to a neighbourhood of $\Omega$. Then $u = v$. 
  \end{theorem}

  \section{Main lemma for uniqueness}
  Here is a heuristic explanation of how we prove uniqueness in the Monge--Amp\`ere case. We consider two solutions $u,v$ and the set where one is greater than the other $\Omega' := \{v < u\}$. The comparison principle paired with an appropriate boundary condition implies $\partial u(\Omega') \subset \partial v(\Omega')$. If this inclusion is strict, in the sense that $|\partial u(\Omega')| < |\partial v(\Omega')|,$ we contradict the PDE and conclude $u =v$.

  Aleksandrov \cite{Aleksandrov42a} showed in the context of the Minkowski problem that a sufficient condition for this strict inclusion is a point $x_0 \in \partial \Omega'$ with $Du(x_0) \neq Dv(x_0)$.   McCann adapted the lemma to the Monge--Amp\`ere case \cite[Lemma 13]{McCann95}. We follow McCann's proof and adapt the result to generated Jacobian equations.

\begin{lemma}\label{lem:u:aleksandrovmccann}
Assume $u,v:\Omega \rightarrow \mathbf{R}$ are $g$-convex and differentiable. Suppose for some $x_0 \in \Omega$ there holds $u(x_0) = v(x_0)$. 
  Let $\Omega':=\{x \in \Omega; u(x)>v(x)\}$ and set $\Xi := Yv^{-1}(Yu (\Omega'))$. Then $\Xi \subset \Omega'$. Furthermore, if  $Du(x_0) \neq Dv(x_0)$ then $x_0$ is a positive distance from $\Xi$. 
\end{lemma}
\begin{proof}
  To prove the subset assertion take $\xi \in \Xi$. The definition of $\Xi$ implies there is $x \in \Omega'$ with $Yv (\xi) = Yu (x)$. We claim $Zu(x) < Zv(\xi)$. If not, $Zu(x) \geq Zv(\xi)$ combined with $Yv (\xi) = Yu (x)$ and $g_z < 0$ implies, for any $x'$ 
  \[ g(x',Yu (x),Zu(x)) \leq g(x',Yv (\xi),Zv(\xi)).\]
  Subsequently
  \begin{align*}
    u(x) &= g(x,Yu (x),Zu(x))\\
         &\leq g(x,Yv (\xi),Zv(\xi))\\
    &\leq v(x),
  \end{align*}
  where the final inequality is because $g(\cdot,Yv (\xi),Zv(\xi))$ is a $g$-support of $v$. Since $x \in \Omega'$ this contradiction establishes $Zu(x) < Zv(\xi)$. Thus for any $x'$
  \begin{align}
   \nonumber u(x') &\geq g(x',Yu (x),Zu(x))\\
   \label{eq:u:ineq}  &> g(x',Yv (\xi),Zv(\xi)).
  \end{align}
  For $x' = \xi$ we have $u(\xi) > v(\xi)$ so $\xi \in \Omega'$ and the subset relation holds.

  Now we show when $Du(x_0)\neq Dv(x_0)$ that $x_0$ is a positive distance from $\Xi$. Suppose to the contrary  there exists a sequence of  $\{\xi_n\}_{n=1}^\infty$ in $\Xi$ with $\xi_n \rightarrow x_0$. The definition of $\Xi$ implies for each $\xi_n$ there exists an $x_n \in \Omega'$ with $Yv (\xi_n) = Yu (x_n).$

  Now, $Du(x_0) \neq Dv(x_0)$ implies in any neighbourhood of $x_0$ there is a particular $z$ for which 
  \begin{equation}
    \label{eq:u:m1}
    u(z) < g(z,Yv (x_0),Zv(x_0)).
  \end{equation}
  for if not 
  \begin{align*}
    u(x_0) &= v(x_0) = g(x_0,Yv (x_0),Zv(x_0)),\\
\text{ and }\quad\quad    u(x) &\geq g(x,Yv (x_0),Zv(x_0)) \text{ in a neighbourhood of }x_0.
  \end{align*}
  This implies $u(\cdot)-g(\cdot,Yv (x_0),Zv(x_0))$ has a local minimum at $x_0$. Thus
  \[ Du(x_0) = g_x(x_0,Yv (x_0),Zv(x_0)) = Dv(x_0),\]
  and this contradiction establishes \eqref{eq:u:m1}. 

  Since  our derivation of \eqref{eq:u:ineq} used only that $x \in \Omega'$ and $\xi \in \Xi$ satisfied $Yv (\xi) = Yu (x)$, (\ref{eq:u:ineq}) also holds for  for $x_n$ and $\xi_n$. That is for any $x'$ we have
  \begin{equation}
  u(x') > g(x',Yv (\xi_n),Zv(\xi_n)) .\label{eq:u:m2}
\end{equation}

Combining \eqref{eq:u:m1} and \eqref{eq:u:m2} we obtain
  \[ g(x',Yv (x_0),Zv(x_0)) > u(x') >g(x',Yv (\xi_n),Zv(\xi_n)), \]
  which, on sending $\xi_n\rightarrow x_0$ yields a contradiction and completes the proof. 
\end{proof}

It was pointed out by an examiner that, infact, the above proof only requires differentiablity at $x_0$. In this case replace $Yu(x),Yv(x)$ by elements of these sets. 
Furthermore, the same result holds assuming the equivalent condition $Yu(x_0) \neq Yv(x_0)$ in place of $Du(x_0) \neq Dv(x_0)$. However, some condition of this form is necessary to obtain the positive distance conclusion. Without this condition we consider the following simple counterexample in the convex case (where $g(x,y,z) = x\cdot y -z$). We consider $u,v:[-1,1]\rightarrow\mathbf{R}$ defined by $u(x) = x^2-1$
  \[v(x) =
    \begin{cases}
      -2x-2 &x \in [-1,-1/2),\\
      2x^2-3/2 & x \in (-1/2,1/2),\\
      2x-2 &x \in (1/2,1].
    \end{cases}
  \]
  Then $u,v \in C^{1,1}([-1,1])$ with $v < u$ but $Du[-1,1] = Dv[-1,1] = [-2,2]$. 

\section{Uniqueness of the second boundary value problem}
\label{sec:u:uniq-second-bound}

The first step in our proof of uniqueness for the second boundary value problem is to show solutions have the same gradient on the set where they intersect.

\begin{corollary}\label{cor:u:dudv}
Assume the conditions of Theorem \ref{thm:u:main}. Then $Du \equiv Dv$ on the set $\{x \in \Omega; u(x) = v(x)\}$.  
\end{corollary}
\begin{proof}
  Suppose otherwise. Then there is $x_0 \in \Omega$ with $u(x_0) = v(x_0)$ and $Du(x_0) \neq Dv(x_0)$. This implies any neighbourhood of $x_0$ contains an $x$ satisfying $u(x) > v(x)$, which is to say $x_0 \in \partial \Omega' \cap \Omega$. By the previous lemma, for $\epsilon$ sufficiently small $B_\epsilon(x_0) \cap \Xi = \emptyset$ and thus $\Xi \subset \Omega'\setminus B_\epsilon(x_0).$ On the other hand, since $x_0 \in \partial \Omega',$ and $u$ is continuous, $|B_\epsilon(x_0) \cap \Omega'| > 0$. Hence
  \[ |Yv^{-1}(Yu (\Omega'))| = |\Xi| \leq  |\Omega' \setminus B_\epsilon(x_0)| < |\Omega'|,\]
  and since $f$ is bounded below, this implies 
  \begin{equation}
   \int_{Yv^{-1}(Yu (\Omega'))}f^*(Yv )\det DYv  \ dx  < \int_{\Omega'} f^*(Yv )\det DYv  \ dx .\label{eq:u:4}
 \end{equation}

  The change of variables formula holds for the mappings $Yu $ and $Yv $ even though they may not be diffeomorphisms. The reasoning here is as in \cite[Theorem A.31]{Figalli17} and \cite[\S 4]{Trudinger14}. In light of this (\ref{eq:u:4}) yields the following contradiction:
  
 \begin{align}
   \nonumber  \int_{\Omega'}f(x) \ dx &= \int_{\Omega'}f^*(Yu ) \det DYu  \ dx \\
   \label{eq:u:gen1}                    &= \int_{Yu (\Omega')}f^*(y) \ dy\\
   \label{eq:u:gen2}                     &= \int_{Yv (Yv^{-1}(Yu (\Omega')))}f^*(y) \ dy\\
   \nonumber                                   &= \int_{Yv^{-1}(Yu (\Omega'))}f^*(Yv )\det DYv  \ dy\\
   \nonumber                          &< \int_{\Omega'} f^*(Yv )\det DYv  \ dy= \int_{\Omega'}f(x) \ dx.
  \end{align}
  Here the equality between \eqref{eq:u:gen1} and \eqref{eq:u:gen2} uses the generalised second boundary value problem to conclude
  \begin{equation}
    \label{eq:u:sets-eq}
       Yv (Yv^{-1}(Yu (\Omega'))) = Yv (\Omega) \cap Yu (\Omega') = Yu (\Omega')\setminus\mathcal{F},
  \end{equation}
  for some set $\mathcal{F}$ with $|\mathcal{F}|=0$. Thus integrals over the sets in \eqref{eq:u:sets-eq} agree. 
\end{proof}

 \subsection*{A weak Harnack inequality}

\begin{lemma}\label{lem:u:harn} Suppose $u,v \in W^{2,n}(\Omega)$ satisfy (\ref{eq:g:gje}) almost everywhere.  Then for any $\tilde{\Omega} \subset\subset \Omega$ there exists $p,C>0$ and independent of $u,v$ such that
\begin{equation}
\Big(\frac{1}{|\tilde{\Omega}|}\int_{\tilde{\Omega}}(u-v)^p\Big)^{\frac{1}{p}} \leq C \inf_{\tilde{\Omega}}(u-v).\label{eq:u:harn}
\end{equation} 
\end{lemma}
\begin{proof}
Provided we are able to show $u-v$ is a supersolution of a homogeneous linear elliptic PDE this is a consequence of the weak Harnack inequality \cite[Theorem 9.22]{GilbargTrudinger01} and a covering argument. We show $w:=u-v$ satisfies
\begin{equation}
 Lw := a^{ij}D_{ij}w+b^kD_kw+cw \leq 0,\label{eq:u:lin}
\end{equation}
where
\begin{align*}
  a^{ij} &=  [D^2u-A(\cdot,u,Du)]^{ij},\\
b^i &= -a^{ij}(A_{ij})_{p_k}-\tilde{B}_{p_k},\\
c &= -a^{ij}(A_{ij})_u-\tilde{B}_u,
\end{align*}
and $\tilde{B} = \log B$. Now using (\ref{eq:g:mate}) we have, almost everywhere, 
\begin{align}
 \label{eq:u:tosub}0=\log\det[D^2v&-A(\cdot,v,Dv)] - \log\det[D^2u-A(\cdot,u,Du)]\\
&+\log B(\cdot,u,Du) - \log B(\cdot,v,Dv).\nonumber
\end{align}
A Taylor series for 
\[ h(t) := \log\det[t(D^2v-A(\cdot,v,Dv))+(1-t)(D^2u-A(\cdot,u,Du))],\]
yields 
\[ h(1) - h(0) = h'(0) + \frac{1}{2}h''(\tau),\]
for some $\tau$ in $[0,1]$. Concavity of $\log\det$ implies $h''(\tau) \leq 0$ and thus on computing $h'(0)$ we obtain
\begin{align}
 \label{eq:u:tosub1} \log\det[D^2v&-A(\cdot,v,Dv)] - \log\det[D^2u-A(\cdot,u,Du)]\\
&\leq a^{ij}D_{ij}(v-u) + a^{ij}(A_{ij}(\cdot,u,Du) - A_{ij}(\cdot,v,Dv)),\nonumber
\end{align}
where $a^{ij} = [D^2u-A(\cdot,u,Du)]^{ij}$. Thus \eqref{eq:u:tosub1} into \eqref{eq:u:tosub} implies
\begin{align}
 \label{eq:u:tosub2}0\leq a^{ij}&D_{ij}(v-u) + a^{ij}(A_{ij}(\cdot,u,Du) - A_{ij}(\cdot,v,Dv))\\
&+\log B(\cdot,u,Du) - \log B(\cdot,v,Dv).\nonumber
\end{align}
The mean value theorem yields
\[ A_{ij}(\cdot,u,Du) - A_{ij}(\cdot,v,Dv) = A_u(\cdot,w,p)(u-v) +A_{p_k}(\cdot,w,p)D_k(u-v), \]
 for some $w = t_1 v+(1-t_1)u$ and $p = t_2Dv + (1-t_2)Du$ and similarly for 
$\log B(\cdot,u,Du) -\log B(\cdot,v,Dv)$. Thus \eqref{eq:u:tosub2} becomes
\[0 \leq a^{ij}D_{ij}(v-u)-(a^{ij}(A_{ij})_{p_k}+\frac{B_{p_k}}{B})D_k(v-u)-(a^{ij}(A_{ij})_u+\frac{B_u}{B})(v-u) ,\]
which is \eqref{eq:u:lin} (multiply by $-1$ since $w = u-v$).
\end{proof}

\subsection*{Proof of Theorem \ref{thm:u:main}: Solutions intersecting on the interior are the same}

We've proved one solution cannot touch another from above. Now we conclude the proof of Theorem \ref{thm:u:main} by showing if two distinct solutions intersect their maximum is a ($W^{2,n}$) solution touching from above --- a contradiction. 

  At the outset we fix $\tilde{\Omega} \subset\subset \Omega$ containing $x_0$. Since $u,v$ are $g$-convex the same is true for $w:=\max\{u,v\}$. Furthermore $Du \equiv Dv$ on $\{u=v\}$, implies $w$  is $C^{1,1}_{\text{loc}}(\Omega)$. Since $C^{1,1}_\text{loc}(\Omega) \subset W^{2,n}_\text{loc}(\Omega)$ we see $w$ solves (\ref{eq:g:mate}) almost everywhere. Hence the weak Harnack inequality \eqref{eq:u:harn} applied to $w-v$ implies $w \equiv v$ in $\tilde{\Omega}$. The same argument yields $w \equiv u$ in $\tilde{\Omega}$ and hence $u \equiv v$ in $\Omega$ via continuity.  \

\subsection*{Solutions intersecting on the boundary are the same}
Next we show, using the $g^*$-convexity of the target domain, that if solutions intersect on the boundary then they are the same solution. Because $\Omega^*$ is a $C^2$ $g^*$-convex domain with respect to $u,v$ there exists a defining function $\phi^* \in C^2(\overline{\Omega^*})$ satisfying
\begin{align*}
  &\phi^* < 0 \text{ in } \Omega^* &&\phi^* = 0 \text{ on }\partial\Omega^*\\
  & D_{p}^2\phi^*(Y(x,u,p)) \geq 0  && |D\phi^*|\neq 0 \text{ on }\partial\Omega^*.
\end{align*}
We construct this function in Section \ref{sec:char-unif-conv}. Then for
\[ G(x,u,p) := \phi^*(Y(x,u,p)) \]  $C^1$ solutions of the second boundary value problem satisfy
\[ G(\cdot,u,Du) = 0 \text{ on }\partial\Omega.\]
Moreover the $g^*$-convexity of $\Omega^*$ implies this condition is oblique i.e satisfies $G_p \cdot \gamma > 0$ on $\partial \Omega$, where $\gamma$ is the outer unit normal. This is proved in Section \ref{sec:boundary-estimates}. With $\phi^*$ in hand we can prove solutions intersecting on the boundary are the same. 
\begin{proof}[Proof (Theorem \ref{thm:u:boundary})]
 Using Theorem \ref{thm:u:main} it suffices to prove there is $x \in \Omega$ with $u(x) = v(x)$.  For a contradiction suppose at some $x_0$ in $\partial \Omega$ we have $u(x_0) = v(x_0)$, yet in $\Omega$ there holds $u > v$. Hopf's lemma (\cite[Lemma 3.4]{GilbargTrudinger01}) yields
\begin{equation}
  \label{eq:u:hopf}
  D_\gamma(u-v)(x_0) < 0,
\end{equation}
(recall from the proof of Proposition \ref{lem:u:harn} $u-v$ solves a linear elliptic inequality). Note no sign condition is needed on the lowest order coefficient in \eqref{eq:u:lin} as $u(x_0)-v(x_0)=0.$ 

Consider $h(t) :=G(x_0,u(x_0),tDv(x_0)+(1-t)Du(x_0))$. By a Taylor series
\[ h(1) = h(0) + h'(0)+h''(\tau)/2,\]
for some $\tau \in [0,1]$. Since $u(x_0) = v(x_0)$ we have $h(1), h(0) =0$. Furthermore convexity implies $h''(\tau) \geq 0$ and hence
\[ 0 \geq h'(0) = G_p \cdot D(v-u),\]
or equivalently $0 \leq G_p\cdot D(u-v).$
Combined with obliqueness we have
\[  D_\gamma(u-v)(x_0) \geq 0,\]
which contradicts \eqref{eq:u:hopf} and thus establishes the result.
\end{proof}

\section{Uniqueness of the Dirichlet problem}
\label{sec:u:uniq-dirichl-probl}

Uniqueness of the Dirichlet problem is a simple Corollary of Lemma \ref{lem:u:aleksandrovmccann} and Hopf's boundary point lemma.

\begin{proof}[Proof (Theorem \ref{thm:u:dirichlet})]
  Throughout the proof we consider the $C^1$ $g$-convex extensions of $u,v$. Assume for a contradiction that there is a nonempty connected component of $\{v < u\}$ denoted $\Omega'$ satisfying $\Omega' \subset \Omega$. We do not discount the possibility $\Omega' = \Omega$. 
  After this extension, the proofs of Lemma \ref{lem:u:aleksandrovmccann} and Corollary \ref{cor:u:dudv} yield $Du \equiv Dv$ on $\partial \Omega'$. In the proof of Corollary \ref{cor:u:dudv} where we used the second boundary value problem to establish equality of \eqref{eq:u:gen1} and \eqref{eq:u:gen2} we now use $Yu(\Omega') \subset Yv(\Omega')$ which follows from the comparison principle. Then $Du \equiv Dv$ on $\partial \Omega'$ contradicts Hopf's lemma applied after linearising as in Lemma \ref{lem:u:harn}.
\end{proof}

\clearpage{}
\clearpage{}\chapter{Global regularity I: Construction of barriers}
\label{chap:r1}

\begin{theorem}\label{thm:r1:main}
  Let $g$ be a generating function satisfying A3w, A4w and A5. Let $u$ be a $g$-convex Aleksandrov solution of \eqref{eq:g:gje} subject to \eqref{eq:g:2bvp}. Assume $f \in C^2(\overline{\Omega}),f^* \in C^2(\overline{\Omega^*})$ and the domains $\Omega,\Omega^*$ are $C^4$, uniformly $g/g^*$-convex with respect to $u$, and satisfy $\overline{\Omega} \subset U, \overline{\Omega^*}\subset V$. Then $u \in C^3(\overline{\Omega})$. 
\end{theorem}

The goal of this chapter and the next is to prove the above theorem.  Here is an outline of the proof: Assume, for now, there exists $v \in C^3(\overline{\Omega})$ solving \eqref{eq:g:gje} subject to \eqref{eq:g:2bvp} and intersecting $u$ at some point in $\Omega$. Our strict $g$-convexity result, Remark \ref{rem:sc:unif-dom}, implies $u$ is strictly $g$-convex and subsequently Theorem \ref{thm:w:regularity} yields $u \in C^3(\Omega)$. Thus $u,v$ are regular enough to apply the uniqueness result, Theorem \ref{thm:u:main}, and conclude $u = v$ so $u \in C^3(\overline{\Omega})$. 

So, the key step remaining is the existence of a globally smooth solution intersecting our Aleksandrov solution. This result is related to Jiang and Trudinger's paper \cite{JiangTrudinger18}. There they used degree theory to prove the existence of $v \in C^3(\overline{\Omega})$ solving \eqref{eq:g:gje} subject to \eqref{eq:g:2bvp}. The degree theory is used in place of the method of continuity because the $u$ dependence implies the linearisations are not \textit{uniquely} solvable.   We modify their construction so as to ensure $v$ takes a prescribed value at a given $x_0$. For the Monge--Amp\`ere and optimal transport cases this is trivial --- we can add a constant. In our case we use the degree theory to obtain a solution which is close to the prescribed value at a given point. Taking a limit of these solutions gives the desired globally smooth solution.  

The degree theory serves, in some sense, as a high-powered version of the method of continuity. So, similarly to the method of continuity, what's required is barrier constructions and apriori estimates. These can all be found in the literature. In particular we use results from \cite{TrudingerWang09,LTW10,LiuTrudinger14,JiangTrudinger14,LTW15,LiuTrudinger16,JiangTrudinger17,JiangTrudinger18,Trudinger20}. For the sake of a unified and complete exposition we are including the details. The reader familiar with the literature can  skip to Section \ref{sec:glob-regul-via} where we give the proof of Theorem \ref{thm:r1:main}.

\section{Construction of a uniformly $g$-convex function and a barrier}
\label{sec:constr-unif-g}

Here we construct a barrier that's required for the $C^2$ Pogorelov estimates in Chapter \ref{chap:r2}. First, we construct a uniformly $g$-convex function that is close to $g$-affine. Heuristically, this function is a $g$-convex analog of $\epsilon|x|^2$ for small $\epsilon$. By uniformly $g$-convex we mean $u \in C^2(\overline{\Omega})$ satisfying $D^2u(x) - g_{ij}(x,Yu(x),Zu(x)) > 0$ on $\overline{\Omega}$. \index{uniformly $g$-convex function}  
\begin{lemma}\label{lem:r1:constr-unif-g}\cite[Lemma 2.1]{JiangTrudinger14}.
  Assume $g$ is a generating function satisfying A3w and $g_0(\cdot):=g(\cdot,y_0,z_0)$ is a $g$-affine function with $y_0 \in \Omega^*$ and $g(\overline{\Omega},y_0,z_0)\subset J$. For $\rho > 0$ sufficiently small there exists $g$-convex $\overline{u} \in C^2(\overline{\Omega})$ satisfying
  \begin{equation}
    \label{eq:r1:u-bar-bound}
     g_0 \leq \overline{u} \leq g_0 + 2\rho\Vert g \Vert_{C^1(\Gamma)}
   \end{equation}
   and
   \begin{equation}
     \label{eq:r1:u-bar-unif}
         D^2\overline{u} - A(\cdot,\overline{u},D\overline{u}) \geq a_0I. 
       \end{equation}
       where $a_0$ depends on $\rho,g$. Finally $Y\overline{u}(\Omega) \subset B_\rho(y_0)$. 
 \end{lemma}
 \begin{proof}
   We assume $\rho$ is small enough to ensure $B_\rho = B_\rho(y_0) \subset\subset\Omega^*$. Set
   \[ v_\rho(y) = z_0 - \sqrt{\rho^2-|y-y_0|^2},\]
   and define $\overline{u}$ as the $g$-transform of $v_\rho$: 
   \begin{equation}
     \label{eq:r1:u-bar-def}
     \overline{u}(x) = \sup_{y \in B_\rho(y_0)}g(x,y,v_\rho(y)) \text{ for }x \in \overline{\Omega}.    
   \end{equation}

   We can immediately prove \eqref{eq:r1:u-bar-bound}. First $v_\rho(y_0) < z_0$ and $g_z < 0$ implies
   \[  g(x,y_0,z_0) <  g(x,y_0,v_\rho(y_0)) \leq \overline{u}(x) .\]
   In addition since the supremum defining $\overline{u}$ is obtained at some $y \in \overline{B_\rho}$ a Taylor series yields
   \begin{align*}
     \overline{u}(x) &= g(x,y,v_\rho(y))\\
     &\leq g(x,y_0,z_0)+|g_y||y-y_0|+|g_z||v_\rho(y)-z_0|\\
     &\leq g_0(x) +  2\rho\Vert g \Vert_{C^1(\Gamma)}.
   \end{align*}

   The proof of \eqref{eq:r1:u-bar-unif} is more involved. The key idea is that, because $\overline{u}$ is the $g$-transform of $v_\rho$, we have $DYu = [DXv_\rho]^{-1}$. So provided $DXv_\rho$ is bounded above we obtain a lower bound for $DYu$ (in the matrix sense). To begin we fix $x \in \Omega$ and show the supremum defining $\overline{u}({x})$ is attained at $y \in B_\rho$ (that is, it is not attained on the boundary). To this end fix $y_1 \in \partial B_\rho$ and set $y_\lambda = \lambda y_1+(1-\lambda) y_0$. A Taylor series implies
   \begin{align*}
     g(x,y_1,v_\rho(y_1))-g(x,y_\lambda,v_\rho(y_\lambda))&= g_y(x,y_\tau,v_\tau)\cdot(y_1-y_\lambda) \\
     &\quad\quad+ g_z(x,y_\tau,v_\tau)(v_\rho(y_1)-v_\rho(y_\lambda)),
   \end{align*}
   for some $\tau \in [0,1]$ and $y_\tau = \tau y_1 + (1-\tau)y_\lambda$, $v_\tau = \tau v_\rho(y_1) + (1-\tau)v_\rho(y_\lambda)$. Observe $|y_1-y_\lambda| = (1-\lambda)\rho$ and  $v_\rho(y_1)-v_\rho(y_\lambda) = \rho\sqrt{1-\lambda^2}$ (since $y_1 \in \partial B_\rho$). Thus
   \begin{equation}
     \label{eq:r1:ineq}
         g(x,y_1,v_\rho(y_1))-g(x,y_\lambda,v_\rho(y_\lambda)) \leq C(1-\lambda)\rho - \delta\rho\sqrt{1-\lambda^2},
   \end{equation}
   for $\delta = \inf_{\overline{\Gamma}}|g_z| > 0$. When $\lambda > \frac{C^2-\delta^2}{C^2+\delta^2}$ the right hand side of \eqref{eq:r1:ineq} is less than 0. Thus the supremum cannot be attained at $y_1$.

   So $y \mapsto g({x},y,v_\rho(y))$ attains an interior max at some $y \in B_\rho$. At this point,
   \[ 0 = D_y[g({x},y,v_\rho(y))] = g_y({x},{y},v_\rho({y})) + g_z({x},{y},v_\rho({y}))Dv_\rho({y}).\]
   Condition A1$^*$ implies for a particular ${y}$ this can only be satisfied for one ${x}$. Necessarily ${x}=X(y,v_\rho(y),Dv_\rho(y))$. Note also, since this implies
   \[ |Dv_\rho({y})| \leq \left|\frac{g_y}{g_z}\right| \leq C,  \]
   and $|Dv_\rho(y)| \rightarrow \infty$ as $|y-y_0|\rightarrow \rho$ we have ${y}\in B_{(1-\lambda)\rho}(y_0)$ for $\lambda$ depending only on $\sup \left|g_y/g_z\right|$.

   Now, as a supremum of $g$-affine functions $\overline{u}$ is $g$-convex for $\rho$ small enough to ensure $g(\overline{\Omega},y,v_\rho(y)) \subset J$. We wish to show $Y\overline{u}({x})$ is a singleton. Since $\overline{u}$ is the $g$-transform of $v_\rho(y)$ it suffices to show $v_\rho$ is strictly $g^*$-convex. This follows because $v_\rho$ is a hemisphere of radius $\rho$ and so has curvature $1/\rho$. Thus for $\rho$ sufficiently small
   \[ D_{ij}v_\rho - A^*_{ij}(\cdot,v_\rho,Dv_\rho) \geq \frac{1}{2\rho}I.\]
   Thus $v$ is uniformly $g^*$-convex on $B_{\rho}$. Subsequently $u$ is $C^1$ and the supremum in \eqref{eq:r1:u-bar-def} occurs at $y = Y(x,\overline{u},D\overline{u})$.

   In particular $Y\overline{u} = (Xv_\rho)^{-1}$ (Lemma \ref{lem:g:gstarsubdiff}). Since $Xv_\rho$ is $C^1$ with nonvanishing Jacobian determinant so is $Y\overline{u}$. Writing
   \[ D\overline{u} = g(x,Y\overline{u}(x),g^*(x,Y\overline{u}(x),\overline{u}(x))),\]
   we see $\overline{u}$ is $C^2$. The estimate \eqref{eq:r1:u-bar-unif} follows from the relations
   \begin{align*}
     DY\overline{u} = [DXv_\rho]^{-1}\\
     DY\overline{u} = Y_p[D^2\overline{u}-A(\cdot,\overline{u},D\overline{u})]\\
     DXv_\rho = X_q[D^2v_\rho-A^*(\cdot,v_\rho,Dv_\rho)].
   \end{align*}

   We note by construction $Y\overline{u}(\Omega) \subset B_\rho(y_0) \subset\subset \Omega^*$ for $\rho$ sufficiently small. 
 \end{proof}

 The function $\overline{u}$ is used in two ways. The first is the following barrier construction, the second is in Section \ref{sec:constr-funct-appr}. 
For our barrier construction we take  a uniformly $g$-convex function lying above a solution of \eqref{eq:g:mate} and construct a barrier for the linearised equation. This construction is due to Jiang and Trudinger \cite{JiangTrudinger18}, though we use a simpler proof from a later paper of theirs \cite{JiangTrudinger17}. The linearized operators we work in this chapter and the next are not the full linearizations of the associated PDE. Indeed we cannot control the sign of the lowest order linearized terms. However these terms are such that, as we will see, they can always be included in the relevant constants. 

 \begin{theorem}\label{thm:r1:lin}
   Let $g$ be a generating function satisfying A3w and A4w. Let $u \in C^2(\Omega)$ satisfy $D^2u-A(\cdot,u,Du) \geq 0$ and associate the linear operator \index[notation]{$\overline{L}$, \ \ differential operator $\overline{L}v :=w^{ij}[D_{ij}v-D_{p_k}A_{ij}(x,u,Du)D_kv]$} $\overline{L}_u[\cdot]$ defined on a function $v$ by 
   \[ \overline{L}v = \overline{L}_uv:=w^{ij}[D_{ij}v-D_{p_k}A_{ij}(x,u,Du)D_kv].\]
   Here $w = D^2u - A(\cdot,u,Du)$. Let $\overline{u} \in C^2(\Omega)$ satisfy $D^2\overline{u} - A(\cdot,\overline{u},D\overline{u}) \geq a_0I$ and $\overline{u} \geq u$. There is $\epsilon_1,C,K>0$ depending on  $g,\Vert u \Vert_{C^1(\Omega)},\det w,$ and $\overline{u}$ such that
   \begin{equation}
     \label{eq:r1:lin}
         \overline{L}(e^{K(\overline{u}-u)}) \geq \epsilon_1 w^{ii} - C.
   \end{equation}
   \end{theorem}
 \begin{proof}
   First we note a consequence of the A3w condition. Recall \eqref{eq:g:a3w-equiv}, which is that A3w implies for any $\xi,\eta \in \mathbf{R}^n$
   \begin{equation}
     \label{eq:r1:a3w-equiv}
         D_{p_kp_l}A_{ij}\xi_i\xi_j\eta_k\eta_l \geq - C|\xi||\eta|(\xi\cdot\eta).
   \end{equation}
    Now express $w^{ij}$  as
   \[ w^{ij} = \sum_{s=1}^n\lambda_s\phi^s_i\phi^s_j,\]
   in terms of its eigenvalues $\lambda_s$ and corresponding normalized eigenvectors $\phi^s$. Then \eqref{eq:r1:a3w-equiv} implies
   \begin{align*}
     w^{ij}D_{p_kp_l}A_{ij}\eta_k\eta_l &= \sum_{s} \lambda_sD_{p_kp_l}A_{ij}\phi^s_i\phi^s_j\eta_k\eta_l\\
     & \geq - C\sum_{s}\lambda_s|\phi^s||\eta|(\phi^s\cdot\eta).
   \end{align*}
   Cauchy's inequality with epsilon implies, for any $\epsilon>0$,
   \begin{align*}
     w^{ij}D_{p_kp_l}A_{ij}\eta_k\eta_l  &\geq - \epsilon\sum_{s} \lambda_s |\phi^s|^2|\eta|^2 - \frac{C}{\epsilon}\sum_{s}\lambda_s(\phi^s\cdot\eta)^2\\
     &\geq -\epsilon w^{ii}|\eta|^2 - \frac{C}{\epsilon}\sum_{s}\lambda_s(\phi^s_i\eta_i)(\phi^s_j\eta_j),
   \end{align*}
   where we've used $|\phi^s| = 1$. Simplifying the final term to $w^{ij}\eta_i\eta_j$ we have
   \begin{equation}
     \label{eq:r1:a3w-use}
     w^{ij}D_{p_kp_l}A_{ij}\eta_k\eta_l \geq -\epsilon w^{ii}|\eta|^2 - \frac{C}{\epsilon}w^{ij}\eta_i\eta_j.
   \end{equation}

   We use this to obtain \eqref{eq:r1:lin}. Since $\overline{u} \geq u$, A4w implies for $\delta=a_0/2$ 
   \[ D^2\overline{u} - A(\cdot,u,D\overline{u}) - \delta I \geq D^2\overline{u} - A(\cdot,\overline{u},D\overline{u}) - \delta I \geq \delta I. \]
   Then
   \begin{align*}
     \overline{L}(\overline{u}-u) &= w^{ij}[D_{ij}(\overline{u}-u) - D_{p_k}A_{ij}(\cdot,u,Du)D_k(\overline{u}-u)]\\
     &= w^{ij}[\delta I + (D_{ij}\overline{u} - A(\cdot,u,D\overline{u}) - \delta I) - (D_{ij}u - A(\cdot,u,Du))]\\
                       &\quad\quad + w^{ij}[A_{ij}(\cdot,u,D\overline{u}) - A_{ij}(\cdot,u,Du) - D_{p_k}A_{ij}(\cdot,u,Du)D_k(\overline{u}-u)]\\
                       &\geq \delta w^{ii} -n + \frac{1}{2}w^{ij}D_{p_kp_l}A_{ij}(\cdot,u,p_\tau)D_k(\overline{u}-u)D_l(\overline{u}-u).
   \end{align*}
   We've used a Taylor series in the last line. Combined with \eqref{eq:r1:a3w-use} we have
   \begin{equation}
     \label{eq:r1:diff-ineq}
      \overline{L}(\overline{u}-u) \geq \delta w^{ii} -n - \epsilon w^{ii}|D(\overline{u}-u)|^2 - \frac{C}{\epsilon}w^{ij}D_i(\overline{u}-u)D_j(\overline{u}-u).
   \end{equation}

   Via direct calculation 
   \[ \overline{L}(e^{K(\overline{u}-u)}) =K e^{K(\overline{u}-u)}\overline{L}(\overline{u}-u) + K^2e^{K(\overline{u}-u)}w^{ij}D_i(\overline{u}-u)D_j(\overline{u}-u).\]
   Which, with \eqref{eq:r1:diff-ineq}, gives
   \begin{align*}
     \overline{L}(e^{K(\overline{u}-u)}) &\geq Ke^{K(\overline{u}-u)}w^{ii}[\delta  -\epsilon |D(\overline{u}-u)|^2 ] - Kne^{K(\overline{u}-u)}\\
     &\quad\quad+e^{K(\overline{u}-u)}w^{ij}D_i(\overline{u}-u)D_j(\overline{u}-u)[K^2 - \frac{C}{\epsilon}] .
   \end{align*}
   Choosing first $\epsilon$ small then $K$ large yields \eqref{eq:r1:lin}.   \end{proof}

\section{Construction of a function approximately satisfying the second boundary value problem}
\label{sec:constr-funct-appr}

The goal of this section is to construct a uniformly $g$-convex function which approximately satisfies the second boundary value problem. We use the notation
\[\Omega_\delta := \{x ; \text{dist}(x,\Omega) < \delta\},\]
and that $\{\Omega_{\delta,\epsilon}^*\}$ is a doubly parametrised family of domains satisfying the following convergence properties with respect to the Hausdorff distance
\[ d_{\mathbf{H}}(\Omega_{\delta,\epsilon}^*,\Omega_{\delta}^*) \leq C\epsilon,\]
where $\Omega_{\delta}^*:= \lim_{\epsilon \rightarrow 0}\Omega_{\delta,\epsilon}^*$ with respect to the Hausdorff distance, is a family of domains satisfying $\Omega_{\delta}^* \rightarrow \Omega^*$. 

\begin{theorem}\label{thm:r1:hom-start}
 Let $0< \epsilon << \delta$ be fixed small . Let $\overline{u}$ be the uniformly $g$-convex function from Lemma \ref{lem:r1:constr-unif-g}. There exists a uniformly $g$-convex function $u_{\epsilon}\in C^\infty(\overline{\Omega_\delta})$ satisfying
  \[ Yu_{\epsilon}(\Omega_\delta) = \Omega^*_{\delta,\epsilon},\]
  and
  \[ \lim_{\epsilon \rightarrow 0}u_{\epsilon} = \overline{u},\]
  pointwise on $\Omega$ and uniformly on compact subsets of $\Omega$. 
\end{theorem}

There are a number of steps to the construction of $u_\epsilon$. The first is to extend $\overline{u}$ to $\Omega_\delta$ by taking a supremum of all $g$-affine functions with $Y$ mapping in $\Omega^*$ and which lie below $\overline{u}$. Then we perturb this function in $\Omega_\delta \setminus \Omega$ to obtain the uniform $g$-convexity and conclude by mollifying. First we need a geometric lemma from \cite{JiangTrudinger18}.

\begin{lemma}\label{lem:r1:geometric}
  Let $g$ be a generating function satisfying A3w and A4w. Assume $u \in C^1(\overline{\Omega})$ is $g$-convex on a domain $\Omega$ which is uniformly $g$-convex with respect to  $y_0,z_0$. Assume at $x_0 \in \partial \Omega$ 
  \[ h(x):= u(x)-g(x,y_0,z_0),\]
  satisfies  $h(x_0) = 0$ and $Dh(x_0) = -s\gamma_0$ for $\gamma_0$ the outer normal to $\Omega$ at $x_0$ and $s>0$. Then $h > 0$ on $\overline{\Omega}\setminus\{x_0\}$.
\end{lemma}
\begin{proof}
We assume $u \in C^2(\overline{\Omega})$ --- if not apply the argument to the unique $g$-support at $x_0$. 
  Fix $x_1 \in \overline{\Omega}$ and let $\{x_\theta\}_{\theta \in [0,1]}$ denote the $g$-segment with respect to $y_0,z_0$ joining $x_0$ to $x_1$. Lemma \ref{lem:g:maindiffineq} and A4w implies $h(\theta) := u(x_\theta)-g(x_\theta,y_0,z_0)$ satisfies $h''(\theta) > -K|h'(\theta)|$ whenever $h(\theta) \geq 0$. In addition via the uniform $g$-convexity $\dot{x}_0$ points into the domain. Thus
  \[ h'(0) = Dh(x_0)\dot{x}_0 = (-s\gamma_0)\dot{x}_0 > 0.\]
  Which combined with $h(0) = 0$ and the differential inequality implies $h(x_1) > 0$. 
\end{proof}

We use this lemma to extend the uniformly $g$-convex function $\overline{u}$ from Lemma \ref{lem:r1:constr-unif-g} to a neighbourhood of $\Omega$. The following lemma appears in \cite{JiangTrudinger18,TrudingerWang09}. 

\begin{lemma}\label{lem:r1:extens-property}
  Let $\overline{u}$ be the uniformly $g$-convex function constructed in Lemma \ref{lem:r1:constr-unif-g}. Assume $\Omega$ is uniformly $g$-convex with respect to $\Omega^* \times Z\overline{u}(\Omega)$ and $\Omega^*$ is uniformly $g^*$-convex with respect to $\overline{u}$. Define
  \begin{align}
    \nonumber u_1&:\Omega_\delta \rightarrow \mathbf{R},\\
    u_1(x) &= \sup\{g(\cdot,y,z); y \in \Omega^*, g(\cdot,y,z) \leq u \text{ in }\Omega\}. \label{lem:r1:extend}
  \end{align}
  For $\delta$ sufficiently small $u_1$ is a $g$-convex extension of $\overline{u}$ to $\Omega_\delta$ satisfying $Yu_1(\Omega_\delta) = \overline{\Omega^*}$. Moreover for each $x \in\Omega_\delta\setminus\overline{\Omega}$ there is unique $x_b \in \partial \Omega, \ y_b=Yu(x_b) \in \partial \Omega^*$ such that if $\{x_\theta\}$ is the $g$-segment joining $x_0:=x_b$ to $x_1:=x$ with respect to $y_b, z_b:=g^*(x_b,y_b,\overline{u}(x_b))$ 
  \begin{align*}
    u_1(x_\theta) = g(x_\theta,y_b,z_b),\\
    Yu_1(x_\theta) = y_b \text{ for }\theta \in (0,1].
  \end{align*}
 Finally if $r<\delta$ is fixed and we consider the mappings from $\partial \Omega_r$ to $\partial \Omega$ and $\partial \Omega^*$ defined, using the above notation, by $x \mapsto x_b$ and $x \mapsto y_b$, then these are $C^2$ diffeomorphisms. 
\end{lemma}
\begin{proof}
  To begin take $y_b \in \partial \Omega^*$. Consider $g_\delta(\cdot):=g(\cdot,y_b,\delta)$ for $\delta$ initially large then decrease $\delta$ until $g_\delta$ first touches $\overline{u}$ from below at some $x_b\in \overline{\Omega} $. Because $y_b \in \partial \Omega^*$ and $\omega^*:=Y\overline{u}(\Omega) \subset\subset \Omega^*$ we have $y_b \not\in Y\overline{u}(\Omega)$ so necessarily $x_b \in \partial \Omega.$ Put $z_b=g^*(x_b,y_b,\overline{u}(x_b))$.

  By construction $g(\cdot,y_b,z_b) \leq u(\cdot)$ in $\Omega$ with equality at $x_b$. Thus
  \begin{equation}
    \label{eq:r1:ybdef}
       g_x(x_b,y_b,z_b) = D\overline{u}(x_b) + s\gamma,
  \end{equation}
  for $\gamma$ the outer unit normal to $\partial \Omega$ at $x_b$ and some $s > 0$. Here $s >0$ because $y_b \not\in Y\overline{u}(\Omega)$ implies $g_x(x_b,y_b,z_b) \neq D\overline{u}(x_b)$. Then Lemma \ref{lem:r1:geometric} implies
  \begin{equation}
    \label{eq:r1:injective}
      g(x,y_b,z_b) < \overline{u} \text{ in }\overline{\Omega} \setminus \{x_b\}.
  \end{equation}
  In addition by \eqref{eq:r1:ybdef}
  \[ y_b = Y(x_b,\overline{u}(x_b),D\overline{u}(x_b)+s\gamma).\]
  We let $\{y_\theta\}_{\theta \in [0,1]}$ denote the $g^*$-segment defined by
  \begin{equation}
    \label{eq:r1:y-theta}
       y_\theta = Y(x_b,\overline{u}(x_b),D\overline{u}(x_b)+\theta s \gamma).
  \end{equation}
  That is, $\{y_\theta\}_{\theta \in [0,1]}$ joins $y_0:=Y\overline{u}(x_b) \in \partial \omega^*$ to $y_1:= y_b \in \partial \Omega^*.$

  Note by the uniform $g$-convexity of $\overline{u}$  and the uniform $g^*$-convexity of $\Omega^*$ the $g^*$-segment $\{y_\theta\}_{\theta \in [0,1]}$ intersects $\omega^*$ only at $y_0$ and $\Omega^*$ only at $y_1.$ Thus the map taking $x \in \partial \Omega$ to the unique $y \in \partial \Omega^*$ at which the $g^*$-segment $Y(x,\overline{u}(x),D\overline{u}(x)+\theta\gamma(x))$ intersects $\Omega^*$ is a bijection from $\partial \Omega$ to $\partial \Omega^*$ (injective by  \eqref{eq:r1:injective}). We call this map $T$, so using the above notation  $T(x_b) = y_b$. Note $T$ is $C^2$ because the boundaries are $C^2$ and $\det DY_p \neq 0$ (i.e we solve for $T$ and obtain the differentiability by the implicit function theorem).

  Now fix $x_b \in \partial \Omega$. Consider a small exterior ball $B$ to $\Omega$ at $x_b$, that is $B\subset\Omega^c$ and $\overline{B} \cap \overline{\Omega} = x_b$.  Let $\tilde{\gamma}$ denote the outer normal to this ball at $x_b$, explicitly $\tilde{\gamma}(x_b) = -\gamma(x_b)$. Then  \eqref{eq:r1:y-theta} implies
  \[ g_x(x_b,y_1,z_1) - g_x(x_b,y_\theta,z_\theta) = -s(1-\theta)\overline{\gamma},\]
  for $z_\theta = g^*(x_b,y_\theta,u(x_b))$. So by another application of Lemma \ref{lem:r1:geometric}, this time applied to $h(x) = g(x,y_1,z_1)-g(x,y_\theta,z_\theta)$ in the exterior ball,
  \[g(x,y_1,z_1) > g(x,y_\theta,z_\theta) \text{ in }B_r. \]
  Note we must choose the exterior ball so small as to be uniformly convex. Thus provided $\delta$ is chosen small it is equivalent to define
  \begin{align}
\label{eq:r1:u1-bound-def}     u_1(x) =
    \begin{cases}
      \overline{u}(x) &\text{ in }\overline{\Omega}\\
      \sup\{g(x,y_b,z_b); y_b = T(x_b) \in \partial \Omega^* \text{ for }x_b \in \partial \Omega\} &\text{ in }\Omega_\delta\setminus\overline{\Omega}.  
    \end{cases}   
  \end{align}

  This setup lets us prove the second part of the theorem. We begin by showing $u_1$ has a unique support at each point in $\Omega^\delta\setminus \Omega$. We've already shown for all $x \in \Omega_\delta \setminus \overline{\Omega}$ that $Yu(x) \subset \partial \Omega^*$. To show the support is unique we consider the $g^*$-transform 
  \[ v_1(y ) = \sup_{x \in \Omega^\delta} g^*(x,y,u_1(x)). \]
In the first part of the proof we've shown if $y \in \Omega^*$ and $y \in Yu(x)$ then $x \in \overline{\Omega}$ (i.e if $Yu(x) \notin\partial \Omega^*$ then $x \notin \Omega_\delta \setminus \overline{\Omega}$). Recalling $x \in Xv_1(y)$ implies $y \in Yu_1(x)$ we have $Xv_1(\Omega^*) = \overline{\Omega}$. Subsequently if $x \in \Omega_\delta \setminus \overline{\Omega}$ and, as before, we shift up the support $g^*(\cdot,x,u)$ until it first touches $v_1$ from below, necessarily  at $y \in \partial \Omega^*$ with $u=g(x,y,v_1(y))$. Then by the same argument as \eqref{eq:r1:ybdef} we have $x = X(y,v_1(y),Dv_1(y) + s\gamma^*)$ for $\gamma^*$ the outer normal to $\Omega^*$ at $y$. Applying the dual version of Lemma \ref{lem:r1:geometric} we have $v_1 >g^*(\cdot,x,u) $ on $\overline{\Omega^*} \setminus \{y\}$. This yields the uniqueness of the support at $x$. Since $\Omega_\delta$ is also uniformly $g$-convex with respect to $u_1$ provided $\delta$ is sufficiently small we can argue as before that the map which takes $x \in \partial \Omega_r$ to the above $y \in Yu_1(x) \cap \partial \Omega^*$ is a $C^2$ diffeomorphism. Moreover for the equality along $g$-segments note if $x \in \Omega_\delta \setminus \Omega$ and $y \in Yu(x)$ then also $y \in Yu(x_b)$ for some $x_b \in \partial \Omega$. Equality along the $g$-segment follows from the $g$-convexity of contact sets (Theorem \ref{thm:g:gconvsection}). 
\end{proof}

The function $u_1$ is tangentially uniformly convex in $\Omega_\delta\setminus\Omega$. That is, if $\tau$ is tangential to $\partial \Omega^r$ we have
\begin{equation}
  \label{eq:r1:tang-unif}
  [D_{ij}u_1-A_{ij}(\cdot,u_1,Du_1)]\tau_i\tau_j \geq \lambda_0.
\end{equation}
Indeed, to see this note we have the tangential uniform convexity on $\partial \Omega$ by the same property for $\overline{u}$, and this extends to a small neighbourhood of $\partial \Omega$ by considering the diffeomorphisms in Lemma \ref{lem:r1:extens-property}. In particular $A_{ij}$ depends on $x,Yu(x),Zu(x)$ where by the previous lemma $Yu(x) = Yu(x_b)$ similarly for $Zu$. 

To obtain the same in the normal direction to $\partial \Omega_r$ we let $d(x) = \text{dist}(x,\partial \Omega)$ and set
\begin{align*}
  u_0(x) :=
  \begin{cases}
    u_1(x) \text{ for }x \in \Omega,\\
    u_1(x)+td(x)^2 \text{ for }x \in\Omega_\delta\setminus\Omega. 
  \end{cases}
\end{align*}
Provided $\delta$ is taken sufficiently small (depending on the $C^2$ norm of $\partial \Omega$) we can ensure for $x \in \partial \Omega_r$
\[ D_{ij}[d(x)^2]\gamma_i\gamma_j \geq \lambda_0 > 0\]
for $\gamma$ the outer normal to $\partial \Omega_r.$
For a choice of $t$ large then $\delta$ small we can ensure $Du_0,u_0$ are as close to $Du_1,u_1$ in $\Omega_\delta \setminus \overline{\Omega}$ as desired. Subsequently $\Omega_\delta^*:=Yu_0(\Omega_\delta)$ is a small perturbation of $\Omega^* = Yu_1(\Omega_\delta)$. That is, $\Omega_\delta^* \rightarrow \Omega^*$ in the Hausdorff distance. Moreover for any $K>0$ by choosing $t$ sufficiently large we can ensure that in $\Omega_\delta \setminus \overline{\Omega}$ there holds
\[ [D_{ij}u_1-A_{ij}(\cdot,u_1,Du_1)]\gamma_i\gamma_j \geq K.\]
In combination with \eqref{eq:r1:tang-unif} we obtain the uniform $g$-convexity 

\begin{equation}
  \label{eq:r2:0-unif}
   [D^2u_0 - A_{ij}(\cdot,u_0,Du_0)]\xi_i\xi_j \geq \lambda_0, 
\end{equation}
in $\Omega_\delta \setminus \partial \Omega$. 

Now for $\rho_\epsilon$ the standard mollifier and $\epsilon < < \delta$ we define $u_\epsilon$ on $\Omega_{\delta}$\footnote{For the mollification to be defined on $\Omega_\delta$ we  assume our earlier extension was to $\Omega_{2\delta}$} by
\[ u_\epsilon(x) = \int_{\mathbf{R^n}}u(x)\rho_{\epsilon}(x-y) \ dy.\]

Let's conclude the proof of Theorem \ref{thm:r1:hom-start} by showing $u_\epsilon$ is the desired function. 
\begin{proof}[Proof. (Theorem \ref{thm:r1:hom-start})]
  The convergence claim, $\lim_{\epsilon \rightarrow 0}u_\epsilon = \overline{u}$ in $\Omega$, is immediate because, in $\Omega$,  $u_\epsilon$ is the mollification of  $\overline{u}$. Moreover because in $\Omega_\delta/\partial \Omega$ we have
  \[ u_\epsilon \rightarrow u_0, Du_\epsilon \rightarrow Du_0,\] the set $\Omega^*_{\delta,\epsilon}:=Yu_\epsilon(\Omega_\delta)$ is a small perturbation of $\Omega_\delta^*,$. This proves the stated properties of $\Omega^*_{\delta,\epsilon}$.

  Now we prove the uniform $g$-convexity. We note 
  \begin{align}
    \label{eq:r1:du-eq}    D_iu_\epsilon(x) &= \int_{B_\epsilon(x)} D_iu_0(x-y)\rho_\epsilon(y) \ dy.
  \end{align}

  The corresponding identity for second derivatives relies on the semiconvexity of $u$. Indeed let $E_u$ denote the set where $u_0$ does not have second derivatives in the Aleksandrov sense. We recall the characterisation of the second derivatives of convex (and subsequently semiconvex) functions as signed measures $\mu^{ij}$ \cite[Theorem 6.8]{EvansGariepy15}. Thus
     \begin{align}
   \label{eq:r1:2deriv1}    D_{ii}u_\epsilon(x)  &= \int_{\mathbf{R}^n \cap \{y ; x-y \not\in E_u\}} D_{ii}u_0(x-y)\rho_\epsilon(y) \ dy \\
             \nonumber        &\quad\quad+\int_{\{y; x-y \in E_u\}} \rho_\epsilon(y) \ d\mu^{ii}, \\
    \label{eq:r1:2deriv2}   &\geq \int_{\mathbf{R}^n \cap \{y ; x-y \not\in E_u\}} D_{ii}u_0(x-y)\rho_\epsilon(y) \ dy
  \end{align}
  because $\mu^{ii}$ is a nonnegative measure at singular points.
  
  Now we consider the uniform $g$-convexity in the sets $U_1 = \{x ; d(x) > \epsilon\}$, $U_2 = \{x ; d(x) \in (\epsilon',\epsilon)\}$, $U_3 = \{x ; d(x) < \epsilon'\}$. Here $\epsilon' = (1-\sigma)\epsilon$ for $\sigma$ to be chosen close to one.

  We begin with $U_1$. Note $u_\epsilon \rightarrow u_0$ uniformly on compact subsets of $\Omega_\delta$ and in particular on  $U_1$. Moreover the same is true for $Du_\epsilon \rightarrow Du_0$ on $U_1$ in the sense that for each $\alpha > 0$ there is $\epsilon$ small such that $|Du_\epsilon(x) - Du_0(x)| < \alpha$ for all $x \in U_1$. To see this observe $Du_0|_{\Omega},Du_0|_{\overline{\Omega_\delta} \setminus \overline{\Omega}}$ are uniformly continuous. Inspecting the usual proof of locally uniform convergence for mollifications we see if the function is initially uniform continuous then the convergence is uniform in the above sense. The uniform $g$-convexity follows by the uniform convergence along with \eqref{eq:r2:0-unif}, and \eqref{eq:r1:2deriv2}. 

  Similarly in $U_2$ we use that $|Du_\epsilon(x) - Du_0(x)|$ can be controlled uniformly in $U_2$ in terms of $\epsilon,\sigma$ sufficiently small. Thus, as in $U_1$, via a choice of $\epsilon$ small and $\sigma$ close to $1$ in conjunction with \eqref{eq:r1:2deriv2}, we obtain the uniform $g$-convexity in $U_2$. Importantly the choice of $\sigma$ is independent of $\epsilon.$

  Finally we prove the uniform $g$-convexity in $U_3$. Take $x \in U_3$. Assume $0$ is the closest point to $x$ in $\partial \Omega$ and choose coordinates so $x = |x|e_n$ with $e_1,\dots,e_{n-1}$ tangential to $\partial \Omega$ at $0$. To begin we focus on a tangential direction, without loss of generality assumed to be $e_1$. We aim to show that for some $x_0>0$
  \begin{align}
      D_{11}u_\epsilon(x)-A_{11}(x,u_\epsilon(x),Du_\epsilon(x)) \geq c_0.\label{eq:r1:11deriv1}
  \end{align}
  We have the uniform convergence $u_\epsilon \rightarrow u$. Moreover by the characterization of the extension in Lemma \ref{lem:r1:extens-property} we see the convergence $D_1u_\epsilon \rightarrow D_1u_0$ is uniform in a neighbourhood of $0$. Thus it suffices to prove
  \begin{align}
    D_{11}u_\epsilon(x)-A_{11}(x,u(x),D_1u(x),D'u_\epsilon(x)) \geq c_0, \label{eq:r1:11deriv2}
  \end{align}
  for $D'u_\epsilon= (D_2u_\epsilon,\dots,D_nu_\epsilon)$. The A3w condition implies codimension one convexity, that is $p' \mapsto A_{11}(x,u,p_1,p')$ is convex. Then by Jensen's inequality
  \[A_{11}(x,u(x),D_1u(x),D'u_\epsilon(x)) \leq \int A_{11}(x,u(x),D_1u(x),D'u(y))\rho_\epsilon(x-y). \]
  This along with \eqref{eq:r1:2deriv2} yields \eqref{eq:r1:11deriv2} and subsequently \eqref{eq:r1:11deriv1}. Finally,  for any $K$ large a choice of $\epsilon$ sufficiently small ensures
  \begin{align}
       D_{nn}u_\epsilon(x)-A_{nn}(x,u(x),D_1u(x),D'u_\epsilon(x)) \geq K. \label{eq:r1:nnderiv}
  \end{align}
  This follows by \eqref{eq:r1:2deriv1}. Indeed $\omega^* \subset\subset \Omega^*$ implies $D_nu$ has a jump discontinuity at $0$. Thus $\mu^{nn} \geq c_0 \delta_0$ for some  $c_0>0$ and $\delta_0$ the dirac measure at $0$. Furthermore $|x-0| < (1-\sigma)\epsilon$ implies $\rho_\epsilon(0) \geq c_n \epsilon^{-n}$. Thus, using \eqref{eq:r1:2deriv1}, the left hand side of \eqref{eq:r1:nnderiv} can be made large by taking $\epsilon$ small. This completes the proof of the uniform convexity.   
\end{proof}

\section{Characterizations of uniformly convex domains }
\label{sec:char-unif-conv}

We conclude this chapter by showing uniform $g/g^*$ convexity of $\Omega/\Omega^*$ implies the existence of defining functions satisfying certain differential inequalities. These computations are well known, however we relied on the details provided by Kitagawa \cite{Kitagawa12}.  
\begin{lemma}\label{lem:r1:phi-def}\index[notation]{$\phi$, \ \ a particular defining function for $\Omega$}
  Assume $u:\overline{\Omega}\rightarrow \mathbf{R}$ is a uniformly $g$-convex function. Assume $\Omega$ is uniformly $g$-convex with respect to $u$. There is a defining function $\phi$ for $\Omega$, that is a function $\phi$ satisfying $ \Omega = \{ \phi < 0\},$ and $D\phi = \gamma$ the outer unit normal on $\partial \Omega$, such that some $c_0 > 0$
  \[D_{ij}\phi - D_{p_k}A_{ij}(x,u,Du)D_k\phi \geq c_0I,\]
  in a neighbourhood $\Omega^\epsilon = \{x \in \overline{\Omega}; \text{dist}(x,\partial \Omega) < \epsilon\}.$ Here $c_0,\epsilon$ depend on $g,\Omega,Yu(\overline{\Omega}),$ and $Zu(\overline{\Omega}).$
\end{lemma}
\begin{proof}
  Our goal is an estimate of the form
  \begin{equation}
    \label{eq:r1:phi-full}
       [D_{ij}\phi - D_{p_k}A_{ij}(x,u,Du)D_k\phi]\xi_i\xi_j \geq c_0,
  \end{equation}
  for unit vectors $\xi$. We take the defining function $\phi = K d(x)^2 - d(x)$ where $d(x)$ is the signed distance function to $\partial \Omega$ (positive in $\Omega$) and $K$ will be chosen large. Note $D\phi = \gamma$. First we prove that on $\partial \Omega$
  \begin{align}
  \label{eq:r1:outer-norm}
   [D_{ij}\phi - D_{p_k}A_{ij}&(x,u,Du)D_k\phi]\tau_i\tau_j  \\
  \nonumber  &=[D_i \gamma_j-A_{ij,p_k}(x,u,p)\gamma_k]\tau_i\tau_j \geq c_0|\tau|^2,
  \end{align}
  for tangent vectors $\tau$. Let $(y,z) = (Yu(x),Zu(x))$ for some $x \in \overline{\Omega}$. Then
  \[ \Omega_{y,z}:= \{q ; X(y,z,q) \in \Omega\}\]
  is a uniformly convex domain. Thus the function $q \mapsto \phi(X(y,z,q))$ is a defining function for the uniformly convex domain $\Omega_{y,z}$. Then if $q_0 \in \partial \Omega_{y,z}$ is given and $\tau$ is tangential to $\partial \Omega_{y,z}$ at $q_0$ there holds
  \[ D_{q_kq_l}(\phi(X(y,z,q))|_{q=q_0}\tau_k\tau_l \geq c_0|\tau|^2.\]
  Computing the derivative yields
  \begin{equation}
    \label{eq:r1:to-sub}
    [\phi_{ij}D_{q_k}X^iD_{q_l}X^j+\phi_iD_{q_kq_l}X^i]\tau_k\tau_l \geq c_0|\tau|^2.
  \end{equation}
  Note $D_{q_k}X\tau_k$ is tangential to $\partial \Omega$ at $X(y,z,q_0)$ (and all tangent vectors to $\partial \Omega$ can be expressed in this way).  Note also $D\phi$ is the outer normal to $\partial \Omega$. Then \eqref{eq:r1:outer-norm}  follows by the following calculations in which we compute
  \begin{equation}
    \label{eq:r1:qq-ident}
       D_{q_kq_l}X^\alpha = \frac{g_{j,z}}{g_z}D_{q_k}X^\alpha D_{q_l}X^j + \frac{g_{i,z}}{g_z}D_{q_k}X^iD_{q_l}X^\alpha-D_{p_\alpha}g_{ij}D_{q_k}X^iD_{q_l}X^j. 
     \end{equation}
     We substitute into \eqref{eq:r1:to-sub} and obtain \eqref{eq:r1:outer-norm} by noting  orthogonality implies that $\phi_\alpha D_{q_k}X^\alpha \tau_k = 0$. 

  To prove \eqref{eq:r1:qq-ident} recall the equation defining $X$, which is
  \begin{align*}
\frac{-g_{y_a}}{g_z}(X(y,z,q),y,z) = -q_a,
  \end{align*}
  and differentiate with respect to $q_k$ to obtain
  \[ \frac{1}{g_z}E_{i,a}D_{q_k}X^i = \delta_{ak}.\]
  Differentiating again yields
  \[ \frac{1}{g_z}E_{ia}D_{q_kq_l}X^i = - D_{x_j}\left(\frac{1}{g_z}E_{ia}\right)D_{q_k}X^iD_{q_l}X^j.\]
  This gives \eqref{eq:r1:qq-ident} by direct calculation. These calculations are straightforward but need an identity from way back, \eqref{eq:g:3invderiv} from Chapter \ref{chap:g}.

  Now we obtain \eqref{eq:r1:phi-full}. The following calculations take place with $D_{p_k}A_{ij}$ replaced by $D_{p_k}g_{ij}(\cdot,Y(\cdot,u,p),Z(\cdot,u,p))|_{Y = y,Z=z}$. We compute
  \begin{align*}
    D_i \phi = 2Kdd_i - d_i\\
    D_{ij}\phi = 2Kdd_{ij} + 2Kd_id_j - d_{ij}
  \end{align*}
  Then on the boundary where $d(x) = 0$
  \begin{align*}
    [D_{ij}\phi - D_{p_k}A_{ij}&(x,u,Du)D_k\phi]\xi_i\xi_j \\
    &= [-d_{ij} + D_{p_k}A_{ij}(x,u,Du)D_kd ]\xi_i\xi_j + 2Kd_id_j\xi_i\xi_j. 
  \end{align*}
  We decompose the vector $\xi = \tau + a\gamma$ into a tangential and outer normal component. Then using \eqref{eq:r1:outer-norm} and $d_i\tau_i = 0,d_i\gamma_i = 1$
  \begin{align*}
    [D_{ij}\phi - D_{p_k}g_{ij}(x,y,z)D_k\phi]\xi_i\xi_j &\geq 2[D_{ij}\phi - D_{p_k}A_{ij}(x,u,Du)D_k\phi]\tau_i(a\gamma_j)
                        \\& +c_0|\tau|^2 +(2K-C)a^2.
  \end{align*}
  A standard application of Cauchy's inequality (with epsilon) yields
  \[ [D_{ij}\phi - D_{p_k}A_{ij}(x,u,Du)D_k\phi]\tau_i(a\gamma_i) \geq -\frac{c_0}{2}|\tau|^2- C a^2,\]
 where $C$ depends on $c_0$. So by a choice of $K$ large
  \[ [D_{ij}\phi - D_{p_k}g_{ij}(x,y,z)D_k\phi]\xi_i\xi_j \geq \frac{c_0}{2}(|\tau|^2+a^2) \geq c_0|\xi|^2.  \]
  Now we've proved the estimate for $x \in \partial \Omega, y \in Yu(\overline{\Omega}),$ and $z \in Zu(\overline{\Omega})$. Continuity allows us to extend this estimate to a neighbourhood of $\partial \Omega$.  
\end{proof}

The corresponding result for $\Omega^*$ is not as simple as inserting $Yu(x)$ into a defining function for $\Omega^*$. Indeed for solutions of \eqref{eq:g:gje} subject to \eqref{eq:g:2bvp} and $\Omega_\delta^*$ a neighbourhood of $\partial \Omega^*$ we cannot choose, independently of $u$, a corresponding neighbourhood of $\partial \Omega$ on which $Yu(x) \in \Omega_\delta^*$. We deal with this via the following lemma.

\begin{lemma} \label{lem:r1:g-def}\index[notation]{$G(x,u,p)$, \ \  function for oblique boundary condition}\index[notation]{$\phi^*$, \ \ a particular defining function for $\Omega^*$}
  Assume $u:\overline{\Omega} \rightarrow \mathbf{R}$ is a uniformly $g$-convex function which satisfies $Yu(\Omega) = \Omega^*$. Suppose $\Omega^*$ is uniformly $g^*$-convex with respect to $u$. There is a a defining function $\phi^*$ for $\Omega^*$ such that
  \[ G(x,u,p) := \phi^*(Y(x,u,p)),\]
  satisfies $D_{p_kp_l}G(x,u,Du) \geq c_0I$ for $x \in \partial \Omega$. Moreover $G$ can be modified to a new function, which agrees with our original on an (unspecified) neighbourhood of $\partial \Omega$, and satisfies
  \begin{align*}
    G(x,u,Du) &< 0 \text{ in }\Omega\\
    G(x,u,Du) &= 0 \text{ on }\partial \Omega\\
    D_{p_kp_l}G(x,u,Du) &\geq c_0I \text{ on } \Omega.
  \end{align*}
\end{lemma}
\begin{proof}
  As before we fix the defining function $\phi^*(y) = Kd^2(y) - d(y)$ for $d(y) = \text{dist}(y,\partial \Omega^*)$. Then for each $x\in \overline{\Omega}, u =u(x)$ set
  \[\Omega^*_{x,u}:= \{ p ; Y(x,u,p) \in \Omega^*\},   \]
  and note this set is uniformly convex. In particular since $p \mapsto \phi^*(Y(x,u,p))$ is a defining function for this set we obtain, as before,
  \[ D_{p_kp_l}\phi^*(Y(x,u,p))\tau_k\tau_l \geq c_0|\tau|^2,\]
  for  $\tau$ a unit vector tangential to $\partial \Omega^*_{x,u}$ and $p \in \partial \Omega^*_{x,u}$. 
  Then arguing as in Lemma \ref{lem:r1:phi-def} and by a choice of $K$ sufficiently large we obtain whenever $\xi \in \mathbf{S}^{n-1}$ and $Y(x,u,Du) \in \partial \Omega^*$
  \[ D_{p_kp_l}\phi^*(Y(x,u,p))\xi_k\xi_l \geq c_0.\]
We obtain this estimate with  $c_0/2$ instead of $c_0$ provided $\text{dist}(Y(x,u,Du),\partial \Omega^*) < \delta$ for some small $\delta.$ However, as noted before the proof, this does not imply an estimate in some fixed neighbourhood of $\partial \Omega$.

  To rectify this take $x^\delta$ satisfying $\text{dist}(x^\delta,\partial \Omega) = \delta$, similarly $x^{\delta/2}$. Set $\overline{a} = \phi^*(x^{\delta/2})$ and $\underline{a} = \phi^*(x^\delta)$.  Now define
  \[G(x,u,p) := \max{\{\phi^*(Y(x,u,p)),a(|p|^2 - K_1)\}}.\]
 Where $K_1,a$ are chosen so that $\underline{a} \leq a(|Du|^2-K_1) \leq \overline{a}$. We further modify $G$ so it equals $a(|p|^2 - K_1)$ whenever $(x,u,p)$ is such that $\text{dist}(Y(x,u,p),\partial \Omega^*) > \delta$.  Now because $D_{p_kp_l}\tilde{G}(x,u,Du) \geq \min\{a,c_0\}I$ on the interior of each of the piecewise domains we obtain that the mollified function is also uniformly convex in $p$ and is the desired function.
\end{proof}

\clearpage{}
\clearpage{}\chapter{Global regularity II: $C^2$ estimates and degree theory}
\label{chap:r2}
In this chapter we obtain $C^2$ estimates for solutions of GJEs. These include global estimates for solutions of the Dirichlet and second boundary value problem, as well as interior estimates for strictly $g$-convex solutions. We recall the interior estimates and estimates for the Dirichlet problem were required in Chapter \ref{chap:w}. The estimates for the second boundary value problem allow us to complete the proof of Theorem \ref{thm:r1:main} and conclude the global regularity of Aleksandrov solutions.  

\section{Pogorelov Estimates}
\label{sec:pogorelov-estimates}

In this section we use a technique for obtaining $C^2$ estimates which dates back to Pogorelov. Here is the basic idea behind these so called Pogorelov type estimates.  We consider a ``test function'', $v$, which controls $|D^2u|$ and assume it obtains an interior maximum. At an interior maximum $v$ satisfies $Lv \leq 0$ for an appropriate differential operator. Provided we've chosen $v$ appropriately  we will be able to manipulate this inequality into an estimate on the second derivatives at this point. 

\begin{theorem}\label{thm:r2:pog-global}
  Assume that $u \in C^4(\Omega) \cap C^2(\overline{\Omega})$ is an elliptic solution of
  \begin{equation}
    \label{eq:r2:mate}
      \log \det[D^2u-A(\cdot,u,Du)] = B(\cdot,u,Du) \text{ in }\Omega,
  \end{equation}
  where the $A,B$ are $C^2$, $A$ satisfies A3w, and $B>0$. Suppose there exists a function $\phi$ satisfying \index[notation]{$L$, \ \ differential operator $Lv:=  w^{ij}[D_{ij}v-D_{p_k}A_{ij}D_kv]-D_{p_k}BD_kv $}
  \begin{equation}
    \label{eq:r2:barrier-req}
       L(\phi) := w^{ij}[D_{ij}\phi-D_{p_k}A_{ij}D_k\phi]-D_{p_k}BD_k\phi \geq w^{ii}-C,
  \end{equation}
  where $w = D^2u - A(\cdot,u,Du)$ and $A,B$ terms evaluated at $(\cdot,u,Du)$. Then $u$ satisfies an estimate
  \begin{equation}
    \label{eq:r2:global}
       \sup_{\Omega}|D^2u| \leq C\big(1+\sup_{\partial \Omega}|D^2u|\big),
  \end{equation}
  where $C$ depends on $A,B,\Vert u \Vert_{C^1(\Omega)},\Vert \phi \Vert_{C^0(\Omega)}$.
\end{theorem}
\begin{proof}
  For $x \in \Omega$ and $\xi \in \mathbf{S}^{n-1}$ define
 \[ v(x) = \kappa \phi(x) + \tau |Du(x)|^2/2 + \log(w_{\xi\xi}(x)).\]
 Here $w_{\xi\xi} = w_{ij}\xi_i\xi_j$. 
  Since $|D^2u| \leq C\sup_{\xi}e^v$ if  $v$ has its maximum on $\partial \Omega$ then the estimate \eqref{eq:r2:global} is immediate. Otherwise we suppose the maximum of $v$ occurs at $x_0 \in \Omega$ and in  a direction $\xi$ assumed, without loss of generality, to be $e_1$. At an interior maximum $Dv=0$ and $D^2v \leq 0$ so that
  \begin{equation}
    \label{eq:p:globmain}
       0 \geq L v  = \kappa L\phi + \tau L(|Du|^2/2) + L(\log(w_{11})).
  \end{equation}
  We assume after a rotation that $w$ is diagonal and proceed to compute  each term in \eqref{eq:p:globmain}.
  
  \textit{Term 1: $L\phi$.} This one's immediate via our assumption on $\phi$
  \begin{equation}
    \label{eq:p:globest1}
     L \phi \geq w^{ii}-C.
  \end{equation}

  \noindent \textit{Term 2: $L(|Du|^2)$}
  We compute
  \begin{align*}
    D_i(|Du|^2/2) &= u_{k}u_{ki}\\
    D_{ii}(|Du|^2/2) &= u_{ki}u_{ki}+u_ku_{kii}.
  \end{align*}
  Then
  \begin{align}
\nonumber    L(|Du|^2/2) &= w^{ii}(u_{ki}u_{ki}-u_ku_{kii} - D_{p_l}A_{ii}u_{k}u_{lk}) - B_{p_l}u_{k}u_{kl}\\
 \label{eq:p:2terms}   &= w^{ii}u_{ki}u_{ki} + u_k[w^{ii}(u_{kii}- D_{p_l}A_{ii}u_{lk})- B_{p_l}u_{kl}].
  \end{align}
  We note
  \begin{equation}
   w^{ii}u_{ki}u_{ki} = w^{ii}(w_{ki}+A_{ki})(w_{ki}+A_{ki}) \geq w_{ii}-C(1+w^{ii}),\label{eq:r2:wuu}
 \end{equation}
  and by differentiating \eqref{eq:r2:mate} in the direction $e_k$
  \begin{align}
 \label{eq:r2:diff-once}   w^{ij}[u_{ijk}-A_{ij,p_l}u_{lk}] - B_{p_l}u_{lk} = w^{ij}(A_{ij,k}+A_{ij,u}u_k)+B_k+B_uu_k.
  \end{align}
  Hence \eqref{eq:p:2terms} becomes
  \begin{equation}
  L(|Du|^2/2) \geq w_{ii}-C(1+w^{ii}).\label{eq:p:globest2}
\end{equation}

\noindent \textit{Term 3: $L(\log(w_{11}))$.} This term is by far the most work. To begin, we differentiate \eqref{eq:r2:mate} twice in the $e_1$ direction and obtain
\begin{align}
  \nonumber w^{ii}&[u_{ii11}-D_{p_k}A_{ii}u_{k11}]- B_{p_k}u_{k11} = w^{ii}w^{jj}w_{ij,1}^2 + w^{ii}\big[A_{ii,11} - 2A_{ii,1u}u_1\\
        \nonumber          &\quad+2A_{ii,1p_k}u_{k1}+A_{ii,uu}u_1^2+A_{ii,u}u_{11}+2A_{ii,p_k}u_1u_{1k}+A_{ii,p_kp_l}u_{1k}u_{1l} \big]\\
  \nonumber           &\quad\quad+ B_{11} +2B_{1u}u_1+2B_{1p_k}u_{1k}+B_{uu}u_1^2 + B_uu_{11}\\
\nonumber  &\quad\quad\quad+ 2B_{up_k}u_1u_{1k} + B_{p_kp_l}u_{1k}u_{1l}\\
        &\geq w^{ii}w^{jj}w_{ij,1}^2+w^{ii}A_{ii,p_kp_l}u_{1k}u_{1l} - C(1+w_{ii}+w^{ii}+w_{ii}w^{ii}+w_{ii}^2). \label{eq:p:1}
\end{align}
We use A3w to deal with the second term. Write
\[ w^{ii}A_{ii,p_kp_l}u_{1k}u_{1l} \geq w^{ii}A_{ii,p_1p_1}w_{11}^2 - C(w^{ii}+w^{ii}w_{ii}).\]
Then by applying A3w with $\xi=e_i,\ \eta=e_j$ for $i \neq j$ we see $A_{ii,p_jp_j} \geq 0$ so that
\begin{align*}
  \label{eq:12}
  w^{ii}A_{ii,p_1p_1}w_{11}^2 &= w^{11}A_{11,p_1p_1}w_{11}^2 + \sum_{i=2}^nw^{ii}A_{ii,p_1p_1}w_{11}^2  \\
  &\geq -Cw_{11}.
\end{align*}
Thus \eqref{eq:p:1} becomes
\[ L(u_{11}) \geq w^{ii}w^{jj}w_{ij,1}^2- C(1+w_{ii}+w^{ii}+w_{ii}w^{ii}+w_{ii}^2). \]

We perform similar calculations for $LA_{11}$. Because we'll use such calculations repeatedly in this chapter we consider $L(F(\cdot,u,Du))$ for an arbitrary $C^2$ function $F$. Indeed direct calculation using \eqref{eq:r2:wuu} and  \eqref{eq:r2:diff-once} yields
\begin{equation}
  \label{eq:r2:f-calcs}
   L(F(\cdot,u,Du)) \geq -C(1+w^{ii}+w_{ii}).
 \end{equation}
 Thus $L(w_{11}) = L(u_{11})-L(A_{11})$ satisfies
\begin{equation}
  \label{eq:p:lw11}
   L(w_{11}) \geq w^{ii}w^{jj}w_{ij,1}^2- C(1+w_{ii}+w^{ii}+w_{ii}w^{ii}+w_{ii}^2).
\end{equation}
Now, proceeding to $L(\log w_{11})$, first compute
\begin{align*}
  D_i\log(w_{11}) = \frac{w_{11,i}}{w_{11}}\\
  D_{ii}\log(w_{11}) = \frac{w_{11,ii}}{w_{11}} - \frac{w_{11,i}^2}{w_{11}^2},
\end{align*}
so that
\[ L(\log w_{11})  = -\frac{w^{ii}w_{11,i}^2}{w_{11}^2} + \frac{L(w_{11})}{w_{11}}.\]
Hence by \eqref{eq:p:lw11}
\begin{equation}
L(\log w_{11}) \geq \frac{w^{ii}w^{jj}w_{ij,1}^2}{w_{11}}-\frac{w^{ii}w_{11,i}^2}{w_{11}^2} - \frac{C}{w_{11}}(1+w_{ii}+w^{ii}+w_{ii}w^{ii}+w_{ii}^2).\label{eq:p:toreturn}
\end{equation}
Note when $i,j = 1$ in the first term and $i=1$ in the second term these terms cancel. Moreover at the expense of an inequality we can discard the terms with neither $i$ nor $j = 1$. Subsequently we estimate the first two terms as follows
\begin{align}
  \label{eq:r2:pog1}   \frac{w^{ii}w^{jj}w_{ij,1}^2}{w_{11}}&-\frac{w^{ii}w_{11,i}^2}{w_{11}^2}\\
\nonumber  &\geq \sum_{i>1}\frac{w^{ii}w_{i1,1}^2}{w_{11}^2} + \sum_{j>1}\frac{w^{jj}w_{1j,1}^2}{w_{11}^2} - \sum_{i>1}\frac{w^{ii}w_{11,i}^2}{w_{11}^2}\\
\nonumber  &=\frac{1}{w_{11}^2}\sum_{i>1}w^{ii}\big[2w_{i1,1}^2-w_{11,i}^2\big]\\
  \nonumber  &= \frac{1}{w_{11}^2}\sum_{i>1}w^{ii}w_{11,i}^2 + \frac{2}{w_{11}^2}\sum_{i>1}w^{ii}[w_{i1,1}^2 - w_{11,i}^2]\\
  \nonumber         &= \frac{1}{w_{11}^2} \sum_{i>1}w^{ii}w_{11,i}^2+ \frac{2}{w_{11}^2}\sum_{i>1}w^{ii}(w_{i1,1}+w_{11,i})(w_{i1,1}-w_{11,i}).
\end{align}
Rewriting the second sum in terms of the $A$ matrix yields
\begin{align}
 \nonumber &\frac{w^{ii}w^{jj}w_{ij,1}^2}{w_{11}}-\frac{w^{ii}w_{11,i}^2}{w_{11}^2}= \frac{1}{w_{11}^2} \sum_{i>1}w^{ii}w_{11,i}^2\\
 \nonumber &\quad\quad+ \frac{2}{w_{11}^2}\sum_{i>1}w^{ii}(D_iA_{11}-D_{1}A_{i1})(2w_{11,i}+D_iA_{11}-D_1A_{ii})\\
  \label{eq:r2:pog2}       &= \frac{1}{w_{11}^2} \sum_{i>1}w^{ii}\big[w_{11,i}^2 + 4w_{11,i}(D_iA_{11}-D_{1}A_{i1}) + 4(D_iA_{11}-D_{1}A_{ii})^2\big]\\ \nonumber &\quad \quad -\frac{2}{w_{11}^2}\sum_{i>1}w^{ii}(D_iA_{11}-D_{1}A_{ii})^2\\
\label{eq:r2:pog3}  &\geq -C w^{ii}.
\end{align}
Returning to \eqref{eq:p:toreturn} and using that we may freely assume $1/w_{11} \leq 1$ we have
  \begin{equation}
  L(\log(w_{11})) \geq -C(1+w_{ii}+w^{ii}).\label{eq:p:globest3}
\end{equation}
Now substituting \eqref{eq:p:globest1}, \eqref{eq:p:globest2} and \eqref{eq:p:globest3} into \eqref{eq:p:globmain} we have
\[ 0 \geq w_{ii} (\tau -C)+ w^{ii}(\kappa -\tau C - C) - C(\kappa + \tau +1).\]
Choosing $\tau \geq C+1$ then $\kappa \geq \tau C + C + 1$ we have the estimate
\[ w_{ii} + w^{ii} \leq C.\]
This gives an estimate for $v$ and subsequently on the second derivatives. 
\end{proof}

The interior estimate for solutions with $g$-affine boundary values is based on a similar calculation. The key difference is the inclusion of a term $g(\cdot,y_0,z_0)-u$ which acts as a cut-off function and forces an interior maximum. 

\begin{theorem} \label{thm:r2:pog-local}
    Assume that $u \in C^4(\Omega) \cap C^2(\overline{\Omega})$ is an elliptic solution of
    \begin{align}
      \det[D^2u-A(\cdot,u,Du)] = B(\cdot,u,Du) \text{ in }\Omega, \nonumber\\
      u = g(\cdot,y,z) \text{ on }\partial \Omega,
    \end{align}
where $A,B$ are $C^2$ with $A$ satisfying A3w and $B>0$. 
  Then there exists $\beta,d,C > 0$ such that provided $\text{diam}(\Omega) < d$ we have the estimate 
  \begin{equation}
    \label{eq:r2:pogloc}
       \sup_{\Omega}(g(\cdot,y,z)-u)^{\beta}|D^2u| \leq C,
  \end{equation}
  where $C$ depends on $\Omega,A,B,\Vert u \Vert_{C^1(\Omega)},g$.
\end{theorem}
\begin{proof}
  The proof is not so different from Theorem \ref{thm:r2:pog-global}. We set $\phi = |x-x_1|^2$ for some $x_1$ in $\Omega$ and note provided the domain is small enough $\phi$ satisfies the barrier requirement \eqref{eq:r2:barrier-req} from Theorem \ref{thm:r2:pog-global}. Moreover by a further choice of $\text{diam} (\Omega)$ small we can ensure $|D\phi|$ is as small as desired. This will be used later.

  This time consider the function
\[ v = \kappa\phi+\tau|Du|^2/2 + \log(w_{\xi\xi}) + \beta \log[g(\cdot,y,z)-u].\]
We use the notation $u_0 = g(\cdot,y,z)$ and $\eta = u_0-u$. Because the nonnegative function $e^v$ is $0$ on $\partial \Omega$,  $v$ attains an interior maximum at $x_0 \in \Omega$ and $\xi$ assumed to be $e_1$. We assume at this point $w$ is diagonal and again note $Lv(x_0) \leq 0$. 

We modify our computations for $L(\log(w_{11}))$. We return to \eqref{eq:r2:pog2} and note Cauchy's inequality implies
\begin{align*}
  4w_{11,i}(D_iA_{11}-D_{1}A_{ii}) \geq -\frac{w_{11,i}^2}{2} - 8 (D_iA_{11}-D_{1}A_{ii})^2.
\end{align*}
Thus in place of \eqref{eq:r2:pog3} we obtain the inequality
\[  \frac{w^{ii}w^{jj}w_{ij,1}^2}{w_{11}}-\frac{w^{ii}w_{11,i}^2}{w_{11}^2} \geq \frac{1}{2w_{11}^2}\sum_{i=2}^nw^{ii}w_{11,i}^2 - Cw^{ii},\]
and subsequently in place of \eqref{eq:p:globest3} we obtain
\[L(\log(w_{11})) \geq \frac{1}{2w_{11}^2}\sum_{i=2}^nw^{ii}w_{11,i}^2 - C(1+w_{ii}+w^{ii}).   \]
Now, the inequality $0 \geq Lv(x_0)$ implies
\begin{align}
  \label{eq:r2:pogloc1} 0 &\geq \kappa(w^{ii}-C) + \tau [w_{ii}-C(1+w^{ii})] +  \frac{1}{2w_{11}^2}\sum_{i=2}^nw^{ii}w_{11,i}^2 \\
\nonumber  & \quad\quad - C(1+w_{ii}+w^{ii})+ \beta L(\log \eta).
\end{align}
So all that remains is to compute $L(\log \eta)$.

To begin write
\begin{equation}
    \label{eq:r2:pogloc2} L(\log \eta) = \frac{L\eta}{\eta} - \sum_{i=1}^nw^{ii}\left(\frac{D_i\eta}{\eta}\right)^2. 
\end{equation}
We compute
\begin{align}
\nonumber  L\eta = w^{ii}&[D_{ii}u_0 - D_{ii}u - D_{p_k}A_{ii}(\cdot,u,Du)D_k\eta] - D_{p_k}BD_k\eta\\
 \nonumber    \geq w^{ii}&[-w_{ii}+A_{ii}(\cdot,u_0,Du_0)- A_{ii}(\cdot,u,Du)- D_{p_k}A_{ii}(\cdot,u,Du)D_k\eta] - C\\
  \nonumber   \geq w^{ii}&[A_{ii,u}\eta+A_{ii}(\cdot,u,Du_0)- A_{ii}(\cdot,u,Du)- D_{p_k}A_{ii}(\cdot,u,Du)D_k\eta] - C\\
  &\geq w^{ii}D_{p_kp_l}A_{ii}D_k\eta D_l \eta - C - Cw^{ii}\eta. \label{eq:r2:pogloc3}
\end{align}
For each $i$ write
\begin{align*}
  w^{ii}D_{p_kp_l}A_{ii}D_k\eta D_l \eta &= \sum_{k,l \neq i}D_{p_kp_l}A_{ii}D_k\eta D_l \eta + 2\sum_{l \neq i}D_{p_ip_l}A_{ii}D_i\eta D_l \eta \\
  &\quad\quad+ D_{p_ip_i}A_{ii}D_i\eta D_i \eta
\end{align*}
Then by A3w the first term is nonnegative, so that
\begin{align*}
w^{ii}D_{p_kp_l}A_{ii}D_k\eta D_l \eta  &\geq -C D_i\eta - C (D_i\eta)^2.
\end{align*}
Returning to \eqref{eq:r2:pogloc3} we see
\[  L\eta \geq -C(1+w^{ii}\eta) - Cw^{ii}D_i\eta - Cw^{ii}(D_i\eta)^2.\]
Which into \eqref{eq:r2:pogloc2} implies
\begin{equation}
  \label{eq:r2:pogloc4}
   L(\log \eta) \geq -\frac{C}{\eta} - Cw^{ii} - C \sum_{i=1}^nw^{ii}\left(\frac{D_i\eta}{\eta}\right)^2.
\end{equation}
Here we've used that we can assume $\eta < 1$, and also used Cauchy's to note
\[ w^{ii}\frac{D_i\eta}{\eta} = \sqrt{w^{ii}}\sqrt{w^{ii}}\frac{D_i\eta}{\eta} \leq w^{ii}+w^{ii}\left(\frac{D_i\eta}{\eta}\right)^2. \]
Now we deal with the final term in \eqref{eq:r2:pogloc4}. We can assume that term $w^{11}(D_1\eta/\eta)^2 \leq 1$, for if not we have \eqref{eq:r2:pogloc} with $\beta = 2$. Since we are at a maximum $D_iv = 0$, that is
\[ \frac{D_i\eta}{\eta} = -\frac{1}{\beta}\left[\frac{D_iw_{11}}{w_{11}}+\kappa D_i\phi + \tau D_k u w_{ik} + \tau D_k u A_{ik}\right].\]
This implies
\[ \sum_{i=1}^nw^{ii}\left(\frac{D_i\eta}{\eta}\right)^2 \leq \frac{C}{w_{11}^2\beta^2}\sum_{i=2}^nw^{ii}w_{11,i}^2 + C\frac{\kappa^2}{\beta^2}w^{ii}|D_i\phi|^2 + \frac{C\tau^2}{\beta^2}[w_{ii}+w^{ii}],\]
where, as stated $w^{11}(D_1\eta/\eta)^2 \leq 1$, so is included in the constant term. 
Choosing $\beta \geq 1,2C$ and returning to \eqref{eq:r2:pogloc4} we obtain
\[ \beta L(\log \eta) \geq \frac{-C\beta}{\eta}-\kappa^2|D\phi|^2 w^{ii} - \frac{\tau^2}{\beta^2}[w^{ii}+w_{ii}] - C\beta w^{ii} - \frac{1}{2w_{11}^2}\sum_{i=2}^nw^{ii}w_{11,i}^2.\]

Substituing into \eqref{eq:r2:pogloc1} completes the proof: We have
\begin{align*}
  0 &\geq \kappa (w^{ii}-C)+\tau[w_{ii}-C(1+w^{ii})] - C(1+w_{ii}+w^{ii}) -\frac{C\beta}{\eta} \\
    &\quad\quad-C \kappa^2|D\phi|^2 w^{ii} - C\frac{\tau^2}{\beta}[w^{ii}+w_{ii}] - C\beta w^{ii}\\
    & = w^{ii}[\kappa - \tau C -C-C\kappa^2|D\phi|^2 -\frac{C\tau^2}{\beta} - C\beta]\\
  &\quad\quad+w_{ii}[\tau - C - \frac{C\tau^2}{\beta}] - C[\kappa+\tau+\frac{\beta}{\eta}].
\end{align*}
Take $\text{diam}(\Omega)$, and subsequently $|D\phi|$, small enough to ensure $\kappa^2|D\phi|^2 \leq 1$ (our choice of $\kappa$ will only depend on allowed quantities). A further choice of $\beta \geq \tau^2$, $\tau$ large depending only on $C$, and finally $\kappa$ large depending on $\tau,C$ implies
\[ 0 \geq w^{ii}+w_{ii}-C(1+\frac{1}{\eta}).\]
This implies $\eta w_{ii} \leq C$ at the maximum point, and the proof is complete. 
\end{proof}

\section{Boundary Estimates}
\label{sec:boundary-estimates}
Theorem \ref{thm:r2:pog-global} reduces global $C^2$ estimates for Monge--Amp\`ere equations to boundary $C^2$ estimates. That is, provided $\sup_{\partial \Omega}|D^2u| \leq C$ then in combination with Theorem \ref{thm:r2:pog-global} we obtain $\sup_{\Omega}|D^2u| \leq C$. We first prove the boundary estimate for the Dirichlet problem, then for the second boundary value problem.

\subsection*{Dirichlet Problem}

The $C^2$ estimates for the Dirichlet problem use techniques from classical elliptic PDE: careful choices of test functions and  barrier arguments. The following theorem is due to Jiang, Trudinger, and Yang \cite{JTY14} who extended the ideas from the Monge--Amp\`ere case \cite{Krylov83,Ivochkina80,Ivochkina83,CNS84,Trudinger95} to general Monge--Amp\`ere type equations. 
\begin{theorem} \label{thm:r2:dir-est}
  Let $u \in C^4(\Omega) \cap C^2(\overline{\Omega})$ be an elliptic solution of
  \begin{align}
   \label{eq:r2:mate-dirichlet} \det [D^2u -A(\cdot,u,Du)] &= B(\cdot,u,Du) \text{ in }\Omega\\
  \label{eq:r2:dirichlet}  u &= \phi \text{ on }\partial \Omega. 
  \end{align}
  where $A,B$ are $C^2$, $A$ satisfies A3w and $B>0$.  Assume $\phi \in C^4(\overline{\Omega})$, $\partial \Omega \in C^4$, and there exists a barrier $\underline{u} \in C^2(\overline{\Omega})$ satisfying $\underline{u} = \phi$ on $\partial \Omega$ along with
  \begin{align}
\label{eq:r2:underu2}   D^2\underline{u} - A(\cdot,u,D\underline{u}) &\geq \delta I,\\
  \label{eq:r2:underu1}  \text{ and }\quad\det[D^2\underline{u} - A(\cdot,u,D\underline{u})] &\geq B(\cdot,u,D\underline{u}),
  \end{align}
    for some $\delta > 0$. Then there is $C$ depending only on $\Omega,\underline{u},A,B,\Vert u \Vert_{C^1(\Omega)}$ such that
\[   \sup_{\partial \Omega}|D^2u| \leq C. \] 
\end{theorem}
\begin{proof}
  First note if $u$ solves \eqref{eq:r2:mate-dirichlet} subject to \eqref{eq:r2:dirichlet} and $\Upsilon$ is a diffeomorphism, then $\tilde{u}:= u \circ \Upsilon$ solves a problem of the same form.  Indeed $\tilde{u}$ satisfies a Dirichlet boundary condition $\tilde{u} = \phi$ on the boundary of $\Omega^\Upsilon := \{x ; \Upsilon(x) \in \Omega\}$.  Furthermore by direct calculation $\tilde{u}$ solves
\[ \det[D^2\tilde{u} - \tilde{A}(\cdot,\tilde{u},D\tilde{u})] = \tilde{B}(\cdot,\tilde{u},D\tilde{u}),\]
  where
  \begin{align*}
    \tilde{A}_{ij}(\cdot,\tilde{u},D\tilde{u}) &= D_i\Upsilon^kD_j\Upsilon^lA_{kl}(\cdot,\tilde{u},[D\Upsilon]^{-1}D\tilde{u})+[D\Upsilon]^{\alpha k}D_\alpha \tilde{u}D_{ij}\Upsilon^k,\\
    \tilde{B}(\cdot,\tilde{u},D\tilde{u}) &= [\det D\Upsilon]^2 B(\cdot,\tilde{u},[D\Upsilon]^{-1}D\tilde{u}).
  \end{align*}
  Because $\tilde{A}(\cdot,\tilde{u},p) = A(\cdot,u,l^{1}_x(p))+l^{(2)}_x(p)$ for functions $p \mapsto l_x^{(1)}(p),l_x^{(2)}(p)$ which are linear in $p$ for each $x$,  $\tilde{A}$ satisfies A3w provided $A$ does. 

  Thus, after diffeomorphism, we fix $x_0 \in \partial \Omega$ assumed to be 0, and assume
  \[ T:= \partial \Omega\cap B_\epsilon(0) \subset \{x; x_n = 0\},\]
  with $e_n$ the inner unit normal at $0$. The plan is to estimate each of the second derivatives (repeated tangential, double normal, and mixed tangential normal) at $0$.

 \textit{Step 1. (Repeated tangential)} These are trivial: in a neighbourhood of $0$, specifically on the boundary portion $T$, we have $u = \underline{u}$. Thus for $\alpha,\beta = 1,\dots,n-1$ we have $|D_{\alpha\beta}u| = |D_{\alpha\beta}\underline{u}| \leq C$.

 \textit{Step 2. (Mixed tangential normal)} The mixed tangential normal estimates are via a barrier argument.  We differentiate the equation once in a tangential direction $e_\alpha$ for $\alpha = 1,\dots,n-1$. As in \eqref{eq:r2:diff-once}
  \begin{equation}
    \label{eq:r2:diff-once-redux}
       |L(D_\alpha u)| \leq C(1+w^{ii}). 
  \end{equation}
  From this
  \begin{align}
    \label{eq:r2:dir-bar} |L[D_\alpha(u-\underline{u})]| \leq C(1+w^{ii})
  \end{align}
  is immediate. Our barrier arguments takes place on $\Omega_\delta:= \Omega \cap B_\delta(0)$. We note $\partial \Omega_\delta = (\partial \Omega \cap \overline{B_\delta}) \cup (\Omega \cap \partial B_\delta)$. We need estimates for $D_\alpha(u - \underline{u})$ on each portion of the boundary. Provided $\delta$ chosen so small as to ensure $\partial \Omega \cap B_\delta \subset T$ we obtain $|D_\alpha(u - \underline{u})| = 0$ on this portion of the boundary. Moreover we trivially obtain $|D_\alpha(u-\underline{u})| \leq C|x|^2$ on $\Omega \cap \partial B_\delta$ for $C$ depending on $\delta, \sup |Du|, \sup |D\underline{u}|$.

  Our barrier is built from $\underline{u}$ over a couple of steps.
  Set $v = 1-e^{k(\underline{u}-u)}$ for a choice of $k$ large. By \eqref{eq:r2:underu1} and linearising, the maximum principle implies $\underline{u} \leq u$ and thus $v \geq 0$. Repeating the calculations in Theorem \ref{thm:r1:lin} we see $v$ satisfies
  \begin{align*}
    v &= 0 \text{ on }\partial \Omega\\
    Lv &\leq -\epsilon w^{ii} + C \text{ in }\Omega
  \end{align*}
  Note whilst Theorem \ref{thm:r1:lin} used $\overline{u}$ satisfying $\overline{u} \geq u$,  this was only to obtain \eqref{eq:r2:underu2}.
  Next set $\psi := v + \mu x_n - K x_n^2$ and compute
  \begin{align*}
    Lx_n &= -w^{ij}A_{ij,p_n}-B_{p_n}\\
    Lx_n^2 &= 2w^{nn} - 2w^{ij}A_{ij,p_n}x_n- 2B_{p_n}x_n
  \end{align*}
  Thus for a choice of $K$ large and $\mu,\delta$ small we obtain
  \begin{align}
   \label{eq:r2:a-inf} L\psi &\leq \frac{-\epsilon}{4}(1+w^{ii}) \text{ in }\Omega_\delta\\
    \psi &\geq 0 \text{ on }\partial \Omega_\delta. \nonumber
  \end{align}
  To provide more detail on the choice of $K,\mu,\delta$ we consider two possibilities for $Lv$, which we know satisfies $Lv \leq -\epsilon w^{ii}+C$. The first case is $-\epsilon w^{ii}+C > -\frac{\epsilon}{2}(w^{ii}+1)$. This implies an upper bound $w^{ii} < C$ depending on $\epsilon$, and then, by the PDE, a lower bound $w^{nn} > C$. We obtain \eqref{eq:r2:a-inf} by a choice of $K$ large depending on $C$, then $\mu$ small depending on $\epsilon$ and $\delta$ small depending on $1/K,\epsilon$. The second case is $-\epsilon w^{ii}+C \leq -\frac{\epsilon}{2}(w^{ii}+1)$. In this case our choice of $\mu,\delta$ small as before ensures \eqref{eq:r2:a-inf}, and we don't need to enforce anything further on $K$. 
  
 Now modify $\psi$ to $\tilde{\psi}:= a\psi + b|x|^2$.
  We have $D_\alpha(u - \underline{u}) = 0 \leq \tilde{\psi}$ on $\partial \Omega \cap \overline{B_\delta}$ and via a choice of $b$ large 
  \[|D_\alpha(u - \underline{u})| \leq b \delta^2 \leq \tilde{\psi} \text{ on }\Omega \cap \partial B_\delta .\]
  Now choosing $a$ large and using \eqref{eq:r2:a-inf} we have
  \[ L\tilde{\psi} \leq -\left(\frac{a\epsilon}{4} - Cb\right)(1+w^{ii}).\]

  All up, with $a > > b$ sufficiently large, we have
  \begin{align*}
    |L[D_\alpha(u-\underline{u})]| - L\tilde{\psi} \geq 0 \text{ in }\Omega_\delta,\\
    |D_\alpha(u-\underline{u})| - \tilde{\psi} \leq 0 \text{ on }\partial \Omega_\delta.
  \end{align*}
  The maximum principle implies $ |D_\alpha(u-\underline{u})| - \tilde{\psi} \leq 0$ in $\Omega_\delta$. Recalling $|D_\alpha(u-\underline{u})(0)|, \tilde{\psi}(0) = 0,$ we see for $t$ small
  \[ |D_\alpha(u-\underline{u})(te_n)-D_\alpha(u-\underline{u})(0)| \leq \tilde{\psi}(te_n) - \tilde{\psi}(0).\]
  Dividing by $t$ and sending $t \rightarrow 0$ we obtain
  \[ |D_{\alpha n}(u-\underline{u})(0)| \leq D_n \tilde{\psi} \leq C. \]
  This completes the mixed tangential normal estimates.

 \textit{Step 3. (Double normal)} 
  We conclude with the double normal estimates. We define, on $T$ and for $\xi \in \mathbf{S}^{n-2}$, the function
  \begin{align}
   \nonumber  W[u]&:= [D_{\alpha\beta}u - A_{\alpha\beta}(x,u,Du)]\xi_\alpha\xi_\beta,\\
  \label{eq:r2:w-extend}  &= [D_{\alpha\beta}\phi - A_{\alpha\beta}(x,\phi,D'\phi,D_nu)]\xi_\alpha\xi_\beta.
  \end{align}
  Here  $\alpha,\beta = 1,\dots,n-1$ and the equality is because $u = \phi$ on $T$ so tangential derivatives agree. Our goal is to obtain an estimate from above for $u_{nn}$. For this it suffices to obtain an estimate $W>c$ for some positive $c$, and use the PDE \eqref{eq:r2:mate-dirichlet} (see \cite[pg. 49]{Figalli17} for details).

  For the lower bound we note that for $K$ sufficiently large $\tilde{w}:= W+K|x|^2$, defined on $T$, attains an interior minimum on $T$ at some $\overline{x},\overline{\xi}$. Now we extend $\phi$ to $\Omega$ in a neighbourhood of $T$ as $\phi(x',x_n) = \phi(x')$ for $x' \in T$. Subsequently we also extend $\tilde{w}$ using \eqref{eq:r2:w-extend}.

  Thanks to A3w, in particular the codimension one convexity, $\tilde{w}$ is concave in $D_{n}u$.
  This let's us compute
  \[ L\tilde{w} \leq C(1+w^{ii}).\]
  This computation is similar to \eqref{eq:p:1} but simpler. We remove third derivatives using \eqref{eq:r2:diff-once-redux} with $i=n$. The only term left to deal with is  $-w^{ij}A_{\alpha\beta,p_np_n}D_{ni}uD_{nj}u\xi_\alpha\xi_\beta$, which via A3w is less than or equal to 0.

  Now we repeat the barrier argument from step 2. Recalling $\psi$ equals $0$ on $\partial \Omega \cap \overline{B_\delta}$ and $\psi \geq 0$ on $\Omega \cap \partial B_\delta$ we see for $a>0$, $a\psi+\tilde{w}|_{\partial \Omega_\delta}$ still has its minimum at  $\overline{x}$. Moreover, via a choice of $a$ large, we have $L(a\psi+\tilde{w}) \leq 0$. Via the maximum (minimum, technically) principle
  \[a\psi+\tilde{w} \geq \inf_{\partial \Omega_\delta} a\psi+\tilde{w} = a\psi(\overline{x})+\tilde{w}(\overline{x}).  \]
  This implies the inner normal at this point is positive. Thus
  \[ D_nW(\overline{x}) \geq -C,\]
  and subsequently
  \begin{equation}
    \label{eq:r2:dir-fin-3}
       A_{\alpha\beta,p_n}\xi_\alpha\xi_\beta D_{nn}u(\overline{x}) \leq C.
  \end{equation}
  Provided we obtain a positive estimate from below for $A_{\alpha\beta,p_n}\xi_\alpha\xi_\beta$, the subsequent upper bound on $ D_{nn}u$ at $\overline{x}$ along with the PDE \eqref{eq:r2:mate-dirichlet} implies the desired  lower bound on $W$.

  To this end using \eqref{eq:r2:underu2} we have on $\partial \Omega \cap B_\delta$,
  \begin{equation}
    \label{eq:r2:dir-fin-1}
      W[\underline{u}]:= [D_{\alpha\beta}\phi - A_{\alpha\beta}(x,\phi,D'\phi,D_n\underline{u})]\xi_\alpha\xi_\beta \geq \delta. 
  \end{equation}
  We assume $W[u](\overline{x}) < \delta/2$, otherwise this is our estimate on the minimum. Then, at $\overline{x}$,
  \begin{align*}
    -\delta/2 &\geq W[u](\overline{x}) - W[\underline{u}](\overline{x}) \\
    &= [ A_{\alpha\beta}(x,\phi,D'\phi,D_n\underline{u})- A_{\alpha\beta}(x,\phi,D'\phi,D_nu)]\xi_\alpha\xi_\beta.
  \end{align*}
  Using the convexity of $A_{\alpha\beta}\xi_\alpha\xi_\beta$ in $D_nu$, more precisely the inequality $h'(0) \leq h(1)-h(0)$ for the convex function $h(t) = A_{\alpha\beta}(x,\phi,D'\phi,t D_n\underline{u} + (1-t)D_nu)$, we have
  \begin{equation}
    \label{eq:r2:dir-fin-2}
        -\delta/2 \geq -A_{\alpha\beta,p_n}\xi_\alpha \xi_\beta D_n(u-\underline{u}).
  \end{equation}
  We trivially have $D_n(u-\underline{u}) \leq \kappa$ for a $\kappa$ depending on $\underline{u}$, $\Vert u \Vert_{C^1}$. Thus \eqref{eq:r2:dir-fin-2} implies
  \[ A_{\alpha\beta,p_n}\xi_\alpha \xi_\beta \geq \frac{\delta}{2\kappa}.\]
  This lower bound yields an upper bound for $u_{nn}$ by \eqref{eq:r2:dir-fin-3}, and subsequently, the lower bound for $W$ on $T$. The upper bound for $u_{nn}$ follows. A lower bound for $u_{nn}$ is immediate from the ellipticity. This concludes the $C^2$ estimate for the Dirichlet problem.  
\end{proof}

This completes the estimates required for Theorem \ref{thm:w:regularity} (which we'd deferred to here).

\subsection*{Second boundary value problem}
We're up to the last of our estimates --- $C^2$  boundary estimates for the second boundary value problem. The ideas in this subsection originated with Urbas \cite{Urbas1997} and were extended to optimal transport by Trudinger and Wang \cite{TrudingerWang09} and to GJEs by Jiang and Trudinger \cite{JiangTrudinger14}.

We first obtain a strict obliqueness estimate. To make sense of this we note, using the function $G(x,u,p)$ from Lemma \ref{lem:r1:g-def}, that $u$ satisfies the boundary condition
\begin{equation}
  \label{eq:r2:bc}
  G(\cdot,u,Du) = 0 \text{ on }\partial \Omega.
\end{equation}
This boundary condition is called \textit{strictly oblique} provided there is $c>0$ such that
\begin{equation}
  \label{eq:r2:strict-oblique}
 G_p(x,u,Du)\cdot \gamma \geq c
\end{equation}
for $x \in \partial \Omega$ and $\gamma$ the outer unit normal to $\partial \Omega$. Once we obtain a strict obliqueness estimate which is independent of $u$ we can estimate second derivatives by decomposing directional derivatives into component in the $G_p$ direction and tangential direction. Versions of the following result in appear \cite{Urbas1997} \cite{TrudingerWang09}, though we follow the details provided by Liu and Trudinger \cite{LiuTrudinger16}.

\begin{theorem}
  Assume $u \in C^2(\overline{\Omega}) \cap C^3(\Omega)$ satifies \eqref{eq:g:mate} subject to \eqref{eq:r2:bc}. Assume the $C^2$ domains $\Omega,\Omega^*$ are, respectively, uniformly $g/g^*$-convex with respect to $u$ and $A,B$ are $C^2$ with $B>0$. Then the boundary condition  \eqref{eq:r2:bc} is strictly oblique. That is, \eqref{eq:r2:strict-oblique} holds for $c$ depending only on $\Omega,\Omega^*,A,B$ and $\Vert u \Vert_{C^1(\overline{\Omega})}$. 
\end{theorem}
\begin{proof}
\textit{Step 1. Urbas type formula:} Since, on $\partial \Omega$, $G_{p_i} = \phi^*_kY^k_{p_i}$ the quantity we need to estimate is
\begin{equation}
  \label{eq:r2:obl2}
  G_p\cdot \gamma = \phi^*_kY^k_{p_i}\gamma_i.
\end{equation}
Because $\phi^*(Yu) = 0$ on $\partial \Omega$ its derivatives in directions tangent to $\partial \Omega$ are 0, i.e. 
\[ \phi^*_kD_jY^k \tau_j = 0,\]
Because $\phi*(Yu) < 0$ in $\Omega$ its outer normal derivative on $\partial \Omega$ is nonnegative, i.e. 
\[\phi^*_kD_jY^k \gamma_j \geq 0. \]
Combined these imply $D(\phi^*(Yu))|_{\partial \Omega}$ has only an outer normal component. Thus
\begin{equation}
  \label{eq:r2:chi-def}
  \phi^*_iD_jY^i = \chi \gamma_j,
\end{equation}
for some $\chi \geq 0$. Infact since $|D\phi^*| \neq 0$ on $\partial \Omega^*$ and $\det DY \neq 0$ we have $\chi > 0.$

We recall \eqref{eq:g:premate}, which is $DY = E^{-1}w$, and also $E^{ij} = D_{p_j}Y^i$. Thus
\begin{equation}
\label{eq:r2:y-defn}  D_jY^i = D_{p_k}Y^iw_{kj}.
\end{equation}
Via which \eqref{eq:r2:chi-def} becomes
\begin{equation}
  \label{eq:r2:chi1}
  \phi^*_iD_{p_k}Y^iw_{kj} = \chi \gamma_j,
\end{equation}
or equivalently
\begin{equation}
 \label{eq:r2:exp1}
  \phi^*_iD_{p_k}Y^i = \chi w^{jk}\gamma_j.
\end{equation}
Because $G_{p_k} = \phi^*_iD_{p_k}Y^i$ equation \eqref{eq:r2:exp1} implies
\begin{equation}
 G_p\cdot \gamma = \chi w^{ij}\gamma_i\gamma_j.\label{eq:r2:urb1}
\end{equation}
Note this implies (non-strict) obliqueness. Alternatively, had we used \eqref{eq:r2:chi1} to obtain $\gamma_j = \chi^{-1}\phi^*_iY^i_{p_k}w_{kj}$ on substituting into \eqref{eq:r2:obl2} we instead obtain
\begin{equation}
G_p\cdot \gamma = \chi^{-1}\phi^*_k\phi^*_lY^l_{p_j}Y^k_{p_i}w_{ij}.\label{eq:r2:urb2}
\end{equation}
Multiplying together \eqref{eq:r2:urb1} and \eqref{eq:r2:urb2} gives
\begin{equation}
 (G_p \cdot \gamma)^2 = ( w^{mn}\gamma_m\gamma_n)(\phi^*_k\phi^*_lY^l_{p_j}Y^k_{p_i}w_{ij}),\label{eq:r2:urb}
\end{equation}
which, following Trudinger and Wang \cite{TrudingerWang09}, we refer to as a formula of Urbas type.  \\
\\
\textit{Step 2. Estimate on second term in Urbas formula:} 
To obtain the strict obliqueness we estimate both terms in \ref{eq:r2:urb} from below. It suffices to obtain these estimates at $x_0$ in $\partial\Omega$ where $G_{p}\cdot\gamma|_{\partial\Omega}$ is assumed to attain its minimum. After a rotation we assume at $x_0$ that $e_{1},\dots,e_{n-1}$ are tangential to $\partial \Omega$ and $\gamma = e_n$. Moreover we assume  $\gamma$ has been extended as a $C^2$ vector field to a neighbourhood of $\partial \Omega$. For $K$ to be chosen large, define
\[ v := G_{p}\cdot \gamma - KG.\]
We recall $G$ has been modified so as to ensure $G_{p_ip_j} \geq c_0I$.
Since $G$ is 0 on $\partial \Omega$, $v|_{\partial \Omega}$ still has its minimum at $x_0$. Thus by differentiating tangentially
\[ D_{i}v = D_{i}(G_p \cdot \gamma) - KD_iG = 0 \text{ for }i=1,\dots,n-1. \]
We claim $D_n v \leq C$, for $C$ depending on the quantities stated in the theorem. The proof is a barrier argument based diversion which we save for the conclusion of this proof (Step 4). Computing $D_i G$ we have
\begin{align*}
 D_i(G_p \cdot \gamma)  - K \phi_{k}^*D_iY^k =0 &\text{ for }i = 1,\dots,n-1,\\
 D_i(G_p \cdot \gamma)  - K \phi_{k}^*D_iY^k \leq C &\text{ for }i=n.
\end{align*}
Hence after multiplying by $\phi^*_{r}Y^r_{p_i}$ and summing over $i=1,\dots,n$ 
\[K\phi^*_{r}Y^r_{p_i}\phi^*_kD_iY^k \geq \phi^*_rY^r_{p_i}D_i(G_p\cdot\gamma) - C\phi_r^*Y^r_{p_n}. \]
 Substituting $D_{i}Y^k = Y^k_{p_j}w_{ij}$, and noting $\gamma = e_n$ at $x_0$ implies $\phi_r^*Y^r_{p_n} = G_{p}\cdot\gamma$, we have
\begin{equation}
 Kw_{ij}Y^k_{p_i}Y^l_{p_j}\phi^*_k\phi^*_l \geq \phi^*_rY^r_{p_i}D_i(G_p\cdot\gamma) - C(G_{p}\cdot\gamma).\label{eq:r2:tosub}
\end{equation}

The left-hand side is one of the terms from the Urbas formula \eqref{eq:r2:urb} --- it's what we want to estimate. So we  focus on $\phi^*_rY^r_{p_i}D_i(G_p\cdot\gamma)$. We require a structure relation for $Y$. From our definition of $A$, specifically that
\[ E^{kl}[D_{li}u-A_{li}] = D_iY^k = Y^k_i+Y^k_uu_i+Y^k_{p_l}D_{li}u,\]
we obtain 
\[ -Y^{k}_{p_l}A_{li} = Y^k_i+Y^k_zp_i,\]
where the indices are chosen to align with a coming substitution. Differentiating with respect to $p_j$ yields the identity
\begin{equation}
 -Y^{k}_{p_lp_j}A_{li}-Y^k_{p_l}A_{li,p_j} = Y^{k}_{ip_j}+Y^k_{up_j}p_i+Y^k_z\delta_{ij}.\label{eq:r2:ystructure}
\end{equation}
Now using \eqref{eq:r2:obl2} and \eqref{eq:r2:y-defn}  yields
\[ D_i(G_p \cdot\gamma) = \phi_{km}^*Y^k_{p_j}Y^m_{p_l}w_{li}\gamma_j + \phi^*_kY^k_{p_j}D_i\gamma_j + \phi_k^*\gamma_j(Y^k_{i,p_j}+Y^k_{u,p_j}u_i+Y^k_{p_lp_j}u_{li}).\]
In this expression we use equation \eqref{eq:r2:ystructure} to substitute for the first two terms in parentheses and obtain
\begin{align}
  D_i(G_p \cdot \gamma) &= \phi_{km}^*Y^k_{p_j}Y^m_{p_l}w_{li}\gamma_j + \phi^*_kY^k_{p_j}D_i\gamma_j \nonumber \\
               &\quad\quad+ \phi_k^*\gamma_j(-Y^{k}_{p_lp_j}A_{li}-Y^k_{p_l}A_{li,p_j}-Y^k_z\delta_{ij}+Y^k_{p_lp_j}u_{li})\nonumber \\
               &= \Big[\phi_{km}^*Y^k_{p_j}Y^m_{p_l}w_{li}\gamma_j + \phi^*_kY^k_{p_lp_j}w_{li}\gamma_j\Big] + \phi^*_kY^k_{p_j}\big(D_i\gamma_j-A_{ij,p_l}\gamma_l\big)\nonumber \\
               &\quad\quad- \phi^*_kY^k_z\gamma_j\delta_{ij} \nonumber \\
               &= D^2_{p_jp_l}(\phi^*\circ Y)w_{li}\gamma_j+ \phi^*_kY^k_{p_j}\big(D_i\gamma_j-A_{ij,p_l}\gamma_l\big)- \phi^*_kY^k_z\gamma_i, \label{eq:r2:Di}
\end{align}
where the second equality uses $w_{li} = u_{li}-A_{li}$. 

We multiply \eqref{eq:r2:Di} by $\phi^*_rY^r_{p_i}$ and sum over $i=1,\dots,n$ to obtain
\begin{align}
 \phi^*_rY^r_{p_i}&D_i(G_p \cdot \gamma) = \phi^*_rY^r_{p_i}D^2_{p_jp_l}(\phi^*\circ Y)w_{li}\gamma_j +  \phi^*_rY^r_{p_i}\phi^*_kY^k_{p_j}\big(D_i\gamma_j-A_{ij,p_l}\gamma_l\big)\nonumber \\&\quad\quad\quad\quad\quad\quad+\phi^*_rY^r_{p_i}\phi^*_kY^k_z\gamma_i\nonumber \\
&= \chi D^2_{p_jp_l}(\phi^*\circ Y)\gamma_j\gamma_l+  \phi^*_rY^r_{p_i}\phi^*_kY^k_{p_j}\big(D_i\gamma_j-A_{ij,p_l}\gamma_l\big) - \phi^*_kY^k_z(G_{p}\cdot\gamma) \label{eq:r2:sum}
\end{align}

From our domain convexity assumptions, namely Lemmas \ref{lem:r1:phi-def} and \ref{lem:r1:g-def}, the first term is positive and the second is greater than some constant $c$. Thus we have the estimate
\begin{equation}
  \label{eq:r2:use-again}
   \phi^*_rY^r_{p_i}D_i(G_p \cdot \gamma) \geq  c-C(G_p\cdot\gamma). 
\end{equation}
Here $c$ and $C$ depend only on allowed quantities. Now if $G_p \cdot \gamma \geq c/4C$ then this is our obliqueness estimate. Otherwise \eqref{eq:r2:tosub} and \eqref{eq:r2:use-again} yield
\[  Kw_{ij}Y^k_{p_i}Y^l_{p_j}\phi^*_k\phi^*_l \geq \frac{c}{2}.\]
This is our estimate for the second term in the Urbas formula. The other term is quicker. \\
\\
\textit{Step 3. Estimate for first term in Urbas formula:}
To estimate the remaining term $w^{ij}\gamma_i\gamma_j$ from below at $x_0$ note our assumption $D_iv = 0$ for $i=1,\dots,n-1$ and $D_nv \leq C$ implies  $Dv = \tau \gamma$ for $\tau \leq C$. Thus it suffices to esimtate $w^{ij}\gamma_iD_jv$ from below.  If we recall from \eqref{eq:r2:exp1} that $w^{ij}\gamma_i = \chi^{-1}\phi^*_kY^k_{p_j}$ we obtain
\begin{align}
  w^{ij}\gamma_iD_j v &= \chi^{-1}\phi^*_kY^k_{p_j}D_j v  \nonumber\\
& = \chi^{-1}\phi^*_kY^k_{p_j}D_j(G_p \cdot \gamma) -  K\chi^{-1}\phi^*_kY^k_{p_j} D_j G.  \label{eq:r2:tosub1}
\end{align}
Noting that on $\partial \Omega$, $D_jG= \phi_k^*D_jY^k =\phi_k^*Y^k_{pl}w_{lj} $, the second term satisfies 
\[ K\chi^{-1}\phi^*_kY^k_{p_j} D_j G =K\chi^{-1}\phi^*_iY^i_{p_j}\phi_k^*Y^k_{pl}w_{lj} =K(G_p\cdot\gamma),  \] where we have employed \eqref{eq:r2:urb2}.
Moreover we estimate the first term of \eqref{eq:r2:tosub1} using \eqref{eq:r2:sum}. We note in particular the factor $\chi$ in the first term on the right-hand side of \eqref{eq:r2:sum} along with the uniform positivity of $D_{p_jp_l}(\phi^*\circ Y)$ implies that \eqref{eq:r2:tosub1} becomes  
\begin{equation}
  \label{eq:r2:urb-2}
  w^{ij}\gamma_iD_j v \geq c_0- K(G_p\cdot\gamma). 
\end{equation}
As before we're free to assume $G_p\cdot\gamma \leq c_0/2K$. Thus \eqref{eq:r2:urb-2} completes the obliqueness estimate, provisional on the claim $D_nv(x_0)\leq C$.\\   
\\
\textit{Step 4. Estimate on $D_nv(x_0)$:} This final step is obtained by a barrier argument, and it is this barrier argument. We consider the linear operator
\[ \overline{L}v := w^{ij}[D_{ij}v - D_{p_k}A_{ij}D_kv].\]
To simplify calculations set $F(x,u,p) = G_p(x,u,p)\cdot\gamma(x) - K G(x,u,p)$ where  $K$ will be chosen large in the course of the argument. Set $v(x) = F(x,u,Du)$ and compute 
\begin{align}
\label{eq:r2:gen-calc}  \overline{L}v &= w^{ij}F_{p_lp_m}D_{li}uD_{mj}u+F_{p_l}w^{ij}(D_{lij}u-D_{p_k}A_{ij}D_{lk}u)\\
\nonumber  &+ w^{ij}(F_{up_l}u_jD_{li}u+F_{up_l}u_iD_{lj}u+F_zu_{ij}+F_{jp_l}D_{li}u+F_{ip_l}D_{lj}u)\\
\nonumber  &+ w^{ij}\big(F_{uu}u_{i}u_j+F_{iu}u_j+F_{ju}u_i+F_{ij}-D_{p_k}A_{ij}(F_k+F_zu_k)\big).
\end{align}
An estimate for each line follows without too much difficulty. The final line is bounded above by $Cw^{ii}$ where $C$ depends on $A,\Omega,\Omega^*, \Vert u \Vert_{C^1(\Omega)}$ and $K$. Similarly, the calculation
\[ w^{ij}D_{li}u = w^{ij}(w_{li}+A_{li}) \leq \delta_{li}+Cw^{ij},\]
implies the second line is bounded above by $C(w^{ii}+1)$. Next, differentiating the equation with respect to $l$ yields
\[F_{p_l}w^{ij}(D_{lij}u-D_{p_k}A_{ij}D_{lk}u) \leq C(1+w^{ii}) + |B_{p_k}|w_{ii}.\]
For the remaining term $w^{ij}F_{p_lp_m}D_{li}uD_{mj}u$, the $g^*$-convexity of the target domain, via Lemma \ref{lem:r1:g-def}, implies $G_{p_kp_l} \geq c_0I$. Thus after a choice of $K$ large
\[ w^{ij}F_{p_lp_m}D_{li}uD_{mj}u \leq  -\frac{Kc_0}{2}w_{ii}+C(w^{ii}+1).\]
Hence
\[ \overline{L}v \leq \frac{-Kc_0}{4}w_{ii}+|B_{p_k}|w_{ii}+C(w^{ii}+1). \]
We have $w^{ii} \geq c$ by recalling that $\det w \leq c$ (though for possibly different $c$). Using this and a choice of $K$ sufficiently large we have
\[ \overline{L}v \leq Cw^{ii}.\]

Now for the other function in our barrier argument. By Lemma \ref{lem:r1:g-def} there is a defining function $\phi$ satisfying
\[ \overline{L}\phi \geq w^{ii}.\]
on $\Omega_\epsilon := \{x \in \Omega; \text{dist}(x,\partial \Omega) < \epsilon\}$. Note that $\phi = 0$ on $\partial \Omega$ and $\phi < - k < 0$  on $\{x; \text{dist}(x,\Omega) = \epsilon\}$ for some $k>0$. Thus for $a$ sufficiently large 
\begin{align*}
  \overline{L}(v - a \phi ) \leq 0 \text{ in } \Omega_\epsilon ,\\
 v-a\phi  \geq v(x_0) \text{ on }\partial \Omega_\epsilon,
\end{align*}
where we recall $x_0$ is the point on $\partial \Omega$ where $v$ attains its minimum.
By the comparison principle $v-a\phi$ attains its minimum over $\Omega_\epsilon$ at $x_0$. Thus the outer normal $D_n(v-a\phi) \leq 0$, which is to say, $D_nv(x_0) \leq C$. 
\end{proof}

The obliqueness estimate is used for boundary second derivative estimates.

\begin{theorem}\label{thm:r2:2bvp-global}
Assume $u \in C^4(\Omega) \cap C^3(\overline{\Omega})$ is an elliptic solution of \eqref{eq:g:mate} subject to \eqref{eq:g:2bvp}. Assume the $C^2$ domains $\Omega,\Omega^*$ are, respectively, uniformly $g/g^*$-convex with respect to $u$ and $A,B$ are $C^2$ with $B>0$. Then there exists $C$ depending on $g,\Omega,\Omega^*,A,B$ with
\[ \sup_{\partial\Omega}{|D^2u|}\leq C(1+\sup_{\Omega}|D^2u|)^{\frac{2n-3}{2n-2}}. \] 
\end{theorem}
\begin{remark}
  The exponent on the right hand side may be perplexing to newcomers. What's important is that the exponent is less than one. So when combined with Theorem \ref{thm:r2:pog-global} (the estimate for second derivatives in terms of their boundary values) we obtain
  \begin{equation}
    \label{eq:r2:comb-est}
      \sup_{\Omega}|D^2u| \leq C(1+\sup_{\Omega}|D^2u|)^{\frac{2n-3}{2n-2}}.
  \end{equation}
  The sub-linear power on the right hand side implies $\sup_{\Omega}|D^2u| \leq C'$ depending only on the $C$ in the expression: Otherwise $ \sup_{\Omega}|D^2u|$ sufficiently large contradicts \eqref{eq:r2:comb-est}
\end{remark}

\begin{proof}[Proof. (Theorem \ref{thm:r2:2bvp-global})]
  \textit{Step 1. Decompose into tangent and oblique direction}\\
  We consider an arbitrary unit vector $\xi \in \mathbf{R}^n$ and estimate $w_{\xi\xi}$ on $\partial \Omega$. To this end, at each $x \in \partial \Omega$, decompose $\xi$ into a tangential component and a component in the direction of $\beta := G_p(x,u,Du)$. Explicitly,
\[ \xi = \tau+b\beta, \quad \quad b = \frac{\xi\cdot\gamma}{\beta\cdot\gamma}, \quad \quad\tau\cdot\gamma = 0,\]
where crucially $|b|$ is controlled by the obliqueness. We obtain
\begin{equation}
 w_{\xi\xi} = w_{\tau\tau}+2bw_{\tau\beta}+b^2w_{\beta\beta}.\label{eq:bbound1}
\end{equation}
Thus we need only estimate $w_{\tau \tau},w_{\tau \beta},w_{\beta\beta}$.\\
\\
\textit{Step 2. The easiest term: $w_{\tau\beta}$.} Differentiate $\phi^*(Y(\cdot,u,Du)) = 0$ on $\partial \Omega$ in a direction tangential to $\partial \Omega$ and obtain
\begin{equation}
  \label{eq:r2:simple-est}
   \phi^*_kD_iY^k\tau_i = 0.
\end{equation}
Recalling \eqref{eq:r2:y-defn}, that is $D_iY^k = Y^k_{p_j}w_{ij}$, equation \eqref{eq:r2:simple-est} implies
\begin{equation}
 w_{\tau\beta} = 0.\label{eq:wtbbound}
\end{equation}
\\
\textit{Step 3. $w_{\beta\beta}$ estimate.}
The $w_{\beta\beta}$ estimates are via a barrier argument. We've used such arguments throughout this chapter to obtain estimates on outer normal derivatives. This is useful here because, by a direct calculation like \eqref{eq:bbound1}, $D_\gamma(\phi^*(Y(\cdot,u,Du))) = w_{\tau\beta}+cw_{\beta\beta}$ for a controlled $c$. Set
\[\overline{L}v := w^{ij}(D_{ij}v - D_{p_k}A_{ij}(x,u,Du)D_kv). \]
Recall when $v = F(\cdot,u,Du)$  by \eqref{eq:r2:f-calcs} we have
\begin{equation}
|\overline{L}v| \leq C(1+w_{ii}+w^{ii}).\label{eq:tracestuff1}
\end{equation}

Now, we claim 
\begin{equation}
w_{ii}^{\frac{1}{n-1}} \leq Cw^{ii}.\label{eq:tracestuff2}
\end{equation}
Too see this is write the eigenvalues of $w$ as $\lambda_1\geq \dots \geq \lambda_n > 0$ and note
$C \geq \det w \geq \lambda_1(\lambda_n)^{n-1}.$ Combine with $\lambda_1 \geq w_{ii}/n$ and $\lambda_n \geq (w^{ii})^{-1}$.  Next, because $w^{ii} \geq C$, \eqref{eq:tracestuff1} simplifies to 
\[|\overline{L}v| \leq C(w^{ii}+w_{ii}).\]
Combined with \eqref{eq:tracestuff2} this becomes
\begin{equation}
| \overline{L}v| \leq C(1+M)^{\frac{n-2}{n-1}}w^{ii},\label{eq:vest}
\end{equation}
where $M = \sup_{\Omega}|D^2u|.$

For our barrier take $\phi,\phi^*$ as in Lemmas \ref{lem:r1:phi-def} and \ref{lem:r1:g-def} and set $v = \phi^*(Y(\cdot,u,Du))$,
\[ \mu := C(1+M)^{\frac{n-2}{n-1}}\phi-v.\]
We consider these on $\Omega_\epsilon = \{x \in \Omega; \text{dist}(x,\partial\Omega) < \epsilon \} .$ By Lemma \ref{lem:r1:phi-def}, a choice of $C$ large, and \eqref{eq:vest} we have $\overline{L}\mu > 0$ in $\Omega_\epsilon$. Set $\mathcal{N} =  \{x \in \Omega; \text{dist}(x,\partial\Omega) = \epsilon\}$. We have 
\[ \partial\Omega_\epsilon = \partial\Omega \cup \mathcal{N}.\]
On $\partial \Omega$ we have $\mu = 0$. A further choice of $C$ large depending only on $\inf_{\Omega^*}\phi^*$ and $\sup_{\mathcal{N}}\phi$ yields $\mu \leq 0$ on $\mathcal{N}$. Since now $\overline{L}\mu > 0$ in $\Omega_\epsilon$ and $\mu \leq 0$ on $\partial\Omega_\epsilon$ the maximum principle implies $\mu \leq 0$ in $\Omega_\epsilon$. Thus on $\partial\Omega$, $D_\gamma \mu > 0$. This implies $D_\gamma[\phi^*(Y(\cdot,u,Du))] \leq C(1+M)^{\frac{n-2}{n-1}}$. Then by our remarks at the start of this step 
\begin{equation}
  \label{eq:wbbbound}
  w_{\beta\beta} \leq C(1+M)^{\frac{n-2}{n-1}}.
\end{equation}

\textit{Step 4. $w_{\tau\tau}$ bound.}
We assume the repeated tangential derivatives have a boundary maximum at $x_0 = 0$ in a tangential direction assumed to be $e_1$, and  estimate $w_{11}(0)$. Differentiating $G(\cdot,u,Du) = 0$ on $\partial\Omega$ twice in the tangential direction $e_1$ we obtain 
\begin{equation}
 G_{p_kp_l}u_{1k}u_{1l}+u_{11\beta} \leq C(1+M).\label{eq:willbound}
\end{equation}
By Lemma \ref{lem:r1:g-def} we have $G_{p_kp_l} \geq c_0I$ on $\partial \Omega$ so \eqref{eq:willbound} implies
\begin{equation}
  \label{eq:r2:final-step}
   |u_{11}|^2 \leq C(1+M) - u_{11\beta}.
\end{equation}
An estimate of for $-D_{\beta}w_{11}$ follows from an outer normal estimate $-D_\gamma w_{11}$. Thus another barrier argument is in order. 

We use \eqref{eq:bbound1} with $\xi = e_1$ along  with \eqref{eq:wtbbound} and \eqref{eq:wbbbound} to obtain that on the boundary
\begin{align}
  w_{11} &\leq w_{\tau\tau}+b^2w_{\beta\beta}\leq |\tau|^2w_{11}(0)+b^2(1+M)^{\frac{n-2}{n-1}}.\label{eq:w11bound}
\end{align}
Now consider
\begin{align*}
  v &:=  w_{11}-|\tau|^2w_{11}(0)-b^2(1+M)^{\frac{n-2}{n-1}}.
\end{align*}
In light of \eqref{eq:w11bound} we have $v \leq 0$ on $\partial \Omega$ and and $v(0) = 0$. Elsewhere we clearly have $|v| \leq C(1+M)$. 

Now, equation \eqref{eq:p:lw11} yields
\[ Lw_{11} \geq -C(w_{ii}^2+w_{ii}w^{ii}+w^{ii}).\]
Because $|\tau|^2,b^2$ are of the form $F(\cdot,u,Du)$ the same calculations as for \eqref{eq:vest} (though now with an additional power of $w_{ii}$) yield
\begin{equation}
  \label{eq:r2:v-diff}
  Lv \geq -C(1+M)^{\frac{2n-3}{n-1}}w^{ii}. 
\end{equation}

Then the barrier argument is as follows. Consider on $\Omega \cap B_r(0)$ for small $r$ the function 
\[ \psi = \phi - \mu \text{dist}(x,\partial \Omega)+K\text{dist}(x,\partial \Omega)^2-a |x|^2.\]
We follow the reasoning used for inequality \eqref{eq:r2:a-inf} to conclude if $K$ is sufficiently large, and $a,\mu,r$ are sufficiently small, we have for some small $c_0$
\begin{align*}
  L\psi \geq c_0w^{ii} \text{ in }\Omega \cap B_r(0)\\
  \psi \leq 0 \text{ on }\partial \Omega \cap \overline{B_r(0)}\\
  \psi \leq -ar^2 \text{ on }\Omega \cap \partial B_r(0).
\end{align*}
Thus for $b$ sufficiently large $b\psi + v \leq 0$ on $\partial (\Omega \cap B_r(0))$ and $L(b\psi + v ) \geq 0$. Thus $b\psi+v$ attains its maximum over $\Omega \cap B_r(0)$ at $0$ and the outer normal derivative at $0$ is nonnegative. That is, $ D_\gamma \mu > 0$, from which it follows that
\begin{equation}
  \label{eq:r2:third}
   -D_{\beta}w_{11}(0) \leq C(1+M)^{\frac{2n-3}{n-1}}.
\end{equation}
We've implicitly used estimates on $D_\beta|b|^2,D_\beta|\tau|^2$ to control $D_\beta v$. Since $|b|^2,|\tau|^2$ are of the form $F(\cdot,u,Du)$ they permit, via the barrier argument of step 3, an estimate like \eqref{eq:wbbbound}. 
Using \eqref{eq:r2:third} in \eqref{eq:r2:final-step} we obtain,
\begin{equation}
 w_{11}(0) \leq  C(1+M)^{\frac{2n-3}{2(n-2)}}\label{eq:wttbound}
\end{equation}
Combining \eqref{eq:wtbbound}, \eqref{eq:wbbbound} and \eqref{eq:wttbound} yields the boundary estimate. 
\end{proof}

\section{Global regularity via degree theory}
\label{sec:glob-regul-via}

Now we have everything needed to prove the global regularity of Aleksandrov solutions (Theorem \ref{thm:r1:main}). We recall our comments after the statment of the theorem: all that's left to do is construct a solution $v \in C^3(\overline{\Omega})$ satisfying $v(x_0) = u(x_0)$. Here $u$ is the given Aleksandrov solution and $x_0$ is some point in $ \Omega$. To this end, let $g(\cdot,y_0,z_0)$ be a support of $u$ at $x_0$ and $\epsilon,\rho,\delta$ be small parameters. Consider the function $\overline{u}$ from Lemma \ref{lem:r1:constr-unif-g} and the function $u_\epsilon$ from Theorem \ref{thm:r1:hom-start}. Let $f_\delta,f^*_\epsilon$ denote smooth extensions of $f,f^*$ to $\Omega_\delta,\Omega^*_{\delta,\epsilon}$  normalized so they still satisfy mass balance. We first obtain the solvability of an approximating problem. Then we complete the proof of Theorem \ref{thm:r1:main} by sending $\epsilon,\delta,\rho \rightarrow 0$.  

\begin{lemma}\label{lem:r2:approx-prob}
  Assume the conditions of Theorem \ref{thm:r1:main}. There exists $v \in C^3(\overline{\Omega}_\delta)$ depending on $\delta,\epsilon,\rho$ and satisfying
  \begin{align}
 \label{eq:r2:f-app-1}   \det DY(\cdot,v,Dv) &= \frac{f_\delta(\cdot)}{f_\epsilon^*(Y(\cdot,v,Dv))}\text{ in }\Omega_\delta\\
 \label{eq:r2:f-app-2}   Yv(\Omega_\delta) &= \Omega^*_{\delta,\epsilon}\\
    v(x_0) &= u_\epsilon(x_0) \label{eq:r2:f-app-3}
  \end{align}
\end{lemma}
\begin{proof}
  Within this proof we'll use the notation $f,f^*,\Omega,\Omega^*$, and $u_0$ for $f_\delta,f^*_\epsilon,$ $\Omega_\delta,\Omega^*_{\delta,\epsilon},$ and $u_\epsilon$. Fix a smooth cutoff function $\eta$ for the unit ball, that is $\eta > 0$ on $B_1(0)$ and $\eta = 0$ on $\overline{B_1}^c$ and set $\eta_a(x) =  a^{4}\eta((x-x_0)/a)$. The $a^{4}$ term ensures $\Vert \eta_a \Vert_{C^4} \leq C$ for $C$ independent of $a$.

  For $t \in [0,1]$ and $\tau$ to be fixed large consider the family of problems
\begin{align}
  \label{eq:r2:homotopy}  f^*(Y(\cdot,v,Dv))&\det DY(\cdot,v,Dv) = e^{[(1-t)\tau+\eta_a(\cdot)](v-u_0)}\big[tf(\cdot)\\
                                   \nonumber                     &\quad+(1-t)f^*(Y(\cdot,u_0,Du_0))\det DY(\cdot,u_0,Du_0)\big]\\
\label{eq:r2:2bvp-approx}  Yu(\Omega) &= \Omega^*. 
\end{align}  
We restrict our attention to $g$-convex $C^{4,\alpha}(\overline{\Omega})$ solutions. For $t=0$ the problem is solved by $u_0$. We aim to show, using the degree theory of Li, Liu, and Nguyen \cite{LLN17} that his problem has a solution for $t=1$. We outline the required $C^{4,\alpha}(\overline{\Omega})$ bounds below. Then the degree theory is applied as follows. By \cite[Theorem 1 (p2)]{LLN17} the problems \eqref{eq:r2:homotopy} and \eqref{eq:r2:2bvp-approx} have the same degree for $t=0$ and $t=1$. Then by \cite[Corollary 2.1 (a)]{LLN17} to show a solution exists for $t=1$ it suffices to show the problem at $t=0$ has nonzero degree. Finally by a combination of Corollary 2.1 (d) and Theorem 1 (p3) from the same paper, the problem has nonzero degree at $t=0$ provided both the problem for $t=0$ has a unique solution and the linearised problem at $t=0$ (linearised about $u_0$) is uniquely solvable. To summarise we need apriori estimates in $C^{4,\alpha}(\overline{\Omega})$ along with the unique solvability for $t=0$ of the problem and its linearisation. 

\textit{Step 1. a-priori estimates} We first show by mass balance that any solution of \eqref{eq:r2:homotopy} subject to \eqref{eq:r2:2bvp-approx} intersects $u_0$ in the domain and, furthermore, for $t=1$ this intersection occurs in $B_a(x_0)$. Assume to the contrary $v > u_0$ on $\Omega$. The proof is similar if $v < u_0$ on $\Omega$. Because $v > u_0$ we have $ e^{[(1-t)\tau+\eta_a(\cdot)](v-u_0)} > 1$. Now \eqref{eq:r2:homotopy} and \eqref{eq:r2:2bvp-approx} along with mass balance yield the following contradiction
  \begin{align*}
    \int_{\Omega^*}f^*(y) \ dy &= \int_{\Omega} f^*(Y(\cdot,v,Dv))\det DY(\cdot,v,Dv)  \\
                         &> \int_{\Omega}tf(\cdot) + (1-t)f^*(Y(\cdot,u_0,Du_0))\det DY(\cdot,u_0,Du_0)\\
    &= \int_{\Omega^*}f^*. 
  \end{align*}
  For $t=1$ we obtain the same contradiction if $v > u_0$ on $B_a$ for in this case 
  \begin{align*}
    f^*(Y(\cdot,v,Dv))\det DY(\cdot,v,Dv) &> f \text{ on }B_a(x_0)\\
    f^*(Y(\cdot,v,Dv))\det DY(\cdot,v,Dv) &= f \text{ on }\Omega\setminus\overline{B_a(x_0)}.
  \end{align*}
  An estimate $\sup_{\Omega}|Dv|\leq C$ follows by A5. We combine the $Dv$ estimate with the contact point $v(x') = u(x')$ to obtain $\Vert v \Vert_{C^1(\overline{\Omega})} \leq C$.  Global apriori $C^2$ estimates follow via Theorem \ref{thm:r2:pog-global} in combination with Theorem \ref{thm:r2:2bvp-global}. Importantly the barrier required for Theorem \ref{thm:r2:pog-global} is provided by Theorem \ref{thm:r1:lin} and the domain convexity hypothesis of Theorem \ref{thm:r1:main}. Subsequently higher order estimates are obtained by the elliptic theory. More precisely $C^{2,\alpha}(\overline{\Omega})$ estimates from \cite[Theorem 1]{LiebermanTrudinger86} then $C^{4,\alpha}(\overline{\Omega})$ estimates by the linear theory. The linear theory is standard, so we just outline its use (these estimates are also obtained in \cite{JiangTrudinger18}). We consider a boundary portion, and a flattening of the boundary as in the proof of Theorem \ref{thm:r2:dir-est}. Then differentiating in a direction tangential to the flattened boundary we see near this boundary portion tangential derivatives satisfy a linear oblique problem. We obtain $C^{2,\alpha}$ estimates (for these first derivatives) by \cite[Lemma 6.29]{GilbargTrudinger01}. Finally the estimates for the remaining derivative are obtained directly from the equation. For the $C^{4,\alpha}$ estimates, as in \cite{JiangTrudinger18}, we  smooth our domains so they are $C^5$. However uniform $C^{3,\alpha}$ estimates are independent of this smoothing.

  We emphasize also that for the solution at $t=1$ the apriori estimate is independent of $u_0,\epsilon,\delta,\rho$ (provided these parameters are initially fixed sufficiently small).\\
 
\textit{Step 2. Unique solvability of the nonlinear problem for $t=0.$}\\
Assume $v$ solves \eqref{eq:r2:homotopy} subject to \eqref{eq:r2:2bvp-approx} for $t=0$. We recall $v = u_0$ somewhere on $\Omega$, and will show $v \equiv u_0$ on $\Omega$.  Set $\Omega' := \{x \in \Omega; v(x) > u_0(x)\}$.If $\Omega'$ is empty then $v \leq u_0$, however by Lemma \ref{lem:u:harn} this implies $v = u_0$ and we're done.  Thus we assume $\Omega' \neq \emptyset$. By \eqref{eq:r2:homotopy} and $v > u_0$ we have
  \[ f^*(Y(\cdot,v,Dv))\det DYv > f^*(Y(\cdot,u_0,Du_0)) \det DYu_0 \]
  on $\Omega'$. In particular, after integrating,
  \[ \int_{Yv(\Omega')}f^* > \int_{Yu_0(\Omega')}f^*.\]
  On the other hand by Lemma \ref{lem:u:aleksandrovmccann} we see $Yu_0^{-1}(Yv(\Omega')) \subset \Omega'$ so with \eqref{eq:g:2bvp}
  \begin{align*}
    \int_{Yv(\Omega')} f^* &= \int_{Yu_0(Yu_0^{-1}(Yv(\Omega')))}f^*\\
                   &= \int_{Yu_0^{-1}(Yv(\Omega'))}f^*(Y(\cdot,u_0,Du_0)) \det DYu_0\\
    &\leq \int_{\Omega'}f^*(Y(\cdot,u_0,Du_0)) \det DYu_0 =  \int_{Yu_0(\Omega')}f^*.
  \end{align*}
This contradiction implies $v \equiv u_0$. 

  \textit{Step 3. Unique solvability of the linearised problem}
The linearization at $t=0$ is a problem of the form
  \begin{align}
  \label{eq:r2:lin-fin}  a^{ij}D_{ij}v +b^iD_iv - cv &= 0 \text{ in }\Omega\\
  \nonumber  \alpha v + \beta \cdot Dv &= 0 \text{ on }\partial \Omega,
  \end{align}
  where $a^{ij},b^i,\alpha,\beta$ depend on $u_0$ and $g$, the equation is uniformly elliptic, $\beta$ is a strictly oblique vector field and, provided $\tau$ is sufficiently large, $c > \tau/2$. We show this problem has the unique solution $v=0$. If not, assume without loss of generality $v$ is a solution positive at some point  in $\overline{\Omega}$. Take the defining function $\phi$, which we recall satisfies $D\phi = \gamma$ (the outer unit normal), and set $w = e^{-\kappa\phi}v$ for large $\kappa$ to be chosen.  If the positive maximum of $w$ occurs at $x_0 \in \partial \Omega$ then by the obliqueness
  \[ 0 \leq \beta \cdot Dw(x_0) = e^{-\kappa \phi}[\beta \cdot Dv - \kappa (\beta \cdot \gamma) v].\]
  With the linearised boundary condition we obtain
  \[ \kappa (\beta \cdot \gamma)v(x_0) \leq \beta \cdot Dv(x_0) = - \alpha v(x_0), \]
  a contradiction for $\kappa$ sufficiently large depending only on $\alpha$. Thus $w$ attains its positive maximum at $x_0 \in \Omega$. At this point
  \begin{align}
 \nonumber   0 &\geq e^{\kappa \phi}a^{ij}D_{ij}w(x_0)  \\
 \label{eq:r2:lin-ineq}     &= a^{ij}D_{ij}v - 2\kappa a^{ij}D_i\phi D_j v - \kappa v a^{ij}D_{ij}\phi  + \kappa ^2 v a^{ij}D_i\phi D_j \phi. 
  \end{align}
  At an interior maximum $Dw(x_0) = 0$ so $D_iv(x_0) = \kappa D_i \phi(x_0) v(x_0)$. Thus we can eliminate $Dv$ terms in \eqref{eq:r2:lin-ineq}. In combination with \eqref{eq:r2:lin-fin} and $c > \tau/2$ we obtain
  \[ 0 \geq e^{\kappa \phi}a^{ij}D_{ij}w(x_0) \geq C(\tau - C_1)v(x_0),\]
  for $C_1$ independent of $\tau$. This is a contradiction for $\tau$ sufficiently large.

  \textit{Conclusion.} By the degree theory there is $v \in C^{4,\alpha}(\overline{\Omega})$ solving \eqref{eq:r2:homotopy} and \eqref{eq:r2:2bvp-approx} for $t=1$. Our $C^{3,\alpha}$ estimates are independent of $a$ and the domain smoothing. We consider a sequence of $a_k \rightarrow 0$ and the corresponding solutions denoted $v_{a_k}$. By Arzela--Ascoli we have uniform convergence to some $v \in C^{3,\alpha}(\overline{\Omega})$ solving \eqref{eq:r2:homotopy}  and \eqref{eq:r2:2bvp-approx} for $t=1$. Because there is $x_k \in B_{a_k}(x_0)$ with $v_{a_k}(x_k) = u_0(x_{k})$, by the uniform convergence, $v$ satisfies $v(x_0) = u_0(x_0)$.     
\end{proof}

The proof of Theorem \ref{thm:r1:main} follows. Take the solution of \eqref{eq:r2:f-app-1}-\eqref{eq:r2:f-app-3}. We have $C^{3,\alpha}$ bounds independent of the parameters $\epsilon,\delta,\rho$ (provided these parameters are initially fixed small, say $<1$). Send first $\epsilon$ then $\delta$ to $0$ and obtain a solution of \eqref{eq:g:gje} subject to \eqref{eq:g:2bvp} equal to $\overline{u}$ at $x_0$. Finally recall $\overline{u}(x_0) \rightarrow g(x_0,y_0,z_0) = u(x_0)$ as $\rho\rightarrow 0$. Thus on sending $\rho \rightarrow 0$ we obtain $v \in C^3(\overline{\Omega})$ solving \eqref{eq:g:gje} subject to \eqref{eq:g:2bvp} and satisfying $u(x_0)=v(x_0)$. By our remarks at the start of this section, the proof of Theorem \ref{thm:r1:main} is complete.

\clearpage{}
\clearpage{}\appendix
\chapter{Omitted proofs}
The following proofs were referenced, but not included, in the main text. 
\begin{lemma}\label{lem:a:1}
  Assume $\{u_k\}_{k=1}^\infty$ is a sequence of $g$-convex functions on $\Omega$ converging to a $g$-convex function $u:\Omega \rightarrow \mathbf{R}$. For each $E^* \subset \Omega^*$ define
  \[ \mu_{u,f}^*(E^*) = \int_{Yu^{-1}(E)}f,\]
and similarly $\mu_{u_k,f}^*$. Then  $\mu_{u_k,f}^*$ converges weakly to $\mu_{u,f}^*$.
\end{lemma}
\begin{proof}
  The proof is similar to Lemma \ref{lem:w:weak_conv}. Following the reasoning there it suffices to show for each compact $K \subset \Omega^*$
  \begin{equation}
    \label{eq:a:1}
       Yu^{-1}(K) \supset \bigcap_{i=1}^\infty\bigcup_{k=i}^\infty Yu^{-1}_k(K),
  \end{equation}
  and for each open $U \subset\subset \Omega^*$ that whenever $V$ is a compact subset of $U$
  \begin{equation}
    \label{eq:a:2}
     Yu^{-1}(V) \setminus E_u \subset \bigcup_{i=1}^\infty\bigcap_{k=i}^\infty Yu_k^{-1}(U).
  \end{equation}
  Here $E_u$ is the points where $u$ is not differentiable. 

  For \eqref{eq:a:1} take $x$ in the right hand side. There is sequence of $y_k \in K$ such that  $y_k \in Yu_k(x)$. This implies $Yu(x)$ also contains an element of $K$, as required

  For \eqref{eq:a:2} take $x$ in the left hand side. Then $Yu(x) = y \in V$. If $x$ is not in the right hand side there is a sequence of $u_k$ such that $Yu_k(x)$ is disjoint from $U$. Take a sequence of $y_k \in Yu_k(x)$, assumed up to a subsequence, to converge to some $\overline{y}$ which necessarily is not in $V$. Taking a limit yields $\overline{y} \in Yu(x)$, a contradiction.
\end{proof}

The next proof is from \cite{Trudinger14}. I fear my only contribution is obfuscating the notation. 
\begin{lemma}\index{A3w$^*$}\index[notation]{$A^*$, $A^*(y,z,q)$, $A_{ij}^*(y,z,q)$, \ \ dual augmenting matrix}
  \label{lem:a:a3w-star} Assume $g$ is a generating function satisfying A3w. Define, on $\mathcal{U}^*$,
  \[ A_{ij}^*(y,z,q) = g^*_{,ij}(X(y,z,q),y,U(y,z,q)). \]
  Then whenever $\xi,\eta \in \mathbf{R}^n$ satisfy $\xi\cdot\eta = 0$ there holds
  \[ D_{q_kq_l}A_{ij}^*(y,z,q)\xi_i\xi_j\eta_k\eta_l \geq 0. \]
\end{lemma}
\begin{proof}
  Our goal is to derive two essentially equivalent formulas for $D_{p_kp_l}A_{ij}$ and $D_{q_iq_j}A^*_{kl}$. Set
  \begin{equation}
    \label{eq:a:q-def}
       Q_j(x,y,z) = -\frac{g_{y_j}}{g_z}(x,y,z),
  \end{equation}
  and compute
  \begin{equation}
    \label{eq:a:q-diff}
    D_{x_i}Q_j = -\frac{E_{ij}}{g_z}. 
  \end{equation}
  Next, when $y=Y(x,u,p),z=Z(x,u,p)$, we have
  \begin{equation}
    \label{eq:a:p-diff}
       D_{p_k} = D_{p_k}Y^rD_{y_r} + D_{p_k}ZD_z = E^{rk}\left(D_{y_r}-\frac{g_{y_r}}{g_z}D_z\right).
  \end{equation}
  We've eliminated $D_{p_k}Z$ using reasoning similar to \eqref{eq:g:dz-sub} and employed $D_{p_k}Y^r = E^{rk}$. Thus
  \begin{equation}
    \label{eq:a:to-sub}
    D_{p_k}A_{ij} = D_{p_k}g_{ij} = E^{r,k}\left(g_{ij,r}-\frac{g_{,r}}{g_z}g_{ij,z}\right).
  \end{equation}
  We wish to substitute for the expression in parentheses. Note 
  \begin{align*}
    D_{x_j}E_{ir} - D_{x_j}(Q_r)g_{i,z} &=  D_{x_j}\left(g_{i,r} - \frac{g_{i,z}g_{,r}}{g_z}\right) + D_{x_j}\left(\frac{g_{,r}}{g_z}\right)g_{i,z}\\
    &= g_{ij,r}- g_{ij,z}\frac{g_{,r}}{g_z}.
  \end{align*}
  Using this and \eqref{eq:a:q-diff}, equation  \eqref{eq:a:to-sub} simplifies as follows:
  \begin{align*}
    D_{p_k}A_{ij} &= E^{rk}\left( D_{x_j}E_{ir} - D_{x_j}(Q_r)g_{i,z} \right)\\
    &= E^{rk}D_{x_j}E_{ir}+\delta_{jk}\frac{g_{i,z}}{g_z}.
  \end{align*}
  By another use of \eqref{eq:a:p-diff} and the formula for differentiating the inverse of a  matrix
  \begin{align}
    \nonumber    D_{p_kp_l}A_{ij}&= E^{rk}E^{sl}\left[D_{x_jy_s}E_{ir} - \frac{g_{,s}}{g_z}D_{x_j,z}E_{ir}\right] \\
    \nonumber                    &\quad + E^{sl}D_{x_j}(E_{ir})\left[D_{y_s}E^{rk}-\frac{g_{,s}}{g_z}D_{z}(E^{rk})\right]  + \delta_{jk}D_{p_l}\left(\frac{g_{i,z}}{g_z}\right) \\
    \label{eq:a:aijkl}    &= E^{rk}E^{sl}\big[D_{x_jy_s}E_{ir}+Q_{s}D_{x_j,z}E_{ir}\\
    \nonumber                             &\quad-E^{ab}D_{x_j}E_{ia}(D_{y_s}E_{br}+Q_{s}D_{z}E_{br})\big] + \delta_{jk}D_{p_l}\left(\frac{g_{i,z}}{g_z}\right).
  \end{align}

  Now we aim for a similar expression for $D_{q_iq_j}A^*_{kl}$. First note because
  \[g^*_{y_k}(x,y,u) = -\frac{g_{,k}}{g_z}(x,y,g^*(x,y,u)),\]
  we have
  \begin{equation}
    \label{eq:a:gyy}
     g^*_{y_ky_l}(x,y,u) = -D_{y_l}\left(\frac{g_{_{,k}}}{g_z}\right)\Big|_{(x,y,g^*(x,y,u))}- g^*_{,l}D_{z}\left(\frac{g_{_{,k}}}{g_z}\right)\Big|_{(x,y,g^*(x,y,u))} .
  \end{equation}
  Now because
  \[ A^{*}_{kl}(y,z,q) = g^*_{y_ky_l}(X(y,z,q),y,g(X(y,z,q),y,z)),\]
  equation \eqref{eq:a:gyy} along with \eqref{eq:a:q-def} implies 
  \[ A^*_{kl}(y,z,q) = D_{y_l}Q_k + q_lD_zQ_k,\]
  where $Q$ is evaluated at $X(y,z,q),y,z$. Thus
  \begin{align}
    \label{eq:a:one-q}
    D_{q_i}A^*_{kl}(y,z,q) &= \frac{\partial X^r}{\partial q_i}D_{x_r,y_l}Q_k+\frac{\partial X^r}{\partial q_i}D_{x_r,z}Q_k Q_l+ \delta_{il}D_z(Q_k).
  \end{align}
  To simplify, note
  \begin{align}
   \label{eq:a:xq-dif} \frac{\partial X^r}{\partial q_i} &= -g_zE^{ir},
  \end{align}
  and from \eqref{eq:a:q-diff}
  \begin{align*}
    D_{x_r,y_l}Q_k &= -\frac{1}{g_z}\left(D_{y_l}E_{rk}-E_{rk}\frac{g_{,l,z}}{g_z}\right),\\
     D_{x_r,z}Q_k &=  -\frac{1}{g_z}\left(D_{z}E_{rk}-E_{rk}\frac{g_{,zz}}{g_z}\right).
  \end{align*}
  Whereby \eqref{eq:a:one-q} implies
  \[   D_{q_i}A^*_{kl}(y,z,q) = E^{ir}(D_{y_l}E_{rk} +D_{z}E_{rk} q_l) - \frac{\delta_{ik}}{g_z}(g_{,l,z}+g_{,zz}q_l) + \delta_{il}D_z(Q_k). \]
  So by differentiating with respect to $q_j$ and employing \eqref{eq:a:xq-dif}
  \begin{align}
    \nonumber   D_{q_iq_j}A^*_{kl} &(y,z,q)\\
  \nonumber  &= -g_zE^{js}\big[-E^{ia}D_{x_s}E_{a,b}E^{br}(D_{y_l}E_{r,k}+D_{z}E_{r,k}q_l)\\
                  \nonumber                        &\quad+E^{ir}(D_{x_s,y_l}E_{rk}+D_{x_s,z}E_{r,k}q_l)]\\
                 \nonumber                         &\quad\quad+E^{ir}D_{z}E_{rk}\delta_{jl}-\delta_{ik}D_{q_j}\left(\frac{1}{g_z}(g_{,l,z}+g_{,zz}q_l)\right) + \delta_{il}D_{q_j,z}(Q_k)\\
 \label{eq:a:astar-fin}   &= -g_zE^{js}E^{ir}[D_{x_s,y_l}E_{rk} + q_lD_{x_s,z}E_{rk}\\
    \nonumber &\quad- E^{ab}D_{x_s}E_{ra}(D_{y_l}E_{b,k}+q_lD_zE_{b,k})]\\
    \nonumber &\quad\quad +E^{ir}D_{z}E_{rk}\delta_{jl}-\delta_{ik}D_{q_j}\left(\frac{1}{g_z}(g_{,l,z}+g_{,zz}q_l)\right) + \delta_{il}D_{q_j,z}(Q_k).
  \end{align}
  Finally we consider the quantities $D_{p_kp_l}A_{ij}\xi_i\xi_j\eta_k\eta_l$ and $D_{q_iq_j}A^*_{kl}\xi_i\xi_j\eta_k\eta_l$ where $\xi,\eta \in \mathbf{R}^n$ satisfy $\xi \cdot \eta = 0$ . Note when we use \eqref{eq:a:aijkl} and \eqref{eq:a:astar-fin} to compute these quantities all terms with a Dirac delta vanish by orthogonality. More explicitly, direct calculation using these identities implies if $\xi \cdot \eta = 0,$ then
\[D_{q_iq_j}A^*_{kl}\xi_i\xi_j\eta_k\eta_l = -g_zD_{p_kp_l}A_{ij}(E^{mi}\xi_m)(E^{nj}\xi_n)(E_{k\alpha}\eta_\alpha)(E_{l\beta}\eta_\beta) \]
 So by A3w, $D_{q_iq_j}A^*_{kl}\xi_i\xi_j\eta_k\eta_l \geq 0$ if $\xi\cdot\eta = 0$. And this is of course A3w$^*$. 
\end{proof}

\begin{lemma}\label{lem:a:mp}
  Assume $u$ is $C^2$ and solves
  \begin{align*}
    0 &\leq a^{ij}D_{ij}u + b^iD_iu+cu \text{ in }B_r(0),\\
    u &= 0 \text{ on }\partial B_r,
  \end{align*}
  where $a^{ij},b^i,c^i$ are measurable functions on $B_r$ and $a^{ij}(x)$ is positive definite for each $x \in B_r$.   Let $a^{ii} = \textrm{trace}(a^{ij})$ and $c= c^+-c^-$ where $c^+ := \text{max}(c,0)$, $c^- = \max\{-c,0\}$.  If
  \begin{equation}
    \label{eq:a:r-choice}
       r \leq \text{min}\left\{\inf_{B_r} \frac{a^{ii}}{2|b|},\inf_{B_r} \sqrt{\frac{a^{ii}}{2 c^+}}\right\},
  \end{equation}
  then $u \leq 0$ in $B_r$. 
\end{lemma}
\begin{proof}
  Set $v = u + K(|x|^2-r^2)/2$ for $K$ to be chosen. Compute
  \[ Lv:=a^{ij}D_{ij}v+b_iD_iv - c^-v \geq-c^+u + Ka^{ii}- K|b||x|. \]
  Choose $K = 2 \sup u\sup (c^+/ a^{ii})$. If \eqref{eq:a:r-choice} is satisfied then $|x||b| \leq a^{ii}/2$. So $Lv \geq 0$ and $v$ satisfies the classical maximum principle. Thus
  \[ 0 \geq \sup v \geq \sup u  -Kr^2/2.  \]
  Then by \eqref{eq:a:r-choice}  $Kr^2/2 \leq \sup u/2$ and we obtain the result. 
  \end{proof}

\begin{lemma}
  \label{lem:a:big-vec} Assume $u\in C^0([0,\epsilon])$ satisfies $u(t) = o(t)$ as $t\rightarrow 0$. For some $\delta>0$ there exists a function $R:(0,\delta)\rightarrow \mathbf{R}$ such that  $R(h) \rightarrow \infty$ as $h \rightarrow0$ and $hR(h) \in \{t ; u(x) < h\}$.
\end{lemma}
\begin{proof}
  Assume $u(t_0) > 0$ for some $t_0 \in [0,\epsilon]$, if no such $t_0$ exists we're done. Set $\delta = u(t_0)$. For $h \in (0,\delta)$ define $\tau_h := \inf\{\tau;u(\tau) = h/2\}$.  Put $R(h)= \tau_h/h$.
  Clearly $u\big(hR(h)\big) < h$, so we need to show $R(h)\rightarrow \infty$ as $h \rightarrow 0$. If not there is a sequence of $h_n \rightarrow 0$ with $R(h_n) \leq C$. Then
  \[ \frac{u(\tau_{h_n})}{\tau_{h_n}} = \frac{h_n}{2\tau_{h_n}} = \frac{1}{2R(h_n)} \geq \frac{1}{2C}, \]
  which is a contradiction as $\tau_{h_n}\rightarrow 0$. 
\end{proof}

\begin{lemma}\label{lem:a:gqq}
Let $g$ be a generating function satisfying A3w. Then the statement of $g$-quasiconvexity, inequality \eqref{eq:sc:gqq} holds. 
\end{lemma}
\begin{proof}
  With all quantities as in that statement set
  \begin{align*}
    h(\theta) &= g(x_\theta,y_1,g^*(x_0,y_1,u_0)) - g(x_\theta,y_0,g^*(x_0,y_0,u_0))\\
         &= g(x_\theta,y_1,z_1)-g(x_\theta,y_0,z_0).
  \end{align*}
  Our goal is to show
  \begin{equation}
    \label{eq:a:h-goal}
       h(\theta) \leq C\theta [h(1)]_{+}.
  \end{equation}
  Note for
  \begin{align*}
    &p_0 := g_x(x_0,y_0,z_0)&&p_1 := g_x(x_0,y_1,z_1),
  \end{align*}
  we have
  \begin{align*}
    &y_0=Y(x_0,u_0,p_0)&&y_1=Y(x_0,u_0,p_1),\\
    &z_0=Z(x_0,u_0,p_0)&&z_1=Z(x_0,u_0,p_1).
  \end{align*}
  Now repeat the calculations in Lemma \ref{lem:g:maindiffineq} up to inequality \eqref{eq:g:need-app}, though with $u =  g(x_\theta,y_1,z_1)$. Note in the subsequent step of that lemma no Taylor series in $u$ is needed so we conclude
  \begin{equation}
    \label{eq:a:h-ineq}
       h''(\theta) \geq -K|h'(\theta)|.
  \end{equation}
Now the result follows from \cite[Corollary 9.4]{GuillenKitagawa17}.
\end{proof}

\chapter{Parabolic generated Jacobian equations}
\label{cha:bound-parab-gener}

Parabolic flows are a standard technique to obtain solutions to generated Jacobian equations. Here we consider $T>0$ and $u:\overline{\Omega}\times[0,T] \rightarrow \mathbf{R}$ where $u(\cdot,t)$ is $g$-convex for $t \in [0,T]$ and $u$ solves
\begin{align}
 \label{eq:a:pgje} \tag{PGJE} \exp{(u_t)} &= \frac{f^*(Y(\cdot,u,Du))\det DY(\cdot,u,Du)}{f(\cdot)} \text{ in }\Omega \times (0,T),\\
\label{eq:a:p2bvp}  Yu(\Omega) &= \Omega^* \text{ for }t \in (0,T),\\
\label{eq:a:ic}  u(\cdot,0) &= u_0(\cdot) \text{ on }\Omega\times \{0\}.
\end{align}
If a solution converges, in the sense that $u_t(x,t) \rightarrow 0$ as $t \rightarrow \infty$, then we expect the resulting function to solve \eqref{eq:g:gje} subject to \eqref{eq:g:2bvp} (at least in some weak sense). This method has been carried out by Schn\"{u}rer and Smoczyk \cite{SchnurerSmoczyk03} for the Monge--Amp\`ere equation and Kitagawa \cite{Kitagawa12} for the optimal transport case. Both results use the maximum principle for parabolic equations and the lack of $u$ dependence in a critical way. The $u$ dependence in \eqref{eq:a:pgje} prevents the same techniques. In this appendix we introduce a convexity based technique that yields $\sup |u|$ estimates depending only on $u_0$ and the constant from the A5 condition. More precisely we show $u$ always intersects $u_0$.

In this chapter of the appendix we assume $g$ satisfies $A3w,A4w,A5$. We also assume a uniformly $g$-convex $u_0$, satisfying $Yu(\Omega) = \Omega^*$ and $u(\Omega) \pm K_0\text{diam}(\Omega) \subset J$ where $K_0$ is from A5, has been given. In general though these arguments take place on $\Omega_\delta,\Omega_{\delta,\epsilon}^*$ using the function $u_\epsilon$ from Theorem \ref{thm:r1:hom-start} and then we send $\epsilon,\delta\rightarrow 0$. We use the convention from parabolic equations that for $C^{k,\alpha}(\overline{\Omega} \times [0,T])$ the time derivatives are weighted with a factor of $2$ \cite{Lieberman96}. For example $C^{2,\alpha}(\overline{\Omega} \times [0,T])$ is the space of functions whose second space derivatives are $\alpha$ holder continuous and whose first time derivatives are $\alpha/2$ holder continuous.

Our estimates independent of $T$ are based on convexity. Specifically an averaging property enforced by \eqref{eq:a:pgje}, \eqref{eq:a:p2bvp}, \eqref{eq:a:ic} and the parabolic equation solved by the $g^*$-transform. We begin with the required properties of the $g^*$-transform. 

\begin{lemma}
  Let $u \in C^2(\overline{\Omega})$ be a $g$-convex solution of
  \begin{align*}
    f^*(Yu(\cdot)) \det DYu &= f(\cdot) \text{ on }\overline{\Omega}\\
    Yu(\Omega) &= \Omega^*,
  \end{align*}
  where $f,f^*$ satisfy $0 < \lambda \leq f,f^* \leq \Lambda < \infty$. 
  Then the $g^*$-transform of $u$, denoted $v$, satisfies $v \in C^2(\overline{\Omega^*})$ and 
  \begin{align}
 \label{eq:a:v-eq}   f(Xv(\cdot))\det DXv =& f^*(\cdot)  \text{ on }\overline{\Omega^*}\\
\label{eq:a:v-bv}    Xv(\Omega^*) &= \Omega.
  \end{align}
\end{lemma}
\begin{proof}
  As noted in Lemma \ref{lem:g:gstarsubdiff} if $y \in Yu(x)$ then
  \[ v(y) = g^*(x,y,u(x)).\]
  Moreover because $u$ is $C^1$ and strictly $g$-convex, $v$ is $C^1$ and strictly $g^*$-convex. Next because $Yu$ is a $C^1$ diffeomorphism from $\Omega$ to $\Omega^*$ it's inverse $(Yu)^{-1}$ is well defined and
  \begin{equation}
    \label{eq:a:v1}
       v(y) = g^*\big(Yu^{-1}(y),y,u(Yu^{-1}(y))\big).
  \end{equation}
  An explicit calculation employing \eqref{eq:g:deriv-equiv} yields
  \begin{equation}
    \label{eq:a:v2}
       Dv(y) = g^*_y\big(Yu^{-1}(y),y,u(Yu^{-1}(y))\big), 
  \end{equation}
  so that $v \in C^2(\overline{\Omega^*})$. Let $y$ be given and $x = Yu^{-1}(y)$. Then \eqref{eq:a:v1} and \eqref{eq:a:v2} are precisely the equations defining $Xv(y) = X(y,v(y),Dv(y))$ and we see $Xv(y) = Yu^{-1}(y)$. Using this, and that $Yu$ is a diffeomorphism, \eqref{eq:a:v-eq} and \eqref{eq:a:v-bv} follow.  
\end{proof}

We also need to consider the time evolution of the $g^*$-transform.
\begin{lemma}\label{lem:a:vt}
  Let $u\in C^2(\overline{\Omega} \times [0,T])$ satisfy  that $u(\cdot,t)$ is uniformly $g$-convex for  $t \in [0,T]$. Assume also $Yu(\Omega,t) = \Omega^*$ for  $t \in [0,T]$. Define the $g^*$-transform by
  \begin{align}
    \label{eq:3}
    v(y,t) &:\overline{\Omega^*} \times [0,T] \rightarrow \mathbf{R}\\
  \label{eq:a:v-time-def}  v(y,t) &= \sup_{x \in \Omega}g^*(x,y,u(x,t)).
  \end{align}
  Let $y \in \Omega$ and $x = Xv(y)$ then
  \[ D_tv(y,t) = g^*_u(x,y,u(x))D_tu(x,t). \]
\end{lemma}
\begin{proof}
  Fix  $y \in \Omega$ and let $x = Xv(y,t)$. By the previous lemma
  \[v(y,t) = g^*(x,y,u(x,t)). \]
  Differentiating with respect to $t$ yields
  \begin{align*}
    D_tv(y,t) &= g^*_x(x,y,u(x,t))\frac{\partial x}{\partial t} + g^*_u(x,y,u(x,t))[Du\frac{\partial x}{\partial t} + u_t].
  \end{align*}
  The result follows because, by \eqref{eq:g:deriv-equiv},
  \[ g^*_x(x,y,u(x,t))\frac{\partial  x}{\partial  t} +g^*_u(x,y,u(x,t))Du\frac{\partial  x}{\partial t} = 0. \]
\end{proof}
As a corollary we obtain by the parabolic equation solved by the $g^*$-transform of a solution of \eqref{eq:a:pgje}, \eqref{eq:a:p2bvp} and \eqref{eq:a:ic}.

\begin{corollary}\label{cor:a:v-eq}
  Let $u:\in C^2(\overline{\Omega}\times[0,T])$. Assume $u$ solves \eqref{eq:a:pgje}, \eqref{eq:a:p2bvp} and \eqref{eq:a:ic} and let $v$ be defined by \eqref{eq:a:v-time-def}. Then $v$ solves
  \begin{align}
    \label{eq:4}
    \exp{[-g_z(Xv,\cdot,u(Xv))v_t(\cdot)]} &= \frac{f(Xv)\det DXv(\cdot)}{f^*(\cdot)} \text{ in } \overline{\Omega^*}\times[0,T]\\
    Xv(\Omega^*) &= \Omega.
  \end{align}
\end{corollary}

Finally we obtain $\sup |u|$ bounds.

\begin{theorem}
  Assume $u \in C^2(\overline{\Omega} \times [0,T])$ solves \eqref{eq:a:pgje}, \eqref{eq:a:p2bvp} and \eqref{eq:a:ic}. Then $\sup_{\overline{\Omega} \times [0,T]} |u| \leq C$ for a constant $C$ depending only on $u_0,K_0$. 
\end{theorem}
\begin{proof}
  Via condition A5 it suffices to show at each time $t$, $u(\cdot,t)$ intersects $u_0$. We achieve this by showing an inequality both ways.\\
  \\
  \textit{Step 1. We do not have $u > u_0$ on $\Omega$} \\
For a contradiction we assume there is $t \in [0,T]$ such that $u(x,t) > u_0(x)$ for all $x \in \Omega$. We have
\begin{align*}
  \int_{\Omega}f \ dx &< \int_{\Omega}f \exp\Big(\frac{u(x,t)-u_0(x)}{t}\Big) \ dx\\
                      &= \int_{\Omega}f \exp\Big(\frac{1}{t}\int_{0}^t\frac{\partial  u}{\partial t} d \tau\Big) \ dx.
\end{align*}
An application of Jensen's inequality (which preserve the above \textit{strict} inequality), then invoking PPJE yields
\begin{align*}
  \int_{\Omega}f \ dx &< \frac{1}{t}\int_{0}^t \int_{\Omega}f \exp\Big(\frac{\partial  u}{\partial  t} \Big) \ dx \ d\tau\\
                      &= \frac{1}{t}\int_{0}^t \int_{\Omega} f^*(Yu)\det DYu \ dx \ d\tau\\
                      &=  \frac{1}{t}\int_{0}^t \int_{\Omega^*}f^* \ dy \ d\tau\\
  &= \int_{\Omega}f(x) \ dx .
\end{align*}
In the last line we've used mass balance to obtain our contradiction.\\
\\
\textit{Step 2. We do not have $u < u_0$ on $\Omega$} \\
  This time for a contradiction we assume there is $t \in [0,T]$ such that $u(x,t) < u_0(x)$ for all $x \in \Omega$. We obtain for every $x \in \Omega$
  \begin{align*}
    1 &< \exp\Big(\frac{u_0(x)-u(x,t)}{t}\Big)\\
      &= \exp\Big(\frac{1}{t}\int_{0}^t-\frac{\partial u}{\partial t}(x,\tau) \ d \tau\Big).    
  \end{align*}
  Jensen's, and Lemma \ref{lem:a:vt} yields
  \begin{align}
   \nonumber 1 &< \frac{1}{t}\int_0^t \exp\Big(-\frac{\partial u}{\partial t}(x,\tau)\Big) \ d \tau\\
  \label{eq:a:important}  &= \frac{1}{t}\int_0^t \exp\Big(-g_z \frac{\partial v}{\partial t}(y,\tau)\Big) \ d \tau,
  \end{align}
  for $y=Y(x,u(x),Du(x))$. Since this holds for every $x \in \Omega$, the boundary condition implies \eqref{eq:a:important} holds for every $y \in \Omega^*$. Hence using Corollary \ref{cor:a:v-eq} we have
  \begin{align*}
    \int_{\Omega}f^*(y) \ dy &< \int_{\Omega^*}f^*(y) \frac{1}{t}\int_0^t \exp\Big(-g_z \frac{\partial v}{\partial t}(y,\tau)\Big) \ d \tau \ dy\\
                               &=\frac{1}{t}\int_0^t\int_{\Omega^*}f^*(y)  \exp\Big(-g_z \frac{\partial v}{\partial t}(y,\tau)\Big) \ d y \ d \tau\\
                               &= \frac{1}{t}\int_0^t\int_{\Omega^*} f(Xv)\det DXv \ d y \ d \tau\\
                               &=  \frac{1}{t}\int_0^t \int_{\Omega}f(x) \ dx \ d \tau\\
    &= \int_{\Omega^*}f^*(y) \ dy .
  \end{align*}
  This contradiction concludes step 2. We conclude for each $t$ there is a point $x_t \in \Omega$ with $u(x_t,t) = u_0(x_t)$. 
\end{proof}

Combined with A5 we have $\Vert u \Vert _{C^1}$ estimates (with respect to the space derivatives), that are independent of $T$. In particular the lower order coefficients of the linearized parabolic equation are bounded. Then by the standard parabolic maximum principle \cite[Theorem 2.2]{LSU1968} applied to the linearized equation for $u_t$ we obtain an estimate of the form $\sup |u_t| \leq Ce^{CT}$. Here $C$ depends only on $u_0,g,\Omega,\Omega^*$. Then the $C^2$ estimates for space derivatives are obtained as in Chapter \ref{chap:r2} with the extension to parabolic equations exactly as in  Kitagawa's work on parabolic optimal transport \cite{Kitagawa12}. The major difference is, unlike the optimal transport case, the $|u_t|$ estimate depending on $T$ means our $C^2$ estimates also have $T$ dependence. Thus using the method of continuity we can obtain the existence of $C^2(\overline{\Omega} \times [0,T])$ for any $T > 0$. It is an open problem whether such a solution converges to a solution of the elliptic problem. The answer depends on finding an estimate on $\sup |u_t|$ which is independent of $T$.

\clearpage{}

\bibliographystyle{plain}
\bibliography{thesis.bib}

\begin{thebibliography}{10}

\bibitem{AbedinGutierrez17}
Farhan Abedin and Cristian~E. Guti\'{e}rrez.
\newblock An iterative method for generated {J}acobian equations.
\newblock {\em Calc. Var. Partial Differential Equations}, 56(4):Paper No. 101,
  14, 2017.

\bibitem{Aleksandrov58}
A.~D. Aleksandrov.
\newblock Dirichlet's problem for the equation $\det || z_{ij} || =$
  $\phi(z_1,\dots,z_n,z,x_1,\dots,x_n)$. {I}.
\newblock {\em Vestnik Leningrad. Univ. Ser. Mat. Meh. Astr.}, 13(1):5--24,
  1958.

\bibitem{Aleksandrov42}
A.~Alexandroff.
\newblock Existence and uniqueness of a convex surface with a given integral
  curvature.
\newblock {\em C. R. (Doklady) Acad. Sci. URSS (N.S.)}, 35:131--134, 1942.

\bibitem{Aleksandrov42a}
A.~Alexandroff.
\newblock Smoothness of the convex surface of bounded {G}aussian curvature.
\newblock {\em C. R. (Doklady) Acad. Sci. URSS (N.S.)}, 36:195--199, 1942.

\bibitem{Ampere1819}
M.~Amp{\`e}re.
\newblock {\em M{\'e}moire contenant l'application de la th{\'e}orie:
  expos{\'e}e dans le XVII. cahier du Journal de l'Ecole polytechnique, a
  l'int{\'e}gration des {\'e}quations aux diff{\'e}rentielles partielles du
  premier et du second order}.
\newblock Ecole polytechnique, 1819.

\bibitem{Bakelman57}
I.~Ya. Bakelman.
\newblock Generalized solutions of {M}onge-{A}mp\`ere equations.
\newblock {\em Dokl. Akad. Nauk SSSR (N.S.)}, 114:1143--1145, 1957.

\bibitem{Bakelman58}
I.~Ya. Bakelman.
\newblock On the theory of {M}onge-{A}mp\`ere's equations.
\newblock {\em Vestnik Leningrad. Univ. Ser. Mat. Meh. Astr.}, 13(1):25--38,
  1958.

\bibitem{Bakelman94}
Ilya~J. Bakelman.
\newblock {\em Convex analysis and nonlinear geometric elliptic equations}.
\newblock Springer-Verlag, Berlin, 1994.
\newblock With an obituary for the author by William Rundell, Edited by Steven
  D. Taliaferro.

\bibitem{Brenier1991}
Yann Brenier.
\newblock Polar factorization and monotone rearrangement of vector-valued
  functions.
\newblock {\em Comm. Pure Appl. Math.}, 44(4):375--417, 1991.

\bibitem{CNS84}
L.~Caffarelli, L.~Nirenberg, and J.~Spruck.
\newblock The {D}irichlet problem for nonlinear second-order elliptic
  equations. {I}. {M}onge-{A}mp\`ere equation.
\newblock {\em Comm. Pure Appl. Math.}, 37(3):369--402, 1984.

\bibitem{Caffarelli}
L.~A. Caffarelli.
\newblock A localization property of viscosity solutions to the
  {M}onge-{A}mp\`ere equation and their strict convexity.
\newblock {\em Ann. of Math. (2)}, 131(1):129--134, 1990.

\bibitem{CaffarelliOliker08}
L.~A. Caffarelli and V.~I. Oliker.
\newblock Weak solutions of one inverse problem in geometric optics.
\newblock {\em J. Math. Sci. (N.Y.)}, 154(1):39--49, 2008.
\newblock Problems in mathematical analysis. No. 37.

\bibitem{Caffarelli92a}
Luis~A. Caffarelli.
\newblock The regularity of mappings with a convex potential.
\newblock {\em J. Amer. Math. Soc.}, 5(1):99--104, 1992.

\bibitem{Caffarelli96}
Luis~A. Caffarelli.
\newblock Boundary regularity of maps with convex potentials. {II}.
\newblock {\em Ann. of Math. (2)}, 144(3):453--496, 1996.

\bibitem{ChenWang16}
Shibing Chen and Xu-Jia Wang.
\newblock Strict convexity and {$C^{1,\alpha}$} regularity of potential
  functions in optimal transportation under condition {A}3w.
\newblock {\em J. Differential Equations}, 2016.

\bibitem{ChengYau76}
Shiu~Yuen Cheng and Shing~Tung Yau.
\newblock On the regularity of the solution of the {$n$}-dimensional
  {M}inkowski problem.
\newblock {\em Comm. Pure Appl. Math.}, 29(5):495--516, 1976.

\bibitem{ChengYau77}
Shiu~Yuen Cheng and Shing~Tung Yau.
\newblock On the regularity of the {M}onge-{A}mp\`ere equation {$\det(\partial
  \sp{2}u/\partial x\sb{i}\partial sx\sb{j})=F(x,u)$}.
\newblock {\em Comm. Pure Appl. Math.}, 30(1):41--68, 1977.

\bibitem{EvansGariepy15}
Lawrence~C. Evans and Ronald~F. Gariepy.
\newblock {\em Measure theory and fine properties of functions}.
\newblock Textbooks in Mathematics. CRC Press, Boca Raton, FL, revised edition,
  2015.

\bibitem{Figalli17}
Alessio Figalli.
\newblock {\em The {M}onge-{A}mp\`ere equation and its applications}.
\newblock Zurich Lectures in Advanced Mathematics. European Mathematical
  Society (EMS), Z\"{u}rich, 2017.

\bibitem{FKM13}
Alessio Figalli, Young-Heon Kim, and Robert~J. McCann.
\newblock H\"{o}lder continuity and injectivity of optimal maps.
\newblock {\em Arch. Ration. Mech. Anal.}, 209(3):747--795, 2013.

\bibitem{FigalliLoeper09}
Alessio Figalli and Gr\'{e}goire Loeper.
\newblock {$C^1$} regularity of solutions of the {M}onge-{A}mp\`ere equation
  for optimal transport in dimension two.
\newblock {\em Calc. Var. Partial Differential Equations}, 2009.

\bibitem{GMT21}
Anatole Gallouët, Quentin Merigot, and Boris Thibert.
\newblock A damped newton algorithm for generated jacobian equations, 2021.

\bibitem{GangboMcCann96}
Wilfrid Gangbo and Robert~J. McCann.
\newblock The geometry of optimal transportation.
\newblock {\em Acta Math.}, 177(2):113--161, 1996.

\bibitem{GilbargTrudinger01}
David Gilbarg and Neil~S. Trudinger.
\newblock {\em Elliptic partial differential equations of second order}.
\newblock Classics in Mathematics. Springer-Verlag, Berlin, 2001.
\newblock Reprint of the 1998 edition.

\bibitem{Guillen19}
Nestor Guillen.
\newblock A primer on generated {J}acobian equations: geometry, optics,
  economics.
\newblock {\em Notices Amer. Math. Soc.}, 66(9):1401--1411, 2019.

\bibitem{GuillenKitagawa15}
Nestor Guillen and Jun Kitagawa.
\newblock On the local geometry of maps with {$c$}-convex potentials.
\newblock {\em Calc. Var. Partial Differential Equations}, 52(1-2):345--387,
  2015.

\bibitem{GuillenKitagawa17}
Nestor Guillen and Jun Kitagawa.
\newblock Pointwise estimates and regularity in geometric optics and other
  generated {J}acobian equations.
\newblock {\em Comm. Pure Appl. Math.}, 70(6):1146--1220, 2017.

\bibitem{Gutierrez16}
Cristian~E. Guti\'{e}rrez.
\newblock {\em The {M}onge-{A}mp\`ere equation}, volume~89 of {\em Progress in
  Nonlinear Differential Equations and their Applications}.
\newblock Birkh\"{a}user/Springer, [Cham], 2016.
\newblock Second edition.

\bibitem{HanLin97}
Qing Han and Fanghua Lin.
\newblock {\em Elliptic partial differential equations}, volume~1 of {\em
  Courant Lecture Notes in Mathematics}.
\newblock New York University, 1997.

\bibitem{Ivochkina80}
N.~M. Ivochkina.
\newblock A priori estimate of {$|u|\sb{C_2(\overline\Omega)}$} of convex
  solutions of the {D}irichlet problem for the {M}onge-{A}mp\`ere equation.
\newblock {\em Zap. Nauchn. Sem. Leningrad. Otdel. Mat. Inst. Steklov. (LOMI)},
  96:69--79, 306, 1980.

\bibitem{Ivochkina83}
N.~M. Ivochkina.
\newblock Classical solvability of the {D}irichlet problem for the
  {M}onge-{A}mp\`ere equation.
\newblock {\em Zap. Nauchn. Sem. Leningrad. Otdel. Mat. Inst. Steklov. (LOMI)},
  131:72--79, 1983.

\bibitem{Jeong21}
Seonghyeon Jeong.
\newblock Local {H}\"{o}lder regularity of solutions to generated {J}acobian
  equations.
\newblock {\em Pure Appl. Anal.}, 3(1):163--188, 2021.

\bibitem{Jhaveri17}
Yash Jhaveri.
\newblock Partial regularity of solutions to the second boundary value problem
  for generated {J}acobian equations.
\newblock {\em Methods Appl. Anal.}, 24(4):445--475, 2017.

\bibitem{JiangTrudinger14}
Feida Jiang and Neil~S. Trudinger.
\newblock On {P}ogorelov estimates in optimal transportation and geometric
  optics.
\newblock {\em Bull. Math. Sci.}, 4(3):407--431, 2014.

\bibitem{JiangTrudinger17}
Feida Jiang and Neil~S. Trudinger.
\newblock Oblique boundary value problems for augmented {H}essian equations
  {II}.
\newblock {\em Nonlinear Anal.}, 154:148--173, 2017.

\bibitem{JiangTrudinger18}
Feida Jiang and Neil~S. Trudinger.
\newblock On the second boundary value problem for {M}onge-{A}mp\`ere type
  equations and optics.
\newblock {\em Arch. Ration. Mech. Anal.}, 229(2):547--567, 2018.

\bibitem{JTY14}
Feida Jiang, Neil~S. Trudinger, and Xiao-Ping Yang.
\newblock On the {D}irichlet problem for {M}onge-{A}mp\`ere type equations.
\newblock {\em Calc. Var. Partial Differential Equations}, 49(3-4):1223--1236,
  2014.

\bibitem{Kantorovich2004}
L.~V. Kantorovich.
\newblock On a problem of {M}onge.
\newblock {\em Zap. Nauchn. Sem. S.-Peterburg. Otdel. Mat. Inst. Steklov.
  (POMI)}, 312(Teor. Predst. Din. Sist. Komb. i Algoritm. Metody. 11):15--16,
  2004.

\bibitem{Kantorovitch1942}
L.~Kantorovitch.
\newblock On the translocation of masses.
\newblock {\em C. R. (Doklady) Acad. Sci. URSS (N.S.)}, 37:199--201, 1942.

\bibitem{KimMcCann10}
Young-Heon Kim and Robert~J. McCann.
\newblock Continuity, curvature, and the general covariance of optimal
  transportation.
\newblock {\em J. Eur. Math. Soc. (JEMS)}, 12(4):1009--1040, 2010.

\bibitem{Kitagawa12}
Jun Kitagawa.
\newblock A parabolic flow toward solutions of the optimal transportation
  problem on domains with boundary.
\newblock {\em J. Reine Angew. Math.}, 672:127--160, 2012.

\bibitem{Krylov83}
N.~V. Krylov.
\newblock Boundedly inhomogeneous elliptic and parabolic equations in a domain.
\newblock {\em Izv. Akad. Nauk SSSR Ser. Mat.}, 47(1):75--108, 1983.

\bibitem{LSU1968}
O.~A. Lady\v{z}enskaja, V.~A. Solonnikov, and N.~N. Ural\'ceva.
\newblock {\em Linear and quasilinear equations of parabolic type}.
\newblock American Mathematical Society, Providence, R.I., 1968.

\bibitem{LLN17}
Yanyan Li, Jiakun Liu, and Luc Nguyen.
\newblock A degree theory for second order nonlinear elliptic operators with
  nonlinear oblique boundary conditions.
\newblock {\em J. Fixed Point Theory Appl.}, 19(1):853--876, 2017.

\bibitem{Lieberman96}
Gary~M. Lieberman.
\newblock {\em Second order parabolic differential equations}.
\newblock World Scientific Publishing Co., Inc., River Edge, NJ, 1996.

\bibitem{LiebermanTrudinger86}
Gary~M. Lieberman and Neil~S. Trudinger.
\newblock Nonlinear oblique boundary value problems for nonlinear elliptic
  equations.
\newblock {\em Trans. Amer. Math. Soc.}, 295(2):509--546, 1986.

\bibitem{Liu09}
Jiakun Liu.
\newblock H\"{o}lder regularity of optimal mappings in optimal transportation.
\newblock {\em Calc. Var. Partial Differential Equations}, 34(4):435--451,
  2009.

\bibitem{LiuTrudinger14}
Jiakun Liu and Neil~S. Trudinger.
\newblock On {P}ogorelov estimates for {M}onge-{A}mp\`ere type equations.
\newblock {\em Discrete Contin. Dyn. Syst.}, 28(3):1121--1135, 2010.

\bibitem{LiuTrudinger16}
Jiakun Liu and Neil~S. Trudinger.
\newblock On the classical solvability of near field reflector problems.
\newblock {\em Discrete Contin. Dyn. Syst.}, 36(2):895--916, 2016.

\bibitem{LTW10}
Jiakun Liu, Neil~S. Trudinger, and Xu-Jia Wang.
\newblock Interior {$C^{2,\alpha}$} regularity for potential functions in
  optimal transportation.
\newblock {\em Comm. Partial Differential Equations}, 35(1):165--184, 2010.

\bibitem{LTW15}
Jiakun Liu, Neil~S. Trudinger, and Xu-Jia Wang.
\newblock On asymptotic behaviour and {$W^{2,p}$} regularity of potentials in
  optimal transportation.
\newblock {\em Arch. Ration. Mech. Anal.}, 215(3):867--905, 2015.

\bibitem{LiuWang15}
Jiakun Liu and Xu-Jia Wang.
\newblock Interior a priori estimates for the {M}onge-{A}mp\`ere equation.
\newblock In {\em Surveys in differential geometry 2014. {R}egularity and
  evolution of nonlinear equations}, volume~19 of {\em Surv. Differ. Geom.},
  pages 151--177. Int. Press, Somerville, MA, 2015.

\bibitem{LoeperTrudinger21}
G.~Loeper and N.~S. Trudinger.
\newblock On the convexity theory of generating functions, 2021.

\bibitem{Loeper09}
Gr\'{e}goire Loeper.
\newblock On the regularity of solutions of optimal transportation problems.
\newblock {\em Acta Math.}, 202(2):241--283, 2009.

\bibitem{MTW05}
Xi-Nan Ma, Neil~S. Trudinger, and Xu-Jia Wang.
\newblock Regularity of potential functions of the optimal transportation
  problem.
\newblock {\em Arch. Ration. Mech. Anal.}, 177(2):151--183, 2005.

\bibitem{McCann95}
Robert~J. McCann.
\newblock Existence and uniqueness of monotone measure-preserving maps.
\newblock {\em Duke Math. J.}, 80(2):309--323, 1995.

\bibitem{Minkowski03}
Hermann Minkowski.
\newblock Volumen und {O}berfl\"{a}che.
\newblock {\em Math. Ann.}, 57(4):447--495, 1903.

\bibitem{Monge1781}
G.~Monge.
\newblock {\em M{\'e}moire sur la th{\'e}orie des d{\'e}blais et des remblais}.
\newblock De l'Imprimerie Royale, 1781.

\bibitem{Monge1784}
G.~Monge.
\newblock {\em M{\'e}moire sur le calcul int{\'e}gral des {\'e}quations aux
  diff{\'e}rences partielles}.
\newblock M\`emoires de l\'Acad\`emie Sciences, 1784.

\bibitem{Pogorelov71a}
A.~V. Pogorelov.
\newblock The {D}irichlet problem for the multidimensional analogue of the
  {M}onge-{A}mp\`ere equation.
\newblock {\em Dokl. Akad. Nauk SSSR}, 201:790--793, 1971.

\bibitem{Pogorelov71}
A.~V. Pogorelov.
\newblock The regularity of the generalized solutions of the equation
  {$\det(\partial \sp{2}u/\partial x\sp{i}\partial x\sp{j})=$} {$\varphi
  (x\sp{1},\,x\sp{2},\dots, x\sp{n})>0$}.
\newblock {\em Dokl. Akad. Nauk SSSR}, 200:534--537, 1971.

\bibitem{Pogorelov78}
Aleksey Vasil\textquotesingle~yevich Pogorelov.
\newblock {\em The {M}inkowski multidimensional problem}.
\newblock V. H. Winston \& Sons, Washington, D.C.; Halsted Press [John Wiley
  \&\ Sons], New York-Toronto-London, 1978.

\bibitem{Rankin2020}
Cale Rankin.
\newblock Distinct solutions to generated {J}acobian equations cannot
  intersect.
\newblock {\em Bull. Aust. Math. Soc.}, 102(3):462--470, 2020.

\bibitem{Rankin2020b}
Cale Rankin.
\newblock Strict convexity and {$C^1$} regularity of solutions to generated
  {J}acobian equations in dimension two.
\newblock {\em Calc. Var. Partial Differential Equations}, 60, 2021.

\bibitem{Rankin2021}
Cale Rankin.
\newblock Strict $g$-convexity for generated {J}acobian equations with
  applications to global regularity.
\newblock {\em ArXiv:2111.00448}, 2021.

\bibitem{Ruschendorf91}
Ludger R\"{u}schendorf.
\newblock Fr\'{e}chet-bounds and their applications.
\newblock In {\em Advances in probability distributions with given marginals
  ({R}ome, 1990)}, volume~67 of {\em Math. Appl.}, pages 151--187. Kluwer Acad.
  Publ., Dordrecht, 1991.

\bibitem{SchnurerSmoczyk03}
Oliver~C. Schn\"{u}rer and Knut Smoczyk.
\newblock Neumann and second boundary value problems for {H}essian and {G}auss
  curvature flows.
\newblock {\em Ann. Inst. H. Poincar\'{e} Anal. Non Lin\'{e}aire}, 20, 2003.

\bibitem{Trudinger95}
Neil~S. Trudinger.
\newblock On the {D}irichlet problem for {H}essian equations.
\newblock {\em Acta Math.}, 1995.

\bibitem{Trudinger14}
Neil~S. Trudinger.
\newblock On the local theory of prescribed {J}acobian equations.
\newblock {\em Discrete Contin. Dyn. Syst.}, 34(4):1663--1681, 2014.

\bibitem{Trudinger20}
Neil~S. Trudinger.
\newblock On the local theory of prescribed {J}acobian equations revisited.
\newblock {\em Math. Eng.}, 3(6):Paper No. 048, 17, 2021.

\bibitem{TrudingerWang08}
Neil~S. Trudinger and Xu-Jia Wang.
\newblock The {M}onge-{A}mp\`ere equation and its geometric applications.
\newblock In {\em Handbook of geometric analysis. {N}o. 1}, pages 467--524.
  Int. Press, 2008.

\bibitem{TrudingerWang09a}
Neil~S. Trudinger and Xu-Jia Wang.
\newblock On strict convexity and continuous differentiability of potential
  functions in optimal transportation.
\newblock {\em Arch. Ration. Mech. Anal.}, 192(3):403--418, 2009.

\bibitem{TrudingerWang09}
Neil~S. Trudinger and Xu-Jia Wang.
\newblock On the second boundary value problem for {M}onge-{A}mp\`ere type
  equations and optimal transportation.
\newblock {\em Ann. Sc. Norm. Super. Pisa Cl. Sci. (5)}, 2009.

\bibitem{Urbas1997}
John Urbas.
\newblock On the second boundary value problem for equations of
  {M}onge-{A}mp\`ere type.
\newblock {\em J. Reine Angew. Math.}, 487:115--124, 1997.

\bibitem{Vetois15}
J\'{e}r\^{o}me V\'{e}tois.
\newblock Continuity and injectivity of optimal maps.
\newblock {\em Calc. Var. Partial Differential Equations}, 52(3-4):587--607,
  2015.

\bibitem{Villani09}
C\'{e}dric Villani.
\newblock {\em Optimal transport}.
\newblock Grundlehren der Mathematischen Wissenschaften. Springer-Verlag, 2009.

\bibitem{Wang96}
Xu-Jia Wang.
\newblock On the design of a reflector antenna.
\newblock {\em Inverse Problems}, 12(3):351--375, 1996.

\bibitem{Zhang18}
Kelvin~Shuangjian Zhang.
\newblock {\em Existence, uniqueness, concavity and geometry of the
  monopolist's problem facing consumers with nonlinear price preferences.}
\newblock PhD thesis, University of Toronto, 2018.

\end{thebibliography}
\printindex[notation]
\printindex

\end{document}